\documentclass{amsart}%
\usepackage{amsfonts}%
\usepackage{amsmath}%
\usepackage{amssymb}%
\usepackage{graphicx}
\usepackage{amsthm}
\usepackage{setspace}
\usepackage{color}
\usepackage{newclude}
\usepackage[top = 1in, bottom = 1in, left = 1.2in, right =
1.2in]{geometry}

\providecommand{\U}[1]{\protect\rule{.1in}{.1in}}
\newtheorem{theorem}{Theorem}[section]
\theoremstyle{plain}

\newtheorem{definition}{Definition}[section]

\newtheorem{lemma}{Lemma}[section]

\newtheorem{proposition}{Proposition}[section]
\newtheorem{remark}{Remark}[section]

\numberwithin{equation}{section}

\begin{document}
\title[]{Nonuniqueness for the kinetic Fokker-Planck equation
with inelastic boundary conditions.}


\author{Hyung Ju Hwang}
\address{Department of Mathematics, Pohang University of Science and
Technology, Pohang, GyungBuk 790-784, Republic of Korea}
\email{hjhwang@postech.ac.kr}
\thanks{}

\author{Juhi Jang}
\address{Department of Mathematics, University of Southern California,
Los Angeles, CA 90089, USA and Korea Institute for Advanced Study, Seoul, Korea}
\email{juhijang@usc.edu}
\curraddr{}
\thanks{}

\author{Juan J. L. Vel\'{a}zquez}
\address{Institute of Applied Mathematics, University of Bonn, Endenicher
Allee 60, 53115 Bonn, Germany.}
\email{velazquez@iam.uni-bonn.de}
\curraddr{}
\thanks{}

\date{today}



\begin{abstract}
We describe the structure of solutions of the kinetic Fokker-Planck
equations in domains with boundaries near the singular set in one-space
dimension. We study in particular the behaviour of the solutions of this
equation for inelastic boundary conditions which are characterized by means
of a coefficient $r$ describing the amount of energy lost in the collisions
of the particles with the boundaries of the domain. A peculiar feature of
this problem is the onset of a critical exponent $r_{c}$ which follows from
the analysis of McKean (cf. \cite{McK}) of the properties of the stochastic
process associated to the Fokker-Planck equation under consideration. In
this paper, we prove rigorously that the solutions of the considered problem
are nonunique if $r<r_{c}$ and unique if $r_{c}<r\leq 1.$ In particular,
this nonuniqueness explains the different behaviours found in the physics
literature for numerical simulations of the stochastic differential equation
associated to the Fokker-Planck equation. In the proof of the results of this paper 
we use several asymptotic formulas and computations in the companion paper \cite{HVJ2}. 
\end{abstract}

\maketitle


%
%
%



\section{Introduction}




A general feature of several kinetic equations is the fact that their
solutions in a domain $\Omega $ with boundaries cannot be infinitely
differentiable at the points of the phase space $\left( x,v\right) $ for
which $x\in \partial \Omega $ and $v$ is tangent to $\partial \Omega ,$ even
if the initial data are arbitrarily smooth. This set of points is 
denoted as the singular set. This property of the solutions of kinetic equations
was found by Guo in his study of the Vlasov-Poisson system (cf. \cite{G1}%
).

In the case of the one-species Vlasov-Poisson system in 3-dimensional
bounded domains with specular reflection boundary conditions, the convexity
of the boundary plays a crucial role in determining the regularity of
solutions (cf. \cite{Hw}, \cite{HV}). If the boundary of the domain is flat
or convex, then one can show that a trajectory on the singular set is
separated from regular one in the interior, so that classical solutions
exist globally in time and are $C^{1}$ in phase space variables except on
the singular set.
However, if a domain is nonconvex, then even in the Vlasov equation without
interactions between the particles and with inflow boundary condition, a
classical $C^{1}$ solution fails to exist in general (cf. \cite{G0}).
Similar singular behaviours near the singular set have been found in the
Boltzmann equation (cf. \cite{Kim}).

This paper is devoted to the analysis of 
the
kinetic Fokker-Planck equation, in which some interesting phenomena related
to the presence of the singular set appear. More precisely, we will be
concerned with the analysis of the following equation:%
\begin{equation}
\partial _{t}P+v\partial _{x}P=\partial _{vv}P\ \ ,\ \ P=P\left(
x,v,t\right) ,\ \ \ x>0,\ \ v\in \mathbb{R}\ \ ,\ \ t>0  \label{S0E1}
\end{equation}%
with the so-called inelastic boundary condition:%
\begin{equation}
P\left( 0,-v,t\right) =r^{2}P\left( 0,rv,t\right) \ \ ,\ \ v>0\ \ ,\ \ t>0
\label{S0E2}
\end{equation}%
where $0<r<\infty .$ We are particularly interested in the case in which $%
r\leq 1.$ 
The physical meaning of (\ref{S0E2}) is that the
particles arriving to the wall $\left\{ x=0\right\}$ with a velocity $-v$
bounce back to the domain $\left\{ x>0\right\}$ with a new velocity $v.$
Notice that in the case $r=1$ the particles bounce elastically, but for $%
r<1,$ the collisions are inelastic and the particles lose a fraction of
their energy in the collisions. We will solve (\ref{S0E1}), (\ref{S0E2})
with the initial condition:%
\begin{equation}
P\left( x,v,0\right) =P_{0}\left( x,v\right)  \label{S0E3}
\end{equation}%
where $P_{0}$ is a probability measure in $\left( x,v\right) \in \mathbb{R}%
^{+}\times \mathbb{R}$.

The singular set reduces in the case of the problem (\ref{S0E1}), (\ref{S0E2}%
) to the point $\left( x,v\right) =\left( 0,0\right) .$

In this paper, 
we will prove that for $%
r<r_{c}$ where 
\begin{equation}
r_{c}=\exp \left( -\frac{\pi }{\sqrt{3}}\right)  \label{S1E3}
\end{equation}%
there exist different measure valued solutions of the boundary value problem
(\ref{S0E1})-(\ref{S0E3}). These solutions differ in their asymptotics near
the singular point $\left( x,v\right) =\left( 0,0\right).$ Roughly
speaking, these solutions differ in the precise interaction between the
singular set and the domain $\left\{ \left( x,v\right) \neq \left(
0,0\right) ,\ x\geq 0,\ v\in \mathbb{R}\right\} .$ The resulting solutions
are $C^{\infty }$ functions in $\left\{ \left( x,v\right) \neq \left(
0,0\right) \right\} ,$ but near the singular point it is only possible to
obtain uniform estimates in suitable H\"{o}lder norms. Some of the solutions
obtained satisfy $\int_{\left\{ \left( 0,0\right) \right\} }P\left(
dxdv,t\right) >0$ for $t>0$ while for some of them $\int_{\left\{ \left(
0,0\right) \right\} }P\left( dxdv,t\right) =0.$

Different solutions of the problem (\ref{S0E1})-(\ref{S0E3}) obtained in
this paper have different natural interpretations in terms of particles
evolving by means of a differential stochastic equation. This has been
discussed in detail in the companion paper \cite{HVJ2}. In that paper,
the asymptotic behaviour of the solutions constructed in this paper was
described by means of formal computations. In fact, several formulas
derived in \cite{HVJ2} will be used in this paper.

The exponent $r_{c}$ in (\ref{S1E3}) was first found by McKean (cf. \cite%
{McK}) and it can be interpreted as the critical value of the recovery
coefficient in the inelastic collisions taking place at the wall $\left\{
x=0\right\} $ below which Brownian particles described by the
Uhlenbeck-Ornstein processes reach the point $\left( x,v\right) =\left(
0,0\right) $ in finite time. See a more detailed discussion in \cite%
{HVJ2}. Intuitively, if $r>r_{c}$ we do not need additional information
to define a unique evolution for the solutions of (\ref{S0E1})-(\ref{S0E3})
because the trajectories associated to the Uhlenbeck-Ornstein processes do
not reach the singular point in finite time, while for $r<r_{c}$ the
singular point is reached in finite time and additional information is
required to determine uniquely the behaviour of the solutions of (\ref{S0E1}%
)-(\ref{S0E3}). This information is encoded in the different asymptotic
behaviours which can be imposed to the solutions of (\ref{S0E1}) near the
singular set (cf. \cite{HVJ2}). The dynamics of Brownian particles
bouncing inelastically at one wall has been also considered in the physical
literature. We refer also to \cite{HVJ2} for details about the relation
between the results of the physical literature and the results in 
\cite{HVJ2} and this paper.


The main result of this paper can be formulated in the following form:


\begin{theorem}
\label{ThIntro}Let  $\mathcal{X}$ be the space of functions $%
C^{\infty }\left( \mathbb{R}^{+}\times \mathbb{R}\times \left( 0,\infty
\right) \right) \cap C\left( \left[ 0,\infty \right) :\mathcal{M}_{+}\left( 
\mathbb{R}^{+}\times \mathbb{R}\right) \right) \cap C\left( \overline{%
\mathbb{R}^{+}\times \mathbb{R}}\times \left( 0,\infty \right) \right) .$
Suppose that $0<r<r_{c}.$ For any Radon measure $P_{0}\in \mathcal{M}%
_{+}\left( \mathbb{R}^{+}\times \mathbb{R}\right) ,\ $such that $\int_{%
\mathbb{R}^{+}\times \mathbb{R}}P_{0}<\infty ,$ there exist infinitely many
different solutions of the problem (\ref{S0E1})-(\ref{S0E3}). The solutions $%
P=P\left( x,v,t\right) \in \mathcal{X}$ satisfy (\ref{S0E1}), (\ref{S0E2})
in classical sense and (\ref{S0E3}) in the sense of distributions. If $%
r_{c}<r\leq 1$ there exists a unique weak solution of (\ref{S0E1})-(\ref%
{S0E3}) with initial data $P_{0}.$
\end{theorem}

\begin{remark}
We have denoted as $\mathcal{M}_{+}\left( B\right) $ the set of Radon
measures in a given Borel set $B\subset \mathbb{R}^{n},\ n\geq 1.$
\end{remark}

Precise definitions of solutions of (\ref{S0E1})-(\ref{S0E3}) as well as
more detailed information about the solutions described in Theorem \ref%
{ThIntro} will be given later (cf. Proposition 4.1 in Section 4, as well as the Definitions 4.2-4.5). 
Roughly speaking, the solutions obtained for $0<r<r_{c}$ differ in the boundary
condition imposed at the singular point $\left( x,v\right) =\left(
0,0\right) .$ Intuitively, in the case of the solution satisfying $\int_{%
\mathbb{R}^{+}\times \mathbb{R}}P\left( \cdot ,\cdot ,t\right) =\int_{%
\mathbb{R}^{+}\times \mathbb{R}}P_{0}$ we assume that the particles arriving
to the origin continue instantaneously their motion. In all the other
solutions, particles arriving to the singular point remain trapped there for
all later times or during some characteristic time before resuming their
motion in the region $\mathbb{R}^{+}\times \mathbb{R}$. The existence and
uniqueness of solutions for $r>r_{c}$ will be also proved in this paper,
although in this case we do not need any additional information about the
solutions near the singular set.

\bigskip 

The main tool used in the proof of the results of this paper is the
classical Hille-Yosida Theorem. Most of the technical difficulties in this
paper arise from the fact that we need to prove that the operators involved
in our problem satisfy the assumptions which allow to apply the Hille-Yosida
Theorem. In particular, the application of this Theorem requires to prove
the solvability of some problems of Partial Differential Equations which 
encode information about the behaviour of the solutions of (\ref%
{S0E1})-(\ref{S0E3}) near the singular set $\left\{ \left( x,v\right)
=\left( 0,0\right) \right\} .$ Due to the fact that we are interested in
Measure Valued solutions of (\ref{S0E1})-(\ref{S0E3}) it will be convenient
to work instead with the solutions of some suitable adjoint problems of (\ref%
{S0E1})-(\ref{S0E3}) which encode the different types of boundary conditions
mentioned above if $r<r_{c}.$ In order to apply Hille-Yosida's Theorem we
need to prove the solvability of some elliptic problems. This will be made
by means of a suitable generalization of the classical Perron's method which
is able to deal with the singular behaviour of the solutions near the
singular point. It is well known that Perron's method allows us to obtain
the solution of some elliptic equations as the supremum of the subsolutions
associated to such equations. However, the application of such ideas to this
problem yields to several technical difficulties. The main ones are the
following ones. Since the operator $D_{v}^{2}+vD_{x}$ contains only the
first order derivative in $x$ instead of the second order ones, it is
possible to construct subsolutions for this problem which are discontinuous
and have discontinuities along subsets of lines $\left\{ x=x_{0}\right\} $
for $x_{0}\in \mathbb{R}$. Moreover, such discontinuous subsolutions arise
naturally applying Perron's method to this class of equations. This fact
will yield several technical difficulties which will be considered in
Section 3. On the other hand the nonuniqueness of the solutions of (\ref%
{S0E1})-(\ref{S0E3}) will be related to the different asymptotic behaviours
of the solutions near the singular point. The application of Perron's method
to the problem under consideration will require to study the properties of
suitably defined sub and supersolutions with different behaviours near the
singular point.

We remark that the Fokker-Planck equation for a general class of boundary conditions but different from the ones considered in this paper has been considered in \cite{Nier}, and the connection between the methods of that paper and this one has been made in the companion paper \cite{HVJ2}. In the case of the absorbing boundary condition, H\"{o}lder continuity of solutions of the Fokker-Planck equation near the singular set has been shown by the authors in \cite{HJJ}, \cite{HJJ2}, \cite{HVJ}.






The plan of this paper is the following. In section \ref{AdjProblems}, we present the adjoint problem 
of the original problem  (\ref{S0E1}), (\ref{S0E2}) with boundary conditions.
The well-possedness of the adjoint problems will be obtained in Section \ref{WellPosed}. 
This section is the most technical part of the paper and it requires some detailed study of the
asymptotics of the solutions of the adjoint problem near the singular point $%
\left( x,v\right) =\left( 0,0\right).$ In Section \ref{weakSolDef}, we prove the existence of measure valued solutions for 
the problem (\ref{S0E1})-(\ref{S0E3}), which gives nonuniqueness for subcritical case and uniqueness for supercritical case. 


%
%
%

\section{Adjoint problem of our original equation \label{AdjProblems}}

\subsection{Definition of some differential operators.}

Our goal now is to obtain suitable adjoint operators for (\ref{S0E1}), (\ref{S0E2}) 
with one of the 
trapping, nontrapping, and partially trapping boundary
conditions, which will be defined in Section \ref{Operators}. These
operators will act over the class of continuous functions on a topological
space $X.$ The action of the operators will be given by a differential
operator $\mathcal{L}$ with suitable boundary conditions. In this Section we
give the precise definitions of $X$ and $\mathcal{L}$.

\begin{definition}
\label{DefX}We define as $X_{0}$ the set obtained identifying the subset of
points $\left[ 0,\infty \right) \times \left( -\infty ,\infty \right) $ such
that $\left( x,v\right) =\left( 0,-v\right) $ and $\left( x,v\right) =\left(
0,rv\right) $, $v>0.$ We then define $X=X_{0}\cup \left\{ \infty \right\} ,$
and we endowed it with the natural topology inherited from $\mathbb{R}^{2}$
complemented with the following set of neighbourhoods of the point $\infty :$%
\begin{equation*}
\mathcal{O}_{M}=\left\{ \left( x,v\right) \in \left[ 0,\infty \right) \times
\left( -\infty ,\infty \right) :v<-M\text{ or }v>rM\ \text{or\ }x>M\right\}
\ \ ,\ \ M>0
\end{equation*}
\end{definition}

The set $X$ is a topological compact set. The continuous functions of this
space can be identified with the bounded continuous functions $\varphi $ in $%
\left[ 0,\infty \right) \times \left( -\infty ,\infty \right) $ such that
\begin{equation}
\varphi \left( 0,-v\right) =\varphi \left( 0,rv\right) ,\ v>0  \label{comp1}
\end{equation}%
and such that the limit $\lim_{x+\left\vert v\right\vert \rightarrow \infty
}\varphi \left( x,v\right) $ exists. We will denote this set of functions as
$C\left( X\right) .$ 

We denote as $\mathcal{U}$ the set:\
\[
\mathcal{U}=\left\{ \left( x,v\right) :x\geq 0,\ v\in \mathbb{R},\ \left(
x,v\right) \neq \left( 0,0\right) \right\} 
\]
Notice that a function $\varphi \in C\left( X\right) $
defines a function in $C\left( \mathcal{U}\right) $ satisfying (\ref{comp1}%
). We will use the same notation $\varphi $ to refer to both functions for
the sake of simplicity.

We need to introduce some local directionality in a neighbourhood of each
point of $X\diagdown \left\{ \left( 0,0\right) ,\infty \right\} $ in order
to compute directional limits.

\begin{definition}
\label{defLeft}Given two points $\left( x_{1},v_{1}\right) ,\left(
x_{2},v_{2}\right) \in X\diagdown \left[ \left\{ \left( x,0\right) :x\geq
0\right\} \cup \left\{ \infty \right\} \right] .$ We will say that $\left(
x_{1},v_{1}\right) $ is to the left of $\left( x_{2},v_{2}\right) $ and we
will write $\left( x_{1},v_{1}\right) \ll \left( x_{2},v_{2}\right) $ if $x_{1}%
\operatorname{sgn}\left( v_{1}\right) <x_{2}\operatorname{sgn}\left( v_{2}\right) .$\
\end{definition}

\bigskip Notice that the previous definition just means that $\left(
x_{1},v_{1}\right) \ll \left( x_{2},v_{2}\right) $ in one of the following
three cases: (i)\ If $v_{1}>0$\ and $v_{2}>0$\ we have $x_{1}<x_{2}.$ (ii) $%
v_{1}<0<v_{2}.$ (iii) If $v_{1}<0$\ and $v_{2}<0$\ we have $x_{1}>x_{2}.$

\begin{definition}
\label{LeftNei}Given a point $\left( x_{0},v_{0}\right) \in X\setminus
\left\{ \left( x,0\right) :x\geq 0\right\} \cup \left\{ \infty \right\} $
and a neighbourhood $\mathcal{B}$ of $\left( x_{0},v_{0}\right) $ in the
topological space $X$ we define the left neighbourhood $\mathcal{B}%
^{-}\left( x_{0},v_{0}\right) $ as:%
\begin{equation*}
\mathcal{B}^{-}\left( x_{0},v_{0}\right) =\left\{ \left( x,v\right) \in
\mathcal{B}:\left( x,v\right) \ll \left( x_{0},v_{0}\right) \right\}
\end{equation*}
\end{definition}

\begin{remark}
Notice that the neighbourhood $\mathcal{B}$ must be understood as a
neighbourhood of the topological space $X.$ In particular, if $\left(
x_{0},v_{0}\right) =\left( 0,v_{0}\right) $ any neighbourhood of $\left(
x_{0},v_{0}\right) $ contains points $\left( x,v\right) $ with $v>0$ and $%
v<0.$
\end{remark}

\begin{definition}
\label{LineLeft}We will say that $L\subset X$ is a vertical segment if it
has the form $\left\{ \left( x_{0},v\right) \in X: v\in \left( \alpha ,\beta \right) \right\} $ 
for some $x_{0}\geq 0,\ \alpha ,\beta \in \mathbb{R}$
with $\alpha \cdot \beta >0,$ $\alpha <\beta .$ Given two vertical segments $%
L_{1},\ L_{2}$ we will say that $L_{1}$ is to the left of $L_{2}$ if for any
$\left( x_{k},v_{k}\right) \in L_{k}$ with $k=1,2$ we have\ $\left(
x_{1},v_{1}\right) \ll \left( x_{2},v_{2}\right) .$ We will then write $%
L_{1}\ll L_{2}$. We will say that $L\subset X$ is a horizontal segment if it
has the form $\left\{ \left( x,v_{0}\right) \in X:x\in \left( \alpha ,\beta
\right) \right\} $ for some $v_{0}\in \mathbb{R}$ $\alpha ,\beta \in \mathbb{%
R}$ with $0\leq \alpha <\beta .$
\end{definition}

It will be convenient to define a suitable concept of convergence in the set
of segments.

\begin{definition}
\label{Distance}Given two segments $L_{1},\ L_{2}$ in $X$ we define a
distance between them as:%
\begin{equation}
\operatorname{dist}{}_{H}\left( L_{1},L_{2}\right) =\inf \left\{ \operatorname{dist}%
\left( \left( x,v\right) ,L_{2}\right) :\left( x,v\right) \in L_{1}\right\}
\label{distH}
\end{equation}
\end{definition}

The action of the operators $\Omega _{\sigma }$ on smooth functions
supported in $\mathcal{U}$ is given by the differential
operator
\begin{equation}
\mathcal{L}=D_{v}^{2}+vD_{x}  \label{diffOp}
\end{equation}%
where the operator $\mathcal{L}$ will be defined in the sense of
distributions as indicated later. Nevertheless, the operators $\Omega
_{\sigma }$ will differ in the different cases (cf. trapping, nontrapping, partially trapping) 
in its domain of definition which will encode
the asymptotic behaviour of $\varphi $ near the singular point $\left(
x,v\right) =\left( 0,0\right) .$

We endow the set $C\left( X\right) $ with a Banach space structure using the
norm:
\begin{equation}
\left\Vert \varphi \right\Vert =\sup_{\left( x,v\right) \in X}|\varphi (x,v)|
\label{Norm}
\end{equation}

We need to impose suitable regularity and compatibility conditions in the
class of test functions in order to take into account the compatibility
conditions imposed in $X.$ Let $V$ be an open subset of $\ \mathcal{U}$. We
will consider functions $\zeta \in C\left( \bar{V}\right) $ satisfying:%
\begin{equation}
\zeta \left( 0,-v\right) =r^{2}\zeta \left( 0,rv\right) \ ;\ v>0\ ,\ \text{%
if }\left( 0,-v\right) ,\left( 0,rv\right) \in \bar{V}  \label{comp1t}
\end{equation}%
\begin{equation}
\text{There exist }\zeta _{x},\zeta _{vv}\in C\left( \bar{V}\right)
\label{comp3}
\end{equation}%
\begin{equation}
\zeta _{x}\left( 0,-v\right) =r^{2}\zeta _{x}\left( 0,rv\right) \ ;\ v>0%
\text{ if }\left( 0,-v\right) ,\left( 0,rv\right) \in \bar{V}\
\label{comp4}
\end{equation}%
\begin{equation}
\operatorname{supp}\left( \zeta \right) \cap \left[ \left\{ x+\left\vert
v\right\vert \geq R\right\} \cup \left\{ 0,0\right\} \right] =\varnothing \
\ \text{for some }R>0  \label{comp2}
\end{equation}

We can define a set of functions
\begin{equation}
\mathcal{F}\left( V\right) \mathcal{=}\left\{ \zeta \in C\left( \bar{V}%
\right) :\text{(\ref{comp1t}),\ (\ref{comp3}), (\ref{comp4}), (\ref{comp2})
hold}\right\}  \label{defF}
\end{equation}

We now define the action of the operator $\mathcal{L}$ in a subset of $%
C\left( X\right) .$

\begin{definition}
\label{LadjDef}Suppose that $W$ is any open subset of $X.$ Given $\varphi
\in C\left( W\right) ,$ we will say that $\mathcal{L}\varphi $ is defined if
there exists $w\in \mathcal{M} \left( W\right) $ such that for any $\zeta \in \mathcal{F%
}\left( \mathcal{U}\cap W\right) $ we have:%
\begin{equation}
\int_{\mathcal{U}\cap W}\varphi \mathcal{L}^{\ast }\left( \zeta \right)
dxdv=\int_{\mathcal{U}\cap W}w\zeta dxdv  \label{Ladjoint}
\end{equation}%
where $\mathcal{L}^{\ast }=D_{v}^{2}-vD_{x}.$ We will then write $w=\mathcal{%
L}\varphi .$ Given $V\subset \mathcal{U}$, we will say that $\mathcal{L}%
\varphi $ is defined in $V$ if there exists $w\in C\left( V\right) $
satisfying (\ref{comp1}) such that for any $\zeta \in \mathcal{F}\left(
\mathcal{U}\right) $ such that $\operatorname{supp}\left( \zeta \right) \cap
\left( \partial V\right) =\varnothing ,$ (\ref{Ladjoint}) holds.
\end{definition}

\subsection{Regularity properties of the solutions: Hypoellipticity.\label%
{HypSection}}

In this Subsection we will formulate some regularity results associated to PDEs
containing the operator $\mathcal{L}$ which will be used repeatedly in the
following.

One of the key features of the solutions of equations like (\ref{S0E1}) is
that their solutions are smooth in spite of the fact that they contain only
second derivatives in the direction of $v$ and not in the direction of $x$.
This smoothness can be proved for a large class of initial data using the
fundamental solution associated to this equation in the whole space, which
was first computed by Kolmogorov (cf. \cite{K}).

On the other hand, hypoellipticity properties for equations with the form (%
\ref{S0E1}) have been known for a long time and they have been formulated in
different functional spaces in several papers (cf. \cite{BCLP}, \cite{HBook}%
, \cite{H}, \cite{Lu}, \cite{Pa}, \cite{RS}). Hypoellipticity properties for
the evolution problem associated to the equation (\ref{S0E1}) with absorbing
boundary conditions\ have been proved in \cite{HVJ}. In \ this paper, we
will need regularizing effects only for the stationary problem. We collect
in this Section some regularity results used in this paper. We first need
some notation to denote a portion of the boundaries of a class of domains
which will be used repeatedly in this paper. More precisely, we will
restrict ourselves to the following class of domains contained in $[0,\infty
)\times (-\infty ,\infty )\smallsetminus \left\{ \left( 0,0\right) \right\}
: $

\begin{definition}
\label{admissible}$\Xi \subset \lbrack 0,\infty )\times (-\infty ,\infty
)\smallsetminus \left\{ \left( 0,0\right) \right\} $ is an admissible domain
if it has one of the following forms:

(a) It is a cartesian product $\left( x_{1},x_{2}\right) \times \left(
v_{1},v_{2}\right) $ with $\left( x_{1},x_{2}\right) \subset \left( 0,\infty
\right) ,\ \left( v_{1},v_{2}\right) \subset \mathbb{R}$ and $0\notin \left[ v_{1},v_{2}\right] .$

(b) It has the form $\left( 0,x_{2}\right) \times \left( v_{1},v_{2}\right)
\diagdown \left( 0,x_{1}\right) \times \left( \bar{v}_{1},\bar{v}_{2}\right)
$ with $0<x_{1}<x_{2},\ v_{1}<\bar{v}_{1}<0<\bar{v}_{2}<v_{2}.$

In the case (a), we will denote\ as $\partial _{a}\Xi $ and term as
admissible boundary the subset of the boundary $\partial \Xi $ defined by
means of:%
\begin{eqnarray*}
\partial _{a}\Xi &=&\partial _{a,h}\Xi \cup \partial _{a,v}\Xi \\
\partial _{a,h}\Xi &=&\left( \left[ x_{1},x_{2}\right] \times \left\{
v_{1}\right\} \right) \cup \left( \left[ x_{1},x_{2}\right] \times \left\{
v_{2}\right\} \right) \\
\partial _{a,v}\Xi &=&\left( \left( \left[ v_{1},v_{2}\right] \cap \left(
0,\infty \right) \right) \times \left\{ x_{2}\right\} \right) \cup \left(
\left( \left[ v_{1},v_{2}\right] \cap \left( -\infty ,0\right) \right)
\times \left\{ x_{1}\right\} \right) .
\end{eqnarray*}%
and we will denote as $\partial _{a}^{\ast }\Xi $ the adjoint admissible
boundary given by:%
\begin{eqnarray*}
\partial _{a}^{\ast }\Xi &=&\partial _{a,h}\Xi \cup \partial _{a,v}^{\ast
}\Xi \\
\partial _{a,v}^{\ast }\Xi &=&\left( \left( \left[ v_{1},v_{2}\right] \cap
\left( -\infty ,0\right) \right) \times \left\{ x_{2}\right\} \right) \cup
\left( \left( \left[ v_{1},v_{2}\right] \cap \left( 0,\infty \right) \right)
\times \left\{ x_{1}\right\} \right) .
\end{eqnarray*}

In the case (b)\ we will denote as $\partial _{a}\Xi $ the set:%
\begin{eqnarray*}
\partial _{a}\Xi &=&\partial _{a,h}\Xi \cup \partial _{a,v}\Xi \\
\partial _{a,h}\Xi &=&\left( \left[ 0,x_{2}\right] \times \left\{
v_{1}\right\} \right) \cup \left( \left[ 0,x_{2}\right] \times \left\{
v_{2}\right\} \right) \cup \left( \left[ 0,x_{1}\right] \times \left\{ \bar{v%
}_{1}\right\} \right) \cup \left( \left[ 0,x_{1}\right] \times \left\{ \bar{v%
}_{2}\right\} \right) \\
\partial _{a,v}\Xi &=&\left( \left[ \bar{v}_{1},0\right] \times \left\{
x_{1}\right\} \right) \cup \left( \left[ 0,v_{2}\right] \times \left\{
x_{2}\right\} \right)
\end{eqnarray*}%
and the adjoint admissible boundary $\partial _{a}^{\ast }\Xi $ as:%
\begin{eqnarray*}
\partial _{a}^{\ast }\Xi &=&\partial _{a,h}\Xi \cup \partial _{a,v}^{\ast
}\Xi \\
\partial _{a,v}^{\ast }\Xi &=&\left( \left[ 0, \bar{v}_{2}\right] \times \left\{
x_{1}\right\} \right) \cup \left( \left[ v_{1},0\right] \times \left\{
x_{2}\right\} \right) .
\end{eqnarray*}
\end{definition}

\begin{figure}[htbp]
\centering \includegraphics[width=0.9\linewidth]{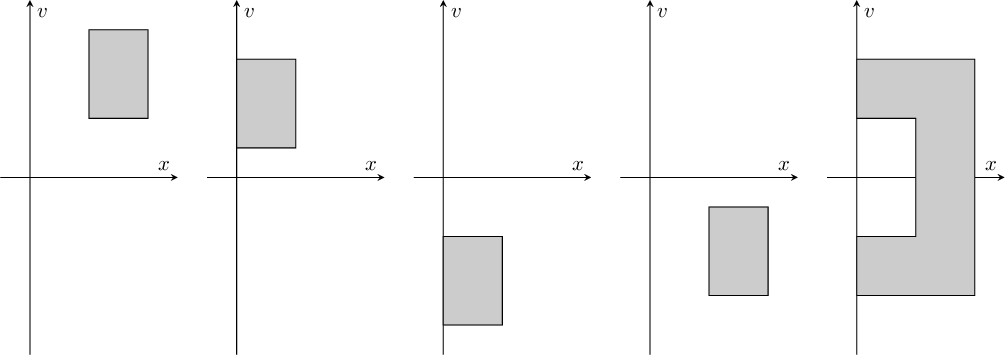}
\caption{Admissible domains in Definition \protect\ref{admissible} }
\label{fig:domain-1}
\end{figure}


\begin{theorem}
\label{Regul}Let $W\subset X$ be one admissible domain in the sense of
Definition \ref{admissible} and $\varphi \in C\left( W\right) .$ Suppose
that $\mathcal{L}\varphi =w\in C\left( W\right) $ in the sense of Definition %
\ref{LadjDef}. Then $D_{v}\varphi \in $ $C\left( W\right) .$ Moreover, if $%
W_{\delta }$ denotes the subset of $W$ such that $\operatorname{dist}\left(
\left( x,v\right) ,\partial _{a}W\right) \geq \delta >0$ we have:%
\begin{equation*}
\left\Vert D_{v}\varphi \right\Vert _{L^{\infty }\left( W_{\delta }\right)
}\leq C_{\delta ,W}\left\Vert \varphi \right\Vert _{L^{\infty }\left(
W\right) }
\end{equation*}%
where $C_{\delta ,W}>0$ depends only on $W,$ $\delta .$
\end{theorem}

\begin{proof}
The regularity at the interior points of $\mathcal{U}\cap W$ is a
consequence of the hypoellipticity results in \cite{HBook}. The regularity
in the sets $\left\{ x=0,\ v>0\right\} \cap W$ follows from classical
parabolic theory \cite{Fr}, assuming that $x$ is the time variable. We then obtain the
desired continuity in $\left\{ x=0,\ v<0\right\} \cap W$ using (\ref{comp1}%
). The uniform regularity in $\left\{ v<0\right\} \cap W$ then follows from
parabolic theory \cite{Fr}.
\end{proof}

We define a family of domains as follows:

\begin{definition}
\label{domainsR}For any given $r>0,$ we define:%
\begin{equation}
\mathcal{R}_{\delta}=\left\{ (x,v):0\leq x^{\frac{1}{3}}\leq 
\delta ,\ -\delta \leq v\leq r\delta \right\} \   \label{domainR}
\end{equation}%
\end{definition}

We will use also the following result for the solutions of the evolution
problem (\ref{S0E1}), (\ref{S0E2}).

\begin{theorem}
\label{RegTime}Let $W=\mathcal{R}_{2}\diagdown \overline{\mathcal{R}_{\frac{1%
}{2}}}$ with $\mathcal{R}_{\delta }$ as in Definition \ref{domainsR}.
Suppose that $P,P_{t},P_{xx},P_{v}\in C\left( \left[ 0,1\right] \times
W\right) \ $and that $P$ solves (\ref{S0E1}), (\ref{S0E2}) in $\left(
0,1\right) \times W.$ Then:%
\begin{equation*}
\left\Vert D_{v}P\right\Vert _{L^{\infty }\left( \left[ \frac{1}{2},1\right]
\times \left( \mathcal{R}_{\frac{3}{2}}\diagdown \overline{\mathcal{R}_{%
\frac{2}{3}}}\right) \right) }\leq C\left\Vert P\right\Vert _{L^{\infty
}\left( \left[ 0,1\right] \times W\right) }
\end{equation*}
\end{theorem}

\begin{proof}
It is just a consequence of the results in \cite{HBook}, \cite{RS} as well
as the boundary condition (\ref{S0E2}).
\end{proof}

We will use also the following hypoellipticity property.

\begin{theorem}
\label{Hypoell}Suppose that $W\subset X$ and $\varphi \in C\left( W\right) ,$
with $\left( 0,0\right) \notin W.$

(i) Suppose that $\mathcal{L}\varphi =w\in C\left( W\right) $ in the sense
of Definition \ref{LadjDef}. Then $D_{x}\varphi ,D_{v}\varphi
,D_{v}^{2}\varphi \in L_{loc}^{p}\left( W\right) $ and these norms can be
estimated by $\left\Vert w\right\Vert _{\infty }.$

(ii) Suppose that $\mathcal{L}\varphi =\nu \in \mathcal{M}\left( W\right) .$
Then the $L^{p}$ norm of $\varphi $ and $D_{v}\varphi $ in each open set, or
in any horizontal curve can be estimated by the sum of the $L^{1}$ norm of $%
\varphi $ and the $\mathcal{M}\left( W\right) $ norm of $\nu .$
\end{theorem}

\begin{proof}
The result (i) follows from \cite{RS}. The result (ii) can be proved by
duality, using (i) for the adjoint operator of $\mathcal{L}$.
\end{proof}


\subsection{Definition of the operators $\Omega _{\protect\sigma }.$ Domains $%
\mathcal{D}(\Omega _{\protect\sigma })$. \label{Operators}}

We now define some operators $\Omega _{\sigma }$ for the different boundary
conditions described in Section \ref{BVP}, 
where $%
\sigma $ is the subindex labelling each set of boundary conditions. In all
the cases the operator $\Omega _{\sigma }$ acts over continuous functions
defined on a compact topological space $X$ in Definition \ref{DefX}.

We will assume that the functions $\varphi \in \mathcal{D}(\Omega _{\sigma
}) $ have the following asymptotic behaviour near the singular set:%
\begin{equation}
\varphi \left( x,v\right) =\varphi \left( 0,0\right) +\mathcal{A}\left(
\varphi \right) F_{\beta }\left( x,v\right) +\psi \left( x,v\right) ,\ \
\lim_{R\rightarrow 0}\frac{\sup_{x+\left\vert v\right\vert ^{3}=R}\left\vert
\psi \left( x,v\right) \right\vert }{R^{\beta }}=0  \label{phi_decom}
\end{equation}%
where $\mathcal{A}\left( \varphi \right) \in \mathbb{R}$ and 
\begin{equation}
F_{\beta }\left( x,v\right) =x^{\beta }\Phi _{\beta }\left( y\right) \text{
with\ }\Phi _{\beta }\left( y\right) =U(-\beta ,\frac{2}{3};y)\ \ ,\ \ y=%
\frac{v^{3}}{9x}  \label{Fbeta}
\end{equation}
with the asymptotics of $\Phi _{\beta }\left( y\right) $ given by:\
\begin{eqnarray}
\Phi _{\beta }\left( y\right) &\sim &\left\vert y\right\vert ^{\beta }\text{
\ \ \ \ as }y\rightarrow \infty ,\   \label{S8E4} \\
\Phi _{\beta }\left( y\right) &\sim &K\left\vert y\right\vert ^{\beta }\text{
\ \ as }y\rightarrow -\infty  \label{S8E5}
\end{eqnarray}%
where $K=r^{3\beta }.$

\subsubsection{The case $r<r_{c}.$ Trapping boundary conditions.\label%
{Absorbing}}

In this case we will define $\Omega _{t,sub}\varphi $ as follows. We
consider the domain:%
\begin{equation}
\mathcal{D}(\Omega _{t,sub})=\{\varphi ,\mathcal{L}\varphi \in C\left(
X\right) :\varphi \text{ satisfies (\ref{phi_decom}),}\lim_{\left(
x,v\right) \rightarrow \left( 0,0\right) }\left( \mathcal{L}\varphi \right)
\left( x,v\right) =0\}  \label{S7E1}
\end{equation}

We then define:

\begin{eqnarray}
\left( \Omega _{t,sub}\varphi \right) \left( x,v\right) &=&\left( \mathcal{L}%
\varphi \right) \left( x,v\right) \ ,\ \ \text{if}\ \ \left( x,v\right) \neq
\left( 0,0\right) ,\infty  \label{S7E1a} \\
\left( \Omega _{t,sub}\varphi \right) \left( 0,0\right) &=&\lim_{\left(
x,v\right) \rightarrow \left( 0,0\right) }\left( \mathcal{L}\varphi \right)
\left( x,v\right) =0\ \ ,\ \ \left( \Omega _{t,sub}\varphi \right) \left(
\infty \right) =\left( \mathcal{L}\varphi \right) \left( \infty \right)
\notag
\end{eqnarray}%
with $\mathcal{L}\varphi $ as in Definition \ref{LadjDef}. We remark that in
this case as well as in the following three cases, we have that $\mathcal{L}%
\varphi \in C\left( X\right) $ for the functions $\varphi $ in the domains
and then the limit $\lim_{\left( x,v\right) \rightarrow \left( 0,0\right)
}\left( \mathcal{L}\varphi \right) \left( x,v\right) $ exists.

\subsubsection{The case $r<r_{c}.$ Nontrapping boundary conditions.\label%
{Reflecting}}

We define $\Omega _{nt,sub}\varphi $ by means of (\ref{diffOp}) in the
domain:%
\begin{equation}
\mathcal{D}(\Omega _{nt,sub})=\left\{ \varphi ,\mathcal{L}\varphi \in
C\left( X\right) :\varphi \text{ satisfies (\ref{phi_decom})\ and\ }\mathcal{%
A}\left( \varphi \right) =0\right\}  \label{S7E2}
\end{equation}

We then define:

\begin{eqnarray}
\left( \Omega _{nt,sub}\varphi \right) \left( x,v\right) &=&\left( \mathcal{L%
}\varphi \right) \left( x,v\right) \ ,\ \ \text{if}\ \ \left( x,v\right)
\neq \left( 0,0\right) ,\ \infty  \label{S7E2a} \\
\left( \Omega _{nt,sub}\varphi \right) \left( 0,0\right) &=&\lim_{\left(
x,v\right) \rightarrow \left( 0,0\right) }\left( \mathcal{L}\varphi \right)
\left( x,v\right) \ \ ,\ \ \left( \Omega _{nt,sub}\varphi \right) \left(
\infty \right) =\left( \mathcal{L}\varphi \right) \left( \infty \right) \
\notag
\end{eqnarray}%
with $\mathcal{L}\varphi $ as in Definition \ref{LadjDef}.

\subsubsection{The case $r<r_{c}.$ Partially trapping boundary conditions.
\label{Mixed}}

Given any $\mu _{\ast }>0,$ we define $\Omega _{pt,sub}\varphi $ by means of
(\ref{diffOp}) in the domain:\
\begin{equation}
\mathcal{D}(\Omega _{pt,sub})=\left\{ \varphi ,\mathcal{L}\varphi \in
C\left( X\right) :\varphi \text{ satisfies (\ref{phi_decom}), there exists\ }%
\lim_{\left( x,v\right) \rightarrow \left( 0,0\right) }\left( \mathcal{L}%
\varphi \right) \left( x,v\right) =\mu _{\ast }\left\vert C_{\ast
}\right\vert \mathcal{A}\left( \varphi \right) \ \right\}  \label{S7E3}
\end{equation}%
Here
\begin{equation}
C_{\ast }=\lim_{\delta \rightarrow 0}\int_{\partial \mathcal{R}_{\delta }}%
\left[ G_{\alpha }\left( n_{v}D_{v}F_{\beta }+n_{x}vF_{\beta }\right)
-F_{\beta }D_{v}G_{\alpha }n_{v}\right] ds  \label{W1E5}
\end{equation}%
where $\mathcal{R}_{\delta }$ is as in (\ref{domainR}), $F_{\beta}$ is as in (\ref{Fbeta}),
and $n=\left(
n_{x},n_{v}\right) $ is the normal vector to $\partial \mathcal{R}_{\delta }$
pointing towards $\mathcal{R}_{\delta }.$  
\begin{equation}
G_{\alpha }\left( x,v\right) =x^{\alpha }\Lambda _{\alpha }\left( \zeta
\right) \ ,\ \ \zeta =\frac{v}{\left( 9x\right) ^{\frac{1}{3}}}\ \ 
\label{Galpha}
\end{equation}%
with%
\begin{equation}
\Lambda _{\alpha }\left( \zeta \right) =U\left( -\alpha ,\frac{2}{3};-\zeta
^{3}\right) >0\ \ \text{for }\zeta \in \mathbb{R\ }  \label{LambdaU}
\end{equation}%
where we denote as $U(a,b;z)$ the classical Tricomi confluent hypergeometric
functions (cf. \cite{Ab}).

Note that from Proposition 4.3 in \cite{HVJ2}
\begin{equation}
\frac{C_{\ast }}{9^{\frac{2}{3}}}=\frac{\pi }{3}\left( \sin \left( \pi
\alpha \right) +\sqrt{3}\cos \left( \pi \alpha \right) \right) -2\cos \left(
\pi \left( \beta +\frac{1}{3}\right) \right) \log \left( r\right)
\label{W1E5a}
\end{equation}.

We will denote from now on $\lim_{\left( x,v\right) \rightarrow \left(
0,0\right) }\left( \mathcal{L}\varphi \right) \left( x,v\right) =\left(
\mathcal{L}\varphi \right) \left( 0,0\right) .$ We then define:%
\begin{eqnarray}
\left( \Omega _{pt,sub}\varphi \right) \left( x,v\right) &=&\left( \mathcal{L%
}\varphi \right) \left( x,v\right) \ ,\ \ \text{if}\ \ \left( x,v\right)
\neq \left( 0,0\right) ,\ \infty  \label{S7E3a} \\
\left( \Omega _{pt,sub}\varphi \right) \left( 0,0\right) &=&\left( \mathcal{L%
}\varphi \right) \left( 0,0\right) =\mu _{\ast }\left\vert C_{\ast
}\right\vert \mathcal{A}\left( \varphi \right) \ \ ,\ \ \left( \Omega
_{pt,sub}\varphi \right) \left( \infty \right) =\left( \mathcal{L}\varphi
\right) \left( \infty \right)  \notag
\end{eqnarray}

We will not make explicit the dependence of the operators $\Omega _{pt,sub}$
in $\mu _{\ast }$ for the sake of simplicity. Notice that trapping boundary
conditions reduce formally to the case $\mu _{\ast }=0$ and nontrapping
boundary conditions to the case $\mu _{\ast }=\infty .$

\subsubsection{The case $r>r_{c}$\label{Supercrit}}

In this case we only consider the case of nontrapping boundary conditions.
We then define $\Omega _{\sup }\varphi $ by means of (\ref{diffOp}) in the
domain:%
\begin{equation}
\mathcal{D}(\Omega _{sup})=\left\{ \varphi ,\mathcal{L}\varphi \in C\left(
X\right) :\text{ there exists\ }\lim_{\left( x,v\right) \rightarrow \left(
0,0\right) }\left( \mathcal{L}\varphi \right) \left( x,v\right) \right\}
\label{S7E4}
\end{equation}

Notice that in this case the condition (\ref{phi_decom}) does not make sense
if\ $\varphi \in C\left( X\right) $ because $\beta <0.$

We then define

\begin{eqnarray}
\left( \Omega _{sup}\varphi \right) \left( x,v\right) &=&\left( \mathcal{L}%
\varphi \right) \left( x,v\right) \ ,\ \ \text{if}\ \ \left( x,v\right) \neq
\left( 0,0\right) ,\ \infty  \label{S7E4a} \\
\left( \Omega _{sup}\varphi \right) \left( 0,0\right) &=&\lim_{\left(
x,v\right) \rightarrow \left( 0,0\right) }\left( \mathcal{L}\varphi \right)
\left( x,v\right) ,\ \left( \Omega _{sup}\varphi \right) \left( \infty
\right) =\left( \mathcal{L}\varphi \right) \left( \infty \right)  \notag
\end{eqnarray}

\subsection{Adjoint problems to be considered.\label{BVP}}

Given the problem (\ref{S0E1}), (\ref{S0E2}) with any of the
trapping, non-trapping, and mixed boundary conditions we define
the adjoint problem as:%
\begin{equation}
\varphi _{t}+\mathbb{A}\left( \varphi \right) =0\ ,\ \ t\in \left(
0,T\right) ,\ \varphi \left( \cdot ,T\right) =\varphi _{0}\left( \cdot
\right) \   \label{T2E1}
\end{equation}%
where $\varphi \left( \cdot, t\right) \in \mathcal{D}\left( \mathbb{A}\right)
$ for any $t\in \left( 0,T\right) .$ We change the time as $t\rightarrow
\left( T-t\right) $ in order to obtain a forward parabolic problem.
Therefore (\ref{T2E1}) becomes:%
\begin{equation}
\varphi _{t}-\mathbb{A}\left( \varphi \right) =0\ ,\ \ t\in \left(
0,T\right) ,\ \varphi \left( \cdot ,0\right) =\varphi _{0}\left( \cdot
\right)  \label{T2E2}
\end{equation}%
with $\varphi \left( \cdot, t\right) \in \mathcal{D}\left( \mathbb{A}\right) $
for any $t\in \left( 0,T\right) .$ We will say that (\ref{T2E2}) is the
adjoint problem of (\ref{S0E1}), (\ref{S0E2}) with the corresponding
boundary condition-trapping, nontrapping, partially trapping.
Notice that the characterizations of the adjoint operators $\mathbb{A}$
make it possible to reformulate the problem (\ref{T2E2}) as a PDE
problem with suitable boundary conditions along the line $\partial \mathcal{U}=\left\{
\left( x,v\right) =\left( 0,v\right) :v\in \mathbb{R}\text{ }\right\} $ as
well as near the singular point $\left( x,v\right) =\left( 0,0\right) .$ We
now describe in detail this set of PDE problems for the different cases for
further reference.

\subsubsection{The case $r<r_{c}$}

In this case we need to distinguish the cases of the three boundary
conditions-trapping, nontrapping, partially trapping. Using the definitions
of the operators $\Omega _{t,sub},\ \Omega _{nt,sub},\ \Omega _{pt,sub}$ in
Sections \ref{Absorbing}, \ref{Reflecting}, \ref{Mixed} we obtain the
following:

\bigskip

The adjoint problem of (\ref{S0E1}), (\ref{S0E2}) with the nontrapping boundary condition 
\begin{align}
\partial _{t}\varphi -v\partial _{x}\varphi -\partial _{vv}\varphi &
=0,\;\;x>0,\;\;v\in \mathbb{R\ }  \label{T2E8} \\
\varphi (0,-v,t)& =\varphi (0,rv,t),\;\;v>0  \label{T2E9}
\end{align}%
\begin{equation}
\varphi \left( x,v,t\right) =\varphi \left( 0,0,t\right) +\psi \left(
x,v,t\right) ,\ \ \lim_{R\rightarrow 0}\frac{\sup_{x+\left\vert v\right\vert
^{3}=R}\left\vert \psi \left( x,v,t\right) \right\vert }{R^{\beta }}=0,\
t\in \left( 0,T\right)  \label{T2E7}
\end{equation}%
which must be complemented with the initial condition:
\begin{equation}
\varphi \left( x,v,0\right) =\varphi _{0}\left( x,v\right) \   \label{T2E10}
\end{equation}

The adjoint problem of (\ref{S0E1}), (\ref{S0E2}) with the trapping boundary condition is:
\begin{align}
\partial _{t}\varphi -v\partial _{x}\varphi -\partial _{vv}\varphi &
=0,\;\;x>0,\;\;v\in \mathbb{R\ \ }  \label{T3E1} \\
\varphi (0,-v,t)& =\varphi (0,rv,t),\;\;v>0\   \label{T3E2}
\end{align}%
\begin{equation}
\lim_{\left( x,v\right) \rightarrow \left( 0,0\right) }\left( \mathcal{L}%
\varphi \right) \left( x,v,t\right) =0,\ t\in \left( 0,T\right)  \label{T3E3}
\end{equation}%
with $\mathcal{L}$ as in (\ref{diffOp}). A natural initial condition for
this problem is:
\begin{equation}
\varphi \left( x,v,0\right) =\varphi _{0}\left( x,v\right)  \label{T3E4}
\end{equation}

The adjoint problem of (\ref{S0E1}), (\ref{S0E2}) with the partially trapping boundary condition is:%
\begin{align}
\partial _{t}\varphi -v\partial _{x}\varphi -\partial _{vv}\varphi &
=0,\;\;x>0\ ,\;\;v\in \mathbb{R\ \ }  \label{T3E5} \\
\varphi (0,-v,t)& =\varphi (0,rv,t)\ ,\;\;v>0  \label{T3E6}
\end{align}%
\begin{equation}
\lim_{\left( x,v\right) \rightarrow \left( 0,0\right) }\left( \mathcal{L}%
\varphi \right) \left(x,v,t\right) =\mu _{\ast }\left\vert C_{\ast
}\right\vert \mathcal{A}\left( \varphi \right) ,\ t\in \left( 0,T\right) \
\label{T3E7}
\end{equation}%
with $\mathcal{L}$ as in (\ref{diffOp}) and $\mathcal{A}\left( \varphi
\right) $ defined by means of:%
\begin{equation}
\varphi \left( x,v,t\right) =\varphi \left( 0,0,t\right) +\mathcal{A}\left(
\varphi \right) F_{\beta }\left( x,v\right) +\psi \left( x,v,t\right) ,\ \
\lim_{R\rightarrow 0}\frac{\sup_{x+\left\vert v\right\vert ^{3}=R}\left\vert
\psi \left( x,v\right) \right\vert }{R^{\beta }}=0  \label{T3E8}
\end{equation}

We will obtain a well defined initial-boundary value problem with:
\begin{equation}
\varphi \left( x,v,0\right) =\varphi _{0}\left( x,v\right) \   \label{T3E9}
\end{equation}

\subsubsection{The case $r>r_{c}$}

In the case $r>r_{c}$ we only need to consider the case in which the
boundary condition is the nontrapping. Due to the definition of $\Omega _{\sup
} $ in Section \ref{Supercrit} we obtain
that the adjoint of the problem (\ref{S0E1}), (\ref{S0E2}) with the nontrapping boundary
condition can be formulated as:%
\begin{align}
\partial _{t}\varphi -v\partial _{x}\varphi -\partial _{vv}\varphi & =0\
,\;\;x>0,\;\;v\in \mathbb{R\ }  \label{T2E3} \\
\varphi (0,-v,t)& =\varphi (0,rv,t),\;\;v>0  \label{T2E4}
\end{align}%
\begin{equation}
\varphi \left( x,v,t\right) =\varphi \left( 0,0,t\right) +\psi \left(
x,v,t\right) ,\ \ ~\lim_{R\rightarrow 0}\sup_{x+\left\vert v\right\vert
^{3}=R}\left\vert \psi \left( x,v,t\right) \right\vert =0,\ t\in \left(
0,T\right)    \label{T2E5}
\end{equation}%
complemented with:%
\begin{equation}
\varphi \left( x,v,0\right) =\varphi _{0}\left( x,v\right)  \label{T2E6}
\end{equation}

As a next step we will prove that the problems (\ref{T2E3})-(\ref{T2E6}), (%
\ref{T2E8})-(\ref{T2E10}), (\ref{T3E1})-(\ref{T3E4}), (\ref{T3E5})-(\ref%
{T3E9}) can be solved for suitable choices of initial data $\varphi _{0}.$
These solvability results will be used to define a suitable concept of
solution for the boundary value problems (\ref{S0E1}), (\ref{S0E2}) with one
of the boundary conditions-trapping, nontrapping, partially trapping.


%
%
%

\section{Well-posedness of the adjoint problems. \label{WellPosed}}

In this section we will prove the solvability of the adjoint problems in Subsection 2.4. 
Using the Hille-Yoshida theorem will reduce the problem of solving the adjoint problems to the solvability 
of some suitable elliptic equations with some boundary conditions imposed at the singular set. 
These equations have good maximum principle properties and 
they will be solved using a generalization of the well known Perron's method for the solution 
of Laplace's equation in general bounded domains.

We state in the following Theorem the solvability of the adjoint problems described in
Section \ref{BVP}. 

\begin{theorem}
\label{AAW}Suppose that $\Omega _{\sigma }$ is one of the operators $\Omega
_{t,sub},\ \Omega _{nt,sub},\ \Omega _{pt,sub},$ $\Omega _{sup}$ defined in (%
\ref{S7E1})-(\ref{S7E1a}), (\ref{S7E2})-(\ref{S7E2a}), (\ref{S7E3})-(\ref%
{S7E3a}), (\ref{S7E4})-(\ref{S7E4a}) respectively. We can define Markov
semigroups $S_{\sigma }\left( t\right) $ having as generator the operator $%
\Omega _{\sigma }.$ For any $\varphi \in \mathcal{D}\left( \Omega _{\sigma
}\right) $ we can define a function $u\in C^{1}\left( \left[ 0,\infty
\right) :C\left( X\right) \right) $ such that:%
\begin{equation*}
\partial _{t}u=\Omega _{\sigma }u\ \ ,\ \ t\in \left[ 0,\infty \right) \ \
,\ \ u\left( t,\cdot \right) \in \mathcal{D}\left( \Omega _{\sigma }\right)
\ \ \text{if\ \ }t\geq 0\ \ ,\ \ u\left( 0,\cdot \right) =\varphi
\end{equation*}

Moreover, $u$\ is a classical solution of the equation $\partial _{t}u\left(
x,v,t\right) =\mathcal{L}u\left( x,v,t\right) $ for $\left( x,v\right) \neq
\left( 0,0\right) ,\ t>0.$
\end{theorem}

We will prove the theorem above at the end of the section.

\begin{remark}
Notice that the boundary condition (\ref{comp1}) is a consequence of
the fact that $u\left( \cdot ,t\right) \in C\left( X\right) $ for any $t\geq
0.$
\end{remark}

In order to prove the well-posedness of the problems (\ref{T2E3})-(\ref{T2E6}), (%
\ref{T2E8})-(\ref{T2E10}), (\ref{T3E1})-(\ref{T3E4}), (\ref{T3E5})-(\ref%
{T3E9}) we will use the version of Hille-Yosida Theorem that we recall in
the next Subsection.

\subsection{Hille-Yosida Theorem.}

We will follow closely the formulation of the Hille-Yosida Theorem in \cite%
{L}, which is particularly well suited for the study of the adjoint problems
summarized in Section \ref{BVP} of Section \ref{AdjProblems}.

The following results are Definition 2.1 and Proposition 2.2, Section 1 of
\cite{L}.

\begin{definition}
\label{MarPre}Let $\Omega $ be a linear operator on the Banach space $%
C\left( X\right) $ and $\mathcal{D}\left( \Omega \right) $ its domain. We
will say that $\Omega $ is a Markov pregenerator if:

\begin{description}
\item[(i)] $1\in \mathcal{D}\left( \Omega \right) ,\ \Omega 1=0.$

\item[(ii)] $\mathcal{D}\left( \Omega \right) $ is dense in $C\left(
X\right) .$

\item[(iii)] If $f\in \mathcal{D}\left( \Omega \right) ,\lambda \geq 0,$ and
$f-\lambda \Omega f=g,$ then $\min_{\zeta \in X}\ f\left( \zeta \right) \geq
\min_{\zeta \in X}\ g\left( \zeta \right) .$
\end{description}
\end{definition}

\begin{lemma}
Suppose that the linear operator $\Omega $ on $C\left( X\right) $ satisfies
that for any $f\in \mathcal{D}\left( \Omega \right) $ such that $f\left(
\eta \right) =\min_{\zeta \in C\left( X\right) }f\left( \zeta \right) ,$ we
have $\Omega f\left( \eta \right) \geq 0.$ Then $\Omega $ satisfies the
property (iii) in the Definition \ref{MarPre}.
\end{lemma}

We recall that an operator $\Omega $ in a Banach space $E$ is closed if its
graph is closed in $E\times E$ (cf. \cite{Brezis}). In our specific setting
this just means the following:

\begin{definition}
$\Omega $ is a closed operator if $\left\{ \left( \varphi ,\Omega \varphi
\right) :\varphi \in \mathcal{D}\left( \Omega \right) \right\} $ is a closed
set in $C\left( X\right) \times C\left( X\right) .$ That is, if $\varphi
_{n}\rightarrow \varphi ^{\ast }$ and $\Omega \varphi _{n}\rightarrow \psi $
in $C\left( X\right) ,$ then $\varphi ^{\ast }\in \mathcal{D}\left( \Omega
\right) $ and $\psi =\Omega \varphi ^{\ast }.$
\end{definition}

The following results are just a reformulation of Proposition 2.6,
Definition 2.7 and Proposition 1.3 of \cite{L}.

\begin{lemma}
Suppose $\Omega $ is a closed Markov pregenerator. Then $\mathcal{R}\left(
I-\lambda \Omega \right) $ is a closed subset of $C\left( X\right) $ for $%
\lambda >0.$
\end{lemma}

\begin{definition}
\label{MarGen}$\Omega $ is a Markov generator if it is a closed Markov
pregenerator such that $\mathcal{R}\left( I-\lambda \Omega \right) =C\left(
X\right) $ for all sufficiently small positive $\lambda $.
\end{definition}

\begin{definition}
$\left\{ S\left( t\right) ~|~t\geq 0\right\} $ is a Markov semigroup if it
is a family of operators on $C\left( X\right) $ satisfying the following:

\begin{description}
\item[(i)] $S\left( 0\right) =I.$

\item[(ii)] $\pi :[0,\infty )\rightarrow C\left( X\right) $ defined by $\pi \left(
t\right) =S\left( t\right) f$ is right-continuous for every $f\in C\left(
X\right) .$

\item[(iii)] $S\left( t+s\right) f=S\left( t\right) S\left( s\right) f$ for
all $f\in C\left( X\right) ,s,t\geq 0.$

\item[(iv)] $S\left( t\right) 1\geq 0,t\geq 0.$

\item[(v)] $S\left( t\right) f\geq 0$ for all nonnegative $f\in C\left(
X\right) .$
\end{description}
\end{definition}

The Hille-Yosida Theorem provides a connection between Markov generators and
Markov semigroups. The following version of this Theorem is the one in
Theorem 2.9 of \cite{L}.

\begin{theorem}
\label{HY}(Hille-Yosida Theorem) There is a one-to-one correspondence
between Markov generators on $C\left( X\right) $ and Markov semigroups on $%
C\left( X\right) .$ The correspondence is given by the following:%
\begin{equation*}
\mathcal{D}\left( \Omega \right) =\left\{ f\in C\left( X\right) \ |\
\lim_{t\rightarrow 0^{+}}\frac{S\left( t\right) f-f}{t}\text{ exists}\right\}
\end{equation*}%
and
\begin{equation*}
\Omega f=\lim_{t\rightarrow 0^{+}}\frac{S\left( t\right) f-f}{t},~f\in
\mathcal{D}\left( \Omega \right)
\end{equation*}

If $f\in \mathcal{D}\left( \Omega \right) ,$ then $S\left( t\right) f\in
\mathcal{D}\left( \Omega \right) $ and%
\begin{equation*}
\frac{d}{dt}\left( S\left( t\right) \right) f = \Omega S\left( t\right) f
\end{equation*}
\end{theorem}

\subsection{Comparison Principles and Trace Properties.}

In this Subsection we collect several Maximum Principle properties of the
operator $\mathcal{L}$.

\subsubsection{Basic Definitions.}

Due to the directionality of the transport terms in the operator $\mathcal{L}
$ we do not need to impose conditions in all the boundary of the domain
where the problem is satisfied in order to obtain comparison results. In
order to keep the statements simple we will formulate the Maximum Principle
only for the particular class of domains in Definition \ref{admissible}.

The purpose of the following Definition is to have a suitable unified
notation for the spaces which will be needed to formulate comparison
principles both for the whole space $X$ and the admissible domains $\Xi .$

\begin{definition}
\label{SpaceB}We will define as $L_{b}^{\infty }\left( X\right) $ the space
of functions $\varphi $ such that, for any compact set $K\subset \lbrack
0,\infty )\times (-\infty ,\infty )$ there exists a constant $C=C\left(
K\right) $ such that $\left\Vert \varphi \right\Vert _{L^{\infty }\left(
K\right) }\leq C\left( K\right) .$ In the case of the domains $\Xi $ in
Definition \ref{admissible} we will define as $L_{b}^{\infty }\left( \Xi
\right) $ just the space $L^{\infty }\left( \Xi \right) .$
\end{definition}

We now define a suitable concept of sub/supersolutions in the domains $\Xi $
and $X.$ We recall that the class of functions $\mathcal{F}\left( Y\right) $
is as in (\ref{defF}).

\begin{definition}
\label{supersolDef}Suppose that $Y$ is either the space $X,$ or one of the
admissible domains $\Xi $ defined in Definition \ref{admissible}. We will
say that $\varphi \in L_{b}^{\infty }\left( Y\right) $ is a supersolution
for the operator $\mathcal{L}\left( \cdot \right) +\kappa ,$ with $\kappa
\in \mathbb{R}$, if, for any $\psi \in \mathcal{F}\left( Y\right) ,\ \psi
\geq 0$ the following inequality holds:%
\begin{equation}
-\int_{X}\varphi \left( \mathcal{L}^{\ast }\psi +\kappa \psi \right)
dxdv\geq 0\   \label{Sineq}
\end{equation}%
where in the case of domains containing all or part of the line $\left\{
x=0\right\} $ we assume in addition that $\psi \left( 0,-v\right) =r^{2}\psi
\left( 0,rv\right) ,v>0.$

We will say that $\varphi \in L_{b}^{\infty }\left( Y\right) $ is a
subsolution for the operator $\mathcal{L}+\kappa $ if $\left(
-\varphi \right) $ is a supersolution for the operator $\mathcal{L}+\kappa .$
\end{definition}

\subsubsection{Traces and their properties.}

We will need to use sub and supersolutions having discontinuities in some
vertical lines. This is the main reason to assume only $L_{b}^{\infty }$
regularity in Definition \ref{supersolDef}. Notice that this Definition does
not make any reference to boundary conditions. This is just due to the fact
that it is not possible to define boundary values for a function $\varphi $
which is only in $L^{\infty }.$ However, it turns out to be possible to
define a suitable concept of trace for the supersolutions and subsolutions
in Definition \ref{supersolDef}. This fact will play a crucial role in the
rest of the paper.

\begin{proposition}
\label{Traces}Suppose that $Y$ is as in Definition \ref{supersolDef} and $%
\varphi \in L_{b}^{\infty }\left( Y\right) $ is a supersolution for the
operator $\mathcal{L}\left( \cdot \right) +\kappa I$ for some $\kappa \in
\mathbb{R}$. Let $\omega \in \mathcal{F}\left( Y\right) .$ It is possible to
modify the function $\varphi $ in a set of zero measure in $Y$ to obtain a
function $\varphi $ (also denoted as $\varphi $)\ with the following
properties:

(i) Let $L_{\ast }\subset \bar{Y}$ be any horizontal segment in the sense of
Definition \ref{LineLeft}. Then there exists $\ell _{L_{\ast }}\left(
\varphi ,\omega \right) \in \mathbb{R}$ such that\
\begin{eqnarray}
&&\lim_{\varepsilon \rightarrow 0^{+}}\left( \sup \left\{ \left\vert
\int_{L}\omega \varphi ds-\ell _{L_{\ast }}\left( \varphi ,\omega \right)
\right\vert :L\subset Y\text{ horizontal segment with }\operatorname{dist}{}%
_{H}\left( L,L_{\ast }\right) \leq \varepsilon \right\} \right)  \label{U1E5}
\\
&=&0  \notag
\end{eqnarray}%
with $\operatorname{dist}{}_{H}\left( L,L_{\ast }\right) $ as in (\ref{distH}).
There exists a function $\psi _{L_{\ast }}\in L^{\infty }\left( L_{\ast
}\right) $ such that $\ell _{L_{\ast }}\left( \varphi ,\omega \right)
=\int_{L_{\ast }}\psi _{L_{\ast }}\omega ds.$ Moreover, (\ref{U1E5}) holds
also for any function $\omega $ with the form $\omega \left( x,v\right)
=\zeta \left( x\right) \chi _{L_{\ast }}\left( v\right) ,$ with $\zeta \in
L^{1}\left( L_{\ast }\right) $ and where $\chi _{L_{\ast }}$ is the
characteristic function of the set $L_{\ast }.$

Moreover, there exist also $m_{L_{\ast }}^{+}\left( \varphi ,\omega \right)
,m_{L_{\ast }}^{-}\left( \varphi ,\omega \right) \in \mathbb{R}$ such that,
for any $v_{0}\in L_{\ast }$ we have:%
\begin{eqnarray}
\lim_{\varepsilon \rightarrow 0^{+}}\frac{1}{\varepsilon }\int_{L_{\ast
}\times \left( v_{0},v_{0}+\varepsilon \right) }\omega \partial _{v}\varphi
dxdv &=&m_{L_{\ast }}^{+}\left( \varphi ,\omega \right) \   \label{U2E1} \\
\lim_{\varepsilon \rightarrow 0^{+}}\frac{1}{\varepsilon }\int_{L_{\ast
}\times \left( -\varepsilon +v_{0},v_{0}\right) }\omega \partial _{v}\varphi
dxdv &=&m_{L_{\ast }}^{-}\left( \varphi ,\omega \right) \   \label{U2E2}
\end{eqnarray}%
There exist Radon measures $Q_{L_{\ast }}^{+},Q_{L_{\ast }}^{-}\in \mathcal{M%
}\left( L_{\ast }\right) $ such that $m_{L_{\ast }}^{\pm }\left( \varphi
,\omega \right) =\int_{L_{\ast }}Q_{L_{\ast }}^{\pm }\omega ds.$

(ii) Let $L_{\ast }\subset \bar{Y}$ be any vertical segment in the sense of
Definition \ref{LineLeft} with $x_{0}\in L_{\ast }$. Then there exists $\ell
_{L_{\ast }}^{+}\left( \varphi ,\omega \right) ,\ell _{L_{\ast }}^{-}\left(
\varphi ,\omega \right) \in \mathbb{R}$ such that%
\begin{equation}
\lim_{\varepsilon \rightarrow 0^{+}}\frac{1}{\varepsilon }\int_{L_{\ast
}\times \left( x_{0},x_{0}+\varepsilon \right) }\omega \varphi dxdv=\ell
_{L_{\ast }}^{+}\left( \varphi ,\omega \right) \   \label{U1E6}
\end{equation}%
\begin{equation}
\lim_{\varepsilon \rightarrow 0^{+}}\frac{1}{\varepsilon }\int_{L_{\ast
}\times \left( -\varepsilon +x_{0},x_{0}\right) }\omega \varphi dxdv=\ell
_{L_{\ast }}^{-}\left( \varphi ,\omega \right)  \label{U1E7}
\end{equation}

There exist functions $\varphi _{L_{\ast }}^{+},\varphi _{L_{\ast }}^{-}\in
L^{\infty }\left( L_{\ast }\right) $ such that $\ell _{L_{\ast }}^{\pm
}\left( \varphi ,\omega \right) =\int_{L_{\ast }}\varphi _{L_{\ast }}^{\pm
}\omega ds.$ Moreover (\ref{U1E6}), (\ref{U1E7}) hold for any function $%
\omega $ with the form $\omega \left( x,v\right) =\chi _{L_{\ast }}\left(
x\right) \zeta \left( v\right) ,$ with $\zeta \in L^{1}\left( L_{\ast
}\right) $ and where $\chi _{L_{\ast }}$ is the characteristic function of
the set $L_{\ast }.$\

The same results hold for subsolutions.
\end{proposition}

\begin{remark}
Notice that in principle, for a function in $L^\infty$, the integrals in  (\ref{U1E5})
cannot be expected to be defined. One of the results obtained in Proposition \ref{Traces} 
is that the functions 
$\phi$ satisfying the assumptions of Proposition \ref{Traces}, have enough regularity, 
modifying them if needed in a measure zero set, to make the integrals computed in 
(\ref{U1E5}) well defined.
\end{remark}

\begin{proof}
Suppose that $L_{\ast }$ is a vertical line as in the statement of the
Proposition and $\omega \in \mathcal{F}\left( Y\right) .$ Let us denote as $%
\zeta =\zeta \left( v\right) $ the restriction of $\omega $ to the line $%
L_{\ast }.$ Due to the continuity of $\omega ,$ to prove (\ref{U1E6}) is
equivalent to prove:%
\begin{equation}
\lim_{\varepsilon \rightarrow 0^{+}}\frac{1}{\varepsilon }\int_{L_{\ast
}\times \left( x_{0},x_{0}+\varepsilon \right) }\zeta \varphi dxdv=\ell
_{L_{\ast }}^{+}\left( \varphi ,\omega \right) \   \label{U1E8}
\end{equation}%
where we assume that $\zeta $ is extended from $L_{\ast }$ to $L_{\ast
}\times \left( 0,\varepsilon \right) $ assuming that $\zeta _{x}=0.$ Due to
the smoothness of $\zeta $ we can write $\zeta =\zeta _{+}-\zeta _{-}+R$
where $\zeta _{+}$ and $\zeta _{-}$ are nonnegative and smooth and $R$ is
small in $L^{1}\left( L_{\ast }\right) .$ Due to the boundedness of $\varphi
$ the contribution of $R$ to the left-hand side of (\ref{U1E8}) can be
bounded by $\left\Vert R\right\Vert _{L^{1}\left( L_{\ast }\right) }.$
Therefore,%
\begin{equation*}
\left\vert \lim_{\varepsilon \rightarrow 0^{+}}\frac{1}{\varepsilon }%
\int_{L_{\ast }\times \left( x_{0},x_{0}+\varepsilon \right) }\zeta \varphi
dxdv-\lim_{\varepsilon \rightarrow 0^{+}}\frac{1}{\varepsilon }\int_{L_{\ast
}\times \left( x_{0},x_{0}+\varepsilon \right) }\left( \zeta _{+}-\zeta
_{-}\right) \varphi dxdv\right\vert \leq C\left\Vert R\right\Vert
_{L^{1}\left( L_{\ast }\right) }
\end{equation*}%
Then, proving formulas like (\ref{U1E8}) for $\zeta _{+}$ and $\zeta _{-}$
for a functional $\ell _{L_{\ast }}^{+}\left( \varphi ,\omega \right)
=\int_{L_{\ast }}\psi _{L_{\ast }}\omega ds$ with $\psi _{L_{\ast }}$ as in
the statement of the Proposition, we would obtain (\ref{U1E8}) with a little
remainder $C\left\Vert R\right\Vert _{L^{1}\left( L_{\ast }\right) }$ which
can be made arbitrarily small and the result would follow for $\zeta .$ We
can then restrict ourselves to the case of $\zeta \geq 0.$

We define an auxiliary test function $\bar{\omega}$ by means of:%
\begin{equation*}
\bar{\omega}\left( x,v\right) =\left\{
\begin{array}{c}
\frac{2}{\varepsilon }\left( x-x_0-\frac{\varepsilon }{2}\right) \frac{\zeta
\left( v\right) }{v}\ \ \text{if\ \ }\frac{\varepsilon }{2}\leq x-x_{0}\leq
\varepsilon \\
\frac{1}{\varepsilon }\left( 2\varepsilon -x+x_0\right) \frac{\zeta \left(
v\right) }{v}\ \ \text{if\ \ }\varepsilon \leq x-x_{0}\leq 2\varepsilon \\
0\text{ otherwise}%
\end{array}%
\right.
\end{equation*}%
A density argument shows that (\ref{Sineq}) holds for the test function $%
\bar{\omega}$ in spite of the fact that this function is not in $\mathcal{F}%
\left( Y\right) .$ Therefore, using (\ref{Sineq}), we obtain that%
\begin{equation*}
\frac{2}{\varepsilon }\int_{L_{\ast }\times \left( x_{0}+\frac{\varepsilon }{%
2},x_{0}+\varepsilon \right) }\varphi \zeta -\frac{1}{\varepsilon }%
\int_{\left( x_{0}+\varepsilon ,x_{0}+2\varepsilon \right) }\varphi \zeta
\geq -C\varepsilon
\end{equation*}%
where $C$ is a constant depending of $\left\Vert \varphi \right\Vert
_{L^{\infty }},\ \left\Vert \zeta _{v}\right\Vert _{L^{\infty }},\left\Vert \zeta _{vv}\right\Vert _{L^{\infty }},$ $%
\operatorname{supp}\left( \zeta \right) ,\ L_{\ast }$ and $\kappa .$ It then
follows that for any $\varepsilon _{0}>0$ small the sequence $\left\{ \frac{%
2^{n}}{\varepsilon _{0}}\int_{L_{\ast }\times \left( x_{0}+2^{-n}\varepsilon
_{0},x_{0}+2^{-n+1}\varepsilon _{0}\right) }\varphi \zeta
-C2^{-n+1}\varepsilon _{0}\right\} $ is an increasing bounded sequence.
Therefore, the limit
\begin{equation*}
\lim_{n\rightarrow \infty }\frac{2^{n}}{\varepsilon _{0}}\int_{L_{\ast
}\times \left( x_{0}+2^{-n}\varepsilon _{0},x_{0}+2^{-n+1}\varepsilon
_{0}\right) }\varphi \zeta =\ell _{0}
\end{equation*}%
exists. This implies the existence of the limit $\lim_{n\rightarrow \infty }%
\frac{2^{n}}{\varepsilon _{0}}\int_{L_{\ast }\times \left(
x_{0},x_{0}+2^{-n}\varepsilon _{0}\right) }\varphi \zeta ,$ since we can
write:%
\begin{equation*}
\frac{2^{n}}{\varepsilon _{0}}\int_{L_{\ast }\times \left(
x_{0},x_{0}+2^{-n}\varepsilon _{0}\right) }\varphi \zeta =\frac{2^{n}}{%
\varepsilon _{0}}\sum_{k=0}^{\infty }\int_{L_{\ast }\times \left(
x_{0}+2^{-n-k-1}\varepsilon _{0},x_{0}+2^{-n-k}\varepsilon _{0}\right)
}\varphi \zeta
\end{equation*}%
Using then the approximation
\begin{equation*}
\int_{L_{\ast }\times \left( x_{0}+2^{-n-k-1}\varepsilon
_{0},x_{0}+2^{-n-k}\varepsilon _{0}\right) }\varphi \zeta
=2^{-n-k-1}\varepsilon _{0}\left( \ell _{0}+o\left( 1\right) \right) \text{
as }n\rightarrow \infty
\end{equation*}
we obtain $\lim_{n\rightarrow \infty }\left( \frac{2^{n}}{\varepsilon _{0}}%
\int_{L_{\ast }\times \left( x_{0},x_{0}+2^{-n}\varepsilon _{0}\right)
}\varphi \zeta \right) =\ell _{0}.$ It only remains to show that the limit
is independent of $\varepsilon _{0}$ for any $\varepsilon _{0}\in \left(
0,1\right) .$ Let us take $\varepsilon _{1}\neq \varepsilon _{0}.$ Let $\ell
_{1}$ be $\lim_{n\rightarrow \infty }\frac{2^{n}}{\varepsilon _{1}}%
\int_{L_{\ast }\times \left( x_{0},x_{0}+2^{-n}\varepsilon _{1}\right)
}\varphi \zeta =\ell _{1}.$ Suppose that $n$ is any large integer. For any
such $n,$ we select another integer $m$\ such that $\varepsilon
_{0}2^{-m}<\varepsilon _{1}2^{-n-1}.$ We consider a new test function $\bar{%
\omega}$ given by:%
\begin{equation*}
\bar{\omega}\left( x,v\right) =\left\{
\begin{array}{c}
\frac{2^{n}}{\varepsilon _{0}}\left(x-x_0\right)\frac{\zeta \left( v\right) }{v}\ \ \ \ \ \ \
\ \text{if\ \ }0\leq x-x_{0}\leq \varepsilon _{0}2^{-n} \\
\frac{\zeta \left( v\right) }{v}\text{\ if }\varepsilon _{0}2^{-m}\leq
x-x_{0}\leq \varepsilon _{1}2^{-n-1} \\
\frac{2^{n+1}}{\varepsilon _{1}}\left( \varepsilon _{1}2^{-n}-x+x_0\right) \frac{%
\zeta \left( v\right) }{v}\text{ if }\varepsilon _{1}2^{-n-1}\leq
x-x_{0}\leq \varepsilon _{1}2^{-n} \\
0\text{ otherwise}%
\end{array}%
\right.
\end{equation*}

Using (\ref{Sineq}) we obtain:
\begin{equation*}
\frac{2^{m}}{\varepsilon _{0}}\int_{L_{\ast }\times \left(
x_{0},x_{0}+\varepsilon _{0}2^{-m}\right) }\varphi \zeta -\frac{2^{n+1}}{%
\varepsilon _{1}}\int_{L_{\ast }\times \left( x_{0}+\varepsilon
_{1}2^{-n-1},x_{0}+\varepsilon _{1}2^{-n}\right) }\varphi \zeta \geq
-C\varepsilon _{1}2^{-n}
\end{equation*}%
for some constant $C$ independent of $n,m.$ Taking the limit $n\rightarrow
\infty $ we obtain $\ell _{0}\geq \ell _{1}.$ Reversing the role of $%
\varepsilon _{0},\varepsilon _{1}$ we would obtain $\ell _{1}\geq \ell _{0}$
and the independence of the limit on the choice of $\varepsilon _{0}$
follows.

Notice that the functional which assigns the function $\zeta $ to the limit
\begin{equation*}
\lim_{n\rightarrow \infty }\frac{2^{n}}{\varepsilon _{0}}\int_{L_{\ast
}\times \left( x_{0},x_{0}+2^{-n}\varepsilon _{0}\right) }\varphi \zeta
\end{equation*}
can be extended to linear functional in $L^{1}\left( \rho ,R\right) $ for
any $0<\rho <R<\infty .$ Since the dual of $L^{1}\left( \rho ,R\right) $ is $%
L^{\infty }\left( \rho ,R\right) $ we obtain the representation formula $%
\ell _{L_{\ast }}^{+}\left( \varphi ,\omega \right) =\int_{L_{\ast }}\varphi
_{L_{\ast }}^{+}\omega ds.$ The proof of (\ref{U1E7}) is similar.

Suppose now that $L_{\ast }\subset \bar{Y}$ is a horizontal line and $\omega
\in \mathcal{F}\left( Y\right) .$ Suppose that $v=v_{0}$ in $L_{\ast }.$ We
denote as $\zeta =\zeta \left( x\right) $ the restriction of $\omega $ to
the line $L_{\ast }$ The continuity of $\omega $ implies that (\ref{U1E5})
is equivalent to:%
\begin{equation*}
\lim_{\varepsilon \rightarrow 0^{+}}\int_{L}\zeta \varphi ds=\ell _{L_{\ast
}}\left( \varphi ,\omega \right)
\end{equation*}%
where $\operatorname{dist}{}_{H}\left( L,L_{\ast }\right) =\varepsilon .$ We can
assume, arguing as in the case of vertical lines, that $\zeta \geq 0.$ We
label the lines $L$ by means of the value of $v$ on them, i.e $L=L\left(
v\right) $. We can define a function $v\rightarrow \Phi \left( v\right)
=\int_{L\left( v\right) }\zeta \varphi ds.$ We then consider a test function
$\omega \left( x,v\right) =\zeta \left( x\right) \beta \left( v\right) $ for
some smooth $\beta \geq 0.$ Inequality (\ref{Sineq}) with $\psi =\omega $
yields:%
\begin{equation*}
\int \Phi \left( v\right) \beta _{vv}dv\leq C\int \beta dv
\end{equation*}%
where the constant $C$ depends on $\zeta $ and its derivatives and $\varphi $%
. This implies that the function $\Phi -Cv^{2}$ is concave and the limit $%
\lim_{v\rightarrow v_{0}}\Phi \left( v\right) $ exists. This limit defines a
linear functional in the set of functions $\zeta $ which can be continuously
extended to $L^{1}\left( L_{\ast }\right) $ due to the boundedness of $%
\varphi .$ Therefore $\ell _{L_{\ast }}\left( \varphi ,\omega \right)
=\int_{L_{\ast }}\varphi _{L_{\ast }}^{+}\omega ds$ and (\ref{U1E5}) follows.

Moreover, the concavity of $\Phi -Cv^{2}$ implies that $\Phi ^{\prime
}\left( v\right) $ exists except at a countable set. Then $\lim_{\varepsilon
\rightarrow 0^{+}}\frac{1}{\varepsilon }\int_{v_{0}}^{v_{0}+\varepsilon
}\Phi ^{\prime }\left( v\right) dv$ exists and we have:%
\begin{equation*}
\Phi ^{\prime }\left( v\right) =\beta _{v}\left( v\right) \int_{L\left(
v\right) }\zeta \left( x\right) \varphi ds+\beta \left( v\right)
\int_{L\left( v\right) }\zeta \left( x\right) \varphi _{v}ds\ \ \,a.e.\
\,v\in \left( v_{0}-\varepsilon ,v_{0}+\varepsilon \right)
\end{equation*}%
where we use the fact that $\int_{L\left( v\right) }\zeta \left( x\right)
\varphi \left( x,v\right) ds=\int_{I}\zeta \left( x\right) \beta \left(
v\right) \varphi \left( x,v\right) ds$ for a fixed interval $I$ for each $v.$
Using (\ref{U1E5}) we obtain the existence of the limit on the left-hand
side of (\ref{U2E1}). This gives (\ref{U2E1}). The proof of (\ref{U2E2}) is
similar. The existence of the measures $Q_{L_{\ast }}^{\pm }$ is a
consequence of the Riesz representation Theorem.

The proof of the results for subsolutions can be obtained in a similar way
with minor changes.
\end{proof}

\begin{remark}
\label{TraceDef}We will denote the function $\varphi _{L_{\ast }}$ in the
case (i) as the trace of $\varphi $ in the horizontal line $L_{\ast }.$ We
will denote the functions $\varphi _{L_{\ast }}^{+},\varphi _{L_{\ast }}^{-}$
in the case (ii)\ as the right-side trace and left-side traces of $\varphi $
respectively. Given a domain $\Xi $ we will use also the terminology\
"interior traces" to denote the trace from the domain $\Xi .$ We will denote
the functions $Q_{L_{\ast }}^{+},Q_{L_{\ast }}^{-}$ as the right-side trace
and left-side traces of $\partial _{v}\varphi $ respectively.
\end{remark}

We remark that Proposition \ref{Traces} allows us to obtain a suitable
definition of boundary data for arbitrary supersolutions of $\mathcal{L}%
\left( \cdot \right) +\kappa $ in the sense of Definition \ref{supersolDef}
defined in an admissible domain $\Xi .$

\begin{definition}
\label{TraceValue}Suppose that $\Xi $ is one admissible domain in the sense
of Definition \ref{admissible} and $\varphi $ is a supersolution of $%
\mathcal{L}\left( \cdot \right) +\kappa $ in the sense of Definition \ref%
{supersolDef}. The admissible boundary $\partial _{a}\Xi $ consists in the
union of some horizontal and vertical segments contained in $\bar{Y}$. The
function $\bar{\varphi}\in L^{\infty }\left( \partial _{a}\Xi \right) $
obtained patching the horizontal and vertical traces will be termed the
boundary value of $\varphi $ in $\partial _{a}\Xi .$
\end{definition}

We will use the following notation. Given any open set $\mathcal{Z}\subset
\left\{ \left( x,v\right) :x>0\text{, }v\in \mathbb{R}\right\} $

\begin{eqnarray}
\mathcal{Z}^{+} &=&\left\{ \left( x,v\right) \in \mathcal{Z}:v>-x\right\}
\notag \\
\mathcal{Z}^{-} &=&\left\{ \left( x,v\right) \in \mathcal{Z}:v<x\right\}
\label{U2E6}
\end{eqnarray}

We will denote as $\mathcal{Z}^{\pm }$ one of the sets $\mathcal{Z}^{+}$ or $%
\mathcal{Z}^{-}.$ Given that one of the main ideas used later is an
adaptation of the classical Perron's method for harmonic functions, we need
to prove that the maximum of subsolutions is a subsolution. Since the
subsolutions considered here might have discontinuities, this requires to
understand the traces of functions that are obtained by means of maxima of
subsolutions. This will be achieved in the following Lemmas.

The following technical Lemma yields a general property of $L^{\infty }$
functions with traces. They are not required to be sub or supersolutions.

\begin{lemma}
\label{TraceInequalities}Suppose that $\varphi _{1},\varphi _{2}$ are
respectively two functions in $L^{\infty }$ defined respectively in the open
sets $\mathcal{W}_{1}^{\pm },\ \mathcal{W}_{2}^{\pm }$\ with $\mathcal{W}%
_{1}^{\pm }\cap \mathcal{W}_{2}^{\pm }\cap \left\{ x=0\right\} \neq
\varnothing .$ Let us assume that the boundary values of $\varphi
_{1},\varphi _{2},\max \left\{ \varphi _{1},\varphi _{2}\right\} $ at $%
\mathcal{W}_{1}^{\pm }\cap \mathcal{W}_{2}^{\pm }\cap \left\{ x=0\right\} $
can be defined in the sense of traces, i.e. there exist functions $\varphi
_{1}^{\pm },\varphi _{2}^{\pm },\left( \max \left\{ \varphi _{1},\varphi
_{2}\right\} \right) ^{\pm }\in L^{\infty }\left( \mathcal{W}_{1}^{\pm }\cap
\mathcal{W}_{2}^{\pm }\cap \left\{ x=0\right\} \right) $ such that%
\begin{eqnarray}
\lim_{\varepsilon \rightarrow 0^{+}}\frac{1}{\varepsilon }\int_{\left[
\mathcal{W}_{1}^{\pm }\cap \mathcal{W}_{2}^{\pm }\cap \left\{ x=0\right\} %
\right] \times \left( 0,\varepsilon \right) }\varphi _{k}\zeta &=&\int_{%
\mathcal{W}_{1}^{\pm }\cap \mathcal{W}_{2}^{\pm }\cap \left\{ x=0\right\}
}\varphi _{k}^{\pm }\zeta \ \ ,\ \ k=1,2  \label{U1E4} \\
\lim_{\varepsilon \rightarrow 0^{+}}\frac{1}{\varepsilon }\int_{\left[
\mathcal{W}_{1}^{\pm }\cap \mathcal{W}_{2}^{\pm }\cap \left\{ x=0\right\} %
\right] \times \left( 0,\varepsilon \right) }\max \left\{ \varphi
_{1},\varphi _{2}\right\} \zeta &=&\int_{\mathcal{W}_{1}^{\pm }\cap \mathcal{%
W}_{2}^{\pm }\cap \left\{ x=0\right\} }\left( \max \left\{ \varphi
_{1},\varphi _{2}\right\} \right) ^{\pm }\zeta  \label{U1E4a}
\end{eqnarray}%
for any function $\zeta \in L_{loc}^{1}\left( \mathcal{W}_{1}^{\pm }\cap
\mathcal{W}_{2}^{\pm }\cap \left\{ x=0\right\} \right) $. Then:%
\begin{equation}
\left( \max \left\{ \varphi _{1},\varphi _{2}\right\} \right) ^{\pm }\geq
\max \left\{ \varphi _{1}^{\pm },\varphi _{2}^{\pm }\right\}  \label{U2E5}
\end{equation}
\end{lemma}

\begin{remark}
Due to the nonlinearity of the operator $\max \left( \cdot \right) ,$ we
cannot hope to replace the sign $\geq $ by $=$ in (\ref{U2E5}), because the
convergence in (\ref{U1E4}) is only a weak convergence. Indeed, if we define
the functions $\varphi _{1}\left( x,v\right) =\operatorname{sgn}\left( \cos
\left( 2^{n}v\right) \right) $ for $2^{-n-1}\leq x<2^{-n}$ and $\varphi
_{2}=0$ and we denote as $\left( \cdot \right) ^{+}$ the traces at $x=0,$ we
have $\varphi _{1}^{+}=\varphi _{2}^{+}=0,$ but $\left( \max \left\{ \varphi
_{1},\varphi _{2}\right\} \right) ^{+}=\frac{1}{2}.$
\end{remark}

\begin{remark}
Proposition \ref{Traces} implies (\ref{U1E4}) for sub and supersolutions for
any function $\zeta \in L^{1}\left( \left( \pm \rho ,\pm R\right) \cap
\mathcal{W}_{1}^{\pm }\cap \mathcal{W}_{2}^{\pm }\right) $ and $k=1,2$ and $%
0<\rho <R<\infty $ arbitrary.
\end{remark}

\begin{proof}
Our goal is to show that $$\lim_{\varepsilon \rightarrow 0^{+}}\frac{1}{%
\varepsilon }\int_{\left[ \mathcal{W}_{1}^{\pm }\cap \mathcal{W}_{2}^{\pm
}\cap \left\{ x=0\right\} \right] \times \left( 0,\varepsilon \right) }\max
\left\{ \varphi _{1},\varphi _{2}\right\} \zeta \geq \int_{\mathcal{W}%
_{1}^{\pm }\cap \mathcal{W}_{2}^{\pm }\cap \left\{ x=0\right\} }\max \left\{
\varphi _{1}^{\pm },\varphi _{2}^{\pm }\right\} \zeta $$ for any test
function $\zeta \geq 0.$ In order to simplify the argument we define $%
F=\varphi _{1}-\varphi _{2}.$ Then (\ref{U1E4}) becomes:%
\begin{equation}
\lim_{\varepsilon \rightarrow 0^{+}}\frac{1}{\varepsilon }\int_{\left[
\mathcal{W}_{1}^{\pm }\cap \mathcal{W}_{2}^{\pm }\cap \left\{ x=0\right\} %
\right] \times \left( 0,\varepsilon \right) }F\zeta =\int_{\mathcal{W}%
_{1}^{\pm }\cap \mathcal{W}_{2}^{\pm }\cap \left\{ x=0\right\} }F^{\pm
}\zeta \ \   \label{U1E9}
\end{equation}%
for any $\zeta \in L^{1}\left( \left( \pm \rho ,\pm R\right) \cap \mathcal{W}%
_{1}^{\pm }\cap \mathcal{W}_{2}^{\pm }\cap \left\{ x=0\right\} \right) ,\
0<\rho <R<\infty $ with $F^{\pm }\in L^{\infty }\left( v\gtrless 0\right) .$
We then need to prove:%
\begin{equation*}
\lim_{\varepsilon \rightarrow 0^{+}}\frac{1}{\varepsilon }\int_{\left[
\mathcal{W}_{1}^{\pm }\cap \mathcal{W}_{2}^{\pm }\cap \left\{ x=0\right\} %
\right] \times \left( 0,\varepsilon \right) }\max \left\{ F,0\right\} \zeta
\geq \int_{\mathcal{W}_{1}^{\pm }\cap \mathcal{W}_{2}^{\pm }\cap \left\{
x=0\right\} }\max \left\{ F^{\pm },0\right\} \zeta
\end{equation*}%
or equivalently $\max \left\{ F^{\pm },0\right\} \leq \left( \max \left\{
F,0\right\} \right) ^{\pm }.$ To this end we argue as follows. We will
restrict all the analysis to arbitrary compact sets of $\left\{ v\gtrless
0\right\} ,$ say $\left[ \pm \rho ,\pm R\right] ,$ with $0<\rho <R<\infty ,$
although this restriction will not be made explicit in the formulas for
simplicity. Then, using (\ref{U1E9}):%
\begin{eqnarray*}
\int_{\mathcal{W}_{1}^{\pm }\cap \mathcal{W}_{2}^{\pm }\cap \left\{
x=0\right\} }F^{\pm }\zeta &=&\lim_{\varepsilon \rightarrow 0^{+}}\frac{1}{%
\varepsilon }\int_{\left[ \mathcal{W}_{1}^{\pm }\cap \mathcal{W}_{2}^{\pm
}\cap \left\{ x=0\right\} \right] \times \left( 0,\varepsilon \right) }F\zeta
\\
&\leq &\lim_{\varepsilon \rightarrow 0^{+}}\frac{1}{\varepsilon }\int_{\left[
\mathcal{W}_{1}^{\pm }\cap \mathcal{W}_{2}^{\pm }\cap \left\{ x=0\right\} %
\right] \times \left( 0,\varepsilon \right) }\max \left\{ F,0\right\} \zeta
\\
&=&\int_{\mathcal{W}_{1}^{\pm }\cap \mathcal{W}_{2}^{\pm }\cap \left\{
x=0\right\} }\left( \max \left\{ F,0\right\} \right) ^{\pm }\zeta
\end{eqnarray*}%
and similarly:%
\begin{equation*}
0\leq \lim_{\varepsilon \rightarrow 0^{+}}\frac{1}{\varepsilon }\int_{\left[
\mathcal{W}_{1}^{\pm }\cap \mathcal{W}_{2}^{\pm }\cap \left\{ x=0\right\} %
\right] \times \left( 0,\varepsilon \right) }\max \left\{ F,0\right\} \zeta
=\int_{\mathcal{W}_{1}^{\pm }\cap \mathcal{W}_{2}^{\pm }\cap \left\{
x=0\right\} }\left( \max \left\{ F,0\right\} \right) ^{\pm }\zeta
\end{equation*}%
for any test function $\zeta \geq 0,$ whence $\left( \max \left\{
F,0\right\} \right) ^{\pm }\geq \max \left\{ F^{\pm },0\right\} .$ Then
\begin{equation*}
\left( \max \left\{ \varphi _{1}-\varphi _{2},0\right\} \right) ^{\pm }\geq
\max \left\{ \varphi _{1}^{\pm }-\varphi _{2}^{\pm },0\right\} .
\end{equation*}
Therefore
\begin{eqnarray*}
\left( \max \left\{ \varphi _{1},\varphi _{2}\right\} \right) ^{\pm }
&=&\left( \varphi _{2}+\max \left\{ \varphi _{1}-\varphi _{2},0\right\}
\right) ^{\pm } \\
&=&\varphi _{2}^{\pm }+\left( \max \left\{ \varphi _{1}-\varphi
_{2},0\right\} \right) ^{\pm } \\
&\geq &\varphi _{2}^{\pm }+\max \left\{ \varphi _{1}^{\pm }-\varphi
_{2}^{\pm },0\right\} =\max \left\{ \varphi _{1}^{\pm },\varphi _{2}^{\pm
}\right\}
\end{eqnarray*}
whence the result follows.
\end{proof}

In order to prove that the maximum of subsolutions is a subsolution we need
to be able to replace the inequality in (\ref{U2E5}) by an equality sign in
the case of the trace operator $\left( \cdot \right) ^{-}$ defined in the
domains $\mathcal{W}^{\pm }.$ It turns out that this is possible for the
subsolutions defined in Definition \ref{SubSuper1}. The key point is the
fact that it is possible to derive some regularity estimates for
subsolutions of (\ref{T6E6}), (\ref{T6E7}). More precisely, the following
Lemma shows that for domains contained in $\left\{ v<0\right\} $ the traces
in some vertical lines can be defined not just in the weak topology but in $%
L^{1}.$

\begin{lemma}
\label{L1traces}Suppose that $\Xi $ is an admissible domain as in Definition %
\ref{admissible} with $\Xi \subset \left\{ v<0\right\} .$ Suppose that $%
\varphi \in L^{\infty }\left( \Xi \right) $ satisfies in the sense of
distributions:\
\begin{equation}
\mathcal{L}\varphi =\nu -F\   \label{B6}
\end{equation}%
where $\nu \geq 0$ is a Radon measure in $\Xi $ and $F\in L^{\infty }\left(
\Xi \right) .$ We assume that $\int_{\Xi }\nu <\infty .$ Suppose that the
admissible boundary of $\Xi $\ is:%
\begin{equation*}
\partial _{a}\Xi =\left( \left\{ x_{1}\right\} \times \left[ -\beta ,-\alpha %
\right] \right) \cup \left( \left\{ -\beta \right\} \times \left[ x_{1},x_{2}%
\right] \right) \cup \left( \left\{ -\alpha \right\} \times \left[
x_{1},x_{2}\right] \right)
\end{equation*}%
for some $0<\alpha <\beta ,$ $0\leq x_{1}<x_{2}$. Then, it is possible to
define $\varphi \left( x_{1},\cdot \right) $ in the sense of trace as in\
Proposition \ref{Traces} and we have:%
\begin{equation}
\varphi \left( x,\cdot \right) \rightarrow \varphi \left( x_{1},\cdot
\right) \text{ as }x\rightarrow \left( x_{1}\right) ^{+}\text{ in }%
L^{1}\left( -\beta ,-\alpha \right)  \label{B8}
\end{equation}
\end{lemma}

\begin{proof}
Equation (\ref{B6}) is a forward parabolic equation for increasing $x,$ due
to the fact that $v$ is negative. Moreover, the variable $\Phi =v\varphi $
satisfies a parabolic equation with the form of conservation law, more
precisely we have:%
\begin{equation}
\partial _{x}\Phi +\partial _{v}^{2}\left( \frac{\Phi }{v}\right) =-F+\nu \
\label{U2E8}
\end{equation}%
where $F$ is bounded and $\int_{\left( x_{1},x_{1}+\delta _{0}\right) \times
K}\nu \left( dxdv\right) \leq C_{0}.$ Let us denote as $G\left(
v,w,x,y\right) $ the fundamental solution associated to the left-hand side
of (\ref{U2E8}) in the domain $\left( x_{1},x_{1}+\delta _{0}\right) \times
K,$ where $K=\left[ -\beta ,-\alpha \right] $ \ Notice that the function $%
\Phi $ is well defined in $\left( x_{1},x_{1}+\delta _{0}\right) \times
\partial K$ and in $\left\{ x_{1}\right\} \times K$ in the sense of traces
due to Proposition \ref{Traces}. We can derive a representation formula for $%
\Phi $ using the fundamental solution $G$ (cf. \cite{Fr}):%
\begin{equation}
\Phi \left( x,v\right) =\int_{\partial _{a}\left( \left( x_{1},x\right)
\times K\right) }G\left( v,w,x,y\right) \Phi \left( y,w\right) +\int_{\left(
x_{1},x\right) \times K}G\left( v,w,x,y\right) \left[ -F+\nu \right]
\label{B7}
\end{equation}

Due to the boundedness of $\Phi ,$ it is enough to prove that $\varphi
\left( x,\cdot \right) \rightarrow \varphi \left( x_{1},\cdot \right) $ as $%
x\rightarrow \left( x_{1}\right) ^{+}$ in $L^{1}\left( \bar{K}\right) $ for
any $\tilde{K}\subset \subset K.$ The convergence of the first term on the
right-hand side of (\ref{B7}) follows from classical results for parabolic
equations, as well as the convergence of $\int_{\left( x_{1},x\right) \times
K}G\left( v,w,x,y\right) F$ to zero as $x\rightarrow \left( x_{1}\right)
^{+}.$ It only remains to prove the convergence of $\int_{\left(
x_{1},x_{1}+\delta _{0}\right) \times K}G\left( v,w,x,y\right) \nu \left(
w,y\right) $ to zero as $x\rightarrow \left( x_{1}\right) ^{+}.$ The
contribution to the trace of this term is given by the limit of:%
\begin{equation}
\frac{1}{\varepsilon }\int_{\left( x_{1},x_{1}+\varepsilon \right]
}dx\int_{\left( x_{1},x\right] }dy\int_{K}dw\nu \left( w,y\right) \int_{\bar{%
K}}dvG\left( v,w,x,y\right)  \label{B9}
\end{equation}%
as $\varepsilon \rightarrow 0^{+}.$ Using Fubini's Theorem it is possible to
rewrite this integral as:%
\begin{equation}
\int_{\left( x_{1},x_{1}+\varepsilon \right] }dy\int_{K}dw\nu \left(
w,y\right) \frac{1}{\varepsilon }\int_{\left( y,x_{1}+\varepsilon \right]
}dx\int_{\tilde{K}}dvG\left( v,w,x,y\right)  \label{B9a}
\end{equation}

Standard regularity properties for Radon measures yield%
\begin{equation}
\int_{\left( x_{1},x_{1}+\varepsilon \right] }dy\int_{\tilde{K}}dw\nu \left(
w,y\right) \rightarrow 0\ \ \text{as\ }\varepsilon \rightarrow 0^{+}
\label{X1}
\end{equation}%
whence the integral in (\ref{B9a}) converges to zero as $\varepsilon
\rightarrow 0$ due to the boundedness of $\int_{\tilde{K}}dvG\left(
v,w,x,y\right) .$ Moreover, using the representation formula (\ref{B7}) with
$K$ replaced by $\tilde{K}$ as well as the fact that $\int G\left(
v,w,x,y\right) dv$ is uniformly bounded for $x,y\in \left(
x_{1},x_{1}+\varepsilon \right) $ and $w\in \tilde{K}$ we then obtain, taking
into account also (\ref{X1}) that
\begin{equation*}
\int_{\tilde{K}}dv\int_{\left( x_{1},x\right) \times \bar{K}}G\left(
v,w,x,y\right) F\rightarrow 0
\end{equation*}%
and
\begin{equation*}
\int_{\bar{K}}dv\int_{\left( x_{1},x\right) \times \bar{K}}G\left(
v,w,x,y\right) \nu \left( w,y\right) \rightarrow 0\text{ as }x\rightarrow
x_{1}^{+}.
\end{equation*}%
Then the result follows.
\end{proof}

\subsection{Well-posedness of the Dirichlet problem in admissible domains.}

We now prove the solvability of the Dirichlet problem:%
\begin{eqnarray}
\left( \varphi -\lambda \mathcal{L}\varphi \right) &=&g\ \ \text{in\ }\Xi \
\ ,\varphi \in C\left( \Xi \right) ,\varphi \left( 0,rv\right) =\varphi
\left( 0,-v\right) \ ,v>0,\ \left( 0,v\right) \in \bar{\Xi}  \label{solv_D1}
\\
\varphi &=&h\ \ \text{on }\partial _{a}\Xi  \label{solv_D2}
\end{eqnarray}%
where $\lambda >0,$ $\Xi $ is any of the admissible domains given in
Definition \ref{admissible}, $\partial _{a}\Xi $ is the corresponding
admissible boundary and $h\in L^{\infty }\left( \partial _{a}\Xi \right) $.
Notice that, if $\bar{\Xi}\cap \left\{ x=0\right\} =\varnothing $ the
condition on the boundary values $\left( 0,v\right) \in \bar{\Xi}$ is empty.
We will derive a representation formula for the solutions of (\ref{solv_D1}%
), (\ref{solv_D2}) by means of the stochastic process associated to the
operator $\mathcal{L}$. To this end, we first introduce some notation for
the solutions of some stochastic differential equations. We will write in
the following $\xi =\left( X,V\right) $ for brevity.

\begin{definition}
\label{StocPr}We will denote as $\xi _{t}=\left( X,V\right) $ the unique
strong solution of the Stochastic Differential Equation:%
\begin{equation}
d\xi =d\left(
\begin{array}{c}
X \\
V%
\end{array}%
\right) =\left(
\begin{array}{c}
-V \\
0%
\end{array}%
\right) dt+\sqrt{2}\left(
\begin{array}{c}
0 \\
dB_{t}%
\end{array}%
\right)  \label{difE}
\end{equation}%
where $B_{t}$ is the Brownian motion. We will denote as $\mathbb{P}^{\left(
x,v\right) },\ \mathbb{E}^{\left( x,v\right) }$ the probability and the mean
value associated to the solution of (\ref{difE}) with initial value $\xi
\left( 0\right) =\left( X,V\right) \left( 0\right) =\left( x,v\right) .$
\end{definition}

The existence of the stochastic process $\xi $ in Definition \ref{StocPr} is
a standard consequence of the theory of Stochastic Differential Equations
(cf. \cite{Oks}). Given an admissible domain $\Xi $ as in Definition \ref%
{admissible}, we will denote as $\tau \left( \Xi \right) $ the stopping time
associated to the process $\xi $ in Definition \ref{StocPr} defined as:%
\begin{equation}
\tau \left( \Xi \right) =\inf \left\{ t\geq 0:\xi \left( t\right) \notin \Xi
\right\}  \label{stopptime}
\end{equation}

We will prove the following result:

\begin{proposition}
\label{DirSolv}Suppose that $g\in C\left( \bar{\Xi}\right) ,$ $\Xi $ is one
of the admissible domains in Definition \ref{admissible} and $h\in L^{\infty
}\left( \partial _{a}\Xi \right) $ where $\partial _{a}\Xi $ is the
corresponding admissible boundary of $\Xi .$ Then, there exists a unique
solution of the problem (\ref{solv_D1}), (\ref{solv_D2}) where the boundary
condition (\ref{solv_D2}) is achieved in the sense of trace defined in
Remark \ref{TraceDef}. Moreover, the solution is given by the following
representation formula:%
\begin{equation}
\varphi \left( x,v\right) =\mathbb{E}^{\left( x,v\right) }\left( e^{-\frac{%
\tau \left( \Xi \right) }{\lambda }}h\left( \xi _{\tau \left( \Xi \right)
}^{\left( x,v\right) }\right) \right) +\frac{1}{\lambda }\mathbb{E}^{\left(
x,v\right) }\left( \int_{0}^{\tau \left( \Xi \right) }e^{-\frac{t}{\lambda }%
}g\left( \xi _{t}^{\left( x,v\right) }\right) dt\right) ,  \label{X8}
\end{equation}%
where $\tau \left( \Xi \right) $ is as in (\ref{stopptime}).
\end{proposition}

\begin{proof}
The representation formula (\ref{X8}) is standard in the Theory of Diffusion
Processes. It might be found, for instance in \cite{SV72}, Theorem 5.1. We
assume that $h$ has been extended as zero at the set $\partial \Xi \diagdown
\partial _{a}\Xi .$ Notice that the stochastic differential equation (\ref%
{difE}) implies that $\mathbb{P}^{\left( x,v\right) }\left( \left\{ \xi
_{\tau \left( \Xi \right) }^{\left( x,v\right) }\in \partial \Xi \diagdown
\partial _{a}\Xi \right\} \right) =0.$ Indeed, this follows from the
continuity of the paths associated to the process $\xi _{t}$ as well as the
fact that (\ref{difE}) implies that with probability one $X\left( \cdot
\right) $ is differentiable and $X^{\prime }\left( t\right) <0$ if $V\left(
t\right) >0$ and $X^{\prime }\left( t\right) >0$ if $V\left( t\right) <0$.
Notice that in order to obtain (\ref{solv_D2}) in the sense of traces it is
enough to prove that the boundary values are taken in the $L^{1}$ sense. On
the other hand, the formula (\ref{X8}) yields a solution of the boundary
value problem (\ref{solv_D1}), (\ref{solv_D2}) for $h\in C\left( \partial
_{a}\Xi \right) .$ Therefore, the fact that (\ref{X8}) yields the solution
of (\ref{solv_D1}), (\ref{solv_D2}) in the sense of traces follows from the
density of the continuous functions in $L^{\infty }\left( \partial _{a}\Xi
\right) $ using the norm $L^{1}\left( \partial _{a}\Xi \right) .$

Uniqueness of solutions follows from the maximum principle. Indeed, the
difference of two solutions of (\ref{solv_D1}), (\ref{solv_D2}) satisfies $%
\varphi -\lambda \mathcal{L}\varphi =0$ in $\Xi ,$ $\varphi =0\ $on $%
\partial _{a}\Xi .$ Using classical regularity theory for parabolic
equations we can prove that $\varphi \in C^{\infty }$ at all the points of
the boundary of $\Xi $, except at the singular points (i.e. the points of
the boundary where $v=0,$ namely $\partial \Xi \cap \left\{ v=0\right\} $)
and at the points of the boundary of $\partial \Xi $ where the boundary is
not $C^{\infty }$ (i.e. the corner-like points of $\partial \Xi $). Notice
that $\psi =e^{\frac{t}{\lambda }}\varphi $ is a bounded solution of $\psi
_{t}=\mathcal{L}\psi $ satisfying $\psi =0$ on $\partial _{a}\Xi .$ We can
then argue exactly as in the proof of Lemma 3.13 of \cite{HVJ} to show that
for any $\varepsilon >0$ we have $\left\vert \psi \right\vert \leq
\varepsilon $ in a small neighbourhood of the singular points $\partial \Xi
\cap \left\{ v=0\right\} $ and $0\leq t\leq 1.$ Similarly, classical
regularity theory for parabolic equations yields also the inequality $%
\left\vert \psi \right\vert \leq \varepsilon $ in a neighbourhood small
enough of the corner-like points of $\partial \Xi .$ Therefore $\left\vert
\varphi \right\vert \leq \varepsilon $ in a neighbourhood of the singular
point and the corner like points of $\partial \Xi .$ Using the regularity
properties of the function $\varphi $ it turns out that we apply the
classical argument yielding the maximum principle (including a regularizing
factor $\pm \delta v^{2}\pm \delta x$ with $\delta >0$ small if the
derivatives do not satisfy strict inequalities at the points where they
reach maximum or minimum values. It then follows that $\left\vert \varphi
\right\vert $ is bounded by its values in $\partial _{a}\Xi $ or at the
neighbourhoods indicated above of the singular point or the corner-like
points. Since we have $\varphi =0\ $on $\partial _{a}\Xi $ it then follows
that $\left\vert \varphi \right\vert \leq \varepsilon $ in $\Xi .$ Since $%
\varepsilon $ is arbitrarily small it then follows that $\varphi =0$ and the
uniqueness of (\ref{solv_D1}), (\ref{solv_D2}) follows.
\end{proof}

\begin{remark}
In the case of admissible domains contained in one of the half-planes $%
\left\{ v>0\right\} $ or $\left\{ v<0\right\} $ we can prove the existence and
uniqueness of solutions using the classical theory of parabolic equations.
\end{remark}

We will use a similar representation formula to the one obtained in
Proposition \ref{DirSolv} for the adjoint problem $\psi -\lambda \mathcal{L}%
^{\ast }\left( \psi \right) =g.$ Notice that in the case of the adjoint
problem the admissible boundary of the domains $\Xi $ is not $\partial
_{a} \Xi $ but the adjoint admissible boundary $\partial _{a}^{\ast
}\Xi $ defined in Definition \ref{admissible}.

\begin{lemma}
\label{comp_nu}Suppose%
\begin{equation}
\varphi -\lambda \mathcal{L}\left( \varphi \right) =\nu ~\ \text{in }\Xi \
,\ \ \varphi =0~\ \text{on }\partial _{a}\Xi ,\   \label{D2}
\end{equation}%
where $\lambda >0,$ $\Xi $ is any of the admissible domains given in
Definition \ref{admissible}, $\partial _{a}\Xi $ is the corresponding
admissible boundary and $\nu \geq 0$ is a Radon measure, and satisfies $%
\int_{S}\nu >0$ for some compact set $S\subset \Xi $ with\ $\operatorname{dist}%
\left( S,\partial \Xi \right) \geq \frac{d\left( \Xi \right) }{4},$ where $%
d\left( \Xi \right) $ is defined as%
\begin{equation}
d\left( \Xi \right) =\min \left\{ \operatorname{dist}\left( L_{1},L_{2}\right)
:L_{1},L_{2}\text{ are parallel lines and }\Xi \text{ is contained between
them}\right\}  \label{T7E3}
\end{equation}%
Then there exists a constant $C_{\ast }>0$ depending only on $\operatorname{diam}%
\left( \Xi \right) $, $d\left( \Xi \right) ,$ and $\max_{\left( x,v\right)
\in \Xi }\left\vert v\right\vert $ such that:
\begin{equation*}
\varphi \geq C_{\ast }\int_{S}\nu >0
\end{equation*}%
in the set $D=\left\{ \left( x,v\right) \in \Xi :\operatorname{dist}\left( \left(
x,v\right) ,\partial _{a}\Xi \right) \geq \frac{d\left( \Xi \right) }{4}%
\right\} .$
\end{lemma}

\begin{remark}
Notice that we assume $\operatorname{dist}\left( S,\partial \Xi \right) \geq
\frac{d\left( \Xi \right) }{4}$ where $\partial \Xi $ is the whole
topological boundary of $\Xi $ and we obtain estimates for $\varphi $ in the
set $D,$ which is separated from the admissible boundary of $\Xi $ at least
the distance $\frac{d\left( \Xi \right) }{4}.$
\end{remark}

\begin{proof}
We prove the result by duality. To this end we define a test function $\psi $
by means of the solution of the problem:%
\begin{equation}
\psi -\lambda \mathcal{L}^{\ast }\left( \psi \right) =\zeta \ \ \text{in\ }%
\Xi \ \ ,\ \ \psi =0\ \ \text{on\ }\partial _{a}^{\ast }\Xi  \label{D6}
\end{equation}%
where $\zeta \in C^{\infty },$ $\zeta =0$ if $\left( x,v\right) \notin \Xi
\diagdown D$, $0\leq \zeta \leq 1,\ \zeta =1$ in an admissible domain $\hat{D%
}\subset \Xi $ with $\operatorname{dist}\left( \hat{D},\partial _{a}\Xi \right) =%
\frac{3d\left( \Xi \right) }{8}$ and $\hat{D}\subset \subset D.$ We can
obtain a representation formula for $\psi $ analogous to (\ref{X8}) in which
we just replace the stochastic process $\xi _{t}=\left( X,V\right) $ by one
solving the Stochastic Differential Equation:%
\begin{equation*}
d\xi =d\left(
\begin{array}{c}
X \\
V%
\end{array}%
\right) =\left(
\begin{array}{c}
V \\
0%
\end{array}%
\right) dt+\sqrt{2}\left(
\begin{array}{c}
0 \\
dB_{t}%
\end{array}%
\right)
\end{equation*}

Notice that in this case we have $\mathbb{P}^{\left( x,v\right) }\left(
\left\{ \xi _{\tau \left( \Xi \right) }^{\left( x,v\right) }\in \partial \Xi
\diagdown \partial _{a}^{\ast }\Xi \right\} \right) =0.$ Therefore, since $%
\psi =0$ at $\partial _{a}^{\ast }\Xi $ we have:%
\begin{equation}
\psi \left( x,v\right) =\frac{1}{\lambda }\mathbb{E}^{\left( x,v\right)
}\left( \int_{0}^{\tau \left( \Xi \right) }e^{-\frac{t}{\lambda }}\zeta
\left( \xi _{t}^{\left( x,v\right) }\right) dt\right)  \label{D7}
\end{equation}

On the other hand, we denote by $\chi _{S}$ the characteristic function
of $S$ and we define $\bar{\varphi}$ as the unique solution of:%
\begin{equation}
\bar{\varphi}-\lambda \mathcal{L}\left( \bar{\varphi}\right) =\nu \chi
_{S}~\ \text{in }\Xi \ ,\ \ \bar{\varphi}=0~\ \text{on }\partial _{a}\Xi
\label{D3}
\end{equation}

Existence of solutions of this problem can be obtained approximating the
nonnegative measure $\nu $ by a sequence of $C^{\infty }$ functions $\nu
_{n} $ such that $\nu _{n}\rightharpoonup \nu $ as $n\rightarrow \infty .$
We denote the corresponding solutions of (\ref{D3}) with source term $\nu
_{n}\chi _{S}$ as $\bar{\varphi}_{n}.$ We have $\bar{\varphi}_{n}\geq 0,$ $%
\bar{\varphi}_{n}$ is a $C^{\infty }$ function at the interior of the domain
$\Xi \diagdown \bar{S}.$ Integrating the corresponding equations (\ref{D3})
in $\Xi $ we obtain:%
\begin{equation}
\int_{\Xi }\bar{\varphi}_{n}dxdv=\int_{S}\nu _{n}dxdv+\lambda \int_{\partial
_{h}\Xi }\partial _{n}\bar{\varphi}_{n}ds+\lambda \int_{\partial _{v}\Xi
\diagdown \partial _{a}\Xi }v\bar{\varphi}_{n}n_{x}ds  \label{D4}
\end{equation}%
where we denote as $\partial _{h}\Xi $,\ $\partial _{v}\Xi $ the horizontal
and vertical parts of the boundary of $\Xi $ respectively.

We remark that deriving this formula we need some careful estimates near the
singular points. More precisely, the results of \cite{HVJ} imply that the
functions $\bar{\varphi}_{n}$ are H\"{o}lder continuous. A rescaling
argument analogous to the one used in \cite{HVJ} implies an estimate with
the form:
\begin{equation}
\left\vert \bar{\varphi}_{n}\right\vert +\left( \left\vert
x-x_{0}\right\vert +\left\vert v\right\vert ^{3}\right) ^{\frac{1}{3}%
}\left\vert \partial _{v}\bar{\varphi}_{n}\right\vert \leq C\left(
\left\vert x-x_{0}\right\vert +\left\vert v\right\vert ^{3}\right) ^{a}
\label{D5}
\end{equation}%
for some $a>0,$ where $\left( x_{0},0\right) $ is one of the singular
points. In order to compute $\int_{\Xi }\bar{\varphi}_{n}dxdv$ we then
approximate the integrals by domains with the form $\int_{\Xi \diagdown
\mathcal{\bar{R}}_{\delta }}\bar{\varphi}_{n}dxdv,$ where the domains $%
\mathcal{\bar{R}}_{\delta }$ are just translated or reflected versions of
the domains $\mathcal{R}_{\delta }$ (cf. (\ref{domainsR})) and they are
choosen in order to exclude the integral \ near the singular points. Using
Gauss's Theorem we obtain boundary terms with the form $\int_{\partial
\mathcal{\bar{R}}_{\delta }}\left[ n_{v}\partial _{v}\bar{\varphi}_{n}+n_{x}v%
\bar{\varphi}_{n}\right] ds$ which converge to zero as $\delta \rightarrow 0$
due to (\ref{D5}). On the other hand $\partial _{n}\bar{\varphi}_{n}\leq 0$
on $\partial _{h}\Xi ,$ and the definition of $\partial _{a}\Xi $ implies
also that $v\bar{\varphi}_{n}n_{x}\leq 0$ on $\partial _{v}\Xi \diagdown
\partial _{a}\Xi .$ Therefore, (\ref{D4}) yields:
\begin{equation*}
\int_{\Xi }\bar{\varphi}_{n}dxdv\leq \int_{S}\nu _{n}dxdv
\end{equation*}

A standard compactness argument then yields that, for a suitable subsequence
$\bar{\varphi}_{n}\rightharpoonup \bar{\varphi}$ where $\bar{\varphi}$ is a
weak solution of (\ref{D3}). Moreover, $\bar{\varphi}$ is smooth in a
neighbourhood of $\partial \Xi ,$ except near the singular points where we
have that $\bar{\varphi}$ is H\"{o}lder due to the results of \cite{HVJ}.
Multiplying (\ref{D3}) by $\psi $ we obtain, using again Gauss's Theorem:%
\begin{equation*}
\int_{\Xi }\left( \psi -\lambda \mathcal{L}^{\ast }\psi \right) \bar{\varphi}%
=\int_{S}\nu \psi
\end{equation*}%
where we use the fact that the contribution near the singular points is zero
due to the h\"{o}lderianity of $\psi ,\ \bar{\varphi}$ and arguing as in the
derivation of (\ref{D4}). Using (\ref{D6}) we obtain $\int_{\Xi }\zeta \bar{%
\varphi}=\int_{S}\nu \psi ,$ whence $\int_{D}\bar{\varphi}\geq \int_{S}\nu
\psi .$ On the other hand, we can estimate a lower estimate of $\psi $ using
(\ref{D7}). Indeed, we have:
\begin{equation}
\mathbb{P}^{\left( x,v\right) }\left( \left\{ \xi _{t}^{\left( x,v\right)
}\in \hat{D},\ \text{for }T\subset \left[ 0,1\right] ,\ t\in T,\ \tau \left(
\Xi \right) \geq 1\right\} \right) \geq b_{\ast }>0\ \ ,\ \ \left(
x,v\right) \in S  \label{Y3}
\end{equation}

This inequality can be proved as follows. Suppose that we denote as $G\left(
x,v,t\right) $ the solution of the problem:%
\begin{eqnarray}
\partial _{t}G+\bar{v}\partial _{\bar{x}}G &=&G_{\bar{v}\bar{v}}\ \text{\
for \ }t\geq 0\ ,\ \ \left( \bar{x},\bar{v}\right) \in \Xi \ \ ,\ \ G=0\
\text{on\ }\partial _{a}^{\ast }\Xi  \label{Y1} \\
G\left( \bar{x},\bar{v},0\right) &=&\delta _{\left( x,v\right) }  \label{Y2}
\end{eqnarray}

Then:%
\begin{equation*}
\mathbb{P}^{\left( x,v\right) }\left( \left\{ \xi _{t^{\ast }}^{\left(
x,v\right) }\in A,\ \xi _{t}^{\left( x,v\right) }\in \Xi \text{ for }t\in %
\left[ 0,t^{\ast }\right] \right\} \right) =\int_{A}G\left( \bar{x},\bar{v}%
,t^{\ast }\right) d\bar{x}d\bar{v}
\end{equation*}%
for any measurable $A\subset \Xi $ and $t^{\ast }>0.$ The solvability of (%
\ref{Y1}), (\ref{Y2}) can be proved with some suitable modification of the
proof in \cite{HVJ} of the existence of solutions for the Kolmogorov
equation in bounded domains with absorbing boundary conditions, or,
alternatively, adapting the representation formula (\ref{D7}) to this case.

We can obtain a subsolution for (\ref{Y1}), (\ref{Y2}) as follows. Suppose
that we denote the fundamental solution of the Kolmogorov operator in the
plane as $Q\left( \bar{x},\bar{v},t\right) .$ We then obtain a family of
subsolutions for (\ref{Y1}), (\ref{Y2}) as
\begin{equation*}
G_{-}\left( \bar{x},\bar{v},t;M\right) =\max \left\{ Q\left( \bar{x},\bar{v}%
,t\right) -M,0\right\} ,\text{ for any }M>0
\end{equation*}
if $0\leq t\leq \bar{t}_{1}$ for a suitable $\bar{t}_{1}>0.$ More precisely,
we will choose $M$ and $\bar{t}_{1}$ as follows. We fix $\bar{t}_{1}>0$
arbitrarily small. We then choose $M>0$ small enough to obtain $Q\left( \bar{%
x},\bar{v},t\right) <M$ in $\partial \Xi $ for $0\leq t\leq \bar{t}_{1}.$
Notice that, since the distance between $\left( x,v\right) $ and $\partial
\Xi $ is positive and fixed and $Q\left( \bar{x},\bar{v},t\right) =\frac{C}{%
t^{2}}\exp \left( -\frac{3\left\vert \bar{x}-x-\left( \frac{v+\bar{v}}{2}%
\right) t\right\vert ^{2}}{t^{3}}-\frac{\left( \bar{v}-v\right) ^{2}}{4t}%
\right) ,$ we have that, if we assume that $\bar{t}_{1}$ is sufficiently
small, there exists a ball $B_{\rho }\left( x,v\right) \subset \Xi $ such
that $G_{-}\left( \bar{x},\bar{v},t_{1};M\right) \geq \delta _{M}>0$ if $%
\left( \bar{x},\bar{v}\right) \in B_{\rho }\left( x,v\right) .$ Moreover,
the ball can be chosen independent of $\bar{t}_{1}$ (even if $\bar{t}_{1}$
is arbitrarily small), reducing the value of $M,\ \delta _{M}$ if needed. We
can construct now a new subsolution with the form $\lambda _{1}G_{-}\left(
\bar{x},\bar{v},t+\bar{t}_{1};M_{1}\right) $ with $\lambda _{1}>0,\ M_{1}>0$
and satisfying $\lambda _{1}G_{-}\left( \bar{x},\bar{v},t+\bar{t}%
_{1};M_{1}\right) \leq G_{-}\left( \bar{x},\bar{v},t_{1};M\right) $ in a
suitable ball $B_{\rho _{1}}\left( x_{1},v_{1}\right) \subset \Xi $ with $%
\left( x_{1},v_{1}\right) $ closer to the set $\hat{D}$ than $\left(
x,v\right) .$ Iterating the argument, and using the fact that the size of
the balls where the subsolutions are positive can be chosen independent on
the times $\bar{t}_{1},\ \bar{t}_{2},...$ we obtain that there exists $%
t_{\ast }\in \left[ 0,\frac{1}{2}\right] $ and a ball $B_{\rho }\subset \hat{%
D}>0$ such that\ $G\left( \bar{x},\bar{v},t_{\ast }\right) >\delta >0.$
Therefore $\mathbb{P}\left( \left\{ \xi _{t_{\ast }}^{\left( x,v\right) }\in
B_{\rho }\right\} \right) >\delta >0.$ We can now estimate $\mathbb{P}\left(
\left\{ \xi _{t}^{\left( x,v\right) }\in \hat{D}\text{ for }t\in \left[
t_{\ast },t_{\ast }+\varepsilon _{0}\right] \right\} \left\vert \xi
_{t_{\ast }}^{\left( x,v\right) }\in B_{\rho }\right. \right) $ for some $%
\varepsilon _{0}>0.$ To this end we define a function $\tilde{G}\left(
x,v,t\right) $ solving:%
\begin{eqnarray}
\partial _{t}\tilde{G}+\bar{v}\partial _{\bar{x}}\tilde{G} &=&\tilde{G}_{%
\bar{v}\bar{v}}\ \text{\ for \ }t\geq 0\ ,\ \ \left( \bar{x},\bar{v}\right)
\in \Xi \ \ ,\ \ \tilde{G}=0\ \text{on\ }\partial _{a}^{\ast }\hat{D}
\label{Y4} \\
\tilde{G}\left( \bar{x},\bar{v},t_{\ast }\right) &=&\frac{1}{\left\vert
B_{\rho }\right\vert }\chi _{B_{\rho }}  \label{Y5}
\end{eqnarray}

Therefore $\mathbb{P}\left( \left\{ \xi _{t}^{\left( x,v\right) }\in \hat{D}%
\text{ for }t\in \left[ t_{\ast },t_{\ast }+\varepsilon _{0}\right] \right\}
\left\vert \xi _{t_{\ast }}^{\left( x,v\right) }\in B_{\rho }\right. \right)
=\tilde{G}\left( \bar{x},\bar{v},t_{\ast }+\varepsilon _{0}\right) .$ Using
subsolutions with the form $G_{-}\left( \bar{x},\bar{v},t;M\right) $ it then
follows that
\begin{equation*}
\mathbb{P}\left( \left\{ \xi _{t}^{\left( x,v\right) }\in \hat{D}\text{ for }%
t\in \left[ t_{\ast },t_{\ast }+\varepsilon _{0}\right] \right\} \left\vert
\xi _{t_{\ast }}^{\left( x,v\right) }\in B_{\rho }\right. \right) \geq
\delta _{\ast }>0
\end{equation*}%
Then (\ref{Y3}) follows.

Then $\psi \left( x,v\right) \geq \frac{e^{-\frac{1}{\lambda }}}{\lambda }%
b_{\ast }$ for $\left( x,v\right) \in S.$ Therefore $\int_{D}\bar{\varphi}%
\geq c_{\ast }\int_{S}\nu .$ Finally the pointwise estimate $\varphi \geq
c_{\ast }\int_{S}\nu $ in the domain $\hat{D}\subset \subset D$ follows from
Harnack inequality for Kolmogorov operators (cf. for instance the results in
\cite{CNP}, \cite{KL}).
\end{proof}

\subsubsection{Comparison Theorems.}

We will use the following version of the weak Maximum Principle.

\begin{proposition}
\label{weak_max}Suppose that $\Xi $ is one of the admissible domains in
Definition \ref{admissible} and that $\varphi \in L_{b}^{\infty }\left( \Xi
\right) $ is a supersolution for the operator $\mathcal{L}\left( \cdot
\right) +\kappa $ for some $\kappa \geq 0.$ Suppose also that $\varphi \geq
0 $ in $\partial _{a}\Xi ,$ where $\varphi $ is defined in $\partial _{a}\Xi
$ in the sense of Definition \ref{TraceValue}.

Then, we have $\varphi \left( x,v\right) \geq 0$ for $a.e.\ \left(
x,v\right) \in \Xi .$ Moreover, there exists $C_{\ast }=C_{\ast }\left( \Xi
\right) ,$ such that
\begin{equation*}
\varphi \left( x,v\right) \geq C_{\ast }\kappa \ \ \text{for any }\left(
x,v\right) \in D
\end{equation*}%
where $d\left( \Xi \right) $ is as in (\ref{T7E3}) and $D=\left\{ \left(
x,v\right) :\operatorname{dist}\left( \left( x,v\right) ,\partial _{a}\Xi \right)
\geq \frac{d\left( \Xi \right) }{4}\right\} $

Moreover, the constant $C_{\ast }$ can be chosen uniformly for sets $\Xi $
contained in a compact set of $[0,\infty )\times (-\infty ,\infty )$ and
have $0<c_{1}\leq d\left( \Xi \right) \leq c_{2}<\infty ,\ 0<c_{1}\leq
\operatorname{diam}\left( \Xi \right) \leq c_{2}<\infty .$
\end{proposition}

\begin{proof}
The proof can be made along the same lines as the proof of Lemma \ref%
{comp_nu}. The problem is linear in $\kappa .$ Then, we can assume that $%
\kappa =1.$ We then assume that $\varphi \in L_{b}^{\infty }\left( \Xi
\right) $ is a supersolution for the operator $\mathcal{L}\left( \cdot
\right) +1.$ We consider any test function $\zeta \geq 0,$ $\zeta \in
C^{\infty }$ supported in the set $D.$ We then define $\psi $ as the
solution of $-\mathcal{L}^{\ast }\left( \psi \right) =\zeta $ in $\Xi ,\
\psi =0$ on $\partial _{a}^{\ast }\Xi .$ We have the representation formula:%
\begin{equation*}
\psi \left( x,v\right) =\mathbb{E}^{\left( x,v\right) }\left( \psi \left(
\xi _{\tau \left( \Xi \right) }^{\left( x,v\right) }\right) dt\right) +%
\mathbb{E}^{\left( x,v\right) }\left( \int_{0}^{\tau \left( \Xi \right)
}\zeta \left( \xi _{t}^{\left( x,v\right) }\right) dt\right) \ ,\ \left(
x,v\right) \in \Xi
\end{equation*}%
with $\xi _{t}^{\left( x,v\right) }$ as in the Proof of Lemma \ref{comp_nu}.
Since the trajectories of the process $\xi _{t}^{\left( x,v\right) }$ leave
the domain across $\partial _{a}^{\ast }\Xi $ we obtain the first term
vanishes. Therefore:
\begin{equation*}
\psi \left( x,v\right) =\mathbb{E}^{\left( x,v\right) }\left( \int_{0}^{\tau
\left( \Xi \right) }\zeta \left( \xi _{t}^{\left( x,v\right) }\right)
dt\right) \geq 0
\end{equation*}

Using $\psi $ as a test function in the definition of supersolution in
Definition \ref{supersolDef} we obtain $-\int \mathcal{L}^{\ast }\left( \psi
\right) \varphi \geq \int \psi ,$ whence $\int \mathcal{\zeta }\varphi \geq
\int \psi .$ We now claim that $\int \psi \geq c_{\ast }\int \zeta .$ To
this end we first notice that%
\begin{equation*}
\psi \left( x,v\right) \geq \mathbb{E}^{\left( x,v\right) }\left( {\large 1}%
_{A_{\left( x,v\right) }}\int_{0}^{1}\zeta \left( \xi _{t}^{\left(
x,v\right) }\right) dt\right)
\end{equation*}%
where $A_{\left( x,v\right) }$ is the set of realizations of the Markov
process $\xi _{t}^{\left( x,v\right) }$ satisfying $\tau \left( \Xi \right)
>1$ and ${\large 1}_{A_{\left( x,v\right) }}$ is the characteristic function
of such a set. Then:%
\begin{equation*}
\psi \left( x,v\right) \geq \int_{0}^{1}\mathbb{E}^{\left( x,v\right)
}\left( {\large 1}_{A_{\left( x,v\right) }}\zeta \left( \xi _{t}^{\left(
x,v\right) }\right) \right) dt
\end{equation*}

Notice that we can then obtain the lower estimate
\begin{equation*}
\mathbb{E}^{\left( x,v\right) }\left( {\large 1}_{A_{\left( x,v\right)
}}\zeta \left( \xi _{t}^{\left( x,v\right) }\right) \right) \geq G\left(
x,v,t\right) ,0\leq t\leq 1
\end{equation*}
and where $G$ is the solution of the problem:%
\begin{eqnarray*}
\partial _{t}G+\bar{v}\partial _{\bar{x}}G &=&G_{\bar{v}\bar{v}}\ \text{\
for \ }t\geq 0\ ,\ \ \left( \bar{x},\bar{v}\right) \in \Xi \ \ ,\ \ G=0\
\text{on\ }\partial _{a}^{\ast }\Xi \\
G\left( \bar{x},\bar{v},0\right) &=&\zeta \left( \bar{x},\bar{v}\right)
\end{eqnarray*}%
Using then that $\zeta $ is supported in $D$ we can argue as in the Proof of
Lemma \ref{comp_nu}, using suitable subsolutions for the fundamental
solution of the problem above to obtain $\psi \left( x,v\right) \geq G\left(
x,v,0\right) \geq c_{0}\int \zeta .$ Then $\int \mathcal{\zeta }\varphi \geq
c_{\ast }\int \zeta $ and since $\zeta $ is arbitrary it follows that $%
\varphi \geq c_{\ast }.$ Using the linearity of the problem the result would
follow.
\end{proof}

We will need also the following version of the strong maximum principle.

\begin{proposition}
\label{strong_max}(Strong maximum principle) Suppose that $\Xi $ is one
admissible domain in the sense of Definition \ref{admissible}. Let $\Phi \in
C\left( \Xi \right) ,$ $\Phi \geq 0$ and and let $\mathcal{L}\Phi =0$ on $%
\Phi \in C\left( \Xi \right) $ with $\Phi $ satisfying (\ref{comp1}) if $%
\overline{\Xi }\cap \left\{ \left( x,v\right) =\left( 0,v\right) ,\ v\in
\mathbb{R}\right\} \neq \varnothing .$ Suppose that there exists a point $%
\left( \hat{x},\hat{v}\right) \in \bar{\Xi}\diagdown \partial _{a}\Xi $ such
that $\Phi \left( \hat{x},\hat{v}\right) =0.$ Then $\Phi \equiv 0.$
\end{proposition}

In order to prove Proposition \ref{strong_max} we will use the following
Lemma.

\begin{lemma}
\label{Max1}Let the function $G\left( v,x\right) \in C\left( \left[ 1,2%
\right] \times \lbrack 0,\infty )\right) ~$satisfy the equation%
\begin{eqnarray*}
vG_{x} &=&G_{vv},\ 1\leq v\leq 2,\ \ x>0, \\
G\left( v,0\right) &\geq &1\text{\ for }7/4\leq v\leq 2. \\
G\left( 7/4,x\right) &\geq &0\ \text{for }x\geq 0
\end{eqnarray*}

Then there exists $\beta >0$ not depending on $G$, such that $G\left(
v,x\right) \geq \beta >0$ for $5/4\leq v\leq 7/4,\ 1\leq x\leq 2.$
\end{lemma}

\begin{proof}
We prove this lemma by constructing a sub-solution $\bar{G}.$ Let $\bar{G}%
\left( v,x\right) =e^{-\theta x}\sin \left( 4\pi \left( v-7/4\right) \right)
.$ Then for $7/4\leq v\leq 2,$ $\bar{G}$ is a sub-solution of $G$ if we
choose $\theta >0$ in such a way that $\theta \geq 64\pi ^{2}/7.$ Since $%
G\left( v,0\right) \geq 1\geq \bar{G}\left( v,0\right) ,~$for $7/4\leq v\leq
2$ and $G\left( v,x\right) \geq 0=\bar{G}\left( v,x\right) $ at $v=7/4,2$
for all $x>0.$ Thus we have \ $G\left( v,x\right) \geq \bar{G}\left(
v,x\right) $ for $7/4\leq v\leq 2.$ Let $Q=\left[ 1,2\right] \times \left[
0,1\right] $ and $\tilde{Q}=\left[ 5/4,7/4\right] \times \lbrack 1,2],$ then
we get%
\begin{equation*}
\sup_{Q}G\left( v,x\right) \geq e^{-\theta }>0.
\end{equation*}%
We then apply the classical Harnack inequality for parabolic equations (cf.
for instance \cite{Fr}) to obtain%
\begin{equation*}
\inf_{\tilde{Q}}G\left( v,x\right) \geq C\sup_{Q}G\left( v,x\right) \geq
Ce^{-\theta }\equiv \beta >0.
\end{equation*}%
This completes the proof.
\end{proof}

\begin{proof}[Proof of Proposition \protect\ref{strong_max}]
We prove by contradiction that if there exists a point at which $\Phi $ is
positive, then $\Phi $ is positive everywhere. This implies the Proposition.
Suppose that there exists $\left( x^{\ast },v^{\ast }\right) \in \lbrack
0,\infty )\times (-\infty ,\infty )\smallsetminus \left( 0,0\right) $ with $%
\Phi \left( x^{\ast },v^{\ast }\right) >0$. $\mathcal{L}\Phi =0$ reads%
\begin{equation}
-v\Phi _{x}=\Phi _{vv}  \label{Kol}
\end{equation}

We may assume that $v^{\ast }>0$ since the cases $v^{\ast }<0$ can be
treated similarly and if $v^{\ast }=0,$ then we can find $v^{\ast \ast }>0$
such that $\Phi \left( x^{\ast },v^{\ast \ast }\right) >0$ due to the
continuity of $\Phi .$ Our goal is to show that $\Phi \left( x,v\right) >0$ for
all $\left( x,v\right) \in \lbrack 0,\infty )\times (-\infty ,\infty
)\smallsetminus \left( 0,0\right) .$

First we prove the positivity of $\Phi $ for $0\leq x<x^{\ast },v>0.$ The
equation (\ref{Kol}) is a parabolic equation for $v\geq \rho >0,$ $\rho $
arbitrary and thus $\Phi \left( x,v\right) >0$ for $0\leq x<x^{\ast },v>0$
due to the strong maximum principle for non-degenerate parabolic equations
(cf. \cite{Fr}).

Next we prove the positivity for $0\leq x<x^{\ast },v=0.$ We rescale as
follows. For $x\in \left[ x_{0},x_{0}+2^{-n}\right] ,v\in \lbrack 2^{-\left(
n+1\right) /3},2^{\left( 2-n\right) /3}],\ $define for some $x_{0}\in \left(
0,x^{\ast }\right) ,$
\begin{equation*}
\bar{x}=-2^{\left( n+1\right) }\left[ x-\left( x_{0}+2^{-n}\right) \right]
,\ \bar{v}=2^{\left( n+1\right) /3}v,
\end{equation*}%
\begin{equation*}
G_{n}\left( \bar{v},\bar{x}\right) =\frac{\Phi \left( x,v\right) }{M_{n}},\
M_{n}=\inf_{v\in I_{n}}\Phi \left( x_{0}+2^{-n},v\right) ,\ I_{n}=\left[
\frac{7}{4}2^{-\left( n+1\right) /3},2\cdot 2^{-\left( n+1\right) /3}\right]
\end{equation*}

Then the following holds%
\begin{equation*}
\bar{v}\partial _{\bar{x}}G_{n}=\partial _{\bar{v}}^{2}G_{n},\ \bar{x}\in %
\left[ 0,2\right] ,\ \bar{v}\in \left[ 1,2\right] ,
\end{equation*}%
\begin{equation*}
G_{n}\left( \bar{v},0\right) \geq 1\text{ for }7/4\leq \bar{v}\leq 2.
\end{equation*}

If we let $J_{n}=\left[ \frac{5}{4}2^{-\left( n+1\right) /3},\frac{7}{4}%
2^{-\left( n+1\right) /3}\right] ,$ then $I_{n+1}\subset J_{n}$ for each $n.$
We now apply Lemma \ref{Max1} to get%
\begin{equation*}
G_{n}\left( \bar{v},\bar{x}\right) \geq \beta >0,\text{ for }\bar{v}\in %
\left[ 5/4,7/5\right] ,\ \ \bar{x}\in \left[ 1,2\right] ,
\end{equation*}

or equivalently,%
\begin{equation}
\Phi \left( x,v\right) \geq \beta M_{n},\text{ for }\left( x,v\right) \in %
\left[ x_{0},x_{0}+2^{-\left( n+1\right) }\right] \times J_{n}.
\label{Phi_b}
\end{equation}

In particular, at $\bar{x}=1,$%
\begin{equation*}
\frac{\Phi \left( x_{0}+2^{-\left( n+1\right) },v\right) }{M_{n}}\geq \beta
,~\ \text{for }v\in J_{n}.
\end{equation*}

Thus we get, since $I_{n+1}\subset J_{n},$%
\begin{equation*}
M_{n+1}=\inf_{v\in I_{n+1}}\Phi \left( x_{0}+2^{-\left( n+1\right)
},v\right) \geq \beta M_{n},
\end{equation*}
which yields%
\begin{equation*}
M_{n}\geq C\beta ^{n},\ 0<\beta <1
\end{equation*}%
Then from (\ref{Phi_b}), we have for each $n,$%
\begin{equation}
\Phi \left( x_{0},v\right) \geq C\beta ^{n+1},\text{ for }v\in J_{n+1}.
\label{Phi_b1}
\end{equation}

We can rewrite (\ref{Phi_b1}) as%
\begin{eqnarray*}
\Phi \left( x_{0},v\right) &\geq &C\beta ^{n+1}=C\exp \left( -\frac{\log
\left( 2\right) }{3}\left( n+1\right) \right) ^{-\frac{3\log \left( \beta
\right) }{\log \left( 2\right) }} \\
&\geq &Cv^{\gamma },\ \gamma =-\frac{3\log \left( \beta \right) }{\log
\left( 2\right) }>0,\text{ for } v\in J_{n+1}.
\end{eqnarray*}

Now let $0<\tilde{x}<x^{\ast },$ then the argument above yields, for $%
x_{0}\in \left( \tilde{x},\tilde{x}+\delta \right) $ with $\delta >0$ small,
\begin{equation*}
\Phi \left( x_{0},v\right) \geq Cv^{\gamma },\text{ for }\gamma >0.
\end{equation*}

We first construct a sub-solution $G\left( \bar{v},\bar{x}\right) =\left(
\bar{x}\right) ^{\gamma /3}H\left( \frac{\bar{v}}{\left( \bar{x}\right)
^{1/3}}\right) $ satisfying%
\begin{equation}
\bar{v}G_{x}=G_{\bar{v}\bar{v}},\ G\left( \bar{v},0\right) =\left( \bar{v}%
\right) ^{\gamma }  \label{G_eqn}
\end{equation}%
where $H\left( \xi \right) ,\ \xi =\frac{\bar{v}}{\left( \bar{x}\right)
^{1/3}}$ satisfy:%
\begin{equation}
\frac{\gamma }{3}\xi H-\frac{\xi ^{2}}{3}H_{\xi }=H_{\xi \xi },\ \ H\left(
0\right) =0,H_{\xi }\left( 0\right) >0,\ \ H\left( \xi \right) \sim \xi
^{\gamma },\text{ as }\xi \rightarrow \infty .  \label{H_eqn}
\end{equation}%
The existence of a solution $\bar{H}$ of the differential equation in (\ref%
{H_eqn}) with $\bar{H}\left( 0\right) =0,$ $\bar{H}_{\xi }\left( 0\right) =1$
follows from standard ODE theory. Then we can show that $\bar{H}_{\xi
}\left( \xi \right) >0,$ for all $\xi \geq 0,$ because otherwise, let $\xi
^{\ast }$ be the first value where $\bar{H}_{\xi }\left( \xi ^{\ast }\right)
=0,$ then $\bar{H}\left( \xi ^{\ast }\right) >0$ and $\bar{H}_{\xi \xi
}\left( \xi ^{\ast }\right) =\frac{\gamma }{3}\xi ^{\ast }\bar{H}\left( \xi
^{\ast }\right) >0,$ which leads to a contradiction. The asymptotic
behaviors of the solutions of the differential equation in (\ref{H_eqn}) can
be computed using the results in Chapter 13 of \cite{Ab}, whence:%
\begin{equation*}
\bar{H}\left( \xi \right) \sim C\xi ^{\gamma },\ \ \xi \rightarrow \infty
\text{ with }C>0
\end{equation*}%
Then $G\left( \bar{v},\bar{x}\right) =\left( \bar{x}\right) ^{\gamma
/3}H\left( \frac{\bar{v}}{\left( \bar{x}\right) ^{1/3}}\right) $ satisfies (%
\ref{G_eqn}) and for all $\bar{x}>0,$%
\begin{equation*}
G_{\bar{v}}\left( 0^{+},\bar{x}\right) =\left( \bar{x}\right) ^{\gamma
/3-1/3}H_{\xi }\left( 0\right) >0.
\end{equation*}

By a comparison principle and Taylor expansion, there exist constant $C_{0}$
and $C_{1}$ satisfying, for $x\in \lbrack \tilde{x},\tilde{x}+\delta ),$%
\begin{equation*}
\Phi \left( x,v\right) \geq C_{0}G\left( v,x-\left( \tilde{x}+\delta \right)
\right) \geq C_{1}\left( \tilde{x}+\delta -x\right) ^{\frac{\gamma -1}{3}}v,%
\text{ for }v\text{ small.}
\end{equation*}

This implies $\Phi \left( x,0\right) >0$ for $x\in \lbrack \tilde{x},\tilde{x%
}+\delta ),$ since otherwise, $\Phi \left( x,v\right) <0$ for some $v<0$, a
contradiction. In particular, we obtain $\Phi \left( \tilde{x},0\right) >0$
for any $0<\tilde{x}<x^{\ast }.$ Finally, the boundary condition (\ref{comp1}%
) for $\Phi $ and similar arguments applying the forward parabolic equation $%
-v\Phi _{x}=\Phi _{vv},\ v<0$ yield that $\Phi \left( x,v\right) >0$ for all
$\left( x,v\right) \neq \left( 0,0\right) .$ This contradicts the assumption
$\Phi \left( \hat{x},\hat{v}\right) =0$ and the result follows.
\end{proof}

We will use also comparison arguments in a class of domains whose boundary
contains the singular point. We define the following class of domains:%
\begin{equation}
\Lambda _{R}=\left\{ \left( x,v\right) :0<x<R\ ,\ -R^{\frac{1}{3}}<v<rR^{%
\frac{1}{3}}\right\} \ ,\ \ R>0  \label{LambdaSing}
\end{equation}%
as well as the admissible boundary:%
\begin{equation*}
\partial _{a}\Lambda _{R}=\left( \left( 0,R\right) \times \left\{ -R^{\frac{1%
}{3}}\right\} \right) \cup \left( \left( 0,R\right) \times \left\{ rR^{\frac{%
1}{3}}\right\} \right) \cup \left( \left\{ R\right\} \times \left( 0,rR^{%
\frac{1}{3}}\right) \right)
\end{equation*}

We have the following result.

\begin{proposition}
\label{ComparisonSingP}Suppose that $\varphi \in C\left( \overline{\Lambda
_{R}}\right) $ satisfies $\mathcal{L}\varphi =0$ in $Y$ in the sense of
Definition \ref{LadjDef} as well as the boundary condition (\ref{comp1}%
), $\varphi \geq 0$ on $\partial _{a}\Lambda _{R}$. If $r<r_{c}$ we will
assume in addition that $\varphi \left( 0,0\right) =0.$ Then $\varphi \geq 0$
in $\overline{\Lambda _{R}}.$
\end{proposition}

\begin{proof}
We can apply the comparison result Proposition \ref{weak_max} in the domain $%
\Lambda _{R}\diagdown \mathcal{R}_{\delta }$. Applying then Proposition \ref%
{weak_max} to the function $\tilde{\varphi}+\varepsilon $ we obtain $\varphi
\geq -\varepsilon $ in $\Lambda _{R}\diagdown \mathcal{R}_{\delta }.$ Taking
the limit $\varepsilon \rightarrow 0$ and $\delta \rightarrow 0$ we obtain
the result.
\end{proof}

\subsubsection{Asymptotic behaviour of the solutions of some PDEs near the singular point.}

In several of the arguments we will need detailed information about the
asymptotics of the solutions of the problem $\mathcal{L}\varphi =h,$ where $%
\varphi ,\ h\in C\left( X\right) $ and $\mathcal{L}\varphi $ is as in
Definition \ref{LadjDef}. We now prove several results which will be used
later in order to derive such information.

\begin{lemma}
\label{super_sol}Suppose that $0<r\leq 1.$ There exists a function $S\in
C^{2}\left( \mathcal{U}\right) \cap C\left( \overline{\mathcal{U}}\right) $
such that%
\begin{eqnarray}
\mathcal{L}S &=&1\text{ in }\mathcal{U}\   \label{T4E3} \\
S\left( 0,-v\right) &=&S\left( 0,rv\right) \ \ \text{for all }v>0
\label{T4E4} \\
S &=&O\left( x^{\frac{2}{3}}+\left\vert v\right\vert ^{2}\right) \ \ \text{%
as\ }\left( x,v\right) \rightarrow 0\text{\ }  \label{T4E5}
\end{eqnarray}

Moreover, if $r=1$ we have $S\left( x,v\right) =\frac{v^{2}}{2}.$ If $r<1,$
the function $S\left( x,v\right) $ takes positive and negative values in $%
\mathbb{R}^{+}\times \mathbb{R}$.
\end{lemma}

\begin{proof}
We look for $S$ of the following form
\begin{equation}
S=x^{2/3}Q(\frac{v}{x^{1/3}})\   \label{T4E7}
\end{equation}%
Since $\partial _{x}S=\frac{2}{3}x^{-1/3}Q-\frac{1}{3}vx^{-2/3}Q^{\prime }$
and $\partial _{vv}S=Q^{\prime \prime }$, $\partial _{v}^{2}S+v\partial
_{x}S=1,$ the equation (\ref{T4E3}) yields:
\begin{equation}
Q^{\prime \prime }-\frac{1}{3}z^{2}Q^{\prime }+\frac{2}{3}zQ=1.  \label{Q}
\end{equation}

This equation has an explicit solution:%
\begin{equation*}
Q_{p}\left( z\right) =\frac{z^{2}}{2}
\end{equation*}

On the other hand we recall that the homogeneous solution of (\ref{Q}) can
be obtained by means of Kummer functions. We write $\Psi \left( z\right)
=Q\left( z\right) -Q_{p}\left( z\right) .$ Then:%
\begin{equation}
\Psi ^{\prime \prime }-\frac{1}{3}z^{2}\Psi ^{\prime }+\frac{2}{3}z\Psi =0\
\label{T4E6}
\end{equation}

It is readily seen that the solution of (\ref{T4E6}) is given by:%
\begin{equation*}
\Psi \left( z\right) =c_{1}M\left( -\frac{2}{3},\frac{2}{3};\frac{z^{3}}{9}%
\right) +c_{2}U\left( -\frac{2}{3},\frac{2}{3};\frac{z^{3}}{9}\right)
\end{equation*}%
with $c_{1},c_{2}\in \mathbb{R}$.\ In order to obtain (\ref{T4E5}) we must
choose $c_{1}=0,$ in order to avoid the exponential growth of $\Psi \left(
z\right) .$ We notice that, since $U\left( -\frac{2}{3},\frac{2}{3};\frac{%
z^{3}}{9}\right) $ solves (\ref{T4E6}), it is analytic in the complex plane $%
z\in \mathbb{C}$. Using Proposition 3.1, and more precisely Remark %
3.3 in \cite{HVJ2}, we obtain:\
\begin{eqnarray*}
S\left( 0,v\right) &=&\left( \frac{1}{2}+c_{2}\right) v^{2}\ \ ,\ \ v>0 \\
S\left( 0,v\right) &=&\left( \frac{1}{2}-2c_{2}\right) v^{2}\ \ ,\ \ v<0
\end{eqnarray*}%
where we use that $K_{\frac{2}{3}}=2\cos \left( \pi \right) =-2.$ We then
impose the boundary condition $S(0,-v)=S(0,rv)$ for $v>0$ (cf. (\ref{T4E4}%
)), whence:%
\begin{equation*}
\left( \frac{1}{2}-2c_{2}\right) =\left( \frac{1}{2}+c_{2}\right) r^{2}
\end{equation*}

Therefore $c_{2}=\frac{\left( 1-r^{2}\right) }{2\left( 2+r^{2}\right) }.$ It
then follows that\
\begin{equation*}
Q\left( z\right) =\frac{z^{2}}{2}+\frac{\left( 1-r^{2}\right) }{2\left(
2+r^{2}\right) }U\left( -\frac{2}{3},\frac{2}{3};\frac{z^{3}}{9}\right)
\end{equation*}

Notice that, if $r=1$ we obtain $Q\left( z\right) =\frac{z^{2}}{2}.$ If $%
r\neq 1$ we obtain, using 13.5.10 in \cite{Ab} that:
\begin{equation}
Q\left( 0\right) =\frac{\left( 1-r^{2}\right) \Gamma \left( \frac{1}{3}%
\right) }{2\left( 2+r^{2}\right) \Gamma \left( -\frac{1}{3}\right) }
\label{P1E1}
\end{equation}

Then $Q\left( 0\right) <0$ if $r<1.$ On the other hand, using that $U\left( -%
\frac{2}{3},\frac{2}{3};\xi \right) \sim \xi ^{\frac{2}{3}}$ as $\xi
\rightarrow \infty $ it follows that $Q\left( z\right) \sim \left( \frac{1}{2%
}+\frac{\left( 1-r^{2}\right) }{2\left( 2+r^{2}\right) 9^{\frac{2}{3}}}%
\right) z^{2}$ as $z\rightarrow \infty .$ Therefore $Q\left( z\right) >0$
for large values of $\left\vert z\right\vert $ and there exists a range of
values for which $Q\left( z\right) <0.$
\end{proof}

We now prove some Liouville's type Theorems which will be useful deriving
asymptotic formulas near the singular point of the solutions of equations
like $\lambda \psi -\mathcal{L}\psi =g$ with boundary conditions $\psi
\left( 0,rv\right) =\psi \left( 0,-v\right) \ \ ,\ \ v>0$.

\begin{theorem}
\label{Liouville}(i) Let $0<r<r_{c}.$ Suppose that $\psi $ is a classical
solution of $\mathcal{L}\psi =0$ in $\left\{ x>0,\ v\in \mathbb{R}\right\} $
satisfying $\psi \left( 0,rv\right) =\psi \left( 0,-v\right) ,\ \ v>0$ as
well as the estimate $\left\vert \psi \left( x,v\right) \right\vert \leq
C\left( \left\vert x\right\vert +\left\vert v\right\vert ^{3}\right) ^{b}$
for some $b\in \left( 0,\beta \right) \cup \left( \beta ,\frac{2}{3}\right] $
and $\left( x,v\right) \in \left\{ x>0,\ v\in \mathbb{R}\right\} .$ Then $%
\psi =0.$

(ii) There is not any bounded function $\psi $ defined in $\mathbb{R}^{2}$
which satisfies also one equation with the form $\mathcal{\bar{L}}\psi =1,$
where $\mathcal{\bar{L}}$ is one of the operators in the set
\begin{equation*}
\left\{ D_{v}^{2}+vD_{x},\ D_{v}^{2}+D_{x},\ D_{v}^{2}-D_{x}\right\} .
\end{equation*}

(iii) Let $r>r_{c}.$ Suppose that we have a solution of $\mathcal{L}\psi =0$
in $\left\{ x>0,\ v\in \mathbb{R}\right\} $ satisfying $\psi \left(
0,rv\right) =\psi \left( 0,-v\right) \ \ ,\ \ v>0$ as well as the estimate $%
\left\vert \psi \left( x,v\right) \right\vert \leq C\left( \left\vert
x\right\vert +\left\vert v\right\vert ^{3}\right) ^{b}$ with $0<b\leq \frac{2%
}{3}.$ Then $\psi =0.$

(iv) Let $0<r<r_{c}.$ Suppose that $\psi $ is a classical solution of $%
\mathcal{L}\psi =0$ in $\left\{ x>0,\ v\in \mathbb{R}\right\} $ satisfying $%
\psi \left( 0,rv\right) =\psi \left( 0,-v\right) ,\ \ v>0$ as well as the
estimate $\left\vert \psi \left( x,v\right) \right\vert \leq C$ for $\left(
x,v\right) \in \left\{ x>0,\ v\in \mathbb{R}\right\} .$ Then $\psi =K_{0}$
for some $K_{0}\in \mathbb{R}$.
\end{theorem}

The proof of Theorem \ref{Liouville} will require to prove some
auxiliary results, which will be required in order to deal with the cases
(i), (iii).

We first prove the following result:

\begin{proposition}
\label{homogen}Suppose that $\varphi $ solves $\mathcal{L}\left( \varphi
\right) =0$ in $x\geq 0,\ v\in \mathbb{R}$ and satisfies the the boundary
conditions $\varphi \left( 0,rv\right) =\varphi \left( 0,-v\right) \ \ ,\ \
v>0.$ Suppose also that we have the estimates:%
\begin{eqnarray}
\left\vert \varphi \left( x,v\right) \right\vert &\leq &K\left( \left\vert
x\right\vert +\left\vert v\right\vert ^{3}\right) ^{b}\ \ \text{with\ }0\leq
b\leq \frac{2}{3}\ \ \text{in\ }x\geq 0\ \ ,\ \ v\in \mathbb{R}  \label{C1}
\\
\sup_{\left( \left\vert x\right\vert +\left\vert v\right\vert ^{3}\right)
=1}\left\vert \varphi \left( x,v\right) \right\vert &=&1  \label{C2}
\end{eqnarray}%
for some $K>1.$ Then $b=\beta $ or $b=0.$
\end{proposition}

This proposition will be proved with the help of the following result.

\begin{proposition}
\label{classif}Suppose that $\varphi $ solves $\mathcal{L}\left( \varphi
\right) =0$ in $x\geq 0,\ v\in \mathbb{R}$ and satisfies the the boundary
conditions $\varphi \left( 0,rv\right) =\varphi \left( 0,-v\right) ,\ v>0.$
Suppose also that $\varphi $ satisfies
\begin{equation}
\varphi \left( \mu ^{3}x,\mu v\right) =\mu ^{3b}\varphi \left( x,v\right)
\label{C3}
\end{equation}%
with $0\leq b\leq \frac{2}{3}.$ Then, if $0<r<r_{c}$ we have that $\varphi
=0 $ unless $b=\beta $ or $b=0.$ If $r>r_{c}$ we have $\varphi =0$ unless $%
b=0$.
\end{proposition}

It is convenient to introduce a different group of variables in order to
simplify the homogeneity of the functions. We define $x=\zeta ^{3}.$ Then $%
\mathcal{L}\left( \varphi \right) =0$ if and only if:%
\begin{equation}
3\zeta ^{2}\varphi _{vv}+v\varphi _{\zeta }=0\ \ ,\ \ \varphi \left(
0,rv\right) =\varphi \left( 0,-v\right) \ \ ,\ \ v>0  \label{EqResc}
\end{equation}

In order to prove Proposition \ref{homogen} we will use the following
auxiliary result.

\begin{lemma}
\label{Mellin}Let us denote as $\mathcal{Y}$ the set of functions with the
form $G:\mathcal{U}\times \mathbb{R}\rightarrow \mathbb{C}$ satisfying the
following properties:

(i) $G\left( X,\omega \right) $ satisfies the homogeneity condition $G\left(
\mu X,\omega \right) =G\left( X,\omega \right) $ for any $\mu >0.$

(ii) The support of $G\left( X,\omega \right) $ is contained in the set $%
K\times \left[ -M,M\right] $ where $K\subset \mathcal{U}$ is a compact set
and $0<M<\infty .$

(iii) $G\in C^{\infty }\left( \mathcal{U}\times \mathbb{R}\right) .$

Then, the set of functions with the form:%
\begin{equation}
F\left( \zeta ,v\right) =F\left( X\right) =\int_{\mathbb{R}}\left\vert
X\right\vert ^{i\omega }G\left( X,\omega \right) d\omega \ \ ,\ \ X=\left(
\zeta ,v\right) \   \label{MeRep}
\end{equation}%
is dense in $L^{1}\left( \mathcal{U};\frac{dX}{\left\vert X\right\vert ^{2}}%
\right) $ where $dX=d\zeta dv.$
\end{lemma}

\begin{proof}
Since $G$ is homogeneous, the representation formula (\ref{MeRep}) is
equivalent to the representation of the function $f\left( r,\theta \right)
=F\left( X\right) $ with $r=\left\vert X\right\vert $ and $\theta \in \left[
-\frac{\pi }{2},\frac{\pi }{2}\right] $ as angular coordinate, in terms of
the Mellin transform. Moreover, using the change of variables $\tilde{f}%
\left( t,\theta \right) =f\left( r,\theta \right) ,$ $t=\log \left( r\right)
$ and writing $G\left( X,\omega \right) =g\left( \omega ,\theta \right) $
the problem becomes equivalent to proving that the class of functions:%
\begin{equation}
\tilde{f}\left( t,\theta \right) =\int_{\mathbb{R}}e^{i\omega t}g\left(
\omega ,\theta \right) d\omega \   \label{SMShape}
\end{equation}%
with $g\in C_{c}^{\infty }\left( \mathbb{R}\times \left[ -\frac{\pi }{2},%
\frac{\pi }{2}\right] \right) $ is dense in $L^{1}\left( \mathbb{R}\times %
\left[ -\frac{\pi }{2},\frac{\pi }{2}\right] \right) .$ This follows using
first that the compactly supported functions are dense in $L^{1}\left(
\mathbb{R}\times \left[ -\frac{\pi }{2},\frac{\pi }{2}\right] \right) $ and
using then that the functions with the form (\ref{SMShape}) are dense in $%
L^{1}\left( J\right) $ for any compact set $J\subset \mathbb{R}\times \left[
-\frac{\pi }{2},\frac{\pi }{2}\right] .$
\end{proof}

The next Lemma shows, by means of an explicit computation, that there are
no homogeneous solutions of (\ref{EqResc}) for a set of homogeneities and
values of $r.$

\begin{lemma}
\label{LiouvHom}Suppose that $\varphi $ satisfies (\ref{EqResc}) and that $%
\varphi $ satisfies the homogeneity condition $\varphi \left( \mu X\right)
=\mu ^{3b}\varphi \left( X\right) ,\ X=\left( \zeta ,v\right) \in \mathcal{U}%
,$ for any $\mu >0$ where $\operatorname{Re}\left( b\right) \in \left[ 0,\beta
\right) \cup \left( \beta ,\frac{2}{3}\right] $ if $0<r<r_{c}$ and $0\leq
\operatorname{Re}\left( b\right) \leq \frac{2}{3}$ if $r_{c}<r\leq 1.$ Then $\varphi
\left( \zeta ,v\right) \equiv 0$ unlesss $b=0.$
\end{lemma}

\begin{remark}
In the case $r_{c}<r$ it is possible to prove the Lemma for $r_{c}<r\leq
\sqrt{2}.$ Notice that, if $\operatorname{Im}\left( b\right) \neq 0$ we must
necessarily have $\varphi $ complex.
\end{remark}

\begin{proof}
Using the change of variables $x=\zeta ^{3},$ it follows that the solutions
with the homogeneity indicated in Lemma \ref{LiouvHom} have the form 
(4.9) in \cite{HVJ2}. Replacing $\beta $ by $z$ to avoid ambiguities, we would obtain that $%
\varphi =\varphi \left( x,v\right) $ has the form $\varphi \left( x,v\right)
=x^{z}F\left( \frac{v^{3}}{9x}\right) $ for some $z\in \mathbb{C}$ where $%
\operatorname{Re}\left( z\right) \in \left[ 0,\frac{2}{3}\right] .$ Our goal is to
show that the only possible values of $z$ are $z=\beta \left( r\right) $ and
$z=0$ if $0<r<r_{c}$ and $z=0$ if $r_{c}<r\leq 1.$ To this end we argue as
in the Proof of Proposition 4.2 in \cite{HVJ2} to prove that
\begin{equation}
r^{3z}=2\sin \left( \pi \left( \frac{1}{6}-z\right) \right)  \label{EqR}
\end{equation}

We then need to show that (\ref{EqR}) has only the roots indicated above. To
this end we argue as follows. If $r>0$ is sufficiently small the only roots
of (\ref{EqR}) in the strip $0\leq \operatorname{Re}\left( z\right) \leq \frac{2}{3}$
are at $z=0$ and $z_{\ast }\left( r\right) $ in a neighbourhood of $z=\frac{1%
}{6}.$ Indeed, it follows from the Implicit Functions Theorem that for any $%
\delta >0$ there is only one solution of (\ref{EqR}) in the region $\operatorname{Re}%
\left( z\right) \geq \delta >0$ if $r$ is small, and such root converges to $%
z=\frac{1}{6}$ as $r\rightarrow 0.$ In order to study the possible roots
with $\operatorname{Re}\left( z\right) <\delta $ we write $z=A+iB,\ A,B\in \mathbb{R}
$ with $0\leq A<\delta .$ Then (\ref{EqR}) becomes:%
\begin{eqnarray}
&&e^{3A\log \left( r\right) }\left( \cos \left( 3B\log \left( r\right)
\right) +i\sin \left( 3B\log \left( r\right) \right) \right)  \label{ReP} \\
&=&2\sin \left( \pi \left( \frac{1}{6}-A\right) \right) \cosh \left( \pi
B\right) -2i\cos \left( \pi \left( \frac{1}{6}-A\right) \right) \sinh \left(
\pi B\right)  \notag
\end{eqnarray}

Estimating $\cos \left( 3B\log \left( r\right) \right) $ and $\cosh \left(
\pi B\right) $ it follows that the real part of the left-hand side of (\ref%
{ReP}) is smaller than $e^{3A\log \left( r\right) }$ and the real part of
the left hand side of (\ref{ReP}) is larger than $2\sin \left( \pi \left(
\frac{1}{6}-A\right) \right) .$ It follows, using Taylor's Theorem to
approximate this functions for $A>0$ small that $2\sin \left( \pi \left(
\frac{1}{6}-A\right) \right) \geq 1-C_{0}A$ for some constant $C_{0}>0$
independent of $r.$ Since $e^{3A\log \left( r\right) }<1-C_{0}A$ for $r$
small enough it follows that the only solution of (\ref{ReP}) with $A<\delta
$ has $A=0$ if $r$ is small. Then (\ref{ReP}) becomes:%
\begin{equation*}
\cos \left( 3B\log \left( r\right) \right) +i\sin \left( 3B\log \left(
r\right) \right) =\cosh \left( \pi B\right) -i\sqrt{3}\sinh \left( \pi
B\right)
\end{equation*}%
and using the fact that $\cos \left( 3B\log \left( r\right) \right) \leq
1<\cosh \left( \pi B\right) $ if $B>0$ we obtain that $B=0.$

Notice that for any fixed $r<1$ we have $\left\vert 2\sin \left( \pi \left(
\frac{1}{6}-z\right) \right) \right\vert >\left\vert r^{3z}\right\vert $ if $%
z$ is large enough. Standard properties of the zeros of analytic functions
then imply that the number of solutions of (\ref{EqR}) for each $r>0$ in $%
0\leq \operatorname{Re}\left( z\right) \leq \frac{2}{3}$ can change for the values
of $r$ for which there are solutions of (\ref{EqR}) in $\operatorname{Re}\left(
z\right) =0$ or $\operatorname{Re}\left( z\right) =\frac{2}{3}.$ Writing $z=iB,\
B\in \mathbb{R}$ we obtain that (\ref{EqR}) becomes (cf. (\ref{ReP})):%
\begin{equation*}
\cos \left( 3B\log \left( r\right) \right) +i\sin \left( 3B\log \left(
r\right) \right) =\cosh \left( \pi B\right) -i\sqrt{3}\sinh \left( \pi
B\right)
\end{equation*}%
and using again that $\cos \left( 3B\log \left( r\right) \right) \leq
1<\cosh \left( \pi B\right) $ it follows that there are not solutions of (%
\ref{EqR}) with $\operatorname{Re}\left( z\right) =0$ except $z=0.$ We now prove
that there are not solutions of (\ref{EqR}) neither with $\operatorname{Re}\left(
z\right) =\frac{2}{3}.$ To this end we write $z=\frac{2}{3}+iB,\ B\in
\mathbb{R}$ and (\ref{EqR}) becomes:%
\begin{equation*}
r^{2}\left( \cos \left( 3B\log \left( r\right) \right) +i\sin \left( 3B\log
\left( r\right) \right) \right) =-2\cosh \left( \pi B\right)
\end{equation*}

The solutions of this equation satisfy $\sin \left( 3B\log \left( r\right)
\right) =0,$ whence $3B\log \left( r\right) =n\pi ,$ $n\in \mathbb{Z}$. Then
$r^{2}\cos \left( n\pi \right) =-2\cosh \left( \frac{n\pi ^{2}}{3\log \left(
r\right) }\right) ,$ whence $n=2\ell +1,\ \ell \in \mathbb{Z}$ and $%
r^{2}=2\cosh \left( \frac{\left( 2\ell +1\right) \pi ^{2}}{3\log \left(
r\right) }\right) .$ This equation does not have any solution with $r\leq
\sqrt{2}.$ It then follows that the number of roots of (\ref{EqR}) in the
strip $0\leq \operatorname{Re}\left( z\right) \leq \frac{2}{3}$ can change only if a
multiple root is formed at $z=0$ for some value of $r.$ A simple computation
shows that this happens only for $r=r_{c},$ in which case there is a double
root (\ref{EqR}) at $z=0.$ Therefore there is one root of (\ref{EqR}) in $0<%
\operatorname{Re}\left( z\right) \leq \frac{2}{3}$ if $r<r_{c}$ and no roots if $%
r_{c}<r\leq 1,$ whence the result follows.
\end{proof}

Proposition \ref{classif} then follows from Lemma \ref{LiouvHom}.

We will use the following standard abstract result of Functional Analysis,
which is just an adaptation of Freldhom alternative adapted to our specific
situation.

\begin{proposition}
\label{FredAl}Suppose that $X$ is a reflexive Banach space. Let $A:D\left(
A\right) \subset X\rightarrow X$ be a linear closed operator with domain $%
D\left( A\right) $ dense in $X$. Let $A^{\ast }$ be the adjoint of $A,$ with
domain $D\left( A^{\ast }\right) \subset X^{\ast },$ with $D\left( A^{\ast
}\right) $ dense in $X^{\ast }.$ If $Ker\left( A\right) =0,$ then $R\left(
D\left( A^{\ast }\right) \right) $ is dense in $X^{\ast }$.
\end{proposition}

\begin{proof}
Suppose that $R\left( D\left( A^{\ast }\right) \right) $ is not dense in $%
X^{\ast }.$ Then $\overline{R\left( D\left( A^{\ast }\right) \right) }\neq
X^{\ast }.$ Therefore, Hahn-Banach implies the existence of an element $\ell
\in X^{\ast \ast }$ such that $\left\langle \ell ,\eta \right\rangle =0$ for
any $\eta \in \overline{R\left( D\left( A^{\ast }\right) \right) }$ and with
$\ell \neq 0.$ Since $X$ is reflexive we have $X^{\ast \ast }=X.$ Then,
there exist $x_{0}\in X,\ x_{0}\neq 0$ such that $\left\langle \eta
,x_{0}\right\rangle =0$ for any $\eta \in \overline{R\left( D\left( A^{\ast
}\right) \right) }.$ In particular, given any $z\in D\left( A^{\ast }\right)
$ we have $0=\left\langle A^{\ast }z,x_{0}\right\rangle .$ Therefore, using
the definition of adjoint operator, it follows that $x_{0}\in D\left(
A^{\ast \ast }\right) $ and $A^{\ast \ast }x_{0}=0.$ Theorem 3.24 in 
\cite{Brezis} implies $A^{\ast \ast }=A$. Then $Ax_{0}=0.$ Since $Ker\left(
A\right) =0$ we have $x_{0}=0$ and this gives a contradiction, whence $%
R\left( D\left( A^{\ast }\right) \right) $ is dense in $X^{\ast }.$
\end{proof}

We can use now Lemma \ref{Mellin} to show that the solutions of (\ref{EqResc}%
) satisfying a suitable boundedness condition are homogeneous. It then
follows, using Lemma\ \ref{LiouvHom} that they vanish for suitable ranges of
parameters. More precisely, we have the following result.

\begin{lemma}
Suppose that $\varphi =\varphi \left( \zeta ,v\right) $ satisfies (\ref%
{EqResc}) as well as the inequality:%
\begin{equation*}
\left\vert \varphi \left( \zeta ,v\right) \right\vert \leq K\left(
\left\vert \zeta \right\vert +\left\vert v\right\vert \right) ^{3b}\ \text{%
in\ }x\geq 0\ \ ,\ \ v\in \mathbb{R}\text{ ,\ }K<\infty
\end{equation*}%
with $\operatorname{Re}\left( b\right) \in \left[ 0,\beta \right) \cup \left( \beta ,%
\frac{2}{3}\right] $ if $0<r<r_{c}$ and $0\leq b\leq \frac{2}{3}$ if $%
r_{c}<r\leq 1.$ Then $\varphi \left( \zeta ,v\right) \equiv 0$ unless $b=0.$
\end{lemma}

\begin{proof}
We multiply (\ref{EqResc}) by a $C^{\infty }$ test function $F\left( \zeta
,v\right) $ with compact support in $\mathcal{U}$ and satisfying
\begin{equation}
F\left( 0,v\right) =r^{-2}F\left( 0,-\frac{v}{r}\right) \ \ ,\ \ v>0
\label{compCond}
\end{equation}%
Integrating by parts we obtain, with the help of (\ref{compCond}) and using
also that $\varphi \left( 0,rv\right) =\varphi \left( 0,-v\right) \ $for$\
v>0$ we obtain:\
\begin{equation}
\int_{\mathcal{U}}\mathcal{B}^{\ast }F\varphi d\zeta dv=0  \label{BF=0}
\end{equation}%
where from now on, we write, by shortness:%
\begin{eqnarray*}
\mathcal{B} &=&3\zeta ^{2}\partial _{vv}+v\partial _{\zeta } \\
\mathcal{B}^{\ast } &=&3\zeta ^{2}\partial _{vv}-v\partial _{\zeta }
\end{eqnarray*}%
We take $F\left( \zeta ,v\right) $ with the following form
\begin{equation*}
F\left( \zeta ,v\right) =F(X)=\int_{\mathbb{R}}\left\vert X\right\vert
^{\alpha +i\omega }G\left( X,\omega \right) d\omega \ \ ,\ \ X=\left( \zeta
,v\right)
\end{equation*}%
with $G\left( X,\omega \right) $ to be specified. Then \eqref{BF=0} becomes
\begin{equation}
\int_{\mathbb{R}}\int_{\mathcal{U}}\mathcal{B}^{\ast }\left( \left\vert
X\right\vert ^{\alpha +i\omega }G\left( X,\omega \right) \right) \varphi
d\zeta dvd\omega =0  \label{BF=01}
\end{equation}%
Notice that if $G$ satisfies:%
\begin{equation}
G\left( 0,v,\omega \right) =\frac{1}{\left\vert r\right\vert ^{i\omega
+2+\alpha }}G\left( 0,-\frac{v}{r},\omega \right) \ \ ,\ \ v>0\   \label{D8}
\end{equation}%
we obtain that $F$ satisfies (\ref{compCond}).

We choose $\alpha =-3b-2.$ Then $g\left( X\right) =\left\vert X\right\vert
^{\alpha +2}\varphi \left( X\right) $ is globally bounded by the assumption
on $\varphi$. We write:%
\begin{eqnarray}
\eta \left( X,\omega \right) &=& \frac{1}{\left\vert X\right\vert ^{\alpha+i
\omega}} \mathcal{B}^{\ast } \left( \left\vert X\right\vert ^{\alpha+i
\omega} G\left( X,\omega \right) \right)  \notag \\
\xi \left( X\right) &=&\int_{\mathbb{R}}\left\vert X\right\vert ^{i\omega
}\eta \left( X,\omega \right) d\omega  \label{D9}
\end{eqnarray}

Then, (\ref{BF=01}) implies:%
\begin{equation}
\int_{\mathcal{U}}\xi \left( X\right) g\left( X\right) \frac{dX}{\left\vert
X\right\vert ^{2}}=0\   \label{D9a}
\end{equation}%
for any function $\xi \left( \cdot \right) $ with the form (\ref{D9}). Our
goal is to show that if $b\leq 0,$ the functions $\xi $ in \eqref{D9} can be
chosen in a dense set in $L^{1}\left( \mathcal{U};\frac{dX}{\left\vert
X\right\vert ^{2}}\right) $, which will imply $g\left( X\right) =0$ and
hence $\varphi =0$. To this end, we define the following operators for each $%
\omega \in \mathbb{R}$:%
\begin{equation}
\mathcal{A}^{\ast }\left( \omega \right) =\frac{1}{\left\vert X\right\vert
^{\alpha +2}}\mathcal{B}^{\ast }\left( \left\vert X\right\vert ^{\alpha
+i\omega }G\left( X,\omega \right) \right)  \label{E1}
\end{equation}%
To make precise the definition of the operators $\mathcal{A}^{\ast }$ we
need to specify their domain. To this end we define a curve in parametric
coordinates by means of $\left\{ r=\lambda \left( \theta \right) :\theta \in %
\left[ -\frac{\pi }{2},\frac{\pi }{2}\right] \right\} ,$ with $\lambda
\left( -\frac{\pi }{2}\right) =\frac{1}{r},\lambda \left( \frac{\pi }{2}%
\right) =1,\lambda ^{\prime }\left( \theta \right) \leq 0.$ We then define
domains%
\begin{equation*}
\mathcal{D}=\left\{ \lambda \left( \theta \right) <\left\vert X\right\vert
<2\lambda \left( \theta \right) :\theta \in \left[ -\frac{\pi }{2},\frac{\pi
}{2}\right] \right\} .
\end{equation*}%
The operator $\mathcal{A}^{\ast }\left( \omega \right) $ in (\ref{E1}) will
be defined in the space of functions:%
\begin{equation*}
Y=\left\{ G\in L^{2}\left( \mathcal{D}\right) :G\left( \mu X\right) =G\left(
X\right) \text{ \ for }\mu >0,~\ X,\mu X\in \mathcal{D}\right\}
\end{equation*}%
with domain:%
\begin{equation*}
D\left( \mathcal{A}^{\ast }\left( \omega \right) \right) =\left\{ G\in
Y,~~\partial _{v}G,\partial _{v}^{2}G,\partial _{\zeta }G\in L^{2}\left(
\mathcal{D}\right) ,~~\text{(\ref{D8}) holds}\right\}
\end{equation*}%
Note that the boundary condition (\ref{D8}) in the sense of traces is
meaningful due to the regularity of the functions in $D\left( \mathcal{A}%
^{\ast }\left( \omega \right) \right) .$ The adjoint of $\mathcal{A}^{\ast
}\left( \omega \right) $ is given by the operator:%
\begin{equation}
\mathcal{A}\left( \omega \right) =\left\vert X\right\vert ^{\alpha +i\omega }%
\mathcal{B}\left( \frac{\tilde{G}\left( X,\omega \right) }{\left\vert
X\right\vert ^{\alpha +2}}\right) \   \label{E2}
\end{equation}%
with the boundary condition%
\begin{equation}
\tilde{G}\left( 0,v,\omega \right) =r^{\alpha +2}\tilde{G}\left( 0,-\frac{v}{%
r},\omega \right) \ ,\ v>0\   \label{E3}
\end{equation}%
and the domain of $\mathcal{A}\left( \omega \right) $ is given by:%
\begin{equation*}
D\left( \mathcal{A}\left( \omega \right) \right) =\left\{ \tilde{G}\in Y,\
\partial _{v}\tilde{G},\partial _{v}^{2}\tilde{G},\partial _{\zeta }\tilde{G}%
\in L^{2}\left( \mathcal{D}\right) ,\ \text{(\ref{E3}) holds}\right\}
\end{equation*}

We can then apply Proposition \ref{FredAl} which implies that $R\left(D\left(
\mathcal{A}^{\ast }\left( \omega \right) \right) \right) $ is dense in $L^{2}\left(
\mathcal{D}\right) $ for any $\omega \in \mathbb{R}$ such that $Ker\left(
\mathcal{A}\left( \omega \right) \right) =\left\{ 0\right\} .$ Notice that $%
\mathcal{A}\left( \omega \right) \tilde{G}\left( \zeta ,v\right) =0$ if and
only if $\varphi \left( \zeta ,v\right) =\varphi \left( X\right) =\left\vert
X\right\vert ^{-\left( \alpha +2\right) }\tilde{G}\left( X,\omega \right) $
solves $\mathcal{L}\left( \varphi \right) =0$ in $x\geq 0,\ v\in \mathbb{R}$
and $\varphi \left( 0,rv\right) =\varphi \left( 0,-v\right) ,\ \ v>0,$ as
well as (\ref{C3}). Since $b\neq 0,~$Proposition \ref{classif} then implies
that $\tilde{G}\left( \zeta ,v\right) =0,$ whence $R \left(D \left( \mathcal{A}^{\ast
}\left( \omega \right) \right) \right) $ is dense in $L^{2}\left( \mathcal{D}\right)
.$\ Now let
\begin{equation*}
\tilde{\mathcal{A}}^{\ast }\left( \omega \right) =\left\vert X\right\vert
^{2-i\omega }\mathcal{A}^{\ast }\left( \omega \right) =\frac{1}{\left\vert
X\right\vert ^{\alpha +i\omega }}\mathcal{B}^{\ast }\left( \left\vert
X\right\vert ^{\alpha +i\omega }G\left( X,\omega \right) \right)
\end{equation*}%
Notice that $\tilde{\mathcal{A}}^{\ast }G(\mu X)=\tilde{\mathcal{A}}^{\ast
}G(X)$. Then $R\left( \tilde{\mathcal{A}}^{\ast }\left( \omega \right)
\right) $ is also dense in $L^{2}\left( \mathcal{D}\right) .$ Therefore, for
any $\varepsilon >0$ and any function $\eta \in C^{\infty }\left( \mathbb{R}%
:L^{2}\left( \mathcal{D}\right) \right) ,$ supported in $\left\vert \omega
\right\vert \leq M,$ we can find a family of functions $G\left( X,\omega
\right) ,\ \omega \in \mathbb{R}$, such that $\left\Vert \tilde{\mathcal{A}}%
^{\ast }\left( \omega \right) G\left( \cdot ,\omega \right) -\eta \left(
\cdot ,\omega \right) \right\Vert _{L^{2}\left( \mathcal{D}\right)
}<\varepsilon $ for any $\left\vert \omega \right\vert \leq M.$ Therefore,
Lemma \ref{Mellin} implies that the functions $\xi \left( X\right) $ in (\ref%
{D9}) can be chosen in a dense set in $L^{1}\left( \mathcal{U};\frac{dX}{%
\left\vert X\right\vert ^{2}}\right) .$ Therefore (\ref{D9a}) implies that $%
g\left( X\right) =0$, whence $\varphi =0$ and the result follows.
\end{proof}

\begin{proof}[Proof of Theorem \protect\ref{Liouville}]
We consider first the cases (i) and (iii). The result in this case follows
from Proposition \ref{homogen}.

In the case (ii) we suppose first that $\psi $ is a bounded function
satisfying $\left( D_{v}^{2}-D_{x}\right) \psi =1$ in $\left( x,v\right) \in
\mathbb{R}^{2}.$ Therefore, using standard parabolic theory we can write:%
\begin{equation*}
\psi \left( x,v\right) =-x+w\left( x,v\right) \ \ \text{for \ }x\geq 0
\end{equation*}%
where $\left( D_{v}^{2}-D_{x}\right) w=0$ and $w\left( 0,v\right) =\psi
\left( 0,v\right) ,\ \psi \left( 0,\cdot \right) $ bounded. The classical
theory of the heat equation establishes that $w$ is uniquely determined and $%
\left\vert w\left( x,v\right) \right\vert \leq \sup_{v\in \mathbb{R}%
}\left\vert \psi \left( 0,v\right) \right\vert .$ But $\lim_{x\rightarrow
\infty }\psi \left( x,v\right) $ $=-\infty ,$ and this contradicts the
boundedness of $\psi .$ The case $\mathcal{\bar{L}=}%
D_{v}^{2}+vD_{x} $ can be studied similarly. Suppose then that $\mathcal{%
\bar{L}=}D_{v}^{2}+vD_{x}.$ In this case we use the fact that $\psi $ is a
bounded solution of the Kolmogorov equation:%
\begin{equation*}
\psi _{t}=\psi _{vv}+v\psi _{x}-1\ \ ,\ \ t\in \mathbb{R}\text{ , }\left(
x,v\right) \in \mathbb{R}^{2}
\end{equation*}

Therefore, we have the following representation formula:\
\begin{equation}
\psi \left( t,x,v\right) =-t+\int_{\mathbb{R}^{2}}G\left( t,x-\xi ,v-\eta
\right) \psi \left( 0,\xi ,\eta \right) d\xi d\eta \ \text{for }t>0
\label{A8}
\end{equation}%
where $G$ is the fundamental solution for the Kolmogorov equation (cf. \cite%
{K}).\ Then, the integral term in (\ref{A8}) is globally bounded, whence $%
\psi \left( t,x,v\right) $ is unbounded for $t\rightarrow \infty .$ The
contradiction concludes the Proof of (ii).

In order to prove (iv) we argue as follows. By assumption $\psi $ is
bounded. Let $m=\inf \left( \psi \right) ,$ $M=\sup \left( \psi \right) .$
If either of them is reached as an interior point we would obtain that $\psi
$ is constant by the strong maximum principle (cf. Proposition \ref%
{strong_max}). Let $L=\lim \sup_{\left\vert \left( x,v\right) \right\vert
\rightarrow \infty }\psi \left( x,v\right) $ and $\lim \inf_{\left(
x,v\right) \rightarrow \left( 0,0\right) }\psi \left( x,v\right) =\kappa
_{0}.$ By assumption $\kappa _{0},L\in \left[ m,M\right] .$ We can then,
arguing by comparison, use as barrier function $\bar{\psi}=\kappa
_{0}+\sigma F_{\beta }$ in the admissible domain $\mathcal{R}_{R}\diagdown
\mathcal{R}_{\varepsilon }$ with $R$ large and $\varepsilon $ small, with $%
\sigma \rightarrow 0$ as $R\rightarrow \infty ,$ due to the fact that $\psi $
is bounded. It then follows that $\kappa _{0}-\delta \leq \psi \leq \kappa
_{0}+C\sigma $ in bounded sets of $\mathcal{U}$ with $\delta >0$ arbitrarily
small. Taking the limit $R\rightarrow \infty ,$ $\sigma \rightarrow 0,\
\delta \rightarrow 0$ we obtain $\psi =\kappa _{0}$ whence the result
follows.
\end{proof}

We will need the following auxiliary result.

\begin{lemma}
\label{LemSeq}Suppose that $\lambda \left( R\right) $ is a positive
continuous function such that
\begin{equation*}
\lim \inf_{R\rightarrow 0}\frac{\log \left( \lambda \left( R\right) \right)
}{\log \left( R\right) }=\zeta >0,\ \zeta <\infty \text{ \ for some }\zeta
>0.
\end{equation*}%
Suppose that $K>1$ is given$.$ Then there exist sequences $\left\{
R_{n}\right\} ,\ \left\{ \mu _{n}\right\} $ with $R_{n}\rightarrow 0,$ $\mu
_{n}\rightarrow \infty $ as $n\rightarrow \infty $ such that $\frac{\lambda
\left( R\right) }{\lambda \left( R_{n}\right) }\leq K\left( \frac{R}{R_{n}}%
\right) ^{\zeta }$ for $\frac{1}{\mu _{n}}\leq \frac{R}{R_{n}}\leq \mu _{n}.$
\end{lemma}

\begin{proof}
We claim the following. Suppose that we fix $L>1.$ We then claim that for
any $\delta >0$ (small), there exists $\bar{R}\leq \delta $ such that $\frac{%
\lambda \left( \rho \right) }{\lambda \left( \bar{R}\right) }\leq K\left(
\frac{\rho }{\bar{R}}\right) ^{\zeta }$ for all $\rho \in \left[ \frac{\bar{R%
}}{L},L\bar{R}\right] .$ The existence of the sequences $\left\{
R_{n}\right\} ,\ \left\{ \mu _{n}\right\} $ then follows inductively by
choosing $L=n=\mu _{n},$ $\delta =1$ if $n=1$ and $\delta =\frac{R_{n-1}}{2}$
if $n>1$ and denoting the corresponding $\bar{R}$ as $R_{n}.$

We prove the claim by contradiction. Suppose that it is false. Then, for any
$L>1$ fixed, there exists $R_{0}>0$ such that, for $\bar{R}\leq R_{0}$ there
exists $\rho =\rho \left( \bar{R}\right) \in \left[ \frac{\bar{R}}{L},L\bar{R%
}\right] $ such that $\frac{\lambda \left( \rho \left( \bar{R}\right)
\right) }{\lambda \left( \bar{R}\right) }>K\left( \frac{\rho \left( \bar{R}%
\right) }{\bar{R}}\right) ^{\zeta }.$ We then define, for each $R_{\ast
}\leq R_{0}$ a sequence inductively as follows. We define $R_{1}=\rho \left(
R_{\ast }\right) $ and then $R_{n+1}=\rho \left( R_{n}\right) $ for $n\geq 1$
as long as $R_{n}\leq R_{0}.$ If $\rho \left( R_{n}\right) >R_{0}$ for some $%
n\geq 1$ the iteration finishes.

We now have three possibilities. (i) Given $R_{\ast }\leq R_{0}$, for the
corresponding sequence $\left\{ R_{n}\right\} $ we have $\lim
\inf_{n\rightarrow \infty }R_{n}=0.$ (ii) Given $R_{\ast }\leq R_{0},$ for
the corresponding sequence $\left\{ R_{n}\right\} $ we have $0<\lim
\inf_{n\rightarrow \infty }R_{n}\leq R_{0}.$ (iii) Given $R_{\ast }\leq
R_{0} $ there exists $N$ such that $R_{N}>R_{0}.$ We define $\lambda
_{n}=\lambda \left( R_{n}\right) .$ We will denote the subsets of $\left(
0,R_{0}\right] $ where each of the three possibilities takes place as $%
A_{1},\ A_{2},\ A_{3}$ respectively.

Suppose first that we are in the case (i), or more precisely that $A_{1}$ is
nonempty. Then, for the corresponding $R_{\ast }$ we have the sequence $%
\left\{ R_{n}\right\} $ as defined above$.$ We then have:%
\begin{equation*}
\lambda _{n}\geq K\lambda _{n-1}\left( \frac{\rho \left( R_{n-1}\right) }{%
R_{n-1}}\right) ^{\zeta }=K\lambda _{n-1}\left( \frac{R_{n}}{R_{n-1}}\right)
^{\zeta }\geq ....\geq K^{n}\lambda \left( R_{\ast }\right) \left( \frac{%
R_{n}}{R_{\ast }}\right) ^{\zeta }
\end{equation*}

Then%
\begin{eqnarray*}
\lim \inf_{R\rightarrow 0}\frac{\log \left( \lambda \left( R\right) \right)
}{\log \left( R\right) } &\leq &\lim \inf_{n\rightarrow \infty }\frac{\log
\left( \lambda _{n}\right) }{\log \left( R_{n}\right) } \\
&\leq &\lim \inf_{n\rightarrow \infty }\frac{n\log \left( K\right) +\zeta
\log \left( R_{n}\right) +\log \left( \lambda \left( R_{\ast }\right)
\right) -\zeta \log \left( R_{\ast }\right) }{\log \left( R_{n}\right) } \\
&=&\lim \inf_{n\rightarrow \infty }\frac{n\log \left( K\right) +\zeta \log
\left( R_{n}\right) }{\log \left( R_{n}\right) } \\
&=&\zeta +\log \left( K\right) \lim \inf_{n\rightarrow \infty }\left( \frac{n%
}{\log \left( R_{n}\right) }\right)
\end{eqnarray*}

Notice that $\rho \left( R\right) \geq \frac{R}{L}.$ Then $R_{n}=\rho \left(
R_{n-1}\right) \geq \frac{R_{n-1}}{L}\geq ...\geq \frac{R_{\ast }}{L^{n}},$
whence $\log \left( R_{n}\right) \geq \log \left( R_{\ast }\right) -n\log
\left( L\right) \geq -\left( n+1\right) \log \left( L\right) $ for $n$ large
enough. Then%
\begin{equation*}
\left\vert \log \left( R_{n}\right) \right\vert \leq \left( n+1\right) \log
\left( L\right) ,
\end{equation*}
whence $\frac{1}{\left\vert \log \left( R_{n}\right) \right\vert }\geq \frac{%
1}{\left( n+1\right) \log \left( L\right) }$ and $-\frac{1}{\left\vert \log
\left( R_{n}\right) \right\vert }\leq -\frac{1}{\left( n+1\right) \log
\left( L\right) },$ i.e. $\frac{1}{\log \left( R_{n}\right) }\leq -\frac{1}{%
\left( n+1\right) \log \left( L\right) }.$ This yields:%
\begin{eqnarray*}
\lim \inf_{R\rightarrow 0}\frac{\log \left( \lambda \left( R\right) \right)
}{\log \left( R\right) } &\leq &\zeta -\log \left( K\right) \lim
\inf_{n\rightarrow \infty }\left( \frac{n}{\left( n+1\right) \log \left(
L\right) }\right) \\
\lim \inf_{R\rightarrow 0}\frac{\log \left( \lambda \left( R\right) \right)
}{\log \left( R\right) } &\leq &\zeta -\frac{\log \left( K\right) }{\log
\left( L\right) }<\zeta
\end{eqnarray*}

This contradicts the fact that $\lim \inf_{R\rightarrow 0}\frac{\log \left(
\lambda \left( R\right) \right) }{\log \left( R\right) }=\zeta $ and then,
it implies that the set $A_{1}$ is empty.

Suppose then that $A_{2}$ is nonempty. We also have the inequality:%
\begin{equation*}
\lambda _{n}\geq K^{n}\lambda \left( R_{\ast }\right) \left( \frac{R_{n}}{%
R_{\ast }}\right) ^{\zeta }
\end{equation*}%
whence:%
\begin{equation*}
\lambda \left( R_{\ast }\right) \leq \frac{\lambda _{n}}{\left( R_{n}\right)
^{\zeta }}\frac{\left( R_{\ast }\right) ^{\zeta }}{K^{n}}
\end{equation*}%
where $n$ is arbitrarily large. Then, since the limit $\lim_{n\rightarrow
\infty }\left( \frac{\lambda _{n}}{\left( R_{n}\right) ^{\zeta }}\right)
=\ell =\ell _{R_{\ast }}>0$ exists, we obtain%
\begin{equation*}
\lambda \left( R_{\ast }\right) \leq \ell _{R_{\ast }}\cdot \left( R_{\ast
}\right) ^{\zeta }\lim_{n\rightarrow \infty }\frac{1}{K^{n}}=0
\end{equation*}

Therefore $\lambda \left( R_{\ast }\right) =0$ if $R_{\ast }\in A_{2},$
which is a contradiction.

If $R_{\ast }\in A_{3}$ we can apply the same argument to obtain:%
\begin{equation*}
\lambda \left( R_{\ast }\right) \leq \frac{\lambda _{N}}{\left( R_{N}\right)
^{\zeta }}\frac{\left( R_{\ast }\right) ^{\zeta }}{K^{N}}
\end{equation*}

We now have the inequality $R_{N}=\rho \left( R_{N-1}\right) \leq
LR_{N-1}\leq ...\leq L^{N}R_{\ast }$ which implies $N\log \left( L\right)
\geq \log \left( \frac{R_{N}}{R_{\ast }}\right) .$ Then $K^{N}\geq \left(
\frac{R_{N}}{R_{\ast }}\right) ^{\frac{\log \left( K\right) }{\log \left(
L\right) }},$ whence:%
\begin{equation*}
\lambda \left( R_{\ast }\right) \leq \frac{\lambda _{N}}{\left( R_{N}\right)
^{\zeta +\frac{\log \left( K\right) }{\log \left( L\right) }}}\left( R_{\ast
}\right) ^{\zeta +\frac{\log \left( K\right) }{\log \left( L\right) }}
\end{equation*}

Given that $R_{N}\geq R_{0}$ we obtain:%
\begin{equation*}
\lambda \left( R_{\ast }\right) \leq C\left( R_{\ast }\right) ^{\zeta +\frac{%
\log \left( K\right) }{\log \left( L\right) }}
\end{equation*}%
with $C$ independent of $R_{\ast }.$ Since $A_{1}=A_{2}=\varnothing $ we
obtain this estimate for $R_{\ast }\leq R_{0}$ arbitrary. Then:%
\begin{equation*}
\lim \inf_{R\rightarrow 0}\frac{\log \left( \lambda \left( R\right) \right)
}{\log \left( R\right) }\geq \zeta +\frac{\log \left( K\right) }{\log \left(
L\right) }>\zeta
\end{equation*}%
and this would give a contradiction.
\end{proof}

The following asymptotic result is interesting for itself, and it will be
used repeatedly in the following.

\begin{theorem}
\label{AsSingPoint}Let $0<r<r_{c}.$ Suppose that $h\in C\left( X\right) ,$ $%
\varphi \in C\left( X\right) $ satisfies (\ref{comp1}) and that we have $%
\mathcal{L}\varphi =h,$ with $\mathcal{L}$ as in Definition \ref{LadjDef}.
Then
\begin{equation}
\left\vert \varphi \left( x,v\right) -\varphi \left( 0,0\right) -\mathcal{A}%
\left( \varphi \right) F_{\beta }\left( x,v\right) -h\left( 0,0\right)
S\left( x,v\right) \right\vert =o\left( x^{\frac{2}{3}}+\left\vert
v\right\vert ^{2}\right) \   \label{P1E2}
\end{equation}%
as $x^{\frac{2}{3}}+\left\vert v\right\vert ^{2}\rightarrow 0$ for some
suitable constant $\mathcal{A}\left( \varphi \right) \in \mathbb{R}$.
\end{theorem}

\begin{proof}
We construct sub and supersolutions with the form:%
\begin{equation*}
\varphi _{\pm }\left( x,v\right) =\varphi \left( 0,0\right) \pm \varepsilon
\pm \bar{C}F_{\beta }\left( x,v\right) \pm AS\left( x,v\right) ,
\end{equation*}%
where $A=2\left\vert h\left( 0,0\right) \right\vert $, $\varepsilon >0,\
\bar{C}>0.$ We can choose $0<\delta _{1}<\delta _{2}$ small and then $\bar{C}
$ large enough to obtain $\varphi _{-}\leq \varphi \leq \varphi _{+}$ in $%
\partial \mathcal{R}_{\delta _{1}}\cup \partial \mathcal{R}_{\delta _{2}}.$
Proposition \ref{weak_max} yields then $\varphi _{-}\leq \varphi \leq
\varphi _{+}$ in $\mathcal{R}_{\delta _{2}}\diagdown \mathcal{R}_{\delta
_{1}}.$ Taking the limit $\delta _{1}\rightarrow 0$ and then $\varepsilon
\rightarrow 0$ we obtain:%
\begin{equation*}
\left\vert \varphi \left( x,v\right) -\varphi \left( 0,0\right) \right\vert
\leq 2\bar{C}F_{\beta }\left( x,v\right) \ \ \text{in\ }\mathcal{R}_{\delta }
\end{equation*}%
for some $\delta >0.$ We can now prove that the limit $\frac{\varphi \left(
x,v\right) -\varphi \left( 0,0\right) }{F_{\beta }\left( x,v\right) }$
exists. To this end we define $C_{-}=\lim \inf_{R\rightarrow 0}\left[
\inf_{\partial \mathcal{R}_{R}}\left( \frac{\varphi \left( x,v\right)
-\varphi \left( 0,0\right) }{F_{\beta }\left( x,v\right) }\right) \right] $
and $C_{+}=\lim \sup_{R\rightarrow 0}\left[ \sup_{\partial \mathcal{R}%
_{R}}\left( \frac{\varphi \left( x,v\right) -\varphi \left( 0,0\right) }{%
F_{\beta }\left( x,v\right) }\right) \right] .$ If $C_{-}=C_{+}$ the result
would follow. Suppose then that $C_{-}<C_{+}$. Suppose first that there
exist sequences $\left\{ R_{n}\right\} ,\ \left\{ \tilde{R}_{n}\right\} $
and positive constants $c_{1},\ c_{2}$ such that
\begin{equation*}
\lim_{n\rightarrow \infty }\left[ \sup_{\partial \mathcal{R}_{R_{n}}}\left(
\frac{\varphi \left( x,v\right) -\varphi \left( 0,0\right) }{F_{\beta
}\left( x,v\right) }\right) \right] =C_{+}
\end{equation*}
and
\begin{equation*}
\lim_{n\rightarrow \infty }\left[ \inf_{\partial \mathcal{R}_{\tilde{R}%
_{n}}}\left( \frac{\varphi \left( x,v\right) -\varphi \left( 0,0\right) }{%
F_{\beta }\left( x,v\right) }\right) \right] =C_{-}\text{ with }%
c_{1}R_{n}\leq \tilde{R}_{n}\leq c_{2}R_{n}.
\end{equation*}%
Notice that we might have in particular cases $R_{n}=\tilde{R}_{n}$. We then
define $\psi _{n}\left( x,v\right) =\frac{\varphi \left(
R_{n}^{3}x,R_{n}v\right) -\varphi \left( 0,0\right) }{\left( R_{n}\right)
^{3\beta }}.$ A standard compactness argument implies that for a suitable
convergence subsequence $\psi _{n}\rightarrow \psi $ in compact sets which
do not contain the singular point $\left( x,v\right) =\left( 0,0\right) $
and where $\mathcal{L}\psi =0.$ Moreover, our assumptions in the sequences $%
\left\{ R_{n}\right\} ,\ \left\{ \tilde{R}_{n}\right\} ,$ as well as the
hypoellipticity properties of the operator $\mathcal{L}$ imply the existence
of points different from the singular point such that $\psi \left(
x_{1},v_{1}\right) =C_{+}F_{\beta }\left( x_{1},v_{1}\right) ,$ $\psi \left(
x_{2},v_{2}\right) =C_{-}F_{\beta }\left( x_{2},v_{2}\right) .$ Moreover $%
C_{-}F_{\beta }\leq \psi \leq C_{+}F_{\beta },\ $Since $C_{-}<C_{+},$ the
strong maximum principle then implies then that $\psi \leq \left(
C_{+}-\delta \right) F_{\beta }$ for some $\delta >0$ in a set $\partial
\mathcal{R}_{\rho }$ for some $0<\rho <\infty .$ A comparison argument then
yields $\psi \leq \left( C_{+}-\delta \right) F_{\beta }$ in $\mathcal{R}%
_{\rho }.$ We remark that the comparison argument is made in a family of
domains $\mathcal{R}_{\rho }\diagdown \mathcal{R}_{\delta }$, adding $%
\varepsilon $ to the barrier functions and taking then the limit $\delta
\rightarrow 0$ and $\varepsilon \rightarrow 0.$ However, using the
definition of $\psi _{n}$ we then obtain that $\lim \sup_{R\rightarrow 0}%
\left[ \sup_{\partial \mathcal{R}_{R}}\left( \frac{\varphi \left( x,v\right)
-\varphi \left( 0,0\right) }{F_{\beta }\left( x,v\right) }\right) \right]
\leq \left( C_{+}-\delta \right) $, a contradiction.

Suppose now that the solutions $\left\{ R_{n}\right\} ,\ \left\{ \tilde{R}%
_{n}\right\} $ with the property stated above do not exist. This implies the
existence of a sequence $\left\{ R_{n}\right\} $ such that%
\begin{equation*}
\lim_{n\rightarrow \infty }\left[ \sup_{\partial \mathcal{R}%
_{R_{n}}}\left\vert \frac{\varphi \left( x,v\right) -\varphi \left(
0,0\right) }{F_{\beta }\left( x,v\right) }-C_{+}\right\vert \right] =0.
\end{equation*}%
A comparison argument then yields $\varphi \left( x,v\right) -\varphi \left(
0,0\right) \geq \left( C_{-}+\delta \right) F_{\beta }$, for some $\delta
>0, $ whence $\lim_{n\rightarrow \infty }\left[ \inf_{\partial \mathcal{R}_{%
\tilde{R}_{n}}}\left( \frac{\varphi \left( x,v\right) -\varphi \left(
0,0\right) }{F_{\beta }\left( x,v\right) }\right) \right] \geq \left(
C_{-}+\delta \right) .$ The contradiction then implies that $C_{-}=C_{+}.$
Therefore we have:%
\begin{equation}
\left\vert \varphi \left( x,v\right) -\varphi \left( 0,0\right) -\mathcal{A}%
\left( \varphi \right) F_{\beta }\left( x,v\right) \right\vert =o\left(
x^{\beta }+\left\vert v\right\vert ^{3\beta }\right)  \label{P1E5}
\end{equation}%
as $\left( x,v\right) \rightarrow \left( 0,0\right) ,$ where $\mathcal{A}%
\left( \varphi \right) =C_{+}.$

As a next step we need to prove that:%
\begin{equation}
\left\vert \varphi \left( x,v\right) -\varphi \left( 0,0\right) -\mathcal{A}%
\left( \varphi \right) F_{\beta }\left( x,v\right) \right\vert \leq C\left(
x^{\frac{2}{3}}+\left\vert v\right\vert ^{2}\right) \   \label{P1E3}
\end{equation}%
for $\left( x,v\right) \in \partial \mathcal{R}_{\delta }$ for suitable $%
\delta >0,\ C>0.$ We define the auxiliary function:%
\begin{equation}
\lambda \left( R\right) =\sup_{x^{\frac{2}{3}}+\left\vert v\right\vert
^{2}=R^{\frac{2}{3}}}\left\vert \varphi \left( x,v\right) -\varphi \left(
0,0\right) -\mathcal{A}\left( \varphi \right) F_{\beta }\left( x,v\right)
\right\vert  \label{P1E4}
\end{equation}

Notice that (\ref{P1E5}) yields $\lambda \left( R\right) =o\left( R^{\beta
}\right) $ as $R\rightarrow 0.$

Suppose now that $\lim \sup_{R\rightarrow 0}\frac{\lambda \left( R\right) }{%
R^{\frac{2}{3}}}=\infty .$ We then have, using also the estimate $\lambda
\left( R\right) =o\left( R^{\beta }\right) $ as $R\rightarrow 0,$ that $\lim
\inf_{R\rightarrow 0}\frac{\log \left( \lambda \left( R\right) \right) }{%
\log \left( R\right) }=\zeta \in \left[ \beta ,\frac{2}{3}\right] .$
Applying Lemma \ref{LemSeq}\ we obtain the existence of a sequence $\left\{
R_{n}\right\} $ such that $\frac{\lambda \left( R\right) }{\lambda \left(
R_{n}\right) }\leq K\left( \frac{R}{R_{n}}\right) ^{\zeta }$ for $R\in \left[
\frac{R_{n}}{\mu _{n}},\mu _{n}R_{n}\right] .$ Rescaling the solutions and
defining $\psi _{n}$ as before we obtain in the limit that $\psi _{n}$
converges to $\psi ,$ solution of $\mathcal{L}\psi =0$ with $\left\vert \psi
\right\vert \leq K\left( x^{\zeta }+\left\vert v\right\vert ^{3\zeta
}\right) $ and $\psi \neq 0.$ We now apply Theorem \ref{Liouville}, which
implies that $\psi =0.$ The contradiction implies that $\lim
\sup_{R\rightarrow 0}\frac{\lambda \left( R\right) }{R^{\frac{2}{3}}}<\infty
.$ We then rescale the solutions again, in order to obtain a sequence of
functions $\psi _{n}$ of order one. The limit of this sequence satisfies $%
\mathcal{L}\psi =h\left( 0,0\right) $ as well as $\left\vert \psi \left(
x,v\right) \right\vert \leq C\left( x^{\frac{2}{3}}+\left\vert v\right\vert
^{2}\right) $ and the boundary conditions $\psi \left( 0,rv\right) =\psi
\left( 0,-v\right) \ $for$\ v>0.$ After substracting from $\psi $ the
function $h\left( 0,0\right) S$ we obtain a new function $R=\psi -h\left(
0,0\right) S$ satisfying $\mathcal{L}R=0$ and $\left\vert \psi \left(
x,v\right) \right\vert \leq C\left( x^{\frac{2}{3}}+\left\vert v\right\vert
^{2}\right) $ as well as similar boundary conditions. Using Theorem \ref%
{Liouville} we obtain $R=0,$ whence the asymptotics (\ref{P1E2}) follows.
\end{proof}

We have a similar result in the supercritical case.

\begin{theorem}
\label{AsSingSuper}Let $r>r_{c}.$ Suppose that $h\in C\left( X\right) ,$ $%
\varphi \in C\left( X\right) $ satisfies (\ref{comp1}) and that we have $%
\mathcal{L}\varphi =h,$ with $\mathcal{L}$ as in Definition \ref{LadjDef}.
Then
\begin{equation}
\left\vert \varphi \left( x,v\right) -\varphi \left( 0,0\right) -h\left(
0,0\right) S\left( x,v\right) \right\vert =o\left( x^{\frac{2}{3}%
}+\left\vert v\right\vert ^{2}\right) \
\end{equation}%
as $x^{\frac{2}{3}}+\left\vert v\right\vert ^{2}\rightarrow 0$ for some
suitable constant $\mathcal{A}\left( \varphi \right) \in \mathbb{R}$.
\end{theorem}

\begin{proof}
The proof is similar to the previous Theorem. We first prove, using the
corresponding Liouville's Theorem for this case, that $\lambda \left(
R\right) \leq CR^{\frac{2}{3}}$ where in this case we define $\lambda \left(
R\right) =\sup_{x^{\frac{2}{3}}+\left\vert v\right\vert ^{2}=R^{\frac{2}{3}%
}}\left\vert \varphi \left( x,v\right) -\varphi \left( 0,0\right)
\right\vert .$ We then obtain, using similar arguments, that the asymptotics
is given by the function $S.$
\end{proof}

We conclude this Section with a technical result, which provides some
uniformicity in the asymptotics of the solutions near the singular point.
Suppose that $\xi =\xi \left( x,v\right) $ is a $C^{\infty }$ cutoff
function, supported in the region $x+\left\vert v\right\vert ^{3}\leq \max
\left\{ \frac{2}{r},2\right\} ,$ and satisfying $\xi \left( x,v\right) =1$
for $x+\left\vert v\right\vert ^{3}\leq 1$ as well as $\xi \left(
0,-v\right) =\xi \left( 0,rv\right) ,$ for $v>0,$ $\xi _{x}\left( x,v\right)
=0$ for $x\leq \frac{1}{4}.$~Let\
\begin{eqnarray}
\mathcal{M} &=&\{\psi :\psi \xi \in C\left( X\right) ~|\ \mathcal{L}\left(
\psi \xi \right) \in C\left( X\right) ,\ \left\vert \mathcal{L}\left( \psi
\xi \right) \right\vert \leq 1,  \notag \\
&&\left\vert \psi \left( x,v\right) \right\vert \leq F_{\beta }\left(
x,v\right) \text{ in }\left\vert x\right\vert +\left\vert v\right\vert
^{3}\leq 1\},\psi \left( x,v\right) =o\left( x^{\beta }+\left\vert
v\right\vert ^{3\beta }\right) ~\text{as }\left\vert x\right\vert
+\left\vert v\right\vert ^{3}\rightarrow 0,  \label{T5E1}
\end{eqnarray}%
where $F_{\beta },$ $\beta $ are as in (\ref{Fbeta}). Notice that
since $r<r_{c}$\ we have $\beta >0.$ The cutoff $\xi $ has been introduced
because in later arguments we will need to use unbounded functions $\psi $
which do not belong to $C\left( X\right) .$

Then we derive the following estimate which shows that the rate of decay of $%
\psi \left( x,v\right) $ as $\left( x,v\right) \rightarrow \left( 0,0\right)
$ is uniform for the class of functions $\mathcal{M}$.

\begin{proposition}
\label{mu_R}\ Let%
\begin{equation}
\mu \left( R\right) =\sup_{\psi \in \mathcal{M}}\sup_{x^{\beta }+\left\vert
v\right\vert ^{3\beta }=R^{\beta }}\left\vert \psi \left( x,v\right)
\right\vert .\   \label{T5E3}
\end{equation}%
with $\mathcal{M}$ as in (\ref{T5E1}). Then we have
\begin{equation*}
\lim_{R\rightarrow 0}\frac{\mu \left( R\right) }{R^{\beta }}=0.
\end{equation*}
\end{proposition}

\begin{proof}
We prove the result by contradiction. Suppose the opposite holds, that is,
there exists $R_{m}\rightarrow 0$ such that $\frac{\mu \left( R_{m}\right) }{%
R_{m}^{\beta }}\geq \eta >0.$ Then there exist sequences $\left\{ \psi
_{m}\right\} \subset \mathcal{M},\ \left\{ x_{m}\right\} $ with $x_{m}>0,\
\left\{ v_{m}\right\} \subset \mathbb{R}$ with $\left\vert x_{m}\right\vert
^{\beta }+\left\vert v_{m}\right\vert ^{3\beta }=R_{m}^{\beta }$ such that
\begin{equation*}
\left\vert \psi _{m}\right\vert \geq \eta \left( \left\vert x_{m}\right\vert
^{\beta }+\left\vert v_{m}\right\vert ^{3\beta }\right) .
\end{equation*}%
Define
\begin{equation*}
\Phi _{m}\left( x,v\right) =\frac{\psi _{m}\left( R_{m}x,R_{m}^{1/3}v\right)
}{R_{m}^{\beta }}.
\end{equation*}

Then $\Phi _{m}$ satisfies the following.

\begin{itemize}
\item[\textbf{(i)}]
\begin{equation}
\left\vert \mathcal{L}\Phi _{m}\right\vert \leq R_{m}^{2/3-\beta
}\rightarrow 0~\text{as }m\rightarrow \infty ,\text{since }\beta <2/3.
\label{Phi_m1}
\end{equation}

\item[\textbf{(ii)}]
\begin{equation}
\Phi _{m}\text{ satisfies (\ref{comp1}) for any }v>0\text{ }  \label{Phi_m2}
\end{equation}

\item[\textbf{(iii)}]
\begin{equation}
\Phi _{m}\left( x,v\right) =o\left( \left\vert x\right\vert ^{\beta
}+\left\vert v\right\vert ^{3\beta }\right) ,\ \ \left\vert x\right\vert
+\left\vert v\right\vert ^{3}\rightarrow 0.  \label{Phi_m3}
\end{equation}

\item[\textbf{(iv)}]
\begin{equation}
\left\vert \Phi _{m}\left( x_{m}^{\ast },v_{m}^{\ast }\right) \right\vert
\geq \eta >0,  \label{Phi_m4}
\end{equation}%
where $\left( x_{m}^{\ast },v_{m}^{\ast }\right) =\left( \frac{x_{m}}{R_{m}},%
\frac{v_{m}}{R_{m}^{1/3}}\right) $ with $\left\vert x_{m}^{\ast }\right\vert
^{\beta }+\left\vert v_{m}^{\ast }\right\vert ^{3\beta }=1.$

\item[\textbf{(v)}]
\begin{equation}
\left\vert \Phi _{m}\left( x,v\right) \right\vert \leq F_{\beta }\left(
x,v\right) .  \label{Phi_m5}
\end{equation}
\end{itemize}

By the regularizing effect of the operator $\mathcal{L}$ and by (i), $\Phi
_{m}$ is uniformly bounded in $W_{loc}^{1,p}\left( \mathbb{R}^{+}\times
\mathbb{R}\right) $ for all $1\leq p<\infty $ (cf. Theorem \ref{Hypoell}).
Then by the Sobolev embedding, $\Phi _{m}$ converges to $\Phi ^{\ast }$
uniformly in $m$ on compact sets (up to subsequences) and $\Phi ^{\ast }$
satisfies the following.
\begin{equation*}
\mathcal{L}\Phi ^{\ast }=0\text{ in }\mathcal{U},~\Phi ^{\ast }\text{
satisfies (\ref{comp1}),}
\end{equation*}%
\begin{equation*}
\left\vert \Phi ^{\ast }\left( x,v\right) \right\vert \leq F_{\beta }\left(
x,v\right) ,\ \left\vert \Phi ^{\ast }\left( x^{\ast },v^{\ast }\right)
\right\vert \geq \eta >0,
\end{equation*}

where $\left( x_{m}^{\ast },v_{m}^{\ast }\right) \rightarrow \left( x^{\ast
},v^{\ast }\right) $ with $\left\vert x^{\ast }\right\vert ^{\beta
}+\left\vert v^{\ast }\right\vert ^{3\beta }=1.$

Let
\begin{equation*}
C^{\ast }=\inf \left\{ C\ |\ \Phi ^{\ast }\left( x,v\right) \leq CF_{\beta
}\left( x,v\right) \text{ for all }\left( x,v\right) \in \mathbb{R}%
^{+}\times \mathbb{R}\right\} .
\end{equation*}

Notice that $\sup_{\left( x,v\right) }\frac{\Phi ^{\ast }\left( x,v\right) }{%
F_{\beta }\left( x,v\right) }=C^{\ast }.$ Indeed, if $\sup_{\left(
x,v\right) }\frac{\Phi ^{\ast }\left( x,v\right) }{F_{\beta }\left(
x,v\right) }>C^{\ast },$ we would obtain that, for any $\delta >0$
sufficiently small there exists $\left( x_{\delta },v_{\delta }\right) \in
\overline{\mathbb{R}^{+}\times \mathbb{R}}$ such that $\Phi ^{\ast }\left(
x_{\delta },v_{\delta }\right) >\left( C^{\ast }+\delta \right) F_{\beta
}\left( x,v\right) ,$ whence $\sup_{\left( x,v\right) }\frac{\Phi ^{\ast
}\left( x,v\right) }{F_{\beta }\left( x,v\right) }\geq C^{\ast }+\delta ,$
and this contradicts the definition of $C^{\ast }.$ If $\sup_{\left(
x,v\right) }\frac{\Phi ^{\ast }\left( x,v\right) }{F_{\beta }\left(
x,v\right) }<C^{\ast }$ we would have $\sup_{\left( x,v\right) }\frac{\Phi
^{\ast }\left( x,v\right) }{F_{\beta }\left( x,v\right) }<C^{\ast }-\delta $
for some $\delta >0$ small, and therefore we would have $C^{\ast }\leq
C^{\ast }-\delta $ which gives a contradiction too.

We now claim that:
\begin{equation}
\Phi ^{\ast }\left( x,v\right) =C^{\ast }F_{\beta }\left( x,v\right)
\label{claimMP}
\end{equation}

To prove this we examine the different possible cases depending on the
values of $\left( x,v\right) $ where $\sup_{\left( x,v\right) }\frac{\Phi
^{\ast }\left( x,v\right) }{F_{\beta }\left( x,v\right) }$ is achieved. At
least one of the following alternatives holds:

\textbf{Case (a)} There exists $\left( \bar{x},\bar{v}\right) \neq \left(
0,0\right) $ with $\left\vert \left( \bar{x},\bar{v}\right) \right\vert
<\infty $ such that $\Phi ^{\ast }\left( \bar{x},\bar{v}\right) =C^{\ast
}F_{\beta }\left( \bar{x},\bar{v}\right) .$

\textbf{Case (b) }%
\begin{equation}
\lim \sup_{R\rightarrow 0}\sup_{\left\vert x\right\vert ^{\beta }+\left\vert
v\right\vert ^{3\beta }=R^{\beta }}\frac{\Phi ^{\ast }\left( x,v\right) }{%
F_{\beta }\left( x,v\right) }=C^{\ast }.  \label{C*_0}
\end{equation}

\textbf{Case (c)}%
\begin{equation*}
\lim \sup_{R\rightarrow \infty }\sup_{\left\vert x\right\vert ^{\beta
}+\left\vert v\right\vert ^{3\beta }=R^{\beta }}\frac{\Phi ^{\ast }\left(
x,v\right) }{F_{\beta }\left( x,v\right) }=C^{\ast }.
\end{equation*}

In the case (a) we notice that $\Phi ^{\ast }\left( x,v\right) \leq C^{\ast
}F_{\beta }\left( x,v\right) $ for all $\left( x,v\right) \in \mathbb{R}%
^{+}\times \mathbb{R}.$ We then obtain a contradiction using the Strong
Maximum Principle (cf. Proposition \ref{strong_max}) if $\Phi ^{\ast }\left(
x,v\right) \neq C^{\ast }F_{\beta }\left( x,v\right) ,$ since we would
obtain a minimum of the function $\left( C^{\ast }F_{\beta }\left(
x,v\right) -\Phi ^{\ast }\right) $ which satisfies $\mathcal{L}\left(
C^{\ast }F_{\beta }\left( x,v\right) -\Phi ^{\ast }\right) =0.$\

In the case (b), suppose $\Phi ^{\ast }\left( x,v\right) \neq C^{\ast
}F_{\beta }\left( x,v\right) $, then we may assume without loss of
generality that $\Phi ^{\ast }\left( x,v\right) <C^{\ast }F_{\beta }\left(
x,v\right) ,$ for $\left\vert x\right\vert ^{\beta }+\left\vert v\right\vert
^{3\beta }=1.$ Then, using Proposition \ref{weak_max} as in the proof of
Theorem \ref{AsSingPoint}, i.e. making the comparison in a region removing a
small set near the singular point, adding a small value $\varepsilon >0$ and
taking the limit as the domain converges to the whole $\mathbb{R}^{+}\times
\mathbb{R}$ we can show that $\Phi ^{\ast }\left( x,v\right) \leq \left(
C^{\ast }-\delta \right) F_{\beta }\left( x,v\right) +\varepsilon ,$ for
some $\delta >0$ small and for arbitrary $\varepsilon >0$ small. By letting $%
\varepsilon \rightarrow 0,$ we get $\Phi ^{\ast }\left( x,v\right) \leq
\left( C^{\ast }-\delta \right) F_{\beta }\left( x,v\right) ,$ which
contradicts to (\ref{C*_0}).

In the case (c), there exists a sequence $\left( \bar{x}_{l},\bar{v}%
_{l}\right) $ such that $\left\vert \left( \bar{x}_{l},\bar{v}_{l}\right)
\right\vert \rightarrow \infty $ and
\begin{equation*}
\lim_{l\rightarrow \infty }\frac{\Phi ^{\ast }\left( \bar{x}_{l},\bar{v}%
_{l}\right) }{F_{\beta }\left( \bar{x}_{l},\bar{v}_{l}\right) }=C^{\ast }.
\end{equation*}%
Define%
\begin{equation*}
\Phi _{l}^{\ast \ast }\left( x,v\right) =\frac{\Phi ^{\ast }\left( \rho
_{l}x,\rho _{l}^{1/3}v\right) }{\rho _{l}^{\beta }},\ \rho _{l}^{\beta
}=\left\vert \bar{x}_{l}\right\vert ^{\beta }+\left\vert \bar{v}%
_{l}\right\vert ^{3\beta }.
\end{equation*}

Then there exist $\Phi _{\infty }^{\ast \ast }\left( x,v\right) \in \mathcal{%
M}$ and $\left( \bar{x}_{\infty },\bar{v}_{\infty }\right) \in \mathbb{R}%
_{+}\times \mathbb{R}$ such that $\left( \frac{\bar{x}_{l}}{\rho _{l}},\frac{%
\bar{v}_{l}}{\rho _{l}^{3}}\right) \rightarrow \left( \bar{x}_{\infty },\bar{%
v}_{\infty }\right) $ with $\left\vert \bar{x}_{\infty }\right\vert ^{\beta
}+\left\vert \bar{v}_{\infty }\right\vert ^{3\beta }=1$ and $\Phi _{l}^{\ast
\ast }\left( x,v\right) \rightarrow \Phi _{\infty }^{\ast \ast }\left(
x,v\right) ,$ as $l\rightarrow \infty ,$ $\mathcal{L}\left( \Phi _{\infty
}^{\ast \ast }\right) =0,\Phi _{\infty }^{\ast \ast }\left( x,v\right) \leq
C^{\ast }F_{\beta }\left( x,v\right) $ and%
\begin{equation*}
\Phi _{\infty }^{\ast \ast }\left( \bar{x}_{\infty },\bar{v}_{\infty
}\right) =C^{\ast }F_{\beta }\left( \bar{x}_{\infty },\bar{v}_{\infty
}\right) .
\end{equation*}

Then by Proposition \ref{strong_max} we have%
\begin{equation*}
\Phi _{\infty }^{\ast \ast }\left( x,v\right) =C^{\ast }F_{\beta }\left(
x,v\right) .
\end{equation*}

Thus for $\delta >0$ given small and $l$ large and for $1/4\leq \left\vert
x\right\vert ^{\beta }+\left\vert v\right\vert ^{3\beta }\leq 4,$%
\begin{equation*}
\Phi _{l}^{\ast \ast }\left( x,v\right) \geq \left( C^{\ast }-\frac{\delta }{%
10^{6}}\right) F_{\beta }\left( x,v\right) ,
\end{equation*}

which in turn implies, for $\rho _{l}/4\leq \left\vert x\right\vert ^{\beta
}+\left\vert v\right\vert ^{3\beta }\leq 4\rho _{l},$
\begin{equation*}
\Phi ^{\ast }\left( x,v\right) \geq \left( C^{\ast }-\frac{\delta }{10^{6}}%
\right) F_{\beta }\left( x,v\right) .
\end{equation*}

Now suppose $\Phi ^{\ast }\left( x,v\right) \neq C^{\ast }F_{\beta }\left(
x,v\right) ,$ then, using the comparison argument above it would follow that
$\Phi ^{\ast }\left( x,v\right) \leq \left( C^{\ast }-\delta \right)
F_{\beta }\left( x,v\right) $ for $\rho _{l}=\left\vert x\right\vert ^{\beta
}+\left\vert v\right\vert ^{3\beta },$ which is a contradiction.

Thus the claim (\ref{claimMP}) is proved. Next we claim $C^{\ast }=0.$
Suppose $C^{\ast }\neq 0.$ Take $C^{\ast }>0$ without loss of generality.
Then for any given $\delta >0$ small, for $\left\vert x\right\vert ^{\beta
}+\left\vert v\right\vert ^{3\beta }=1,$ $\Phi _{m}\geq \left( C^{\ast
}-\delta \right) F_{\beta }$ for large $m$ since $\Phi _{m}\rightarrow \Phi
^{\ast }$ as $m\rightarrow \infty $ and $\Phi ^{\ast }\left( x,v\right)
=C^{\ast }F_{\beta }\left( x,v\right) .$ We recall a super-solution $S\left(
x,v\right) $ with $S\sim \left\vert x\right\vert ^{2/3}+\left\vert
v\right\vert ^{2}$ near $\left( 0,0\right) $ and $\mathcal{L}S=1$ as in
Lemma \ref{super_sol}$.$

Then we construct a sub-solution of the form%
\begin{equation*}
Z\left( x,v\right) =\left( C^{\ast }-2\delta \right) F_{\beta }\left(
x,v\right) -R_{m}^{2/3-\beta }S\left( x,v\right) ,
\end{equation*}

so that we have by comparison principle,%
\begin{equation*}
\Phi _{m}\left( x,v\right) \geq Z\left( x,v\right) ,\text{ for }\left\vert
x\right\vert ^{\beta }+\left\vert v\right\vert ^{3\beta }\leq 1.
\end{equation*}

Therefore, we have%
\begin{equation*}
\Phi _{m}\left( x,v\right) \geq \left( C^{\ast }-3\delta \right) F_{\beta
}\left( x,v\right) ,\text{ for }\left\vert x\right\vert ^{\beta }+\left\vert
v\right\vert ^{3\beta }\leq 1,
\end{equation*}

which contradicts to (\ref{Phi_m3}). Thus the claim is proved, that is, $%
C^{\ast }=0.$ Then $\Phi ^{\ast }\left( x,v\right) =0,$ which contradicts to
(\ref{Phi_m4}). Thus the proof is complete.
\end{proof}

\subsection{The operators $\Omega _{\protect\sigma }$ are Markov pregenerators.
\label{SuMP}}

We will denote in the following as $\Omega _{\sigma }$ any of the operators $%
\Omega _{t,sub},\ \Omega _{nt,sub},\ \Omega _{pt,sub},\ \Omega _{sup}$\
defined in Sections \ref{Absorbing}, \ref{Reflecting}, \ref{Mixed}, \ref%
{Supercrit}.\ In the rest of this Section we check that the four operators $%
\Omega _{\sigma }$ defined above are Markov generators in the sense of
Definition \ref{MarGen}, and therefore that they define Markov generators
due to Hille-Yosida Theorem (cf. Theorem \ref{HY}). We first check that the
operators $\Omega _{\sigma }$ are Markov pregenerators.

\begin{proposition}
\label{PropPG}The operators $\Omega _{t,sub},\ \Omega _{nt,sub},\ \Omega
_{pt,sub},\ \Omega _{sup}$\ defined in Sections \ref{Absorbing}, \ref%
{Reflecting}, \ref{Mixed}, \ref{Supercrit} are Markov pregenerators in the
sense of Definition \ref{MarPre}.
\end{proposition}

We will prove Proposition \ref{PropPG} with the help of several\ Lemmas.

\begin{lemma}
\label{density} Suppose that $\Omega _{\sigma }\in \left\{ \Omega _{t,sub},\
\Omega _{nt,sub},\ \Omega _{pt,sub},\ \Omega _{sup}\right\} .$ The domains $%
\mathcal{D}\left( \Omega _{\sigma }\right) $ defined in (\ref{S7E1}), (\ref%
{S7E2}), (\ref{S7E3}), (\ref{S7E4}) are dense in $C\left( X\right) $ endowed
with the uniform topology.
\end{lemma}

\begin{proof}
We need to check that the domains $\mathcal{D}\left( \Omega _{\sigma
}\right) $ are dense in $C\left( X\right) $ with the uniform norm $%
\left\Vert \cdot \right\Vert $ in (\ref{Norm}). To this end we argue as
follows. Let $\xi \in C\left( X\right) ,$ then we want to find an
approximate sequence in $\mathcal{D}(\Omega )$ converging uniformly to $\xi $
in $X.$ To this end, for given $\rho >0$, first we introduce the following
function $\lambda \in C^{\infty }[0,\infty )$ so that
\begin{equation*}
0\leq \lambda (s)\leq 1;\quad \lambda (s)=1\text{ for }0\leq s\leq \rho \
;\quad \lambda (s)=0\text{ for }s\geq 2\rho
\end{equation*}%
and we define $\eta (x,v)$ by
\begin{equation}
\eta (x,v)=%
\begin{cases}
\lambda (x+|v|^{3}), & v<0 \\
\lambda (x+|\frac{v}{r}|^{3}), & v\geq 0.%
\end{cases}
\label{eta}
\end{equation}

Note that $\eta $ is continuous in $X$, $\eta (0,-v)=\eta (0,rv)$ for $v>0$,
and $\psi =\mathcal{A}\left( \eta \right) =\mathcal{L}\left( \eta \right) =0$
in (\ref{phi_decom}) since $\eta $ is constant near the origin. Let $%
\varepsilon >0$ be given. Choose $\tilde{\xi}=\eta \xi \left( 0,0\right)
+\left( 1-\eta \right) \xi ,$ where $\eta $ is given in \eqref{eta}. Then $%
\eta =1$ on $\{x+\left\vert v\right\vert ^{3}\leq \rho \text{ and }v<0,\text{
or }x+\left\vert \frac{v}{r}\right\vert ^{3}\leq \rho \text{ and }v\geq 0\}$
and $\eta =0$ for $\{x+\left\vert v\right\vert ^{3}\geq 2\rho \text{ and }%
v<0,\text{ or }x+\left\vert \frac{v}{r}\right\vert ^{3}\geq 2\rho \text{ and
}v\geq 0\}$ and $0\leq \eta \leq 1$ otherwise. Then it is easy to see, since
$\xi $ is continuous up to the boundary,%
\begin{equation*}
\left\Vert \xi -\tilde{\xi}\right\Vert =\left\Vert \eta \left( \xi -\xi
\left( 0,0\right) \right) \right\Vert \leq \varepsilon /2\text{,}
\end{equation*}%
for sufficiently small $\rho $. Write
\begin{equation*}
\tilde{\xi}=\xi \left( 0,0\right) +\left( 1-\eta \right) (\xi -\xi \left(
0,0\right) )=:\xi _{1}+\xi _{2}.
\end{equation*}%
Then $\xi _{1}\in \mathcal{D}(\Omega )$ since it is a constant function. We
now consider $\xi _{2}=\left( 1-\eta \right) (\xi -\xi \left( 0,0\right) )$
and will find an approximation in $\mathcal{D}(\Omega )$ to it. Note that $%
\xi _{2}$ is $0$ near the origin. We may assume that $\xi _{2}$ is compactly
supported. Let us consider the standard mollifiers $\{\zeta ^{\delta
}\}_{\delta >0}$. By extending $\xi _{2}$ to zero outside of $X$, we
consider the following approximation of $\xi _{2}$ via convolution
\begin{equation*}
\xi _{2}^{\delta }(x,v)=\int_{|x-y|^{2}+|v-w|^{2}\leq \delta }\zeta ^{\delta
}(x-y,v-w)\xi _{2}(y,w)dydw
\end{equation*}%
Then for sufficiently small $\delta (<r\rho /2)$, we obtain $\xi
_{2}^{\delta }=0$ in a small neighborhood of the origin and that $\Vert \xi
_{2}^{\delta }-\xi _{2}\Vert \leq \varepsilon /2$. However, this $\xi
_{2}^{\delta }$ does not belong to $\mathcal{D}(\Omega )$ in general, since $%
\xi _{2}^{\delta }$ and $\mathcal{L}\xi _{2}^{\delta }$ do not satisfy the
boundary condition \eqref{comp1}. In order to find the approximation of $\xi
_{2}$ that lies in $\mathcal{D}(\Omega )$, we will construct a smooth
function $\varphi ^{\delta }$ in a thin strip $Y=\{0<x<\delta /2,\;v<-\delta
/2\}$ near the lower boundary but away from the singular set so that it is
close to $\xi _{2}$ in that strip $Y$ and moreover, $\varphi ^{\delta
}(0,v)=\xi _{2}^{\delta }(0,-rv)$ and $\mathcal{L}\varphi ^{\delta }(0,v)=%
\mathcal{L}\xi _{2}^{\delta }(0,-rv)$. This will be achieved by solving the
parabolic equation in $Y$:
\begin{equation*}
\begin{split}
\partial _{x}\varphi ^{\delta }(x,v)& =-\frac{1}{v}\partial _{v}^{2}\varphi
^{\delta }-r\partial _{x}\xi _{2}^{\delta }(x,-rv)+\frac{1}{v}\partial
_{v}^{2}\xi _{2}^{\delta }(x,-rv)\;\text{ for }0<x<\delta /2,\;v<-\delta /2
\\
\varphi ^{\delta }(0,v)& =\xi _{2}^{\delta }(0,-rv),\quad \varphi ^{\delta
}(x,-\frac{\delta }{2})=0\;\text{ for }0<x<\delta /2.
\end{split}%
\end{equation*}%
Here we view $x$ as a time-like variable. Since $-v$ is away from zero, the
existence and regularity of the solution follow from classical parabolic
theory. Note that the parabolic equation guarantees $\mathcal{L}\varphi
^{\delta }(0,v)=\mathcal{L}\xi _{2}^{\delta }(0,-rv)$ and the initial data
gives $\varphi ^{\delta }(0,v)=\xi _{2}^{\delta }(0,-rv)$. In addition,
since $\Vert \xi _{2}^{\delta }-\xi _{2}\Vert \leq \varepsilon /2$ and $%
\varphi ^{\delta }(0,v)=\xi _{2}^{\delta }(0,-rv)\sim \xi _{2}(0,v)$, and by
continuity of $\xi _{2}$, we deduce that $\sup_{Y}|\varphi ^{\delta }-\xi
_{2}|\leq \varepsilon /2$ for sufficiently small $\delta $.

By introducing another smooth cutoff function $\chi $ so that $\chi (x,v)=1$
in $Y^{\prime }=\{0<x<\delta /4,\;v<-3\delta /4\}$, $\chi (x,v)=0$ in $%
X\setminus Y$, and $0\leq \chi \leq 1$, we now let $\tilde{\xi _{2}^{\delta }%
}:=\chi \varphi ^{\delta }+(1-\chi )\xi _{2}^{\delta }$. Then by
construction, we deduce that $\tilde{\xi _{2}}\in \mathcal{D}(\Omega )$ and
moreover,
\begin{equation*}
\left\Vert \xi _{2}-\tilde{\xi _{2}^{\delta }}\right\Vert \leq \varepsilon /2%
\text{,}
\end{equation*}%
for sufficiently small $\delta >0$. Therefore,
\begin{equation*}
\left\Vert \xi -(\xi _{1}+\tilde{\xi _{2}^{\delta }})\right\Vert =\left\Vert
\xi -\tilde{\xi}+\xi _{2}-\tilde{\xi _{2}^{\delta }}\right\Vert \leq
\varepsilon .
\end{equation*}%
Since $\varepsilon $ is arbitrary, we are done.
\end{proof}

We now prove the following:

\begin{lemma}
\label{minProp}Suppose that $\Omega _{\sigma }\in \left\{ \Omega _{t,sub},\
\Omega _{nt,sub},\ \Omega _{pt,sub},\ \Omega _{sup}\right\} .$ Let $g\in
C\left( X\right) $ and suppose that there exists $\varphi \in \mathcal{D}%
\left( \Omega _{\sigma }\right) $ such that $\varphi -\lambda \Omega
_{\sigma }\varphi =g$ for some $\lambda \geq 0.$ Then $\min_{\zeta \in X}\
\varphi \left( \zeta \right) \geq \min_{\zeta \in X}\ g\left( \zeta \right)
. $
\end{lemma}

\begin{proof}
In order to prove the result we consider several (not mutually exclusive)
subcases, namely:

(a) $\varphi \left( x_{0},v_{0}\right) =\min_{X}\varphi ,\ $with $x_{0}>0,\
v_{0}\in \mathbb{R}$.

(b) $\varphi \left( 0,v_{0}\right) =\min_{X}\varphi ,\ $with $v_{0}\in
\mathbb{R}.$

(c) $\varphi \left( \infty \right) =\min_{X}\varphi .$

(d) $\varphi \left( 0,0\right) =\min_{X}\varphi .$

By assumption, in the case (a) we have that $\varphi \left( x,v\right) -%
\mathcal{L}\varphi \left( x,v\right) =g\left( x,v\right) $ for $\left(
x,v\right) $ in a neighbourhood of $\left( x_{0},v_{0}\right) $ contained in
$\mathcal{U}$. We now claim that $\mathcal{L}\varphi \left(
x_{0},v_{0}\right) \geq 0.$ Suppose the opposite, i.e. $\mathcal{L}\varphi
\left( x_{0},v_{0}\right) <0.$ Then, since $g,\varphi \in C\left( X\right) $
there exists an admissible domain $\Xi $ (cf. Definition \ref{admissible}),
such that $\left( x_{0},v_{0}\right) \in \Xi $ and $\mathcal{L}\varphi
\left( x,v\right) \leq -\kappa $ for $\left( x,v\right) \in \Xi $ with $%
\kappa >0.$ Moreover, we can assume also that $\operatorname{dist}\left( \left(
x_{0},v_{0}\right) ,\partial _{a}\Xi \right) \geq \frac{d\left( \Xi \right)
}{4}$ with $d\left( \Xi \right) $ as in (\ref{T7E3}). Due to (a) we have $%
\varphi \left( x_{0},v_{0}\right) \leq \min_{\partial _{a}\Xi }\varphi .$
Applying Proposition \ref{weak_max} to $\bar{\varphi}=\varphi -\varphi
\left( x_{0},v_{0}\right) $ it then follows that $0=\bar{\varphi}\left(
x_{0},v_{0}\right) \geq C_{\ast }\kappa >0,$ a contradiction. Therefore $%
\mathcal{L}\varphi \left( x_{0},v_{0}\right) \geq 0$ and since $\Omega
_{\sigma }\varphi \left( x_{0},v_{0}\right) =\mathcal{L}\varphi \left(
x_{0},v_{0}\right) $ the result stated in the Lemma follows in this case.

In the case (b)\ we argue similarly. We just take a domain $\Xi $ with the
form $\left( 0,x_{2}\right) \times \left( v_{1},v_{2}\right) $ such that $%
\left( x_{0},v_{0}\right) \in \partial \Xi $, and $\operatorname{dist}\left(
\left( x_{0},v_{0}\right) ,\partial _{a}\Xi \right) \geq \frac{d\left( \Xi
\right) }{4}.$ Notice that we use the identification of the points in the
definition of $X$ (cf. (\ref{DefX})) and that the points of the domain $\Xi $
are to the left of $\left( x_{0},v_{0}\right) $ in the sense of Definition %
\ref{defLeft}. If $\mathcal{L}\varphi \left( x_{0},v_{0}\right) <0$ we
obtain a contradiction as in the case (a) using Proposition \ref{weak_max}.

In the case (c), using the fact that $\mathcal{L}\varphi $ is continuous in $%
X$ \ it follows that, for any given any $\varepsilon ^{\ast }>0$ we can
choose $R$ such that if $\left\vert \left( x,v\right) \right\vert \geq R$ we
have $\varphi \left( x,v\right) \leq \min_{X}\varphi +\varepsilon ^{\ast }.$
Suppose that $\mathcal{L}\varphi \left( \infty \right) <0,$ we claim that
this implies the existence of $\left( \bar{x},\bar{v}\right) $ with $%
\left\vert \left( \bar{x},\bar{v}\right) \right\vert \geq R$ such that $%
\varphi \left( \bar{x},\bar{v}\right) >\min_{X}\varphi +C_{0},$ with $C_{0}>0
$ independent of $\varepsilon ^{\ast }.$ To prove this we choose $\left(
x_{0},v_{0}\right) =\left( x_{0},1\right) $\ with $x_{0}$ large and we
define an admissible domain $\Xi =\left( x_{0}-1,x_{0}+1\right) \times
\left( -2,2\right) .$ By assumption $\mathcal{L}\varphi \left( x,v\right)
\leq -\kappa $ for some $\kappa >0$ independent of $\varepsilon ^{\ast }$
and $\left( x,v\right) \in \Xi .$ Suppose that $\varphi \left( x,v\right)
\leq \min_{X}\varphi +\varepsilon ^{\ast }$ for any $\left( x,v\right) \in
\partial _{a}\Xi .$ Then, a comparison argument, as in the previous cases,
yields $\varphi \left( x_{0},1\right) \leq \min_{X}\varphi +\varepsilon
^{\ast }-C_{1}\kappa $ for some constant $C_{1}>0$ independent of $%
\varepsilon ^{\ast }.$ Notice that the independence of $C_{1}$ follows from
the invariance of the operator $\mathcal{L}$ under translations in $x.$
Since $\varepsilon ^{\ast }$ can be chosen arbitrarily small it then follows
that $\varphi \left( x_{0},1\right) \leq \min_{X}\varphi -\frac{C\kappa }{2}$
if $x_{0}$ is large enough. However, since $\varphi \in C\left( X\right) $
this implies $\varphi \left( \infty \right) \leq \min_{X}\varphi -\frac{%
C\kappa }{2},$ but this contradicts the fact that $\min_{X}\varphi =\varphi
\left( \infty \right) .$ Therefore $\mathcal{L}\varphi \left( \infty \right)
\geq 0$ whence $\varphi \left( \infty \right) \geq \varphi \left( \infty
\right) -\mathcal{L}\varphi \left( \infty \right) =g\left( \infty \right) $
and the result follows in this case.

Suppose now that (d) holds. Notice that in the case of the four operators%
\begin{equation*}
\Omega _{t,sub},\ \Omega _{nt,sub},\ \Omega _{pt,sub},\Omega _{sup}
\end{equation*}
we have that $\Omega _{\sigma }\left( 0,0\right) =\left( \mathcal{L}\varphi
\right) \left( 0,0\right) .$ Moreover, all the functions in the domain of
the four operators satisfy (\ref{phi_decom}).

In the case of the operator $\Omega _{t,sub}$ we have, by definition $\left(
\Omega _{t,sub}\varphi \right) \left( 0,0\right) =0,$ whence the result
immediately follows. In order to show this property for $\Omega _{pt,sub}$,
notice that the asymptotics (\ref{phi_decom}), combined with the fact that $%
F_{\beta }\left( x,v\right) >0$ (cf. Proposition 4.2 in \cite{HVJ2})\ implies
that, since $\varphi \left( 0,0\right) =\min_{X}\varphi ,$ we have $\mathcal{%
A}\left( \varphi \right) \geq 0.$ Therefore, using that $\left( \Omega
_{pt,sub}\varphi \right) \left( 0,0\right) =\mu _{\ast }\left\vert C_{\ast
}\right\vert \mathcal{A}\left( \varphi \right) $, with $\mu _{\ast }>0,$ we
obtain that (iii) also holds for the operator $\Omega _{pt,sub}.$\ In the
case of the operators $\Omega _{nt,sub}$ we have, due to the definition of
its domain, that $\mathcal{A}\left( \varphi \right) =0.$ We now claim that
this implies that $\left( \mathcal{L}\varphi \right) \left( 0,0\right) =0.$
Suppose that $\left( \mathcal{L}\varphi \right) \left( 0,0\right) =\kappa
\neq 0.$ Theorem \ref{AsSingPoint}, combined with the boundary condition $%
\mathcal{A}\left( \varphi \right) =0$ yields the asymptotics:%
\begin{equation}
\varphi \left( x,v\right) -\varphi \left( 0,0\right) -\kappa S\left(
x,v\right) =o\left( x^{\frac{2}{3}}+\left\vert v\right\vert ^{2}\right) \
\label{P1E6}
\end{equation}%
as $\left( x,v\right) \rightarrow \left( 0,0\right) .$ Lemma \ref{super_sol}
implies that $S\left( x,v\right) $ takes positive and negative values in any
neighbourhood of $\left( x,v\right) =\left( 0,0\right) .$ Therefore, if $%
\kappa \neq 0$ we cannot have $\varphi \left( 0,0\right) =\min_{X}\varphi .$
This rules out the possibility (d) in the case of the operator $\Omega
_{nt,sub}.$

It only remains to consider the operator $\Omega _{sup}.$ The asymptotics of
$\varphi \left( x,v\right) $ near the singular point can be computed in this
case using Theorem \ref{AsSingSuper}. We obtain then the asymptotics (\ref%
{P1E6}). Since $S$ changes sign in any neighbourhood of the singular point
we obtain that $\varphi \left( 0,0\right) \neq \min_{X}\varphi $ if $\kappa
\neq 0.$
\end{proof}

\begin{proof}[Proof of Proposition \protect\ref{PropPG}]
Notice that for all the operators
\begin{equation*}
\Omega _{\sigma }\in \left\{ \Omega _{t,sub},\ \Omega _{nt,sub},\ \Omega
_{pt,sub},\ \Omega _{sup}\right\}
\end{equation*}
we have that $1\in \mathcal{D}\left( \Omega _{\sigma }\right) $ and $\Omega
_{\sigma }1=0$ (cf. (\ref{S7E1a}), (\ref{S7E2a}), (\ref{S7E3a}), (\ref{S7E4a}%
)). Therefore (i) in Definition \ref{MarPre} follows. The property (ii)\ in
Definition \ref{MarPre} follows from Lemma \ref{density} and the property
(iii)\ is a consequence of Lemma \ref{minProp}.
\end{proof}

\subsection{The operators $\Omega _{\protect\sigma }$ are Markov generators.}

In this Subsection we prove that the operators $\Omega _{\sigma }$ defined in
Sections \ref{Absorbing}, \ref{Reflecting}, \ref{Mixed}, \ref{Supercrit} are
Markov generators in the sense of Definition \ref{MarGen}. Due to
Proposition \ref{PropPG} we just need to prove that the operators $\Omega
_{\sigma }$ are closed and that $\mathcal{R}\left( I-\lambda \Omega _{\sigma
}\right) =C\left( X\right) $ for $\lambda >0$ small. The proof of these
results is a bit technical, due to the conditions imposed at the singular
point $\left( x,v\right) =\left( 0,0\right) $ in the definition of the
domains $\mathcal{D}(\Omega _{\sigma }).$

\subsubsection{Closure of $\Omega _{\protect\sigma }.$}

We first show that the operators $\Omega _{\sigma }$ are closed, that is,
the graph\\
$\left\{ \left( \varphi ,\Omega _{\sigma }\varphi \right) \ |\
\varphi \in \mathcal{D}(\Omega _{\sigma })\right\} $ are closed in $C\left(
X\right) \times C\left( X\right) $. First we establish the closure property
away from the singular set $\left( 0,0\right) .$

\begin{lemma}
\label{cl_out}Let $\Omega _{\sigma }$ be one of the operators $\Omega
_{t,sub},\ \Omega _{nt,sub},\ \Omega _{pt,sub},\ \Omega _{sup}.$ Assume $%
\left( \varphi _{m},\Omega _{\sigma }\varphi _{m}\right) $ with $\varphi
_{m}\in \mathcal{D}(\Omega _{\sigma })$ converges to $\left( \varphi
,w\right) $ in $C\left( X\right) \times C\left( X\right) $. Then
\begin{equation*}
\varphi ,\ D_{x}\varphi ,\ D_{v}\varphi ,\;D_{v}^{2}\varphi \in
L_{loc}^{p}\left( \mathcal{U}\right)
\end{equation*}
for any $p\in \left( 1,\infty \right) $ and $w=\left(
D_{v}^{2}+vD_{x}\right) \varphi $ in $\mathcal{U}.$
\end{lemma}

\begin{proof}
Since $w_{n}=\mathcal{L}\left( \varphi _{n}\right) \rightarrow w\in C\left(
X\right) ,$ we have the following, for any $\zeta \in \mathcal{F}\left(
\mathcal{U}\right) $:
\begin{eqnarray*}
\int \varphi \mathcal{L}^{\ast }\left( \zeta \right) &\leftarrow &\int
\varphi _{n}\mathcal{L}^{\ast }\left( \zeta \right) =\int \mathcal{L}\left(
\varphi _{n}\right) \zeta \\
&=&\int w_{n}\zeta \rightarrow \int w\zeta .
\end{eqnarray*}

Thus we get \
\begin{equation*}
w=\mathcal{L}\left( \varphi \right) \in C\left( X\right) .
\end{equation*}

The regularity properties of $\varphi $ are a consequence of the
hypoellipticity properties of the operator $\mathcal{L}$ (cf. Theorem \ref%
{Hypoell}).
\end{proof}

The uniform estimate obtained in Proposition \ref{mu_R} enables us to derive
information near the singular point for sequences $\left\{ \varphi
_{m}\right\} $ such that $\left\{ \Omega _{\sigma }\varphi _{m}\right\} $ is
convergent in $C\left( X\right) .$

\begin{lemma}
\label{limitA}Let $\left\{ \left( \varphi _{m},\Omega _{\sigma }\varphi
_{m}\right) \right\} $ be a sequence of $C\left( X\right) \times C\left(
X\right) $ with $\left\{ \varphi _{m}\right\} \subset \mathcal{D}(\Omega
_{\sigma })$ where $\Omega _{\sigma }$ is any of the operators $\Omega
_{t,sub},\ \Omega _{nt,sub},\ \Omega _{pt,sub}.$ Let us assume that the
sequence $\left\{ \left( \varphi _{m},\Omega _{\sigma }\varphi _{m}\right)
\right\} $ converges to $\left( \varphi ,w\right) $ in $C\left( X\right)
\times C\left( X\right) .$ Then the sequence of numbers $\left\{ \mathcal{A}%
\left( \varphi _{m}\right) \right\} $ defined by means of (\ref{phi_decom})
is a Cauchy sequence and therefore $\lim_{m\rightarrow \infty }\mathcal{A}%
\left( \varphi _{m}\right) =L$ exists.
\end{lemma}

\begin{proof}
In the case of the operators $\Omega _{nt,sub}$ we have, due to (\ref{S7E2})
that $\mathcal{A}\left( \varphi _{m}\right) =0$ and the result follows
trivially. Suppose then that $\Omega _{\sigma }$ is one of the operators $%
\Omega _{t,sub},\ \Omega _{pt,sub}.$ We first show that $\left\{ \mathcal{A}%
\left( \varphi _{m}\right) \right\} $ is uniformly bounded in $m$. Let $%
\varphi _{n}\left( x,v\right) =\varphi _{n}\left( 0,0\right) +\mathcal{A}%
\left( \varphi _{n}\right) F_{\beta }\left( x,v\right) +\psi _{n}\left(
x,v\right) .$ By assumption$~\varphi _{n}\rightarrow \varphi ,~\mathcal{L}%
\varphi _{n}\mathcal{\rightarrow L}\varphi $ in $C\left( X\right) .$ Then $%
\left\vert \mathcal{L}\varphi _{n}\right\vert \leq A,$ for some $0<A<\infty $
and $\left\vert \varphi _{n}\left( x,v\right) -\varphi _{n}\left( 0,0\right)
\right\vert \leq C.$ We consider a super-solution $\bar{\varphi}_{+}$ and a
sub-solution $\bar{\varphi}_{-}$ of the form%
\begin{equation*}
\bar{\varphi}_{\pm }\left( x,v\right) =\bar{C}F_{\beta }\left( x,v\right)
\pm AS\left( x,v\right) ,
\end{equation*}%
where $S\left( x,v\right) $ satisfies $\mathcal{L}S=1$ as in Lemma \ref%
{super_sol}. Then $\bar{\varphi}_{+}\left( x,v\right) \geq \frac{\bar{C}}{2}%
F_{\beta }\left( x,v\right) $ and $\bar{\varphi}_{-}\left( x,v\right) \leq
\frac{\bar{C}}{2}F_{\beta }\left( x,v\right) $ in $x+\left\vert v\right\vert
^{3}=\rho $ if $\rho >0$ is sufficiently small. Then $\left\vert \varphi
_{n}\left( x,v\right) -\varphi _{n}\left( 0,0\right) \right\vert \leq K\bar{%
\varphi}_{\pm }\left( x,v\right) $ for $x+\left\vert v\right\vert ^{3}=\rho ,
$ with $K>0$ sufficiently large. On the other hand since $\varphi _{n}\in
C\left( X\right) ,$ given $\varepsilon >0$ we have $\left\vert \varphi
_{n}\left( x,v\right) -\varphi _{n}\left( 0,0\right) \right\vert \leq
\varepsilon $ if $x+\left\vert v\right\vert ^{3}=\delta $ for any $\delta >0$
sufficiently small, depending on $\varepsilon .$ Then the Weak Maximum
Principle in Proposition \ref{weak_max} applied to $\varphi _{n}\left(
x,v\right) -\varphi _{n}\left( 0,0\right) $ and $\left( \bar{\varphi}_{\pm
}\left( x,v\right) \pm \varepsilon \right) $ yields \
\begin{equation*}
\left\vert \varphi _{n}\left( x,v\right) -\varphi _{n}\left( 0,0\right)
\right\vert \leq \frac{K\bar{C}}{2}F_{\beta }\left( x,v\right) +\varepsilon ,%
\text{ for all }\left( x,v\right) \text{ with }\delta \leq x+\left\vert
v\right\vert ^{3}\leq \rho
\end{equation*}%
and taking the limit $\delta \rightarrow 0,$ $\varepsilon \rightarrow 0$ we
obtain:%
\begin{equation*}
\left\vert \varphi _{n}\left( x,v\right) -\varphi _{n}\left( 0,0\right)
\right\vert \leq \frac{K\bar{C}}{2}F_{\beta }\left( x,v\right) ,\text{ for
all }\left( x,v\right) \text{ with }0\leq x+\left\vert v\right\vert ^{3}\leq
\rho .
\end{equation*}%
This implies%
\begin{equation*}
\left\vert \mathcal{A}\left( \varphi _{n}\right) +\frac{\psi _{n}\left(
x,v\right) }{F_{\beta }\left( x,v\right) }\right\vert \leq \frac{K\bar{C}}{2}%
,~\text{for all }\left( x,v\right) \text{ with }x+\left\vert v\right\vert
^{3}\leq \rho .
\end{equation*}%
Then applying Proposition \ref{mu_R} yields, as $x+\left\vert v\right\vert
^{3}\rightarrow 0,$
\begin{equation*}
\left\vert \mathcal{A}\left( \varphi _{n}\right) \right\vert \leq \frac{K%
\bar{C}}{2}.
\end{equation*}%
Next, we show that $\left\{ \mathcal{A}\left( \varphi _{m}\right) \right\} $
is a Cauchy sequence. A similar argument can be applied to $\varphi
_{n}\left( x,v\right) -\varphi _{m}\left( x,v\right) -\varphi _{n}\left(
0,0\right) -\varphi _{m}\left( 0,0\right) $ for large $n,m$\ to obtain, for
all $\left( x,v\right) $ with $x+\left\vert v\right\vert ^{3}\leq \rho ,$%
\begin{equation}
\left\vert \varphi _{n}\left( x,v\right) -\varphi _{m}\left( x,v\right)
-\varphi _{n}\left( 0,0\right) -\varphi _{m}\left( 0,0\right) \right\vert
\leq \varepsilon F_{\beta }\left( x,v\right) ,\text{ }\   \label{T5E2}
\end{equation}%
where the $\varepsilon >0$ on the right-hand side of (\ref{T5E2}) can be
chosen arbitrarily small due to the fact that $\left\{ \varphi _{n}\right\} $
converges uniformly in $X,$ in particular at $\left( x,v\right) =0$ and for $%
x+\left\vert v\right\vert ^{3}=\rho >0.$ Then:%
\begin{equation*}
\left\vert \mathcal{A}\left( \varphi _{n}\right) -\mathcal{A}\left( \varphi
_{m}\right) +\frac{\psi _{n}\left( x,v\right) -\psi _{m}\left( x,v\right) }{%
F_{\beta }\left( x,v\right) }\right\vert \leq \varepsilon ,~\text{for all }%
\left( x,v\right) \text{ with }x+\left\vert v\right\vert ^{3}\leq \rho
\end{equation*}%
where $\psi _{n},\ \psi _{m}$ are chosen as in (\ref{phi_decom}) and $m,n$
large enough. The definition of $\mathcal{M}$ in (\ref{T5E1}) implies that
\begin{equation}
\left\vert \mathcal{A}\left( \varphi _{n}\right) -\mathcal{A}\left( \varphi
_{m}\right) \right\vert \leq \varepsilon +\frac{\mu \left( R\right) }{%
F_{\beta }\left( x,v\right) },~\text{for all }\left( x,v\right) \text{ with }%
x+\left\vert v\right\vert ^{3}=R\leq \rho   \label{T5E4}
\end{equation}

Notice that the right-hand side of (\ref{T5E4}) is independent of $n,m.$
Using Proposition \ref{mu_R} and taking the limit $R\rightarrow \infty $ we
obtain:%
\begin{equation*}
\left\vert \mathcal{A}\left( \varphi _{n}\right) -\mathcal{A}\left( \varphi
_{m}\right) \right\vert \leq \varepsilon
\end{equation*}%
if $n,m$ are sufficiently large. Therefore $\left\{ \mathcal{A}\left(
\varphi _{m}\right) \right\} $ is a Cauchy sequence and the result follows.
\end{proof}

We show in the following lemma that the limit of functions in the domain of
the operators $\Omega _{\sigma }$ has the asymptotics required near the
singular set.

\begin{lemma}
\label{cl_in}Let $\left\{ \left( \varphi _{m},\Omega _{\sigma }\varphi
_{m}\right) \right\} $ be a sequence of $C\left( X\right) \times C\left(
X\right) $ with $\left\{ \varphi _{m}\right\} \subset \mathcal{D}(\Omega
_{\sigma })$ where $\Omega _{\sigma }$ is any of the operators $\Omega
_{t,sub},\ \Omega _{nt,sub},\ \Omega _{pt,sub},\ \Omega _{sup}.$ Suppose
that the functions $\varphi _{m}$ satisfy (\ref{phi_decom}), i.e., they can
be written as:%
\begin{equation}
\varphi _{m}\left( x,v\right) =\varphi _{m}\left( 0,0\right) +\mathcal{A}%
\left( \varphi _{m}\right) F_{\beta }\left( x,v\right) +\psi _{m}\left(
x,v\right) \   \label{T5E5}
\end{equation}%
where $\psi _{m}$ satisfies the last identity in (\ref{phi_decom}). Let us
assume that the sequence $\left\{ \left( \varphi _{m},\Omega _{\sigma
}\varphi _{m}\right) \right\} $ converges to $\left( \varphi ,w\right) $ in $%
C\left( X\right) \times C\left( X\right) $. Then $\varphi $ satisfies (\ref%
{phi_decom}) for suitable $\mathcal{A}\left( \varphi \right) $ and $\psi .$
Moreover, we have:%
\begin{eqnarray}
\lim_{m\rightarrow \infty }\varphi _{m}\left( 0,0\right) &=&\varphi \left(
0,0\right) ,  \label{T6E1} \\
\lim_{m\rightarrow \infty }\mathcal{A}\left( \varphi _{m}\right) &=&\mathcal{%
A}\left( \varphi \right) ,  \label{T6E2} \\
\lim_{m\rightarrow \infty }\psi _{m}\left( x,v\right) &=&\psi \left(
x,v\right) ,  \label{T6E3} \\
\lim_{m\rightarrow \infty }\mathcal{L}\varphi _{m}\left( x,v\right)
&=&w\left( x,v\right) =\mathcal{L}\varphi \left( x,v\right)  \label{T6E4} \\
\lim_{m\rightarrow \infty }\left( \mathcal{L}\psi _{m}\right) \left(
x,v\right) &=&\left( \mathcal{L}\psi \right) \left( x,v\right) ,
\label{T6E5}
\end{eqnarray}
\end{lemma}

\begin{proof}
We define $\psi _{m}\left( x,v\right) $ by means of (\ref{T5E5}). Using the
assumptions in Lemma \ref{cl_in} as well as Lemma \ref{limitA} we obtain
that $\lim_{m\rightarrow \infty }\psi _{m}\left( x,v\right) $ exists,
uniformly in $X$ and:%
\begin{equation}
\psi \left( x,v\right) :=\lim_{m\rightarrow \infty }\psi _{m}\left(
x,v\right) =\varphi \left( x,v\right) -\varphi \left( 0,0\right) -LF_{\beta
}\left( x,v\right) \   \label{T5E6}
\end{equation}%
with $L$ as in Lemma \ref{limitA}. Notice that, using comparison arguments
as in the Proof of Lemma \ref{limitA} we obtain the uniform estimate $%
\left\vert \psi _{m}\left( x,v\right) \right\vert \leq CF_{\beta }\left(
x,v\right) $ for some suitable $C>0$ independent of $m$ and $x+\left\vert
v\right\vert ^{3}\leq 1.$ Moreover, we have $\psi _{m}\left( x,v\right)
=o\left( x+\left\vert v\right\vert ^{3}\right) $ as $\left( x,v\right)
\rightarrow 0.$ Then $\frac{\psi _{m}}{C}\in \mathcal{M}$ with $\mathcal{M}$
as in (\ref{T5E1}). Using Proposition \ref{mu_R} we obtain the uniform
estimate $\left\vert \psi _{m}\left( x,v\right) \right\vert \leq \mu \left(
R\right) $ if $x+\left\vert v\right\vert ^{3}=R,$ with $\lim_{R\rightarrow 0}%
\frac{\mu \left( R\right) }{R^{\beta }}=0.$ Then $\left\vert \psi \left(
x,v\right) \right\vert \leq \mu \left( R\right) $ if $x+\left\vert
v\right\vert ^{3}=R$, whence $\psi $ satisfies the last identity in (\ref%
{phi_decom}) and therefore $\varphi $ satisfies (\ref{phi_decom}). Using
then (\ref{T5E6}) we obtain $L=\mathcal{A}\left( \varphi \right) .$ We then
have (\ref{T6E2}), (\ref{T6E3}). The identity (\ref{T6E1}) is a consequence
of the convergence of $\left\{ \varphi _{m}\right\} $ in $C\left( X\right) .$
The identity (\ref{T6E4}) follows from the definition of the operator $%
\mathcal{L}$ in Definition \ref{LadjDef} since we can take the limit on the
left-hand side of that formula. Finally (\ref{T6E5}) follows from the fact
that $\mathcal{L}F_{\beta }=0.$
\end{proof}

The closure of the operators $\Omega _{\sigma }$ is just a consequence of
the two previous Lemmas.

\begin{proposition}
\label{Closed}Let $\left\{ \left( \varphi _{m},\Omega _{\sigma }\varphi
_{m}\right) \right\} $ be a sequence of $C\left( X\right) \times C\left(
X\right) $ with $\left\{ \varphi _{m}\right\} \subset \mathcal{D}(\Omega
_{\sigma })$ where $\Omega _{\sigma }$ is any of the operators $\Omega
_{t,sub},\ \Omega _{nt,sub},\ \Omega _{pt,sub},\ \Omega _{sup}.$Assume that $%
\left\{ \left( \varphi _{m},\Omega _{\sigma }\varphi _{m}\right) \right\} $
converges to $\left( \varphi ,w\right) $ in $C\left( X\right) \times C\left(
X\right) $. Then $\varphi \in \mathcal{D}(\Omega _{\sigma })$ and $w=%
\mathcal{L}\varphi .$
\end{proposition}

\begin{remark}
Notice that Proposition \ref{Closed} just means that the operators $\Omega
_{\sigma }$ are closed.
\end{remark}

\begin{proof}
It is an immediate consequence from Lemmas \ref{limitA} and \ref{cl_in}.
\end{proof}

\subsubsection{$\mathcal{R}\left( I-\protect\lambda \Omega _{\protect\sigma %
}\right) =C\left( X\right) $ for $\protect\lambda >0$.}

In order to conclude the Proof of the fact that the operators $\Omega
_{\sigma }$ are Markov generators it only remains to prove that $\mathcal{R}%
\left( I-\lambda \Omega _{\sigma }\right) =C\left( X\right) $ for $\lambda >0
$ small (cf. Definition \ref{MarGen}). This is equivalent to prove that it
is possible to solve the problems
\begin{equation}
\left( \varphi -\lambda \Omega _{\sigma }\varphi \right) =g,\ \ g\in C\left(
X\right)   \label{solvability}
\end{equation}%
with $\varphi \in C\left( X\right) \cap \mathcal{D}\left( \Omega _{\sigma
}\right) $ for any $\lambda >0$ small, $\Omega _{\sigma }\in \left\{ \Omega
_{t,sub},\ \Omega _{nt,sub},\ \Omega _{pt,sub},\ \Omega _{sup}\right\} $.

The definition of the operators $\Omega _{\sigma }$ in Section \ref%
{Operators} and in particular the choice of domains $\mathcal{D}\left(
\Omega _{\sigma }\right) $ (cf. (\ref{S7E1}), (\ref{S7E2}), (\ref{S7E3}), (%
\ref{S7E4})) allows us to reformulate (\ref{solvability}) by means of PDE
problems with suitable boundary conditions at the singular point $\left(
x,v\right) =\left( 0,0\right) .$\ More precisely, the equation (\ref%
{solvability}) is equivalent in the case of the four operators $\Omega
_{t,sub},\ \Omega _{nt,sub},\ \Omega _{pt,sub},\ \Omega _{sup}$ to the
equation:%
\begin{equation}
\left( \varphi -\lambda \mathcal{L}\varphi \right) =g\ \ \text{in\ \ }%
\mathcal{U}\ \ ,\ \varphi \in C\left( X\right) \ \ ,\ \ \varphi \left(
0,rv\right) =\varphi \left( 0,-v\right) \ \ ,\ \ v>0\   \label{T6E6}
\end{equation}%
where $\varphi $ satisfies (\ref{phi_decom}) and we impose the following
boundary conditions for each of the cases:

\begin{eqnarray}
\varphi \left( 0,0\right) &=&g\left( 0,0\right) \ \ \text{for }\Omega
_{t,sub}\ \ \text{(cf. (\ref{S7E1}))}  \label{T6E7} \\
\mathcal{A}\left( \varphi \right) &=&0\ \ \text{for\ \ }\Omega _{nt,sub}\ \
\text{(cf. (\ref{S7E2}))}  \label{T6E8} \\
\lambda \mu _{\ast }\left\vert C_{\ast }\right\vert \mathcal{A}\left(
\varphi \right) &=&\lambda \left( \mathcal{L}\varphi \right) \left(
0,0\right) =\left( \varphi \left( 0,0\right) -g\left( 0,0\right) \right)
\text{ for }\Omega _{pt,sub}\ \text{(cf. (\ref{S7E3}))\ }  \label{T6E9} \\
&&\text{No boundary condition\ at }\left( 0,0\right) \ \text{for }\Omega
_{sup}\ \ \text{(cf. (\ref{S7E4}))}  \label{T6E10}
\end{eqnarray}

The operator $\mathcal{L}\varphi $ is understood as in Definition \ref%
{LadjDef}. Notice that the function $\mathcal{L}\varphi $ is continuous in $%
\overline{\mathcal{U}}$ and therefore the conditions (\ref{T6E7})-(\ref{T6E9}%
) are meaningful.

We summarize the result just obtained as follows:

\begin{proposition}
\label{Equivalence}The problem (\ref{solvability}) with $\Omega _{\sigma }$
as one of the operators
\begin{equation*}
\Omega _{t,sub},\ \Omega _{nt,sub},\ \Omega _{pt,sub},\ \Omega _{sup}
\end{equation*}
is equivalent to solving the PDE problem (\ref{T6E6}) in the class of
functions $\varphi \in C\left( X\right) ,$ with $\mathcal{L}\varphi \in
C\left( X\right) $ with the boundary conditions (\ref{T6E7}), (\ref{T6E8}), (%
\ref{T6E9}) and (\ref{T6E10}) respectively and where the operator is
understood as in Definition \ref{LadjDef}.
\end{proposition}

We now consider the solvability of the problem (\ref{T6E6}) with boundary
conditions (\ref{T6E7}), (\ref{T6E8}), (\ref{T6E9}). This will be made using
suitable adaptations of the classical Perron's method (cf. for instance 
\cite{GT}) for harmonic functions for each of the specific problems under
consideration. To this end, we will use the solution for the Dirichlet
problem in admissible domains obtained in Proposition \ref{DirSolv}.

\subsubsection{Some technical results\label{Tech}}

In order to solve the PDE problems stated in Proposition 5.50, 
we will need the following technical results.

\begin{lemma}
\label{TraceIde}Suppose that $\varphi _{1},\varphi _{2}$ with $\varphi
_{k}\in L^{\infty }\left( \mathcal{W}_{k}^{-}\right) ,$ $k=1,2$ are two
subsolutions of (\ref{T6E6}), (\ref{T6E7}) in the sense of Definition \ref%
{SubSuper1}, where\ the domains $\mathcal{W}_{1},\ \mathcal{W}_{2}$ satisfy $%
\mathcal{W}_{1}^{-}\cap \mathcal{W}_{2}^{-}\cap \left\{ x=0\right\} \neq
\varnothing .$ Then the boundary values of $\varphi _{1},\varphi _{2}$
defined in $\mathcal{W}_{1}^{-}\cap \mathcal{W}_{2}^{-}\cap \left\{
x=0\right\} $ in the sense of traces by means of Proposition \ref{Traces}
satisfy:%
\begin{equation}
\left( \max \left\{ \varphi _{1},\varphi _{2}\right\} \right) ^{+}=\max
\left\{ \varphi _{1}^{+},\varphi _{2}^{+}\right\} \ \ a.e.\ \text{in }%
\mathcal{W}_{1}^{-}\cap \mathcal{W}_{2}^{-}\cap \left\{ x=0\right\}
\label{X3}
\end{equation}
\end{lemma}

\begin{proof}
The inequality (\ref{sub_sol}) implies that there exists a measure $\nu \in
\mathcal{M}_{+}\left( \mathcal{W}_{1}\cap \mathcal{W}_{2}\right) $ such that:%
\begin{equation}
\varphi -\lambda \mathcal{L}\left( \varphi \right) -g=-\nu  \label{U2E7}
\end{equation}

Due to Lemma \ref{L1traces} we obtain $\varphi _{k}\rightarrow \left(
\varphi _{k}\right) ^{+}$ as $x\rightarrow 0^{+}$ in $L^{1}\left( K\right) $
for $K\subset \mathcal{W}_{1}^{-}\cap \mathcal{W}_{2}^{-}\cap \left\{
x=0\right\} ,\ k=1,2.$ Then $\max \left\{ \varphi _{1},\varphi _{2}\right\}
\rightarrow \max \left\{ \varphi _{1}^{+},\varphi _{2}^{+}\right\} $ in $%
L^{1}\left( K\right) $ due to Lebesgue dominated convergence Theorem. Using
the definition of Traces (see Remark \ref{TraceDef}) it then follows that $%
\left( \max \left\{ \varphi _{1},\varphi _{2}\right\} \right) ^{+}$ is well
defined and (\ref{X3}) holds.
\end{proof}

We will use also the following continuity result for the traces.

\begin{lemma}
\label{TraceConvergence}Suppose that we have a sequence of bounded functions
$\left\{ \varphi _{n}\right\} $ satisfying $-\mathcal{L}\varphi _{n}=\mu
_{n}+g_{n}$ where $g_{n}\in C\left( \Xi \right) ,$ the functions $g_{n}$ are
uniformly bounded and $\mu _{n}\geq 0$ are Radon measures. We assume that $%
\Xi $ is an admissible domain . Suppose also that the sequence $\left\{
\varphi _{n}\right\} $ converges to $\varphi $ in the weak topology and $\mu
_{n}\rightarrow 0$ also in the weak topology. Suppose that $\Gamma $ is a
curve made of horizontal and vertical lines contained in $\Xi .$ Then, given
any test function $\psi \in C^{\infty }\left( \Xi \right) ,$ the following
convergence property holds for the traces of $\varphi _{n}$ and its
derivatives at the curve $\Gamma $%
\begin{equation*}
\int_{\Gamma }\left( n_{v}\psi \partial _{v}\varphi _{n}-n_{v}\partial
_{v}\psi \varphi _{n}+v\varphi _{n}\psi n_{x}\right) ds\rightarrow
\int_{\Gamma }\left( n_{v}\psi \partial _{v}\varphi -n_{v}\partial _{v}\psi
\varphi +v\varphi \psi n_{x}\right) ds
\end{equation*}%
as $n\rightarrow \infty $.
\end{lemma}

\begin{remark}
The traces of the functions $\varphi _{n}$ can be defined at any side of the
boundary $\Gamma .$
\end{remark}

\begin{proof}
We can first substract from the functions $\varphi _{n}$ the solutions of
the equations $-\mathcal{L}\varphi _{n}=g_{n}$\ with zero boundary
conditions in the admissible boundary of $\Xi .$ These solutions are
uniformly bounded and smooth at the curve $\Gamma $ and then, the
corresponding traces have the desired convergence properties. We can then
assume without loss of generality that $-\mathcal{L}\varphi _{n}=\mu _{n}.$

Given any other curve $\tilde{\Gamma}$ with dimensions comparable to $\Gamma
$ such that $\Gamma \cup \tilde{\Gamma}=\partial \Xi $ for some admissible
domain $\Xi $ we can estimate the difference between the fluxes on $\tilde{%
\Gamma}$ and $\Gamma $ integrating by parts and using the equation $-%
\mathcal{L}\varphi _{n}=\mu _{n}.$ The contribution due to $\mu _{n}$ tends
to zero, due to the weak convergence of the measures. Some of the integrals $%
\int_{\Gamma }\left[ \cdot \cdot \cdot \right] ds$ contain integrations of
the functions $\varphi _{n}.$ The integrals of $\varphi _{n}v$ and related
functions can be represented by means of the integrals in horizontal and
vertical lines using Fubini. In particular, it is possible to approximate
the integral $\int_{\Gamma }\left[ \cdot \cdot \cdot \right] ds$ by similar
integrals in other contours $\tilde{\Gamma}$ due to the weak convergence of
the functions $\varphi _{n}$ as well as the fact that the difference of
values between these integrals converges to zero as $n\rightarrow \infty ,$
because these differences are proportional to the integrals of $\mu _{n}$ on
two-dimensional domains which converge to zero by assumption as $%
n\rightarrow \infty .$ In order to conclude the argument we need to obtain
suitable convergences for the integrals containing the derivatives $\partial
_{v}\varphi _{n}.$ This regularity follows from the hypoellipticity
properties in Theorem \ref{Hypoell}. The same argument above, which allows
to transform convergences in two-dimensional domains into convergences for $%
\int_{\Gamma }\left[ \cdot \cdot \cdot \right] ds$ using the fact that the
difference between these integrals converges to zero, gives the result.\
\bigskip
\end{proof}

\subsubsection{Operator $\Omega _{t,sub}:$ Solvability of (\protect\ref{T6E6}),
(\protect\ref{T6E7}).\label{SolvAbs}}

We first solve (\ref{T6E6}), (\ref{T6E7}). We will prove the following
result:

\begin{proposition}
\label{Case_a_sub}For any $g\in C\left( X\right) $ there exists a unique $%
\varphi \in C\left( X\right) $ which solves (\ref{T6E6}), (\ref{T6E7}).
\end{proposition}


In order to prove Proposition \ref{Case_a_sub} we need to define a suitable
concept of subsolution and supersolutions for (\ref{T6E6}), (\ref{T6E7}) for
which the boundary condition at the singular point $\left( 0,0\right) $
holds. We recall that the space of functions $L_{b}^{\infty }\left( X\right)
$ has been defined in Definition \ref{SpaceB}. It is worth noticing that the
definition of sub/supersolutions which will be made in the following involve
also the boundary conditions at the singular point, differently from the
sub/supersolutions for the operator $\mathcal{L}$ in Definition \ref%
{supersolDef} where such boundary conditions at $\left( 0,0\right) $ were
not included. The sub/supersolutions for the operator $\mathcal{L}$ in
Definition \ref{supersolDef} are suitable for comparison arguments in
admissible domains, while the sub/supersolutions defined here are suitable
to obtain global results in the domain $X,$ in particular well-posedness
results.

\begin{definition}
\label{SubSuper1}Suppose that $g\in C\left( X\right) .$ We will say that a
function $\varphi \in L_{b}^{\infty }\left( X\right) $ such that the limit $%
\lim_{\left( x,v\right) \rightarrow \left( 0,0\right) }\varphi \left(
x,v\right) =\varphi \left( 0,0\right) $ exists, is a subsolution of (\ref%
{T6E6}), (\ref{T6E7}) if $\varphi \left( 0,0\right) =g\left( 0,0\right) $
and for all $\psi \in \mathcal{F}\left( X\right) $ (cf. (\ref{defF})) with $%
\psi \geq 0$ we have:%
\begin{equation}
\int \left( \psi -\lambda \mathcal{L}^{\ast }\left( \psi \right) \right)
\varphi \leq \int g\psi  \label{sub_sol}
\end{equation}

Given $g\in C\left( X\right) ,$ we will say that $\varphi \in L_{b}^{\infty
}\left( X\right) $ is a supersolution of (\ref{T6E6}), (\ref{T6E7}) if $%
\varphi \left( 0,0\right) =g\left( 0,0\right) $ and for all $\psi \in
\mathcal{F}\left( X\right) $ with $\psi \geq 0$ we have:%
\begin{equation}
\int \left( \psi -\lambda \mathcal{L}^{\ast }\left( \psi \right) \right)
\varphi \geq \int g\psi  \label{super_sol2}
\end{equation}

Suppose that $\mathcal{W}$ is an open subset of $X$. We will say that $%
\varphi \in L_{b}^{\infty }\left( \mathcal{W}\right) $ is a subsolution of (%
\ref{T6E6}), (\ref{T6E7}) in $\mathcal{W}$ if the following conditions hold:
(1) If $\left( 0,0\right) \in \mathcal{W}$ we have $\varphi \left(
0,0\right) =g\left( 0,0\right) $. (2) The inequality (\ref{sub_sol}) holds
for any function $\psi \in \mathcal{F}\left( \mathcal{W}\right) ,$ $\psi
\geq 0$.
\end{definition}

A key property of sub and supersolutions of (\ref{T6E6}), (\ref{T6E7}) is
the following:

\begin{lemma}
\label{maxima}Suppose that $\mathcal{W}_{1},\ \mathcal{W}_{2}$ are two open
subsets of $X$ and that $\varphi _{1}\in L_{b}^{\infty }\left( \mathcal{W}%
_{1}\right) ,\ \varphi _{2}\in L_{b}^{\infty }\left( \mathcal{W}_{2}\right) $
are two subsolutions of (\ref{T6E6}), (\ref{T6E7}) in the sense of
Definition \ref{SubSuper1} in their respective domains. Then the function $%
\varphi $ defined by means of $\varphi =\max \left\{ \varphi _{1},\varphi
_{2}\right\} $ in $\mathcal{W}_{1}\cap \mathcal{W}_{2},\ \varphi =\varphi
_{1}$ in $\mathcal{W}_{1}\diagdown \left( \mathcal{W}_{1}\cap \mathcal{W}%
_{2}\right) ,\ \varphi =\varphi _{2}$ in $\mathcal{W}_{2}\diagdown \left(
\mathcal{W}_{1}\cap \mathcal{W}_{2}\right) $ is a subsolution of (\ref{T6E6}%
), (\ref{T6E7}) in $\mathcal{W}=\left( \mathcal{W}_{1}\cup \mathcal{W}%
_{2}\right) $ in the sense of Definition \ref{SubSuper1}.

Suppose that $\mathcal{W}_{1},\ \mathcal{W}_{2}$ are two open subsets of $X$
and that $\varphi _{1}\in L_{b}^{\infty }\left( \mathcal{W}_{1}\right) ,\
\varphi _{2}\in L_{b}^{\infty }\left( \mathcal{W}_{2}\right) $ are two
supersolutions of (\ref{T6E6}), (\ref{T6E7}) in the sense of Definition \ref%
{SubSuper1} in their respective domains. Then the function $\varphi $
defined by means of $\varphi =\min \left\{ \varphi _{1},\varphi _{2}\right\}
$ in $\mathcal{W}_{1}\cap \mathcal{W}_{2},\ \varphi =\varphi _{1}$ in $%
\mathcal{W}_{1}\diagdown \left( \mathcal{W}_{1}\cap \mathcal{W}_{2}\right)
,\ \varphi =\varphi _{2}$ in $\mathcal{W}_{2}\diagdown \left( \mathcal{W}%
_{1}\cap \mathcal{W}_{2}\right) $ is a supersolution of (\ref{T6E6}), (\ref%
{T6E7}) in $\mathcal{W}=\left( \mathcal{W}_{1}\cup \mathcal{W}_{2}\right) $
in the sense of Definition \ref{SubSuper1}.
\end{lemma}

\begin{remark}
Notice that the functions $\varphi _{1},\varphi _{2}$ are defined only
almost everywhere. Nevertheless it is well known that the function $\max
\left\{ \varphi _{1},\varphi _{2}\right\} $ can be defined as a $%
L_{b}^{\infty }$ function defined also almost everywhere.
\end{remark}

\begin{proof}
We will follow different strategies in order to obtain the subsolution
inequality for test functions $\psi $ whose support does not intersect $%
\left\{ x=0\right\} $ and test functions $\psi $ with support intersecting $%
\left\{ x=0\right\} .$ Suppose that $\varphi _{1},\varphi _{2}$ are two
subsolutions as in the statement of the Lemma. Suppose that $\zeta
_{\varepsilon }\left( x,v\right) $ is a $C^{\infty }$ mollifier with the
form:%
\begin{equation*}
\zeta _{\varepsilon }\left( x,v\right) =\frac{1}{\varepsilon ^{3}}\zeta
\left( \frac{x}{\varepsilon },\frac{v}{\varepsilon ^{2}}\right) ,\ \
\varepsilon >0,\text{ }\zeta \geq 0,\int_{\mathbb{R}^{2}}\zeta =1,\ \operatorname{%
supp}\left( \zeta \right) \subset \left\{ 0\leq x\leq 1,\ \left\vert
v\right\vert \leq 1\right\}
\end{equation*}%
Given any $\delta >0$ we define:%
\begin{equation}
\mathcal{Z}_{\delta }=\left\{ \left( x,v\right) \in \mathcal{Z}:\operatorname{dist%
}\left( \left( x,v\right) ,\partial \mathcal{Z}\right) \geq \delta \right\}
\label{Ddomain}
\end{equation}

We fix $\delta >0$ and we then define functions $\varphi _{1,\varepsilon
},\varphi _{2,\varepsilon }$ in the sets $\left( \mathcal{W}_{1}^{\pm
}\right) _{\delta },\ \left( \mathcal{W}_{2}^{\pm }\right) _{\delta }\ $%
(cf.\ (\ref{U2E6})) with $2\varepsilon <\delta $ by means of $\varphi
_{k,\varepsilon }=\zeta _{\varepsilon }\ast \varphi _{k},$ $k=1,2.$ Given $%
\psi \in \mathcal{F}\left( \left( \mathcal{W}_{k}^{\pm }\right) _{\delta
}\right) $ with $\psi \geq 0$ we have:%
\begin{eqnarray*}
&&\int_{\left( \mathcal{W}_{k}^{\pm }\right) _{\frac{\delta }{2}}}\left(
\varphi _{k,\varepsilon }-\lambda \mathcal{L}\left( \varphi _{k,\varepsilon
}\right) \right) \psi \\
&=&\int_{\left( \mathcal{W}_{k}^{\pm }\right) _{\frac{\delta }{2}}}\left(
\psi -\lambda \mathcal{L}^{\ast }\left( \psi \right) \right) \varphi
_{k,\varepsilon } \\
&=&\int_{\left( \mathcal{W}_{k}^{\pm }\right) _{\frac{\delta }{2}}}\left(
\psi -\lambda \mathcal{L}^{\ast }\left( \psi \right) \right) \left( \varphi
_{k}\ast \zeta _{\varepsilon }\right) \\
&=&\int_{\left( \mathcal{W}_{k}^{\pm }\right) _{\frac{\delta }{2}}}\left(
\psi _{\varepsilon }-\lambda \left( \zeta _{\varepsilon }\ast \mathcal{L}%
^{\ast }\left( \psi \right) \right) \right) \varphi _{k} \\
&=&\int_{\left( \mathcal{W}_{k}^{\pm }\right) _{\frac{\delta }{2}}}\left(
\psi _{\varepsilon }-\lambda \left( \mathcal{L}^{\ast }\left( \psi
_{\varepsilon }\right) \right) \right) \varphi _{k}-\lambda \int_{\left(
\mathcal{W}_{k}^{\pm }\right) _{\frac{\delta }{2}}}\left( \left( \zeta
_{\varepsilon }\ast \mathcal{L}^{\ast }\left( \psi \right) \right) -\left(
\mathcal{L}^{\ast }\left( \zeta _{\varepsilon }\ast \psi \right) \right)
\right) \varphi _{k}
\end{eqnarray*}

It is relevant to remark that the functions $\psi _{\varepsilon }$ are well
defined in the sets $\left( \mathcal{W}_{k}^{\pm }\right) _{\delta }$ even
if these sets have a nonempty intersection with the line $\left\{
x=0\right\} ,$ due to our choice of the mollifiers $\zeta _{\varepsilon }$
which are supported in the region $x\geq 0.$ On the other hand, the
functions $\psi _{\varepsilon }$ do not belong in general to $\mathcal{F}%
\left( \left( \mathcal{W}_{k}\right) _{\delta }\right) $ because the
condition (\ref{comp1t}) does not necessarily holds.

Notice that:%
\begin{equation*}
\left\vert \left( \zeta _{\varepsilon }\ast \mathcal{L}^{\ast }\left( \psi
\right) \right) -\left( \mathcal{L}^{\ast }\left( \zeta _{\varepsilon }\ast
\psi \right) \right) \right\vert =\left\vert \zeta _{\varepsilon }\ast
\left( v\partial _{x}\psi \right) -v\partial _{x}\left( \zeta _{\varepsilon
}\ast \psi \right) \right\vert
\end{equation*}%
Using then the definition of $\zeta _{\varepsilon }$ we obtain:
\begin{equation*}
\int_{\left( \mathcal{W}_{k}^{\pm }\right) _{\frac{\delta }{2}}}\left\vert
\zeta _{\varepsilon }\ast \left( v\partial _{x}\psi \right) -v\partial
_{x}\left( \zeta _{\varepsilon }\ast \psi \right) \right\vert \leq
C\varepsilon \int_{\left( \mathcal{W}_{k}^{\pm }\right) _{\frac{\delta }{2}%
}}\psi
\end{equation*}%
where we use:%
\begin{eqnarray*}
&&\int_{\left( \mathcal{W}_{k}^{\pm }\right) _{\frac{\delta }{2}}}\left\vert
\zeta _{\varepsilon }\ast \left( v\partial _{x}\psi \right) -v\partial
_{x}\left( \zeta _{\varepsilon }\ast \psi \right) \right\vert dxdv \\
&=&\int_{\left( \mathcal{W}_{k}^{\pm }\right) _{\frac{\delta }{2}%
}}\left\vert \int \left[ \partial _{x}\zeta _{\varepsilon }\left(
x-y,v-w\right) w\psi \left( y,w\right) -v\partial _{x}\zeta _{\varepsilon
}\left( x-y,v-w\right) \psi \left( y,w\right) \right] dydw\right\vert dxdv \\
&=&\int_{\left( \mathcal{W}_{k}^{\pm }\right) _{\frac{\delta }{2}%
}}\left\vert \int \partial _{x}\zeta _{\varepsilon }\left( x-y,v-w\right)
\left( w-v\right) \psi \left( y,w\right) dydw\right\vert dxdv \\
&\leq &C\varepsilon ^{2}\int dydw\psi \left( y,w\right) \int_{\left(
\mathcal{W}_{k}^{\pm }\right) _{\frac{\delta }{2}}}\left\vert \partial
_{x}\zeta _{\varepsilon }\left( x-y,v-w\right) \right\vert dxdv\leq \frac{%
C\varepsilon }{\varepsilon }^{2}\int dydw\psi \left( y,w\right)
\end{eqnarray*}

Then:%
\begin{equation}
\left\vert \int_{\left( \mathcal{W}_{k}^{\pm }\right) _{\frac{\delta }{2}%
}}\left( \varphi _{k,\varepsilon }-\lambda \mathcal{L}\left( \varphi
_{k,\varepsilon }\right) \right) \psi -\int_{\left( \mathcal{W}_{k}^{\pm
}\right) _{\frac{\delta }{2}}}\left( \psi _{\varepsilon }-\lambda \left(
\mathcal{L}^{\ast }\left( \psi _{\varepsilon }\right) \right) \right)
\varphi _{k}\right\vert \leq C\varepsilon \left\Vert \varphi _{k}\right\Vert
_{L^{\infty }}\int_{\left( \mathcal{W}_{k}^{\pm }\right) _{\frac{\delta }{2}%
}}\psi  \label{Point}
\end{equation}

Using that $\varphi _{k}$ are subsolutions in $\mathcal{W}_{k}$ it then
follows that:%
\begin{equation*}
\int_{\left( \mathcal{W}_{k}^{\pm }\right) _{\frac{\delta }{2}}}\left(
\varphi _{k,\varepsilon }-\lambda \mathcal{L}\left( \varphi _{k,\varepsilon
}\right) \right) \psi \leq \int_{\left( \mathcal{W}_{k}^{\pm }\right) _{%
\frac{\delta }{2}}}g\psi +C\varepsilon \left\Vert \varphi _{k}\right\Vert
_{L^{\infty }}\int_{\left( \mathcal{W}_{k}^{\pm }\right) _{\frac{\delta }{2}%
}}\psi
\end{equation*}%
whence the following pointwise estimate follows:%
\begin{equation*}
\left( \varphi _{k,\varepsilon }-\lambda \mathcal{L}\left( \varphi
_{k,\varepsilon }\right) \right) \leq g+C\varepsilon \left\Vert \varphi
_{k}\right\Vert _{L^{\infty }}\ \ \text{in\ }\left( \mathcal{W}_{k}\right)
_{\delta }\text{ , }k=1,2
\end{equation*}

We next obtain the subsolution inequality for test functions $\psi $
supported in $\left( \mathcal{W}_{1}^{\pm }\right) _{\delta }\cup \left(
\mathcal{W}_{2}^{\pm }\right) _{\delta }$ for any $\delta >0.$ Indeed, we
can assume that the set $\left\{ \varphi _{1,\varepsilon }=\varphi
_{2,\varepsilon }\right\} \cap \left[ \left( \mathcal{W}_{1}^{\pm }\right)
_{\delta }\cap \left( \mathcal{W}_{2}^{\pm }\right) _{\delta }\right] $ is
non-empty, since otherwise the result would follow trivially. Then, Sard's
Lemma (cf. \cite{M}) implies that, for any $\varepsilon >0$ arbitrarily
small, there exists a sequence $\rho _{n}\left( \varepsilon \right)
\rightarrow 0$ as $n\rightarrow \infty ,$ such that the curves $\left\{
\varphi _{1,\varepsilon }=\varphi _{2,\varepsilon }+\rho _{n}\left(
\varepsilon \right) \right\} \subset \left( \mathcal{W}_{1}\right) _{\frac{%
\delta }{2}}\cap \left( \mathcal{W}_{2}\right) _{\frac{\delta }{2}}$ are
smooth. These curves separate the regions $\left\{ \varphi _{1,\varepsilon
}>\varphi _{2,\varepsilon }+\rho _{n}\left( \varepsilon \right) \right\} $
and $\left\{ \varphi _{1,\varepsilon }<\varphi _{2,\varepsilon }+\rho
_{n}\left( \varepsilon \right) \right\} .$ We define functions:%
\begin{eqnarray}
\varphi _{\varepsilon }^{\left( n\right) } &=&\max \left\{ \varphi
_{1,\varepsilon },\varphi _{2,\varepsilon }+\rho _{n}\left( \varepsilon
\right) \right\} \ \text{in\ }\left( \mathcal{W}_{1}^{\pm }\right) _{\frac{%
\delta }{2}}\cap \left( \mathcal{W}_{2}^{\pm }\right) _{\frac{\delta }{2}}
\notag \\
\varphi _{\varepsilon }^{\left( n\right) } &=&\varphi _{1,\varepsilon }\text{
in }\left( \mathcal{W}_{1}^{\pm }\right) _{\frac{\delta }{2}}\diagdown
\left( \left( \mathcal{W}_{1}^{\pm }\right) _{\frac{\delta }{2}}\cap \left(
\mathcal{W}_{2}^{\pm }\right) _{\frac{\delta }{2}}\right)   \label{DefMax} \\
\varphi _{\varepsilon }^{\left( n\right) } &=&\varphi _{2,\varepsilon }+\rho
_{n}\left( \varepsilon \right) \ \text{in }\left( \mathcal{W}_{2}^{\pm
}\right) _{\frac{\delta }{2}}\diagdown \left( \left( \mathcal{W}_{1}^{\pm
}\right) _{\frac{\delta }{2}}\cap \left( \mathcal{W}_{2}^{\pm }\right) _{%
\frac{\delta }{2}}\right)   \notag
\end{eqnarray}

We then compute:%
\begin{eqnarray}
&&\int_{\left( \mathcal{W}_{1}^{\pm }\right) _{\frac{\delta }{2}}\cup \left(
\mathcal{W}_{2}^{\pm }\right) _{\frac{\delta }{2}}}\left( \psi -\lambda
\mathcal{L}^{\ast }\left( \psi \right) \right) \varphi _{\varepsilon
}^{\left( n\right) }  \label{Sum} \\
&=&\int_{\left[ \left( \mathcal{W}_{1}^{\pm }\right) _{\frac{\delta }{2}%
}\diagdown \left( \left( \mathcal{W}_{1}^{\pm }\right) _{\frac{\delta }{2}%
}\cap \left( \mathcal{W}_{2}^{\pm }\right) _{\frac{\delta }{2}}\right) %
\right] \cup \left\{ \varphi _{1,\varepsilon }>\varphi _{2,\varepsilon
}+\rho _{n}\left( \varepsilon \right) \right\} }\left( \psi -\lambda
\mathcal{L}^{\ast }\left( \psi \right) \right) \varphi _{1,\varepsilon }+
\notag \\
&&+\int_{\left[ \left( \mathcal{W}_{2}^{\pm }\right) _{\frac{\delta }{2}%
}\diagdown \left( \left( \mathcal{W}_{1}^{\pm }\right) _{\frac{\delta }{2}%
}\cap \left( \mathcal{W}_{2}^{\pm }\right) _{\frac{\delta }{2}}\right) %
\right] \cup \left\{ \varphi _{1,\varepsilon }<\varphi _{2,\varepsilon
}+\rho _{n}\left( \varepsilon \right) \right\} }\left( \psi -\lambda
\mathcal{L}^{\ast }\left( \psi \right) \right) \left( \varphi
_{2,\varepsilon }+\rho _{n}\left( \varepsilon \right) \right)  \notag
\end{eqnarray}

Using the fact that the functions $\varphi _{k,\varepsilon }$ are smooth and
integrating by parts we obtain:%
\begin{eqnarray*}
&&\int_{\left[ \left( \mathcal{W}_{1}^{\pm }\right) _{\frac{\delta }{2}%
}\diagdown \left( \left( \mathcal{W}_{1}^{\pm }\right) _{\frac{\delta }{2}%
}\cap \left( \mathcal{W}_{2}^{\pm }\right) _{\frac{\delta }{2}}\right) %
\right] \cup \left\{ \varphi _{1,\varepsilon }>\varphi _{2,\varepsilon
}+\rho _{n}\left( \varepsilon \right) \right\} }\left( \psi -\lambda
\mathcal{L}^{\ast }\left( \psi \right) \right) \varphi _{1,\varepsilon } \\
&=&\int_{\left[ \left( \mathcal{W}_{1}^{\pm }\right) _{\frac{\delta }{2}%
}\diagdown \left( \left( \mathcal{W}_{1}^{\pm }\right) _{\frac{\delta }{2}%
}\cap \left( \mathcal{W}_{2}^{\pm }\right) _{\frac{\delta }{2}}\right) %
\right] \cup \left\{ \varphi _{1,\varepsilon }>\varphi _{2,\varepsilon
}+\rho _{n}\left( \varepsilon \right) \right\} }\left( \varphi
_{1,\varepsilon }-\lambda \mathcal{L}\left( \varphi _{1,\varepsilon }\right)
\right) \psi + \\
&&-\lambda \int_{\partial \left\{ \varphi _{1,\varepsilon }>\varphi
_{2,\varepsilon }+\rho _{n}\left( \varepsilon \right) \right\} }\left[
n_{v}\varphi _{1,\varepsilon }\partial _{v}\psi -\left( n_{v}\partial
_{v}\varphi _{1,\varepsilon }+vn_{x}\varphi _{1,\varepsilon }\right) \psi %
\right] ds
\end{eqnarray*}%
where $n=\left( n_{x},n_{v}\right) $ is the normal vector to $\partial
\left\{ \varphi _{1,\varepsilon }>\varphi _{2,\varepsilon }+\rho _{n}\left(
\varepsilon \right) \right\} $ pointing outwards from the domain $\left\{
\varphi _{1,\varepsilon }>\varphi _{2,\varepsilon }+\rho _{n}\left(
\varepsilon \right) \right\} .$ Using a similar argument to compute the last
integral in (\ref{Sum}) we obtain, after some cancellations of terms in the
boundary $\partial \left\{ \varphi _{1,\varepsilon }>\varphi _{2,\varepsilon
}+\rho _{n}\left( \varepsilon \right) \right\} :$%
\begin{eqnarray*}
&&\int_{\left( \mathcal{W}_{1}^{\pm }\right) _{\frac{\delta }{2}}\cup \left(
\mathcal{W}_{2}^{\pm }\right) _{\frac{\delta }{2}}}\left( \psi -\lambda
\mathcal{L}^{\ast }\left( \psi \right) \right) \varphi _{\varepsilon
}^{\left( n\right) } \\
&=&\int_{\left( \mathcal{W}_{1}^{\pm }\right) _{\frac{\delta }{2}}\cup
\left( \mathcal{W}_{2}^{\pm }\right) _{\frac{\delta }{2}}}\left( \varphi
_{\varepsilon }^{\left( n\right) }-\lambda \mathcal{L}^{\ast }\left( \varphi
_{\varepsilon }^{\left( n\right) }\right) \right) \psi  \\
&&+\lambda \int_{\partial \left\{ \varphi _{1,\varepsilon }>\varphi
_{2,\varepsilon }+\rho _{n}\left( \varepsilon \right) \right\} }n_{v}\left[
\partial _{v}\varphi _{1,\varepsilon }-\partial _{v}\left( \varphi
_{2\varepsilon }+\rho _{n}\left( \varepsilon \right) \right) \right] \psi ds
\end{eqnarray*}%
where $n$ is the same normal vector as before. Using the pointwise
inequality (\ref{Point}) as well as the fact that, $n_{v}\partial
_{v}\varphi _{1,\varepsilon }\leq n_{v}\partial _{v}\left( \varphi
_{2\varepsilon }+\rho _{n}\left( \varepsilon \right) \right) $ we obtain:%
\begin{equation}
\int_{\left( \mathcal{W}_{1}^{\pm }\right) _{\frac{\delta }{2}}\cup \left(
\mathcal{W}_{2}^{\pm }\right) _{\frac{\delta }{2}}}\left( \psi -\lambda
\mathcal{L}^{\ast }\left( \psi \right) \right) \varphi _{\varepsilon
}^{\left( n\right) }\leq \int_{\left( \mathcal{W}_{1}^{\pm }\right) _{\frac{%
\delta }{2}}\cup \left( \mathcal{W}_{2}^{\pm }\right) _{\frac{\delta }{2}%
}}g\psi +C\varepsilon \left\Vert \varphi _{k}\right\Vert _{L^{\infty
}}\int_{\left( \mathcal{W}_{1}^{\pm }\right) _{\frac{\delta }{2}}\cup \left(
\mathcal{W}_{2}^{\pm }\right) _{\frac{\delta }{2}}}\psi \   \label{IntPo}
\end{equation}%
for any $\psi \in \mathcal{F}\left( \left( \mathcal{W}_{1}^{\pm }\right)
_{\delta }\cup \left( \mathcal{W}_{2}^{\pm }\right) _{\delta }\right) $ with
$\psi \geq 0.$ Notice that these functions might be different from zero in
the line $\left\{ x=0\right\} $ if $\left( \left( \mathcal{W}_{1}^{\pm
}\right) _{\delta }\cup \left( \mathcal{W}_{2}^{\pm }\right) _{\delta
}\right) \cap \left\{ x=0\right\} \neq \varnothing .$ However, they do not
satisfy the condition (\ref{comp1t}) because by definition, these functions
are only defined in $\left\{ v>-x\right\} $ or $\left\{ v<x\right\} .$

We now use Lebesgue's Dominated Convergence Theorem to take the limit $%
\varepsilon \rightarrow 0$ in (\ref{IntPo}). Therefore, using also the fact
that $\delta $ can be chosen arbitrarily small, we obtain:%
\begin{equation}
\int_{\mathcal{W}_{1}^{\pm }\cup \mathcal{W}_{2}^{\pm }}\left( \psi -\lambda
\mathcal{L}^{\ast }\left( \psi \right) \right) \varphi \leq \int_{\mathcal{W}%
_{1}^{\pm }\cup \mathcal{W}_{2}^{\pm }}g\psi  \label{U1E1}
\end{equation}%
where $\psi \in \mathcal{F}\left( \mathcal{W}_{1}^{\pm }\cup \mathcal{W}%
_{2}^{\pm }\right) $ with $\psi \geq 0$ and $\operatorname{supp}\left( \psi
\right) \cap \left\{ x=0\right\} =\varnothing .$

It only remains to extend the validity of (\ref{U1E1}) to arbitrary
functions $\psi \in \mathcal{F}\left( \mathcal{W}_{1}\cup \mathcal{W}%
_{2}\right) $ with $\psi \geq 0$ (which in particular satisfy (\ref{comp1t}))%
$.$ To this end we use Proposition \ref{Traces} which proves the existence
of traces for subsolutions for the operator $\mathcal{L}$. Since $\varphi
_{1},\varphi _{2}$ are subsolutions (\ref{T6E6}), (\ref{T6E7}) and they are
bounded, we obtain that $\left( -\varphi _{1}\right) ,\ \left( -\varphi
_{2}\right) $ are supersolutions of $\mathcal{L}\left( \cdot \right) +\kappa
$ for a suitable constant $\kappa $ in the sense of Definition \ref%
{supersolDef} and any admissible domain $\Xi \subset \mathcal{W}_{k}$ with $%
k=1,2$ respectively. Therefore we can define $\varphi _{1},\varphi _{2}$ in
the sense of traces (see Remark \ref{TraceDef}) at $\left\{ x=0\right\} $ as
$x\rightarrow 0^{+}$ for $v>0$ and $v<0$ respectively. We will denote these
limits as $\varphi _{k,+},\varphi _{k,-}$ with $k=1,2.$ Notice that these
functions are in the spaces $L^{\infty }\left( \left\{ x=0,v>0\right\}
\right) $ and $L^{\infty }\left( \left\{ x=0,v<0\right\} \right) $
respectively. Moreover, we now claim that:%
\begin{equation}
\varphi _{k,+}\left( 0,rv\right) \geq \varphi _{k,-}\left( 0,-v\right) \ \
,\ a.e.\ v>0,\ \left( 0,v\right) \in \left\{ x=0\right\} \cap \mathcal{W}%
_{k}\ \ ,\ \ k=1,2  \label{U1E2}
\end{equation}

The inequality (\ref{U1E2}) is equivalent to
\begin{equation}
\int_{\left\{ v>0\right\} }\varphi _{k,+}\left( 0,r\cdot \right) \zeta
\left( \cdot \right) \geq \int_{\left\{ v>0\right\} }\varphi _{k,-}\left(
0,-\cdot \right) \zeta \left( \cdot \right)  \label{U1E3}
\end{equation}%
for any $\zeta =\zeta \left( v\right) \geq 0,\ \zeta \in C^{\infty }\left(
\left\{ v>0\right\} \right) .$ To prove (\ref{U1E3}) we argue as follows.
Proposition \ref{Traces} implies that $\lim_{x\rightarrow
0^{+}}\int_{\left\{ v>0\right\} }\varphi _{k,\pm }\left( x,r\cdot \right)
\zeta \left( \cdot \right) =\int_{\left\{ v>0\right\} }\varphi _{k,\pm
}\left( 0,r\cdot \right) \zeta \left( \cdot \right) .$ We construct a
function $\psi \in \mathcal{F}\left( \mathcal{W}_{k}\right) $ as follows. We
define a function $\eta =\eta \left( \frac{x}{\varepsilon }\right) $ with $%
\varepsilon >0,\ \eta \in C^{\infty }\left( \left\{ x\geq 0\right\} \right)
, $ $\eta ^{\prime }\leq 0,$ $\eta \left( x\right) =0$ if $x\geq 1,$ $\eta
\left( 0\right) =1.$ We then define $\psi \left( x,v\right) =\frac{\zeta
\left( \frac{v}{r}\right) \eta \left( x\right) }{rv}$ for $v>0$ and $\psi
\left( x,v\right) =\frac{\zeta \left( -v\right) \eta \left( x\right) }{%
\left\vert v\right\vert }$ for $v<0.$ Notice that $r^{2}\psi \left(
0^{+},rv\right) =\psi \left( 0^{+},-v\right) =\frac{\zeta \left( v\right)
\eta \left( x\right) }{v}$ for $v>0.$ Therefore $\psi \in \mathcal{F}\left(
\mathcal{W}_{k}\right) $ (cf. (\ref{defF})). Integrating by parts in (\ref%
{sub_sol})

\begin{equation*}
\int_{\mathcal{W}_{k}}\left( \psi -\lambda D_{v}^{2}\psi -\lambda vD_{x}\psi
\right) \varphi _{k}+\lambda \int_{\left\{ x=0,v>0\right\} }v\varphi _{k,+}%
\frac{\zeta \left( \frac{\cdot }{r}\right) }{rv}+\lambda \int_{\left\{
x=0,v<0\right\} }v\varphi _{k,-}\frac{\zeta \left( v\right) }{\left\vert
v\right\vert }\leq \int g\psi
\end{equation*}

Taking the limit $\varepsilon \rightarrow 0$ we obtain $\frac{1}{r}%
\int_{\left\{ x=0,v>0\right\} }\varphi _{k,+}\zeta \left( \frac{\cdot }{r}%
\right) -\int_{\left\{ x=0,v<0\right\} }\varphi _{k,-}\zeta \left( v\right)
\geq 0,$ whence, after using the change of variables $\frac{\cdot }{r}%
\rightarrow \cdot $ in the first integral we obtain (\ref{U1E3}) whence (\ref%
{U1E2}) follows.

We now claim that it is possible to define the limit values $\varphi
_{+}=\varphi \left( 0^{+},v\right) ,\ \varphi _{-}=\varphi \left(
0^{+},-v\right) $ for $v>0$ in the sense of traces, with $\varphi =\max
\left\{ \varphi _{1},\varphi _{2}\right\} .$ Indeed, this is a consequence
of the fact that $\varphi $ satisfies the subsolution inequality (\ref%
{sub_sol}) for any test function $\psi \geq 0$ with $\operatorname{supp}\left(
\psi \right) \cap \left\{ x=0\right\} \neq \varnothing .$ Therefore, the
argument yielding Proposition \ref{Traces} implies the existence of $\varphi
_{+},$ $\varphi _{-}.$

We now use Lemmas \ref{TraceInequalities}, \ref{TraceIde} as well as the
inequalities (\ref{U1E2}) to prove that:%
\begin{eqnarray*}
\varphi ^{-}\left( 0,-v\right) &=&\max \left\{ \varphi _{1}^{-}\left(
0,-v\right) ,\varphi _{2}^{-}\left( 0,-v\right) \right\} \\
&\leq &\max \left\{ \varphi _{1}^{+}\left( 0,rv\right) ,\varphi
_{2}^{+}\left( 0,rv\right) \right\} \leq \left( \max \left\{ \varphi
_{1}\left( 0,rv\right) ,\varphi _{2}\left( 0,rv\right) \right\} \right)
^{+}=\varphi ^{+}\left( 0,rv\right)
\end{eqnarray*}%
for $a.e.$ $v>0.$ We can then argue as in the derivation of (\ref{U1E2}) to
show that $\varphi $ satisfies also the subsolution inequality at the line $%
\left\{ x=0\right\} .$ This shows that $\varphi $ is a subsolution and
concludes the proof of the result.
\end{proof}

The main idea in Perron's method is that it is possible to construct a
solution of the problem if we can obtain one subsolution $\varphi ^{\operatorname{%
sub}}$ and a supersolution $\varphi ^{\sup }$ satisfying $\varphi ^{\operatorname{%
sub}}\leq \varphi ^{\sup }$. Such sub and supersolutions can be easily
obtained in the case of the problem (\ref{T6E7}), (\ref{T6E8}) and
Definition \ref{SubSuper1}:

\begin{lemma}
\label{ExSubSup}For any $g\in C\left( X\right) $ there exist at least one
subsolution $\varphi ^{\mathrm{sub}}$ and one supersolution $\varphi ^{\sup
}$ in the sense of Definition \ref{SubSuper1} such that:%
\begin{equation}
\varphi ^{\mathrm{sub}}\leq \varphi ^{\sup }  \label{T7E1}
\end{equation}
\end{lemma}

\begin{proof}
Let $\left\vert g\right\vert \leq \left\Vert g\right\Vert _{L^{\infty
}\left( X\right) }$ and let $\varphi ^{\sup }=\min \left\{ g\left(
0,0\right) +KF_{\beta }+\left\Vert g\right\Vert _{L^{\infty }\left( X\right)
}S,\left\Vert g\right\Vert _{L^{\infty }\left( X\right) }\right\} ,\
\allowbreak \varphi ^{\mathrm{sub}}=\max \left\{ g\left( 0,0\right)
-KF_{\beta }-\left\Vert g\right\Vert _{L^{\infty }\left( X\right)
}S,-\left\Vert g\right\Vert _{L^{\infty }\left( X\right) }\right\} ,$ where $%
S\left( x,v\right) $ is a super-solution with $\mathcal{L}S=1$, constructed
in Lemma \ref{super_sol} and $K>0$ must be determined. We have that $g\left(
0,0\right) +KF_{\beta }\left( x,v\right) +\left\Vert g\right\Vert
_{L^{\infty }\left( X\right) }S\left( x,v\right) >0$ in a neighbourhood of
the singular point which will be denoted as $\mathcal{W}_{1},$ if we take $%
K>0$ sufficiently large. Moreover, this function is a supersolution of (\ref%
{T6E6}), (\ref{T6E7}) in $\mathcal{W}_{1}$ by construction. On the other
hand $\left\Vert g\right\Vert _{L^{\infty }\left( X\right) }$ is a
supersolution of (\ref{T6E6}), (\ref{T6E7}) in an open set $\mathcal{W}_{2}$
such that $\mathcal{W}_{1}\cup \mathcal{W}_{2}=X$ and $\mathcal{W}_{1}\cap
\mathcal{W}_{2}\neq \varnothing .$ Therefore $\varphi ^{\sup }$ is a
supersolution due to Lemma \ref{maxima}. To prove that $\varphi ^{\operatorname{%
sub}}$ is a subsolution we use a similar argument.
\end{proof}

We define the following subset of $L_{b}^{\infty }\left( X\right) :$

\begin{definition}
\label{SubSet1}Suppose that $\varphi ^{\mathrm{sub}},\ \varphi ^{\sup }$
are respectively one subsolution and one supersolution in the sense of
Definition \ref{SubSuper1} satisfying (\ref{T7E1}). We then define $\mathcal{%
\tilde{G}}_{sub}\subset L_{b}^{\infty }\left( X\right) $ as:%
\begin{equation}
\mathcal{\tilde{G}}_{sub}\equiv \left\{ \varphi \in L_{b}^{\infty }\left(
X\right) :%
\begin{array}{c}
\varphi \text{ sub-solution of (\ref{T6E6}),\ (\ref{T6E7})} \\
\text{ in the sense of Definition \ref{SubSuper1}}|\ \varphi ^{\mathrm{sub}%
}\leq \varphi \leq \varphi ^{\sup }%
\end{array}%
\right\} .  \label{T7E4}
\end{equation}
\end{definition}

Notice that we have:

\begin{lemma}
\label{ContinuityWeak}The set $\mathcal{\tilde{G}}_{sub}$ is closed in the
weak topology, defined by means of the functionals $\ell _{\psi }\left(
\varphi \right) =\int_{X}\varphi \psi $ with $\psi \in \mathcal{F}\left(
X\right) $ (cf. (\ref{defF})).

\begin{proof}
It is just a consequence of the definition \ref{SubSuper1}, as well as the
fact that the inequalities $\varphi ^{\mathrm{sub}}\leq \varphi \leq
\varphi ^{\sup }$ are preserved by weak limits.
\end{proof}
\end{lemma}

\begin{remark}
\label{distW}We recall that the bounded set $\mathcal{\tilde{G}}_{sub}$
endowed with the weak$-\ast $ topology is metrizable(cf. \cite{Brezis}). We
will denote the corresponding metric as $\operatorname{dist}_{\ \ast }.$
\end{remark}

The following Lemma will be useful in order to show that the supremum of the
set $\mathcal{\tilde{G}}_{sub}$ can be obtained by means of limits of
subsequences.

\begin{lemma}
\label{countable}There exists a countable and dense subset $\mathcal{G}%
_{sub} $ of $\mathcal{\tilde{G}}_{sub}$ in the weak topology in Lemma \ref%
{ContinuityWeak}.
\end{lemma}

\begin{proof}
We will prove the result by showing that there exists a countable subset of $%
\mathcal{\tilde{G}}_{sub}$ which is dense in the $L_{loc}^{2}\left( X\right)
$ topology. Since the test functions in $\mathcal{F}\left( X\right) $ are
compactly supported, this is enough to prove the required density property
in the weak topology. We find a countable subset $\mathcal{G}_{sub}$\ of $%
\mathcal{\tilde{G}}_{sub}$ by using the fact that $L^{2}\left( K\right) $ is
separable for any compact subsets $K$: Choose a sequence of compact sets $%
\left\{ K_{n}\right\} _{n=1}^{\infty }$ such that $X=\cup _{n=1}^{\infty
}K_{n},$ $K_{n}\subset K_{n+1}.$ For each $K_{n},$ $L^{2}\left( K_{n}\right)
=\cup _{k_{N}=1}^{\infty }B_{1/N}\left( f_{k_{N}}\right) ,$ where $%
f_{k_{N}}\in L^{2}\left( K_{n}\right) ,k_{N}=1,2,...$ and the distance is
measured in $L^{2}$ norm. Then we choose one element from each $%
B_{1/N}\left( f_{k_{N}}\right) $ if $\mathcal{\tilde{G}}_{sub}\cap
B_{1/N}\left( f_{k_{N}}\right) $ is nonempty and choose none if $\mathcal{%
\tilde{G}}_{sub}\cap B_{1/N}\left( f_{k_{N}}\right) $ is empty. This
selection can be made for each $N=1,2,3,..$ and then for $n=1,2,3...$We call
this subset $\mathcal{G}_{sub}.$ This set is countable and dense in $%
\mathcal{\tilde{G}}_{sub}$ in $L^{2}$ norm. This completes the proof.
\end{proof}

Next we show that the set $\mathcal{\tilde{G}}_{sub}$ is closed under the
maximum function.

\begin{lemma}
\label{max_sub}Let $\varphi _{1},\varphi _{2},...,\varphi _{L}\in \mathcal{%
\tilde{G}}_{sub}$ with $L<\infty $ and $\mathcal{\tilde{G}}_{sub}$ as in
Definition \ref{SubSet2}. Define:%
\begin{equation*}
\bar{\varphi}:=\max \{\varphi _{1},\varphi _{2},...,\varphi _{L}\}
\end{equation*}

Then $\bar{\varphi}\in \mathcal{\tilde{G}}_{sub}$.
\end{lemma}

\begin{proof}
It is just a consequence of Lemma \ref{maxima}.
\end{proof}

We want to give a definition of the largest subsolution in the set $\mathcal{%
\tilde{G}}_{sub}.$ Given that the functions $\varphi $ in $\mathcal{\tilde{G}%
}_{sub}$ are not defined pointwise we cannot just take the supremum $\sup_{%
\mathcal{\tilde{G}}_{sub}}\varphi .$ However, the function that will play
the role of such supremum is the following. We can assume that the countable
family $\mathcal{G}_{sub}$ constructed in Lemma \ref{countable} is $\left\{
\varphi _{j}\right\} _{j\in \mathbb{N}}.$ We then construct the following
finite families of subsolutions:%
\begin{equation*}
\mathcal{G}_{sub}\left( M\right) =\left\{ \varphi _{j}\right\} _{j=1}^{M}
\end{equation*}%
and we then define the following function which will play the role of $\sup_{%
\mathcal{\tilde{G}}_{sub}}\varphi :$%
\begin{equation}
\varphi _{\ast }=\lim_{M\rightarrow \infty }\max_{\mathcal{G}_{sub}\left(
M\right) }\varphi  \label{T7E2}
\end{equation}

Notice that $\max_{\mathcal{G}_{sub}\left( M\right) }\varphi $ is the
maximum of a finite number of functions and therefore is a well defined
quantity. On the other hand, the sequence $\left\{ \max_{\mathcal{G}%
_{sub}\left( M\right) }\varphi \right\} $ is increasing in $M$ and uniformly
bounded by $\varphi ^{\sup },$ therefore the limit on the right-hand side of
(\ref{T7E2}) exists in $L_{loc}^{1}$ and then also in the weak topology
defined in Lemma \ref{ContinuityWeak}.

Our next goal is to show that $\varphi _{\ast }$ is the desired solution of
the problem (\ref{T6E6}), (\ref{T6E7}). To this end we need an auxiliary
result, namely the solvability of the Dirichlet problem in the admissible
domains defined in Definition \ref{admissible} with boundary values in the
admissible boundaries.


The following result will be used to prove that if a subsolution is not a
solution of (\ref{T6E6}), (\ref{T6E7}), it is possible to construct a larger
subsolutions.

\begin{lemma}
\label{IncSub}Suppose that $\bar{\varphi}$ is a subsolution of (\ref{T6E6}),
(\ref{T6E7}) in the sense of Definition \ref{SubSuper1}. Let $\Xi $ be any
admissible domain in the sense of Definition \ref{admissible} contained in $%
X.$ Let us denote as $h$ the boundary values of $\bar{\varphi}$ in $\partial
_{a}\Xi $ obtained from the interior of $\Xi $ (cf. Proposition \ref{Traces}%
). Let$\ \varphi $ be the unique solution of (\ref{solv_D1}), (\ref{solv_D2}%
) obtained in Proposition \ref{DirSolv}. We construct a function $\Phi $ by
means of:\
\begin{equation}
\Phi =\left\{
\begin{array}{c}
\bar{\varphi},\ \left( x,v\right) \notin X\diagdown \Xi  \\
\varphi ,\ \left( x,v\right) \in \Xi
\end{array}%
\right.   \label{T7E5}
\end{equation}%
Then $\Phi $ is a sub-solution of (\ref{T6E6}),(\ref{T6E7}) in the sense of
Definition \ref{SubSuper1}.
\end{lemma}

\begin{proof}
It follows using some arguments analogous to those in the Proof of Lemma \ref%
{maxima}, and therefore we will skip the details of the proof. Notice that
the function $\Phi $ is larger than $\bar{\varphi}$ in $\Xi .$ Therefore the
contribution of the terms arising at the boundary of $\Xi $ in the
definition of subsolution, including at the possible points of discontinuity
of $\Phi ,$ have the sign required for $\Phi $ to be a subsolution.
\end{proof}

\begin{proof}[End of the Proof of Proposition \protect\ref{Case_a_sub}.]
We define $\varphi _{\ast }$ as in (\ref{T7E2}). The limit in (\ref{T7E2})
is pointwise $a.e$ in $X$. Due to Lebesgue's Theorem it then follows that $%
\varphi _{M}$ converges to $\varphi _{\ast }$ in the weak topology
introduced in Lemma \ref{ContinuityWeak}. Applying this Lemma, as well as
the fact that the definition of $\mathcal{\tilde{G}}_{sub}$ (cf. (\ref{T7E4}%
)) and $\varphi ^{\mathrm{sub}},\ \varphi ^{\sup }$ (cf. Lemma \ref%
{ExSubSup}) it follows that $\varphi _{\ast }$ is a subsolution of (\ref%
{T6E6}), (\ref{T6E7}) with $\varphi _{\ast }\in \mathcal{\tilde{G}}_{sub}.$
This implies that $\nu =-\left( \varphi _{\ast }-\lambda \mathcal{L}\left(
\varphi _{\ast }\right) -g\right) $ defines a nonnegative Radon measure. If $%
\nu =0,$ it follows that $\varphi _{\ast }$ is a solution of (\ref{T6E6}), (%
\ref{T6E7}) and the Proposition would easily follow. Suppose then that $\nu
\neq 0.$ Then, there exists an admissible domain $\Xi \subset X$ such that $%
\nu \left( S\right) >0$ with $S=\left\{ \left( x,v\right) :\operatorname{dist}%
\left( \left( x,v\right) ,\partial \Xi \right) \geq \frac{d\left( \Xi
\right) }{4}\right\} $\ with $d\left( \Xi \right) $ as in (\ref{T7E3}).
Suppose that $\varphi $ is the unique solution of (\ref{solv_D1}), (\ref%
{solv_D2}) obtained in Proposition \ref{DirSolv} where $h$ is the trace of $%
\varphi _{\ast }$ in $\partial _{a}\Xi $ obtained from the interior of $\Xi .
$\ We then define $\Phi $ as in (\ref{T7E5}) with $\bar{\varphi}=\varphi
_{\ast }.$ Lemma \ref{IncSub} implies that $\Phi $ is a subsolution of (\ref%
{T6E6}), (\ref{T6E7}) in the sense of Definition \ref{SubSuper1}. On the
other hand, since $\varphi _{\ast }\in \mathcal{\tilde{G}}_{sub}$ we have $%
h\leq \varphi ^{\sup }$ on $\partial _{a}\Xi ,$ where $\varphi ^{\sup }$ is
understood in $\partial _{a}\Xi $ in the sense of trace from the interior of
$\Xi .$ Then $\varphi \leq \varphi ^{\sup }$ due to Proposition \ref%
{weak_max}, whence $\Phi \leq \varphi ^{\sup }.$ Therefore $\Phi \in
\mathcal{\tilde{G}}_{sub}.$ Due to Lemma \ref{comp_nu} and using that $\nu
\left( S\right) >0$ we have that $\Phi >\varphi _{\ast }+\delta $ for some $%
\delta >0$ in an open subset of $\Xi .$ Since $\varphi _{\ast }\geq \varphi
_{j}$ for any $\varphi _{j}\in \mathcal{G}_{sub}$ with $\mathcal{G}_{sub}$
as in Lemma \ref{countable} it then follows that $\operatorname{dist}_{\ast
}\left( \Phi ,\mathcal{G}_{sub}\right) >0,$ where $\operatorname{dist}_{\ast }$
is the metric associated to the weak$-\ast $ topology used in Lemma \ref%
{ContinuityWeak} (cf. Remark \ref{distW}). However, Lemma \ref{countable}
implies that $\mathcal{G}_{sub}$ is dense in $\mathcal{\tilde{G}}_{sub}$ in
the weak$-\ast $ topology, and therefore this gives a contradiction. Then $%
\nu =0$ and thus $\varphi _{\ast }-\lambda \mathcal{L}\left( \varphi _{\ast
}\right) -g=0$ in the sense of distributions. Since $\varphi _{\ast }$ is
bounded, we can use Theorem \ref{Hypoell} to prove that $\varphi _{\ast }\in
W^{1,p}\left( U\right) $ in any bounded set $U\subset X$ whose closure does
not intersect $\left( 0,0\right) $ and any $1<p<\infty .$ Therefore $\varphi
_{\ast }$ is continuous away from the origin. Since $\varphi _{\ast }\in
\mathcal{\tilde{G}}_{sub}$ it also follows that $\varphi _{\ast }$ is
continuous at $\left( 0,0\right) .$ It remains to check that $\varphi _{\ast
}$ is also continuous at the point $\infty .$ Notice that our choice of
sub/supersolutions imply that $\left\vert \varphi _{\ast }\right\vert \leq
\left\Vert g\right\Vert _{L^{\infty }\left( X\right) }.$ We can now prove
that $\varphi _{\ast }\left( x,v\right) \rightarrow g\left( \infty \right) $
as $\left( x,v\right) \rightarrow \infty $ as follows. In any large domain
contained in the regions $\left\{ v>0\right\} $ or $\left\{ v<0\right\} $ we
then obtain, using parabolic theory, that $\varphi _{\ast }$ converges to $%
g\left( \infty \right) .$ Indeed, in the case of $\left\{ v<0\right\} $ we
just integrate the parabolic equation in the direction of increasing $x$ and
it readily follows that $\varphi _{\ast }\left( x,v\right) $ approaches to $%
g\left( \infty \right) .$ If $\left\{ v>0\right\} $, we argue similarly, but
with decreasing $x.$ It then follows also, due to the boundedness of $%
\left\vert \varphi _{\ast }\right\vert $ that $\varphi _{\ast }$ converges
to $g\left( \infty \right) $ if $\left\vert \left( x,v\right) \right\vert
\rightarrow \infty $ and$\ v\rightarrow \infty .$ Therefore $\varphi _{\ast
}\left( x,v\right) \rightarrow g\left( \infty \right) $ if $\left\vert
\left( x,v\right) \right\vert \rightarrow \infty $ and$\ \left\vert
v\right\vert \rightarrow \infty .$ In order to prove the convergence for $%
x\rightarrow \infty $ and $\left\vert v\right\vert $ bounded we argue as
follows. After substracting from $\varphi _{\ast }$ quantities like $g\left(
\infty \right) \pm \varepsilon ,$ with $\varepsilon >0$ it follows that we
only need to show that for large admissible domains $\Xi $ with bounded
values of $\psi $ at the admissible boundary $\partial _{a}\Xi $ and $\psi $
satisfying $\psi -\lambda \mathcal{L}\left( \psi \right) =0$ we have $%
\left\vert \psi \left( x,v\right) \right\vert $ small if the size of $\Xi $
tends to infinity and the distance from $\left( x,v\right) $ to $\partial
_{a}\Xi $ tends to infinity too. This follows from instance from (\ref{X8})
using the fact that $\tau \left( \Xi \right) $ tends to infinity for the
points under consideration with a large probability, and estimating the
remainder by the small probability of having $\tau \left( \Xi \right) $ of
order one.

Therefore $\varphi _{\ast }\left( \infty \right) =g\left( \infty \right) $
and the result follows.
\end{proof}

\subsection{Operator $\Omega _{nt,sub}.$ Solvability of (\protect\ref{T6E6}), (%
\protect\ref{T6E8}).}

We now argue as in the previous case in order to show that the problem (\ref%
{T6E6}), (\ref{T6E8}) can be solved for any $g\in C\left( X\right) .$ The
main difference arises in the analysis of the behaviour of the solution $%
\varphi $ in a neighbourhood of the singular point $\left( 0,0\right) $ and
in particular showing that (\ref{T6E8}) holds. The key point will be to show
that for the maximum of a suitable family of subsolutions $\mathcal{\tilde{G}%
}_{sub}$, we have that $\mathcal{A}\left( \varphi \right) $ is well defined
and $\mathcal{A}\left( \varphi \right) \geq 0.$ If $\mathcal{A}\left(
\varphi \right) >0,$ it turns out that it is possible to construct a larger
subsolution in the family $\mathcal{\tilde{G}}_{sub}.$ This contradiction
will imply $\mathcal{A}\left( \varphi \right) =0.$

Several arguments needed to prove the solvability of the problem (\ref{T6E6}%
), (\ref{T6E8}) are similar to the ones in the previous Section. We will
emphasize just the points where differences arise. The main result that we
prove in this Section is the following.

\begin{proposition}
\label{Case_r}For any $g\in C\left( X\right) $ there exists a unique $%
\varphi \in C\left( X\right) $ which solves (\ref{T6E6}), (\ref{T6E8}).
\end{proposition}

We will prove Proposition \ref{Case_r} using a method analogous to the proof
of Proposition \ref{Case_a_sub}. However we need to define a different class
of subsolutions, which allows us to identify the boundary condition (\ref%
{T6E8}). The main novelty is that the test functions might have a power law
singularity near the singular point $\left( x,v\right) =\left( 0,0\right).$ We first recall $G_\gamma$ 
 in \cite{HVJ2} (cf. Proposition 3.1 in \cite{HVJ2}): 
\[
G_{\gamma }\left( x,v\right) =x^{\gamma }\Lambda _{\gamma }\left( \zeta
\right) \ ,\ \ \zeta =\frac{v}{\left( 9x\right) ^{\frac{1}{3}}}\ \ \text{%
with }\gamma \in \left\{ -\frac{2}{3},\alpha \left( r\right) \right\}
\]
with%
\[
\Lambda _{\gamma }\left( \zeta \right) =U\left( -\gamma ,\frac{2}{3};-\zeta
^{3}\right) >0\ \ \text{for }\zeta \in \mathbb{R\ }  \label{LambdaU}
\]
where we denote as $U(a,b;z)$ the classical Tricomi confluent hypergeometric
functions (cf. \cite{Ab}). 

\begin{definition}
\label{SubSuper2}Suppose that $g\in C\left( X\right) .$ Let $\zeta =\zeta
\left( x,v\right) $ be a nonnegative $C^{\infty }$ test function supported
in $0\leq x+\left\vert v\right\vert ^{3}\leq 2$, satisfying $\zeta =1$ for $%
0\leq x+\left\vert v\right\vert ^{3}\leq 1$. We will say that a function $%
\varphi \in L_{b}^{\infty }\left( X\right) $ is a subsolution of (\ref{T6E6}%
), (\ref{T6E8}) if for all $\psi $ with $\psi \geq 0,\ \psi \left(
0,-v\right) =r^{2}\psi \left( 0,rv\right) ,\ v>0$ and having the form $\psi
=\theta \zeta G_{\gamma }+\bar{\psi}$ where $\bar{\psi}\in C^{\infty },$ $%
\theta \geq 0$ 
with $\psi $ supported in a set
contained in the ball $\left\vert \left( x,v\right) \right\vert \leq R$ for
some $R>0$ we have:%
\begin{equation}
\int \left( \psi -\lambda \mathcal{L}^{\ast }\left( \psi \right) \right)
\varphi \leq \int g\psi   \label{sub_sol2}
\end{equation}

Given $g\in C\left( X\right) ,$ we will say that $\varphi \in L_{b}^{\infty
}\left( X\right) $ is a supersolution of (\ref{T6E6}), (\ref{T6E8}) if for
all $\psi $ with the same properties as above\ we have:%
\begin{equation}
\int \left( \psi -\lambda \mathcal{L}^{\ast }\left( \psi \right) \right)
\varphi \geq \int g\psi  \label{super_sol3}
\end{equation}
\end{definition}

Notice that there are two differences between the Definitions \ref{SubSuper1}
and \ref{SubSuper2}. In this second definition we do not impose $\varphi
\left( 0,0\right) =g\left( 0,0\right) .$ On the other hand, we have a larger
class of admissible test functions in Definition \ref{SubSuper2}. Notice
that the integral on the left hand side of (\ref{sub_sol2}) is well defined
in spite of the singularity of $G_{\gamma }$ near the singular point because
$\mathcal{L}^{\ast }\left( G_{\gamma }\right) =0.$

Our next goal is to prove the analogous of Lemma \ref{maxima} for the
operator $\Omega _{nt,sub}.$ This will require to prove that it is possible
to define in a suitable sense the quantity $\mathcal{A}\left( \varphi
\right) $ (cf. (\ref{phi_decom})) for sub or supersolutions in the sense of
Definition \ref{SubSuper2}. This will be made by means of the limit as $%
\delta \rightarrow 0^{+}$ of the quantity $\Psi \left( \delta \right) $
defined in the next Lemma.

\begin{lemma}
\label{TraceAphi}Let $g\in C\left( X\right) $ and $\varphi $ be a
subsolution of (\ref{T6E6}), (\ref{T6E8})\ in the sense of Definition \ref%
{SubSuper2}. Let:%
\begin{equation}
\Psi \left( \delta \right) =\int_{\partial \mathcal{R}_{\delta }}\left(
n_{v}G_{\alpha }\partial _{v}\varphi -n_{v}\partial _{v}G_{\alpha }\varphi
+v\varphi G_{\alpha }n_{x}\right) ds\   \label{W6E3}
\end{equation}%
which is well defined in the sense of Traces (cf. Remark \ref{TraceDef}).
Here $n=\left( n_{x},n_{v}\right) $ is the unit normal vector pointing
toward $\mathcal{R}_{\delta }$, where $\mathcal{R}_{\delta }$
is defined in (\ref{domainR}).
Then the limit $\lim_{\delta \rightarrow
0^{+}}\Psi \left( \delta \right) $ exists and we have:%
\begin{equation*}
\lim_{\delta \rightarrow 0^{+}}\Psi \left( \delta \right) \leq 0
\end{equation*}%
Let $\mu =-\varphi +\lambda \mathcal{L}\left( \varphi \right) +g.$ Then $\mu
\geq 0$ defines a Radon measure in $X\diagdown \left\{ \left( 0,0\right)
\right\} $ satisfying:%
\begin{equation}
\int_{\mathcal{R}_{1}\diagdown \left\{ \left( 0,0\right) \right\} }G_{\alpha
}\mu dxdv<\infty   \label{P2E3}
\end{equation}
\end{lemma}

\begin{proof}
Using test functions $\psi \geq 0,$ with support disjoint from $\left(
x,v\right) =\left( 0,0\right) $ we obtain that (\ref{sub_sol2}) implies the
existence of a measure $\mu \geq 0$ such that:%
\begin{equation}
-\varphi +\lambda \mathcal{L}\left( \varphi \right) +g=\mu \geq 0  \label{M1}
\end{equation}%
Notice that part of the support of $\mu $ can be contained in the line $%
\left\{ x=0\right\} .$ Notice that it might be possible to have $\mu \left(
\mathbb{R}^{+}\times \mathbb{R}\right) =\infty .$

Suppose that $\varphi $ is a subsolution of (\ref{T6E6}), (\ref{T6E8}) in
the sense of Definition \ref{SubSuper2}. Multiplying (\ref{M1}) by $%
G_{\alpha }$ and integrating in domains $\mathcal{R}_{\delta _{2}}\diagdown
\mathcal{R}_{\delta _{1}}$ with $\delta _{1}<\delta _{2}$ as well as the
fact that the traces of the function $\varphi $ and some of the derivatives $%
\partial _{v}\varphi $ exist, we obtain:%
\begin{eqnarray}
&&-\Psi \left( \delta _{2}\right) +\Psi \left( \delta _{1}\right)  \notag \\
&=&\frac{1}{\lambda }\int_{\mathcal{R}_{\delta _{2}}\diagdown \mathcal{R}%
_{\delta _{1}}}G_{\alpha }\mu dxdv+\frac{1}{\lambda }\int_{\mathcal{R}%
_{\delta _{2}}\diagdown \mathcal{R}_{\delta _{1}}}G_{\alpha }\varphi dxdv\ -%
\frac{1}{\lambda }\int_{\mathcal{R}_{\delta _{2}}\diagdown \mathcal{R}%
_{\delta _{1}}}G_{\alpha }gdxdv,  \label{Ma2}
\end{eqnarray}%
where $\Psi \left( \delta \right) $ is as in (\ref{W6E3}). Notice that the
existence of the traces of $\varphi ,\ \partial _{v}\varphi $ in the
required boundaries $\partial \mathcal{R}_{\delta _{1}},\ \partial \mathcal{R%
}_{\delta _{2}}$ due to Proposition \ref{Traces}. Notice that the proof of (%
\ref{Ma2}) must be made approximating the characteristic function of the
domain $\mathcal{R}_{\delta _{2}}\diagdown \mathcal{R}_{\delta _{1}}$ by
smooth functions and taking the limit. Combining (\ref{Ma2}) with the fact
that $\mu \geq 0,$ and $\varphi $ and $g$ are bounded implies:%
\begin{equation}
-\Psi \left( \delta _{2}\right) +\Psi \left( \delta _{1}\right) \geq -\frac{C%
}{\lambda }\int_{\mathcal{R}_{\delta _{2}}\diagdown \mathcal{R}_{\delta
_{1}}}G_{\alpha }dxdv  \label{Ma3}
\end{equation}%
We can compute the integrals $\int_{\mathcal{R}_{\delta }}G_{\alpha }dxdv.$
Indeed, using Proposition 3.1 in \cite{HVJ2} and (\ref{domainR}) we obtain:%
\begin{equation*}
\int_{\mathcal{R}_{\delta }}G_{\alpha }dxdv=\int_{\mathcal{R}_{\delta
}}x^{\alpha }\Phi \left( -\frac{v^{3}}{9x}\right) dxdv=A\delta ^{3\left(
1+\alpha \right) +1}
\end{equation*}%
where $A=\int_{-1}^{r}dv\int_{0}^{1}x^{\alpha }\Phi \left( -\frac{v^{3}}{9x}%
\right) dx.$ It is readily seen that $0<A<\infty $ and also that $3\left(
1+\alpha \right) >0.$ The inequality (\ref{Ma3}) then implies that the
function $\Psi \left( \delta \right) -\frac{CA\delta ^{3\left( 1+\alpha
\right) +1}}{\lambda }$ is decreasing. Then the limit $\lim_{\delta
\rightarrow 0^{+}}\left( \Psi \left( \delta \right) -\frac{A\delta ^{3\left(
1+\alpha \right) +1}}{\lambda }\right) $ exists, whence the limit $%
\lim_{\delta \rightarrow 0^{+}}\Psi \left( \delta \right) $ exists too,
where the value of this limit might be $\infty $. Our goal is to show that:%
\begin{equation}
\lim_{\delta \rightarrow 0^{+}}\Psi \left( \delta \right) \leq 0
\label{P1E7}
\end{equation}%
We will consider separately the cases $\lim_{\delta \rightarrow 0^{+}}\Psi
\left( \delta \right) <\infty $ and $\lim_{\delta \rightarrow 0^{+}}\Psi
\left( \delta \right) =\infty .$ Suppose first that $\lim_{\delta
\rightarrow 0^{+}}\Psi \left( \delta \right) <\infty .$ Using (\ref{Ma2})
with $\delta _{2}=1$,\ and taking the limit $\delta _{1}\rightarrow 0$ we
obtain:%
\begin{eqnarray*}
&&-\Psi \left( 1\right) +\Psi \left( 0^{+}\right) \\
&=&\frac{1}{\lambda }\int_{\mathcal{R}_{1}\diagdown \left\{ \left(
0,0\right) \right\} }G_{\alpha }\mu dxdv+\frac{1}{\lambda }\int_{\mathcal{R}%
_{1}\diagdown \left\{ \left( 0,0\right) \right\} }G_{\alpha }\varphi dxdv\ -%
\frac{1}{\lambda }\int_{\mathcal{R}_{1}\diagdown \left\{ \left( 0,0\right)
\right\} }G_{\alpha }gdxdv.\
\end{eqnarray*}

The boundedness of $\varphi $ implies that $\int_{\mathcal{R}_{1}\diagdown
\left\{ \left( 0,0\right) \right\} }G_{\alpha }\varphi dxdv=\int_{\mathcal{R}%
_{1}}G_{\alpha }\varphi dxdv$ is also bounded. Then, using the fact that $%
\mu \geq 0$ and $G_{\alpha }>0$ we obtain:%
\begin{equation}
0\leq \int_{\mathcal{R}_{1}\diagdown \left\{ \left( 0,0\right) \right\}
}G_{\alpha }\mu dxdv<\infty  \label{P1E8}
\end{equation}

We now argue as follows. We choose a test function $\psi =G_{\alpha }\xi ,$
with $\xi =1$ in $\overline{\mathcal{R}_{\delta }}$ and $\xi =0$ in $%
X\diagdown \mathcal{R}_{2\delta },$ with $\delta >0,$ $\xi \geq 0.$ Using
Definition \ref{SubSuper2} we obtain (\ref{M1}). Notice that $\mathcal{L}%
^{\ast }\left( \psi \right) $ is integrable near the origin and $\varphi $
is bounded. Then:%
\begin{equation}
\int \left( \left( \psi -\lambda \mathcal{L}^{\ast }\left( \psi \right)
\right) \varphi -g\psi \right) =\lim_{\delta \rightarrow 0}\int_{X\diagdown
\mathcal{R}_{\delta }}\left( \left( \psi -\lambda \mathcal{L}^{\ast }\left(
\psi \right) \right) \varphi -g\psi \right)  \label{P1E9}
\end{equation}%
Integrating by parts and using (\ref{M1}), (\ref{W6E3}) we obtain:%
\begin{eqnarray*}
&&\int_{X\diagdown \mathcal{R}_{\delta }}\left( \left( \psi -\lambda
\mathcal{L}^{\ast }\left( \psi \right) \right) \varphi -g\psi \right) \\
&=&\lambda \int_{\partial \mathcal{R}_{\delta }}\left( n_{v}G_{\alpha
}\partial _{v}\varphi -n_{v}\partial _{v}G_{\alpha }\varphi +v\varphi
G_{\alpha }n_{x}\right) ds+ \\
&&-\int_{X\diagdown \mathcal{R}_{\delta }}\psi \mu dxdv=\lambda \Psi \left(
\delta \right) -\int_{X\diagdown \mathcal{R}_{\delta }}\psi \mu dxdv
\end{eqnarray*}%
Using Definition \ref{SubSuper2} and (\ref{P1E9}) we arrive at:%
\begin{equation*}
\lim_{\delta \rightarrow 0}\left[ \lambda \Psi \left( \delta \right)
-\int_{X\diagdown \mathcal{R}_{\delta }}\psi \mu dxdv\right] \leq 0
\end{equation*}%
and (\ref{P1E8}) yields $\lim_{\delta \rightarrow 0}\int_{X\diagdown
\mathcal{R}_{\delta }}\psi \mu dxdv=0,$ whence (\ref{P1E7}) follows if $%
\lim_{\delta \rightarrow 0^{+}}\Psi \left( \delta \right) <\infty .$

Suppose now that $\lim_{\delta \rightarrow 0^{+}}\Psi \left( \delta \right)
=\infty .$ Then $\Psi \left( \delta \right) >0$ can be made arbitrarily
large for $\delta $ small. We consider test functions with the form $\psi
=G_{\alpha }\xi $ where $\xi \geq 0$ is a function which takes constant
values in each set $\partial \mathcal{R}_{\rho }$ for each $\rho >0.$
Moreover, we will assume that $\xi =\bar{\xi}\left( \rho \right) $ satisfies
$\bar{\xi}^{\prime }\left( \rho \right) \leq 0$ for $\rho $ small, and $\bar{%
\xi}^{\prime }$ globally bounded. We will assume that $\bar{\xi}$ is
constant for small $\rho ,$ whence $\psi $ is an admissible test function.
We have $\int \left( \psi -\lambda \mathcal{L}^{\ast }\left( \psi \right)
\right) \varphi =\lim_{\delta \rightarrow 0}\int_{X\diagdown \mathcal{R}%
_{\delta }}\left( \psi -\lambda \mathcal{L}^{\ast }\left( \psi \right)
\right) \varphi .$ Integrating by parts we obtain:%
\begin{equation*}
\int_{X\diagdown \mathcal{R}_{\delta }}\left( \psi -\lambda \mathcal{L}%
^{\ast }\left( \psi \right) \right) \varphi =-\lambda \int_{\partial
\mathcal{R}_{\delta }}n_{v}D_{v}\psi \varphi ds+\int_{X\diagdown \mathcal{R}%
_{\delta }}\left( \psi \varphi +\lambda \left( D_{v}\psi D_{v}\varphi
+vD_{x}\psi \varphi \right) \right)
\end{equation*}%
Using then that $\psi =G_{\alpha }\xi $ we rewrite this expression as:%
\begin{eqnarray*}
&&\int_{X\diagdown \mathcal{R}_{\delta }}\left( \psi -\lambda \mathcal{L}%
^{\ast }\left( \psi \right) \right) \varphi  \\
&=&-\lambda \int_{\partial \mathcal{R}_{\delta }}n_{v}\left( \xi
D_{v}G_{\alpha }+G_{\alpha }D_{v}\xi \right) \varphi ds+ \\
&&+\int_{X\diagdown \mathcal{R}_{\delta }}\left( \psi \varphi +\lambda
\left( \xi D_{v}G_{\alpha }D_{v}\varphi +G_{\alpha }D_{v}\xi D_{v}\varphi
+vD_{x}G_{\alpha }\xi \varphi +vG_{\alpha }D_{x}\xi \varphi \right) \right)
\end{eqnarray*}%
We now use that $D_{vv}G_{\alpha }=vD_{x}G_{\alpha }$ to rewrite this
formula as:%
\begin{eqnarray*}
&&\int_{X\diagdown \mathcal{R}_{\delta }}\left( \psi -\lambda \mathcal{L}%
^{\ast }\left( \psi \right) \right) \varphi  \\
&=&-\lambda \int_{\partial \mathcal{R}_{\delta }}n_{v}\left( \xi
D_{v}G_{\alpha }+G_{\alpha }D_{v}\xi \right) \varphi ds+ \\
&&+\int_{X\diagdown \mathcal{R}_{\delta }}\left( \psi \varphi +\lambda
\left( \xi D_{v}G_{\alpha }D_{v}\varphi +G_{\alpha }D_{v}\xi D_{v}\varphi
+D_{vv}G_{\alpha }\xi \varphi +vG_{\alpha }D_{x}\xi \varphi \right) \right)
\end{eqnarray*}%
Then, integrating by parts in the term containing $D_{vv}G_{\alpha }$ we
obtain:%
\begin{eqnarray*}
&&\int_{X\diagdown \mathcal{R}_{\delta }}\left( \psi -\lambda \mathcal{L}%
^{\ast }\left( \psi \right) \right) \varphi  \\
&=&-\lambda \int_{\partial \mathcal{R}_{\delta }}n_{v}\left( \xi
D_{v}G_{\alpha }+G_{\alpha }D_{v}\xi \right) \varphi ds+\lambda
\int_{\partial \mathcal{R}_{\delta }}n_{v}\xi D_{v}G_{\alpha }\varphi ds+ \\
&&+\int_{X\diagdown \mathcal{R}_{\delta }}\left( \psi \varphi +\lambda
\left( \xi D_{v}G_{\alpha }D_{v}\varphi +G_{\alpha }D_{v}\xi D_{v}\varphi
-D_{v}G_{\alpha }D_{v}\left( \xi \varphi \right) +vG_{\alpha }D_{x}\xi
\varphi \right) \right)
\end{eqnarray*}%
whence, after some cancellations and rearrangements of terms:%
\begin{equation}
\int_{X\diagdown \mathcal{R}_{\delta }}\left( \psi -\lambda \mathcal{L}%
^{\ast }\left( \psi \right) \right) \varphi =-\lambda \int_{\partial
\mathcal{R}_{\delta }}n_{v}G_{\alpha }D_{v}\xi \varphi ds+\int_{X\diagdown
\mathcal{R}_{\delta }}\left( \psi \varphi +\lambda Q\right) \   \label{P2E1}
\end{equation}%
where:%
\begin{equation*}
Q=D_{v}\xi \left[ G_{\alpha }D_{v}\varphi -D_{v}G_{\alpha }\varphi \right]
+vG_{\alpha }D_{x}\xi \varphi
\end{equation*}%
We now use the following Fubini's like formula:%
\begin{equation}
\int_{X\diagdown \mathcal{R}_{\delta }}\left( \psi \varphi +\lambda Q\right)
=\int_{X\diagdown \mathcal{R}_{\delta }}\psi \varphi +\lambda \int_{\delta
}^{\infty }d\rho \int_{\partial \mathcal{R}_{\rho }}J_{\rho }Qds_{\rho }\
\label{P2E2}
\end{equation}%
where $ds_{\rho }$ is the arc-length in $\partial \mathcal{R}_{\rho }$ and $%
J_{\rho }$ is obtained by means of the formula:%
\begin{equation*}
\lim_{h\rightarrow 0}\frac{1}{h}\int_{\mathcal{R}_{\rho +h}\diagdown
\mathcal{R}_{\rho }}\omega dxdv=\int_{\partial \mathcal{R}_{\rho }}J_{\rho
}\omega ds_{\rho }
\end{equation*}%
for any continuous test function $\omega .$ It is readily seen that the
function $J_{\rho }$ is well defined by means of elementary arguments.
Moreover, it is given by:%
\begin{equation*}
J_{\rho }=\left\{
\begin{array}{c}
1\ ,\ x\in \left[ 0,\rho ^{3}\right] \ ,\ v=-\rho  \\
r\ ,\ x\in \left[ 0,\rho ^{3}\right] \ ,\ v=r\rho  \\
3\rho ^{2}\ ,\ x=\rho ^{3}\ ,\ -\rho \leq v\leq r\rho
\end{array}%
\right\}
\end{equation*}%
On the other hand, using our choice of $\xi $ we obtain:%
\begin{equation*}
D_{x}\xi =\frac{\bar{\xi}^{\prime }\left( \rho \right) }{3\rho ^{2}}\ \text{%
if}\ -x\leq v^{3}\leq r^{3}x\ ,\ \ D_{x}\xi =0\ \ \text{if\ \ }v^{3}<-x\
\text{or }v^{3}>r^{3}x
\end{equation*}%
\begin{equation*}
D_{v}\xi =-\bar{\xi}^{\prime }\left( \rho \right) \ \ \text{if\ }v^{3}<-x\ \
,\ \ D_{v}\xi =r\bar{\xi}^{\prime }\left( \rho \right) \ \ \text{if\ }%
v^{3}>r^{3}x\ \ ,\ D_{v}\xi =0\ \text{if\ }-x\leq v^{3}\leq r^{3}x\text{\ }\
\end{equation*}%
Then, using also that the normal vector $\left( n_{x},n_{v}\right) $ points
towards $\mathcal{R}_{\rho }$ we obtain:%
\begin{equation*}
J_{\rho }Q=-\bar{\xi}^{\prime }\left( \rho \right) \left[ n_{v}\left[
G_{\alpha }D_{v}\varphi -D_{v}G_{\alpha }\varphi \right] +v\varphi G_{\alpha
}n_{x}\right]
\end{equation*}%
Using this formula in (\ref{P2E1}), (\ref{P2E2}) we obtain:%
\begin{eqnarray*}
\int_{X\diagdown \mathcal{R}_{\delta }}\left( \psi -\lambda \mathcal{L}%
^{\ast }\left( \psi \right) \right) \varphi  &=&-\lambda \int_{\partial
\mathcal{R}_{\delta }}n_{v}G_{\alpha }D_{v}\xi \varphi ds+\int_{X\diagdown
\mathcal{R}_{\delta }}\psi \varphi - \\
&&-\lambda \int_{\delta }^{\infty }\bar{\xi}^{\prime }\left( \rho \right)
d\rho \int_{\partial \mathcal{R}_{\rho }}\left[ n_{v}\left[ G_{\alpha
}D_{v}\varphi -D_{v}G_{\alpha }\varphi \right] +v\varphi G_{\alpha }n_{x}%
\right] ds_{\rho }\
\end{eqnarray*}%
whence, using (\ref{W6E3}) as well as the fact that $\bar{\xi}^{\prime
}\left( \rho \right) =0$ if $\rho $ is sufficiently small, we obtain:
\begin{equation*}
\int_{X\diagdown \mathcal{R}_{\delta }}\left( \psi -\lambda \mathcal{L}%
^{\ast }\left( \psi \right) \right) \varphi =\int_{X\diagdown \mathcal{R}%
_{\delta }}\psi \varphi -\lambda \int_{\delta }^{\infty }\bar{\xi}^{\prime
}\left( \rho \right) \Psi \left( \rho \right) d\rho \
\end{equation*}%
if $\delta $ is sufficiently small. Notice that $\Psi \left( \rho \right) $
is unbounded as $\rho \rightarrow 0.$ Choosing $\bar{\xi}^{\prime }\left(
\rho \right) <0$ for small $\rho ,$ but globally bounded we can make $%
-\int_{\delta }^{\infty }\bar{\xi}^{\prime }\left( \rho \right) \Psi \left(
\rho \right) d\rho $ sufficiently large. On the other hand, we have $%
\lim_{\delta \rightarrow 0}\int_{X\diagdown \mathcal{R}_{\delta }}\left(
\psi -\lambda \mathcal{L}^{\ast }\left( \psi \right) \right) \varphi =\int
\left( \psi -\lambda \mathcal{L}^{\ast }\left( \psi \right) \right) \varphi
\leq \int g\psi \leq C$ and $-\int_{X\diagdown \mathcal{R}_{\delta }}\psi
\varphi \leq C$ for some suitable constant $C$ independent of $\delta .$
Then:%
\begin{equation*}
-\int_{\delta }^{\infty }\bar{\xi}^{\prime }\left( \rho \right) \Psi \left(
\rho \right) d\rho \leq \frac{3C}{\lambda }
\end{equation*}%
for any $\bar{\xi}\geq 0$ bounded$.$ However, as indicated above, it is
possible to choose $\bar{\xi},\ \delta $ yielding $-\int_{\delta }^{\infty }%
\bar{\xi}^{\prime }\left( \rho \right) \Psi \left( \rho \right) d\rho $
arbitrarily large with $\bar{\xi}\geq 0$ bounded. This contradiction then
yields $\lim_{\delta \rightarrow 0^{+}}\Psi \left( \delta \right) <\infty $
whence the Lemma follows.
\end{proof}

We will show now also that it is possible to define the limit value $\varphi
\left( 0,0\right) $ for sub/supersolutions in the sense of Definition \ref%
{SubSuper2}. This will be made by means of the limit as $\delta \rightarrow
0 $ of the quantity $H\left( \delta \right) $ defined below.

\begin{lemma}
\label{LimitHzero}Assume $r<r_{c}.$ Let $g\in C\left( X\right) $ and $%
\varphi $ a subsolution of (\ref{T6E6}), (\ref{T6E8})\ in the sense of
Definition \ref{SubSuper2}. Suppose that the normal vector $\left(
n_{x},n_{v}\right) $ points towards $\mathcal{R}_{\delta }.$ Let:%
\begin{equation*}
H\left( \delta \right) =\int_{\partial \mathcal{R}_{\delta }}\left( n_{v}G_{-%
\frac{2}{3}}\partial _{v}\varphi -n_{v}\partial _{v}G_{-\frac{2}{3}}\varphi
+v\varphi G_{-\frac{2}{3}}n_{x}\right) ds
\end{equation*}%
which is well defined in the sense of Traces (cf. Proposition \ref{Traces}).
Then the limit $\lim_{\delta \rightarrow 0^{+}}H\left( \delta \right)
=H\left( 0\right) $ exists. We also have:%
\begin{equation}
\lim_{\delta \rightarrow 0^{+}}\varphi \left( \delta ^{3}x,\delta v\right) =-%
\frac{H\left( 0\right) }{9^{\frac{2}{3}}\left[ \log \left( r\right) +\frac{%
\pi }{\sqrt{3}}\right] }=\varphi \left( 0,0\right) \   \label{W6E2}
\end{equation}%
in $L_{loc}^{p}$ for some $p>1.$
\end{lemma}

\begin{proof}
We use in Definition \ref{SubSuper2} the test function $\psi =\xi G_{-\frac{2%
}{3}}$ where $\xi =1$ for $\left( x,v\right) \in \mathcal{R}_{\delta
_{2}}\diagdown \mathcal{R}_{\delta _{1}}$ and $\xi =0$ otherwise. We then
obtain:%
\begin{eqnarray*}
-H\left( \delta _{2}\right) +H\left( \delta _{1}\right)  &=&\frac{1}{\lambda
}\int_{\mathcal{R}_{\delta _{2}}\diagdown \mathcal{R}_{\delta _{1}}}G_{-%
\frac{2}{3}}\mu dxdv+ \\
&&+\frac{1}{\lambda }\int_{\mathcal{R}_{\delta _{2}}\diagdown \mathcal{R}%
_{\delta _{1}}}G_{-\frac{2}{3}}\varphi dxdv-\frac{1}{\lambda }\int_{\mathcal{%
R}_{\delta _{2}}\diagdown \mathcal{R}_{\delta _{1}}}G_{-\frac{2}{3}}gdxdv
\end{eqnarray*}%
Using (\ref{P2E3}) combined with the fact that $\alpha <-\frac{2}{3}$ we
obtain $0\leq \frac{1}{\lambda }\int_{\mathcal{R}_{\delta _{2}}\diagdown
\mathcal{R}_{\delta _{1}}}G_{-\frac{2}{3}}\mu dxdv\leq C\left( \delta
_{2}\right) ^{b}$ for $b=-2-3\alpha >0.$ Then, using the boundedness of $%
\varphi $ and $g$ we obtain $\left\vert H\left( \delta _{2}\right) -H\left(
\delta _{1}\right) \right\vert \leq C\left( \delta _{2}\right) ^{b}.$
Choosing $\delta _{1}=\frac{\delta _{2}}{2}$ and iterating, we obtain that
the limit $\lim_{\delta \rightarrow 0}H\left( \delta \right) $ exists and it
is bounded.

In order to prove the convergence of $\varphi _{\delta }\left( x,v\right)
=\varphi \left( \delta ^{3}x,\delta v\right) $ we notice that the functions $%
\varphi _{\delta }$ satisfy:%
\begin{equation}
-\delta ^{2}\varphi _{\delta }+\lambda \mathcal{L}\left( \varphi _{\delta
}\right) +\delta ^{2}g_{\delta }=\mu _{\delta }\   \label{P2E4}
\end{equation}%
where:%
\begin{equation*}
\mu _{\delta }\left( x,v\right) =\delta ^{2}\mu \left( \delta ^{3}x,\delta
v\right) \text{ and}~g_{\delta }\left( x,v\right) =\delta ^{2}g\mu \left(
\delta ^{3}x,\delta v\right) .
\end{equation*}%
Notice that (\ref{P2E3}) yields:\
\begin{equation*}
\int_{\mathcal{R}_{1}\diagdown \left\{ \left( 0,0\right) \right\} }G_{-\frac{%
2}{3}}\mu _{\delta }dxdv\leq C\left( \delta \right) ^{b}
\end{equation*}%
The measures $\mu _{\delta }$ converge to zero in the weak topology in
compact sets. On the other hand we have:\
\begin{equation*}
H\left( \delta \right) =\int_{\partial \mathcal{R}_{1}}\left( n_{v}G_{-\frac{%
2}{3}}\partial _{v}\varphi _{\delta }-n_{v}\partial _{v}G_{-\frac{2}{3}%
}\varphi _{\delta }+v\varphi _{\delta }G_{-\frac{2}{3}}n_{x}\right) ds
\end{equation*}%
The functions $\varphi _{\delta }$ are bounded. Taking suitable subsequences
we obtain convergence to a limit $\hat{\varphi}$ in the weak topology.
Taking the limit in (\ref{P2E4}) we obtain that $\hat{\varphi}$ solves the
following equation in the sense of distributions:%
\begin{equation*}
\mathcal{L}\left( \hat{\varphi}\right) =0
\end{equation*}%
The traces of the functions $\varphi _{\delta }$ at the boundaries of $%
\partial \mathcal{R}_{1}$ are defined uniformly in $\delta ,$ in the sense
that the derivatives of the integrals of $\varphi _{\delta }$ in vertical
lines (and horizontal lines with some derivatives), have uniformly bounded
derivatives (in $L^{1}$ norms), with derivatives converging to zero as $%
\delta \rightarrow 0$) (cf. Lemma \ref{TraceConvergence}). Then, we can take
the limit of $H\left( \delta \right) $ to obtain:%
\begin{equation}
H\left( 0^{+}\right) =\int_{\partial \mathcal{R}_{1}}\left( n_{v}G_{-\frac{2%
}{3}}\partial _{v}\hat{\varphi}-n_{v}\partial _{v}G_{-\frac{2}{3}}\hat{%
\varphi}+v\hat{\varphi}G_{-\frac{2}{3}}n_{x}\right) ds  \label{P2E5}
\end{equation}%
Moreover, the same argument implies that $\hat{\varphi}$ satisfies (\ref%
{comp1}). The point (iv) in Theorem \ref{Liouville} then implies that $\hat{%
\varphi}$ is constant (since it is bounded). We will denote this constant as
$\varphi \left( 0,0\right) .$ Due to (\ref{P2E5}) the limit is independent
of the subsequence and we obtain, using (3.28) in \cite{HVJ2}:%
\begin{equation*}
H\left( 0^{+}\right) =-9^{\frac{2}{3}}\varphi \left( 0,0\right) \left[ \log
\left( r\right) +\frac{\pi }{\sqrt{3}}\right]
\end{equation*}

The convergence takes place in $L^{p}$ for some $p>1,$ due to the
regularizing effects for hypoelliptic operators in Theorem \ref{Hypoell}.
\end{proof}

The following computation shows that the subsolutions of (\ref{T6E6}), (\ref%
{T6E8}) in the sense of Definition \ref{SubSuper2} with the asymptotics (\ref%
{phi_decom}) must have a restriction on the sign of $\mathcal{A}\left(
\varphi \right) .$

\begin{lemma}
\label{ConstStar}Suppose that $\varphi $ is a subsolution of (\ref{T6E6}), (%
\ref{T6E8}) in the sense of Definition \ref{SubSuper2}. Let us assume also
that $\varphi $ satisfies (\ref{phi_decom}). Then:%
\begin{equation}
\lim_{\delta \rightarrow 0^{+}}\Psi (\delta )=C_{\ast }\mathcal{A}\left(
\varphi \right) \   \label{R2}
\end{equation}%
where $\Psi (\delta )$ is as in \eqref{W6E3} and $C_{\ast }$ is as in
(\ref{W1E5a}). Moreover, we have $\mathcal{A}\left( \varphi
\right) \geq 0.$
\end{lemma}

\begin{proof}
Taking the limit $\delta \rightarrow 0$ and using (\ref{phi_decom}), we can
approximate the right-hand side of \eqref{W6E3} as:
\begin{equation*}
\mathcal{A}\left( \varphi \right) \int_{\partial \mathcal{R}_{\delta }}\left[
n_{v}G_{\alpha }D_{v}F_{\beta }-n_{v}D_{v}G_{\alpha }F_{\beta }+vF_{\beta
}G_{\alpha }n_{x}\right] ds\rightarrow C_{\ast }\mathcal{A}\left( \varphi
\right) \text{ as }\delta \rightarrow 0,
\end{equation*}%
where we have used (\ref{W1E5}). This gives (\ref{R2}). We recall that $%
C_{\ast }<0,$ and since $\varphi $ is a subsolution of (\ref{T6E6}), (\ref%
{T6E8}) in the sense of Definition \ref{SubSuper2}, by Lemma \ref{TraceAphi}%
, we obtain that $\mathcal{A}\left( \varphi \right) \geq 0$ whence the
result follows.
\end{proof}

\begin{lemma}
\label{maximaRef}Suppose that $\mathcal{W}_{1},\ \mathcal{W}_{2}$ are two
open subsets of $X$ and that $\varphi _{1}\in L_{b}^{\infty }\left( \mathcal{%
W}_{1}\right) ,\ \varphi _{2}\in L_{b}^{\infty }\left( \mathcal{W}%
_{2}\right) $ are two subsolutions of (\ref{T6E6}), (\ref{T6E8}) in the
sense of Definition \ref{SubSuper2} in their respective domains. Then the
function $\varphi $ defined by means of $\varphi =\max \left\{ \varphi
_{1},\varphi _{2}\right\} $ in $\mathcal{W}_{1}\cap \mathcal{W}_{2},\
\varphi =\varphi _{1}$ in $\mathcal{W}_{1}\diagdown \left( \mathcal{W}%
_{1}\cap \mathcal{W}_{2}\right) ,\ \varphi =\varphi _{2}$ in $\mathcal{W}%
_{2}\diagdown \left( \mathcal{W}_{1}\cap \mathcal{W}_{2}\right) $ is a
subsolution of (\ref{T6E6}), (\ref{T6E8}) in $\mathcal{W}=\left( \mathcal{W}%
_{1}\cup \mathcal{W}_{2}\right) $ in the sense of Definition \ref{SubSuper2}.

Suppose that $\mathcal{W}_{1},\ \mathcal{W}_{2}$ are two open subsets of $X$
and that $\varphi _{1}\in L_{b}^{\infty }\left( \mathcal{W}_{1}\right) ,\
\varphi _{2}\in L_{b}^{\infty }\left( \mathcal{W}_{2}\right) $ are two
supersolutions of (\ref{T6E6}), (\ref{T6E8}) in the sense of Definition \ref%
{SubSuper2} in their respective domains. Then the function $\varphi $
defined by means of $\varphi =\min \left\{ \varphi _{1},\varphi _{2}\right\}
$ in $\mathcal{W}_{1}\cap \mathcal{W}_{2},\ \varphi =\varphi _{1}$ in $%
\mathcal{W}_{1}\diagdown \left( \mathcal{W}_{1}\cap \mathcal{W}_{2}\right)
,\ \varphi =\varphi _{2}$ in $\mathcal{W}_{2}\diagdown \left( \mathcal{W}%
_{1}\cap \mathcal{W}_{2}\right) $ is a supersolution of (\ref{T6E6}), (\ref%
{T6E8}) in $\mathcal{W}=\left( \mathcal{W}_{1}\cup \mathcal{W}_{2}\right) $
in the sense of Definition \ref{SubSuper2}.
\end{lemma}

\begin{proof}
It is enough to prove the result for subsolutions, since for supersolutions
the argument is similar. We recall that we use test functions with the form $%
\psi =\theta \zeta G_{\alpha }+\bar{\psi}$ where $\bar{\psi}\in C^{\infty },$
$\theta \geq 0$\ with $G_{\alpha }$ as in (3.8) in \cite{HVJ2}. If $\theta =0$
and the support of $\bar{\psi}$ does not intersect the singular point the
inequality (\ref{sub_sol2}) follows arguing exactly as in the Proof of Lemma %
\ref{maxima}. On the other hand, since $\varphi $ is bounded, we can use a
limit argument as well as the fact that the contributions of the integrals $%
\int \mathcal{L}^{\ast }\left( \bar{\psi}\right) \varphi $ in the region
close to the origin are small if $\bar{\psi}$ is smooth near the singular
point to prove (\ref{sub_sol2}) for $\bar{\psi}\in C^{\infty }$ compactly
supported. The only difficulty is to derive the inequality for $\psi =\theta
\zeta G_{\alpha }+\bar{\psi}$ with $\theta >0$ and $\zeta $ as in Definition %
\ref{SubSuper2}. We will assume in the following that $\theta =1$.

We define $\varphi =\max \left\{ \varphi _{1},\varphi _{2}\right\} .$ We
need to prove that $\int \left( \psi -\lambda \mathcal{L}^{\ast }\left( \psi
\right) \right) \varphi -\int g\psi \leq 0$ for $\psi $ as above. Notice
that:%
\begin{equation*}
\int \left( \psi -\lambda \mathcal{L}^{\ast }\left( \psi \right) \right)
\varphi -\int g\psi =\lim_{\varepsilon \rightarrow 0^{+}}\left[
\int_{X\diagdown \mathcal{R}_{\varepsilon }}\left( \psi -\lambda \mathcal{L}%
^{\ast }\left( \psi \right) \right) \varphi -\int_{X\diagdown \mathcal{R}%
_{\varepsilon }}g\psi \right]
\end{equation*}%
Integrating by parts and using Proposition \ref{Traces} we obtain for $%
\varepsilon >0$:%
\begin{eqnarray*}
&&\int_{X\diagdown \mathcal{R}_{\varepsilon }}\left( \psi -\lambda \mathcal{L%
}^{\ast }\left( \psi \right) \right) \varphi -\int_{X\diagdown \mathcal{R}%
_{\varepsilon }}g\psi \\
&=&\lambda \int_{\partial \mathcal{R}_{\varepsilon }}\left( n_{v}G_{\alpha
}\partial _{v}\varphi -n_{v}\partial _{v}G_{\alpha }\varphi +v\varphi
G_{\alpha }n_{x}\right) ds+ \\
&&+\int_{X\diagdown \mathcal{R}_{\varepsilon }}\psi \left( \varphi -\lambda
\mathcal{L}\left( \varphi \right) \right) -\int_{X\diagdown \mathcal{R}%
_{\varepsilon }}g\psi
\end{eqnarray*}%
We have that $\left( \varphi -\lambda \mathcal{L}\left( \varphi \right)
\right) -g\leq 0$ in the sense of measures as it might be seen arguing as in
the proof of Lemma \ref{maxima} in sets which do not contain the singular
point $\left( 0,0\right) .$ We then obtain:\
\begin{equation}
\int_{X\diagdown \mathcal{R}_{\varepsilon }}\left( \psi -\lambda \mathcal{L}%
^{\ast }\left( \psi \right) \right) \varphi -\int_{X\diagdown \mathcal{R}%
_{\varepsilon }}g\psi \leq \lambda \int_{\partial \mathcal{R}_{\varepsilon
}}\left( n_{v}G_{\alpha }\partial _{v}\varphi -n_{v}\partial _{v}G_{\alpha
}\varphi +v\varphi G_{\alpha }n_{x}\right) ds  \label{W6E1}
\end{equation}%
We now need to examine the sign of the right-hand side of (\ref{W6E1}). To
this end we first prove the existence of $\varphi \left( 0,0\right) $ in
some suitable sense. Using that $-\mu =\left( \varphi -\lambda \mathcal{L}%
\left( \varphi \right) \right) -g\leq 0$ we obtain, after multiplying by $%
G_{-\frac{2}{3}}$ and integrating by parts in the domain $\mathcal{R}%
_{\varepsilon _{2}}\setminus \mathcal{R}_{\varepsilon _{1}}$ with $%
\varepsilon _{1}<\varepsilon _{2}$ sufficiently small:%
\begin{equation*}
\Phi \left( \varepsilon _{2}\right) -\Phi \left( \varepsilon _{1}\right)
\leq \int_{\mathcal{R}_{\varepsilon _{2}}\setminus \mathcal{R}_{\varepsilon
_{1}}}\left( g-\varphi \right) G_{-\frac{2}{3}}\ \ ,\ \ \varepsilon
_{1}<\varepsilon _{2}
\end{equation*}%
where:%
\begin{equation*}
\Phi \left( \varepsilon \right) =\lambda \int_{\partial \mathcal{R}%
_{\varepsilon }}\left( n_{v}G_{-\frac{2}{3}}\partial _{v}\varphi
-n_{v}\partial _{v}G_{-\frac{2}{3}}\varphi +v\varphi G_{-\frac{2}{3}%
}n_{x}\right) ds
\end{equation*}%
and where $n$ is as usual the normal vector pointing towards the origin. It
then follows that $\Phi \left( \varepsilon \right) +C\varepsilon $ is
decreasing for a suitable constant $C$ and $\varepsilon $ sufficiently
small. We have now two possibilities, either $\lim_{\varepsilon \rightarrow
0}\Phi \left( \varepsilon \right) =\infty $ or $\lim_{\varepsilon
\rightarrow 0}\Phi \left( \varepsilon \right) <\infty .$ We now claim that
the first possibility contradicts the boundedness of $\varphi .$ Indeed, to
check this we argue as follows. Suppose that $\lim_{\varepsilon \rightarrow
0}\Phi \left( \varepsilon \right) =\infty .$ We define $\tilde{\varphi}%
=\varphi +R-m$ where $m=\sup_{\mathcal{R}_{1}}\left( \varphi +R\right) $ and
where $R$ satisfies $-\lambda \mathcal{L}\left( R\right) \leq \left\Vert
g-\varphi \right\Vert _{\infty }$ in order to have $-\lambda \mathcal{L}%
\left( \tilde{\varphi}\right) \leq -\mu .$ Notice that $R$ can be chosen as
a quadratic function in $v$ near the singular point as in Lemma \ref%
{super_sol}. Then $\tilde{\varphi}\leq 0$ in $\mathcal{R}_{1}.$ We define a
sequence $\mu _{k}=\int_{\mathcal{R}_{2^{-k}}\setminus \mathcal{R}%
_{2^{-\left( k+1\right) }}}\mu G_{-\frac{2}{3}}.$ We remark that $%
\lim_{\varepsilon \rightarrow 0}\Phi \left( \varepsilon \right) =\infty $ is
equivalent to $\sum_{k}^{\infty }\mu _{k}=\infty .$ Notice that:%
\begin{equation*}
-\lambda \mathcal{L}\left( \tilde{\varphi}\right) \leq -\mu
\end{equation*}%
We can obtain an upper estimate for $\tilde{\varphi}$ as follows. We define
functions $\tilde{\varphi}_{\ell }$ as the solutions of:%
\begin{eqnarray*}
\lambda \mathcal{L}\left( \tilde{\varphi}_{\ell }\right) &=&\mu \chi _{%
\mathcal{R}_{2^{-\ell }}\setminus \mathcal{R}_{2^{-\left( \ell +1\right)
}}}\ \ \text{in\ }\mathcal{R}_{1}\diagdown \mathcal{R}_{2^{-m}}\text{\ ,\ }%
\ell +1<m\ \ ,\ \ell >1 \\
\tilde{\varphi}_{\ell } &=&0\ \ \text{in\ }\partial _{a}\left( \mathcal{R}%
_{1}\diagdown \mathcal{R}_{2^{-m}}\right)
\end{eqnarray*}%
Since $\tilde{\varphi}_{\ell }\leq 0$ we can compare with the function $\hat{%
\varphi}_{\ell }$ which solves:%
\begin{eqnarray*}
\lambda \mathcal{L}\left( \hat{\varphi}_{\ell }\right) &=&\mu \chi _{%
\mathcal{R}_{2^{-\ell }}\setminus \mathcal{R}_{2^{-\left( \ell +1\right)
}}}\ \ \text{in\ }\mathcal{R}_{2^{-\left( \ell -1\right) }}\setminus
\mathcal{R}_{2^{-\left( \ell +2\right) }}\text{\ ,\ }\ell +1<m\ \ ,\ \ \ell
>1 \\
\hat{\varphi}_{\ell } &=&0\ \ \text{in\ }\partial _{a}\left( \mathcal{R}%
_{2^{-\left( \ell -1\right) }}\setminus \mathcal{R}_{2^{-\left( \ell
+2\right) }}\right)
\end{eqnarray*}%
Then $\tilde{\varphi}_{\ell }\leq \hat{\varphi}_{\ell }$ in $\mathcal{R}%
_{2^{-\left( \ell -1\right) }}\setminus \mathcal{R}_{2^{-\left( \ell
+2\right) }}.$ Proposition \ref{weak_max} combined with a rescaling argument
yields $\hat{\varphi}_{\ell }\leq -c_{0}\mu _{\ell }$\ in $\partial \mathcal{%
R}_{2^{-\ell }}$ for some $c_{0}>0$ independent on $\ell .$ Then $\tilde{%
\varphi}_{\ell }\leq -c_{0}\mu _{\ell }$\ in $\partial \mathcal{R}_{2^{-\ell
}}.$ On the other hand we have $\tilde{\varphi}_{\ell }=0$ in $\partial _{a}%
\mathcal{R}_{1}.$ We can then use comparison with a function with the form $%
w_{\ell }=-c_{1}\mu _{\ell }\left( 1-KF_{\beta }\right) $ with $K$ chosen
large enough and $c_{1}$ sufficiently small to guarantee that $w_{\ell }\geq
0$ in $\partial _{a}\mathcal{R}_{1}$ and $w_{\ell }\geq \tilde{\varphi}%
_{\ell }$ in $\partial _{a}\mathcal{R}_{2^{-\ell }}.$ This implies $w_{\ell
}\geq \tilde{\varphi}_{\ell }$ in $\mathcal{R}_{1}\diagdown \mathcal{R}%
_{2^{-\ell }}.$ Therefore, since $K$ is independent of $\ell ,$ it follows
that $\tilde{\varphi}_{\ell }\leq -\frac{c_{1}}{2}\mu _{\ell }$ in $\mathcal{%
R}_{2^{-J}}\diagdown \mathcal{R}_{2^{-\ell }}$ for some $J>0$ independent of
$\ell .$ We now use that $\tilde{\varphi}\leq \sum_{\ell =J}^{m}\tilde{%
\varphi}_{\ell }$ for any $m.$ Then $\tilde{\varphi}\leq -\frac{c_{1}}{2}%
\sum_{\ell =J}^{m}\mu _{\ell }$ in $\mathcal{R}_{2^{-J}}\diagdown \mathcal{R}%
_{2^{-\left( J+1\right) }}$ and since $\sum_{k}^{\infty }\mu _{k}=\infty $
it then follows that $\tilde{\varphi}=-\infty $ in $\mathcal{R}%
_{2^{-J}}\diagdown \mathcal{R}_{2^{-\left( J+1\right) }}.$ However, we know
that $\tilde{\varphi}$ is bounded. This contradiction yields $%
\sum_{k}^{\infty }\mu _{k}<\infty .$ Then $\mu _{k}\rightarrow 0$ as $%
k\rightarrow \infty .$ We can then argue as in the derivation of (\ref{W6E2}%
) to prove that the limit $\lim_{\delta \rightarrow 0^{+}}\varphi \left(
\delta ^{3}x,\delta v\right) =\varphi \left( 0,0\right) $ exists in $%
L_{loc}^{p}$ for some $p>1.$

We now claim that $\varphi \left( 0,0\right) =\max \left\{ \varphi
_{1}\left( 0,0\right) ,\varphi _{2}\left( 0,0\right) \right\} .$ Suppose
without loss of generality that $\varphi _{2}\left( 0,0\right) \geq \varphi
_{1}\left( 0,0\right) .$ We then have $\varphi \left( \delta ^{3}x,\delta
v\right) \geq \varphi _{2}\left( \delta ^{3}x,\delta v\right) $ whence $%
\varphi \left( 0,0\right) \geq \varphi _{2}\left( 0,0\right) .$ On the other
hand we have $\varphi _{1}\left( \delta ^{3}x,\delta v\right) =\varphi
_{1}\left( 0,0\right) +\varepsilon _{1,\delta }\left( x,v\right) ,$ $\varphi
_{2}\left( \delta ^{3}x,\delta v\right) =\varphi _{2}\left( 0,0\right)
+\varepsilon _{2,\delta }\left( x,v\right) ,$ where $\varepsilon _{k,\delta
}\left( x,v\right) \rightarrow 0$ as $\delta \rightarrow 0$ in $%
L_{loc}^{p},\ k=1,2.$ Then $\varphi \left( \delta ^{3}x,\delta v\right) \leq
\varphi _{2}\left( 0,0\right) +\left\vert \varepsilon _{1,\delta }\left(
x,v\right) \right\vert +\left\vert \varepsilon _{2,\delta }\left( x,v\right)
\right\vert ,$ whence $\varphi \left( 0,0\right) \leq \varphi _{2}\left(
0,0\right) $ follows. Therefore $\varphi \left( 0,0\right) =\max \left\{
\varphi _{1}\left( 0,0\right) ,\varphi _{2}\left( 0,0\right) \right\} .$
Arguing now as in the Proof of Lemma \ref{TraceAphi} we obtain that there
exists the limit $\lim_{\delta \rightarrow 0^{+}}\Psi \left( \delta \right) $
with $\Psi \left( \delta \right) $ defined as in (\ref{W6E3}). This limit
might take perhaps the value $+\infty .$ We will denote this limit as $\Psi
_{\varphi }\left( 0\right) $ in order to make explicit the dependence on the
function $\varphi .$ Lemma \ref{TraceAphi} implies the existence of the
corresponding limits $\Psi _{\varphi _{1}}\left( 0\right) ,\ \Psi _{\varphi
_{2}}\left( 0\right) $ for the subsolutions $\varphi _{1},\ \varphi _{2}$
respectively. Moreover, we have $\Psi _{\varphi _{1}}\left( 0\right) \leq
0,\ \Psi _{\varphi _{2}}\left( 0\right) \leq 0.$

We now distinguish two different cases. Suppose first that $\varphi
_{1}\left( 0,0\right) <\varphi _{2}\left( 0,0\right) ,$ where the values $%
\varphi _{k}\left( 0,0\right) ,\ k=1,2$ are defined as in (\ref{W6E2}). We
then argue as follows. Taking into account (\ref{P2E3}) we obtain that $%
\varepsilon _{k,\delta }\left( x,v\right) =\left\vert \varphi _{k}\left(
\delta ^{3}x,\delta v\right) -\varphi _{k}\left( 0,0\right) \right\vert ,\
k=1,2$ satisfies:\
\begin{equation*}
\left\Vert \varepsilon _{k}\right\Vert _{L^{p}\left( K\right) }<<C\left(
\delta \right) ^{-3\alpha -2}\ \ \text{as\ \ }\delta \rightarrow 0^{+}
\end{equation*}%
for some $p>1.$ This estimate follows from (\ref{P2E3}) using a rescaling
argument in the equation satisfied by $\varphi _{k}.$

Notice that $-3\alpha -2>0.$ We have also that:%
\begin{equation*}
-\varepsilon _{1}\leq \varphi \left( \delta ^{3}x,\delta v\right) -\varphi
_{2}\left( \delta ^{3}x,\delta v\right) \leq \varepsilon _{2}
\end{equation*}%
whence:%
\begin{equation}
\left\Vert \varphi \left( \delta ^{3}\cdot ,\delta \cdot \right) -\varphi
_{2}\left( \delta ^{3}\cdot ,\delta \cdot \right) \right\Vert _{L^{p}\left(
K\right) }<<C\left( \delta \right) ^{-3\alpha -2}  \label{P3E2}
\end{equation}%
We now study the behaviour of the function $\Psi _{\varphi }\left( \delta
\right) $ as $\delta \rightarrow 0.$ We denote as $\varphi ^{\delta }$ the
rescaled function $\varphi ^{\delta }=\varphi \left( \delta ^{3}\cdot
,\delta \cdot \right) $ and we define similarly $\varphi _{k}^{\delta }$ for
$k=1,2.$ Then it follows that:
\begin{equation}
\Psi _{\varphi }\left( \delta \zeta \right) =\delta ^{2+3\alpha
}\int_{\partial \mathcal{R}_{\zeta }}\left( n_{v}G_{\alpha }\partial
_{v}\varphi ^{\delta }-n_{v}\partial _{v}G_{\alpha }\varphi ^{\delta
}+v\varphi ^{\delta }G_{\alpha }n_{x}\right) ds\   \label{P3E1}
\end{equation}%
with $\zeta \in \left[ \frac{1}{2},2\right] .$ We now need to approximate
this integral by the one in which $\varphi ^{\delta }$ is replaced by $%
\varphi _{2}^{\delta }.$ This requires to obtain some continuity of the
integrand of the right-hand side of (\ref{P3E1}) using (\ref{P3E2}). This
requires to obtain estimates for $\left( \partial _{v}\varphi ^{\delta
}-\partial _{v}\varphi _{2}^{\delta }\right) .$ Since we need these
estimates only in the part of the boundaries $\partial \mathcal{R}_{\zeta }$
where $\left\vert n_{v}\right\vert =1$ and $n_{x}=0$ we can obtain the
result using classical parabolic estimates. Notice that the problem reduces
to obtaining estimates for the derivatives with respect to $v$ of the
solutions of $\mathcal{L}\left( w\right) =\nu $ where $\nu $ is a measure.
Standard theory yields estimates of the form $\int \left\vert \partial
_{v}w\right\vert dxdv\leq C\left\Vert \nu \right\Vert ,$ in bounded sets. It
then follows, combining this estimate with (\ref{P3E2}), (\ref{P3E1}) as
well as Fubini's Theorem that $\int_{\left[ \frac{1}{2},2\right] }\left\vert
\Psi _{\varphi }\left( \delta \zeta \right) -\Psi _{\varphi _{2}}\left(
\delta \zeta \right) \right\vert d\zeta \rightarrow 0$ as $\delta
\rightarrow 0.$ Therefore, there exists $\zeta =\zeta _{\delta }$ such that $%
\left\vert \Psi _{\varphi }\left( \delta \zeta _{\delta }\right) -\Psi
_{\varphi _{2}}\left( \delta \zeta _{\delta }\right) \right\vert \rightarrow
0$ as $\delta \rightarrow 0,$ whence $\Psi _{\varphi }\left( 0\right) =\Psi
_{\varphi _{2}}\left( 0\right) .$ In this case, we can use (\ref{W6E1}) to
obtain, taking the limit $\varepsilon \rightarrow 0$ that:
\begin{equation*}
\int_{X\diagdown \left\{ 0,0\right\} }\left( \psi -\lambda \mathcal{L}^{\ast
}\left( \psi \right) \right) \varphi -\int_{X\diagdown \left\{ 0,0\right\}
}g\psi \leq \lambda \Psi _{\varphi _{2}}\left( 0\right)
\end{equation*}%
Using then Lemma \ref{TraceAphi} we obtain $\Psi _{\varphi _{2}}\left(
0\right) \leq 0,$ whence the Subsolution inequality (\ref{sub_sol2}) follows.

The case $\varphi _{1}\left( 0,0\right) >\varphi _{2}\left( 0,0\right) $ is
similar. Suppose then that $\varphi _{1}\left( 0,0\right) =\varphi
_{2}\left( 0,0\right) .$ Lemma \ref{TraceAphi} implies the existence of the
limits $\Psi _{\varphi _{1}}\left( 0\right) ,\ \Psi _{\varphi _{2}}\left(
0\right) .$ Notice that we have also the existence of $\Psi _{\varphi
}\left( 0\right) $ although this limit can take the value $\infty .$

We define rescaled functions as $\Phi _{k}^{\delta }\left( x,v\right) =\frac{%
\varphi _{k}\left( \delta ^{3}x,\delta v\right) -\varphi _{k}\left(
0,0\right) }{\delta ^{3\beta }},\ k=1,2.$ We then have, using Lemma \ref%
{ConstStar}, that $\Phi _{k}^{\delta }$ converges to $\mathcal{A}\left(
\varphi _{k}\right) F_{\beta }$ with $\mathcal{A}\left( \varphi _{k}\right) =%
\frac{1}{C_{\ast }}\Psi _{\varphi _{k}}\left( 0\right) .$ The convergence
takes place in $L_{loc}^{p}$ with $p>1$.\ Moreover, arguing as in the
previous case we obtain also convergence of $\partial _{v}\Phi _{k}^{\delta
} $ to $\mathcal{A}\left( \varphi _{k}\right) \partial _{v}F_{\beta }$ in $%
L_{loc}^{p}.$ We now have two possibilities. Either $\Psi _{\varphi
_{1}}\left( 0\right) \neq \Psi _{\varphi _{2}}\left( 0\right) $ or $\Psi
_{\varphi _{1}}\left( 0\right) =\Psi _{\varphi _{2}}\left( 0\right) .$ In
the first case, suppose without loss of generality that $\Psi _{\varphi
_{1}}\left( 0\right) >\Psi _{\varphi _{2}}\left( 0\right) .$ We recall that $%
\varphi =\max \left\{ \varphi _{1},\varphi _{2}\right\} .$ We define also $%
\Phi ^{\delta }\left( x,v\right) =\frac{\varphi \left( \delta ^{3}x,\delta
v\right) -\varphi \left( 0,0\right) }{\delta ^{3\beta }}.$ Notice that $%
\varphi _{1}\left( 0,0\right) =\varphi _{2}\left( 0,0\right) =\varphi
_{k}\left( 0,0\right) .$ We then have, arguing as in the case $\varphi
_{1}\left( 0,0\right) \neq \varphi _{2}\left( 0,0\right) $ that $\Phi
^{\delta }$ converges to $\Phi _{1}^{\delta }$ in $L_{loc}^{p}.$ Using the
continuity of the functionals $\Psi _{\varphi }\left( 0\right) $ on the
functions $\varphi $ obtained above, it then follows that $\Psi _{\varphi
}\left( 0\right) =\Psi _{\varphi _{1}}\left( 0\right) \leq 0$ whence the
inequality (\ref{sub_sol2}) follows also in this case. It remains to examine
the case $\Psi _{\varphi _{1}}\left( 0\right) =\Psi _{\varphi _{2}}\left(
0\right) .$ In this case we have that the functions $\Phi _{k}^{\delta }$
for $k=1,2$ converge to $\mathcal{A}\left( \varphi _{1}\right) F_{\beta }=%
\mathcal{A}\left( \varphi _{2}\right) F_{\beta }$ in $L_{loc}^{p}.$ Then $%
\Phi ^{\delta }=\max \left\{ \Phi _{1}^{\delta },\Phi _{2}^{\delta }\right\}
$ converges to the same limit in $L_{loc}^{p}.$ We have also convergence for
the derivatives $\partial _{v}\Phi ^{\delta }$ using the regularizing
effects as usual. Then, using the continuity of the functionals $\Psi
_{\varphi }\left( 0\right) $ with respect to this topology we obtain $\Psi
_{\varphi }\left( 0\right) =\Psi _{\varphi _{1}}\left( 0\right) =\Psi
_{\varphi _{2}}\left( 0\right) ,$ whence $\Psi _{\varphi }\left( 0\right)
\leq 0$ and (\ref{sub_sol2}) follows. Therefore $\varphi $ is a subsolution
and the result follows.
\end{proof}

We can now conclude the proof of Proposition \ref{Case_r}, arguing as in the
case of trapping boundary conditions.

We can obtain in this case easily one subsolution and one supersolution
which are ordered.

\begin{lemma}
\label{ExSubSup2}For any $g\in C\left( X\right) $ there exist at least one
subsolution $\varphi ^{\mathrm{sub}}$ and one supersolution $\varphi ^{\sup
}$ in the sense of Definition \ref{SubSuper2} such that:%
\begin{equation}
\varphi ^{\mathrm{sub}}\leq \varphi ^{\sup }  \label{ineqSubSup}
\end{equation}
\end{lemma}

\begin{proof}
We can just take $\varphi ^{\mathrm{sub}}=-\left\Vert g\right\Vert
_{L^{\infty }\left( X\right) },\ \varphi ^{\sup }=\left\Vert g\right\Vert
_{L^{\infty }\left( X\right) }.$ We need to prove that $\int \mathcal{L}%
^{\ast }\left( \psi \right) \varphi \leq 0$ if $\varphi $ is a constant and $%
\psi $ is any test function as in the Definition \ref{SubSuper2}. We can
assume that $\varphi =1.$ We write $\int \mathcal{L}^{\ast }\left( \psi
\right) =\lim_{\delta \rightarrow 0^{+}}\int_{X\diagdown \mathcal{R}_{\delta
}}\mathcal{L}^{\ast }\left( \psi \right) $ with $\mathcal{R}_{\delta }$ as
in (\ref{domainR}). Integrating by parts and using that $\psi $ is compactly
supported, as well as the fact that $\psi \left( 0,-v\right) =r^{2}\psi
\left( 0,rv\right) ,\ v>0$ we obtain:%
\begin{equation*}
\int_{X\diagdown \mathcal{R}_{\delta }}\mathcal{L}^{\ast }\left( \psi
\right) =\int_{\partial \mathcal{R}_{\delta }}\left[ D_{v}\psi n_{v}-v\psi
n_{x}\right]
\end{equation*}%
where $n=\left( n_{x},n_{v}\right) $ is the normal vector to $\partial
\mathcal{R}_{\delta }$ pointing towards $\mathcal{R}_{\delta }$ and where $%
\psi =\theta G_{\gamma }+\bar{\psi}$ if $\delta $ is sufficiently small and $%
\gamma \in \left\{ -\frac{2}{3},\alpha \right\} $. The regularity of $\bar{%
\psi}$ implies that
\begin{equation*}
\lim_{\delta \rightarrow 0^{+}}\int_{\partial \mathcal{R}_{\delta }}\left[
D_{v}\bar{\psi}n_{v}-v\bar{\psi}n_{x}\right] =0.
\end{equation*}
On the other hand we have
\begin{equation*}
\lim_{\delta \rightarrow 0^{+}}\int_{\partial \mathcal{R}_{\delta }}\left[
D_{v}G_{\alpha }n_{v}-vG_{\alpha }n_{x}\right] =0
\end{equation*}%
due to Proposition 3.3  in \cite{HVJ2}. Finally we notice that%
\begin{equation*}
\lim_{\delta \rightarrow 0^{+}}\int_{\partial \mathcal{R}_{\delta }}\left[
D_{v}G_{-\frac{2}{3}}n_{v}-vG_{-\frac{2}{3}}n_{x}\right] \leq 0
\end{equation*}
due to Proposition 3.2 in \cite{HVJ2} whence the result follows.
\end{proof}

We can then define $\mathcal{\tilde{G}}_{sub}$ as in Definition the
following subset of $L_{b}^{\infty }\left( X\right) :$

\begin{definition}
\label{SubSet2}Suppose that $\varphi ^{\mathrm{sub}},\ \varphi ^{\sup }$
are respectively one subsolution and one supersolution in the sense of
Definition \ref{SubSuper2} satisfying (\ref{ineqSubSup}). We then define $%
\mathcal{\tilde{G}}_{sub}\subset L_{b}^{\infty }\left( X\right) $ as:%
\begin{equation}
\mathcal{\tilde{G}}_{sub}\equiv \left\{ \varphi \in L_{b}^{\infty }\left(
X\right) :%
\begin{array}{c}
\varphi \text{ sub-solution of (\ref{T6E6}),\ (\ref{T6E8})} \\
\text{ in the sense of Definition \ref{SubSuper2}}\ |\ \varphi ^{\operatorname{sub%
}}\leq \varphi \leq \varphi ^{\sup }%
\end{array}%
\right\} .  \label{T7E6b}
\end{equation}
\end{definition}

We now remark that Lemmas \ref{ContinuityWeak} and \ref{countable} hold
without any changes. Moreover, the proof of Lemma \ref{max_sub} can be
adapted to include this case.

\begin{lemma}
\label{max_sub_ref}Let $\varphi _{1},\varphi _{2},...,\varphi _{L}\in
\mathcal{\tilde{G}}_{sub}$ with $L<\infty $ and $\mathcal{\tilde{G}}_{sub}$
as in Definition \ref{SubSet2}. Define:%
\begin{equation*}
\bar{\varphi}:=\max \{\varphi _{1},\varphi _{2},...,\varphi _{L}\}
\end{equation*}

Then $\bar{\varphi}\in \mathcal{\tilde{G}}_{sub}$.
\end{lemma}

\begin{proof}
It is enough to prove the result for $L=2.$ The result then follows from
Lemma \ref{maximaRef}. Notice that the inequalities $\varphi ^{\mathrm{sub}%
}\leq \varphi \leq \varphi ^{\sup }$ in Definition (\ref{T7E6b}) are
preserved by the maxima.
\end{proof}

We remark that Proposition \ref{DirSolv} and Lemma \ref{comp_nu} are
independent of the boundary conditions imposed at the singular point $\left(
x,v\right) =\left( 0,0\right) $ and therefore can be applied to subsolutions
as in Definition \ref{SubSet2}. Lemma \ref{IncSub} can be proved for the
subsolutions in the class $\mathcal{\tilde{G}}_{sub}$ given in Definition %
\ref{SubSet2} with minor cases with respect to the previous case, because
the definition of $\Phi $ in (\ref{T7E5}) modifies the subsolution $\bar{%
\varphi}$ just in a set $\Xi $ which does not intersect the singular point $%
\left( 0,0\right) .$ We can now conclude the Proof of Proposition \ref%
{Case_r}.

\begin{proof}[End of the Proof of Proposition \protect\ref{Case_r}]
We define $\varphi _{\ast }$ as in (\ref{T7E2}) with $\mathcal{\tilde{G}}%
_{sub}$ as in Definition \ref{SubSet2}. Arguing as in the Proof of
Proposition \ref{Case_a_sub} we obtain that $\varphi _{\ast }\in \mathcal{%
\tilde{G}}_{sub}$ with $\mathcal{\tilde{G}}_{sub}$ as in Definition \ref%
{SubSet2}. We define $\nu =-\left( \varphi _{\ast }-\lambda \mathcal{L}%
\left( \varphi _{\ast }\right) -g\right) .$ Due to Definition \ref{SubSuper2}
we have that $\nu $ is a nonnegative Radon measure. If $\nu \left(
X\diagdown \left\{ 0,0\right\} \right) >0$ we can argue exactly as in the
Proof of Proposition \ref{Case_a_sub} to derive a contradiction. Therefore $%
\nu \left( X\diagdown \left\{ 0,0\right\} \right) =0$ and then, $\varphi
_{\ast }$ satisfies the problem:%
\begin{equation*}
\varphi _{\ast }-\lambda \mathcal{L}\left( \varphi _{\ast }\right) -g=0\ ,\
\left( x,v\right) \in X\diagdown \left\{ 0,0\right\} \ ,\ \varphi _{\ast
}\in C\left( X\right)
\end{equation*}
Theorem \ref{AsSingPoint} then implies:%
\begin{equation*}
\varphi _{\ast }\left( x,v\right) =\varphi _{\ast }\left( 0,0\right) +%
\mathcal{A}\left( \varphi _{\ast }\right) F_{\beta }\left( x,v\right) +\psi
_{\ast }\left( x,v\right)
\end{equation*}

Using Lemma \ref{TraceAphi} and Lemma \ref{ConstStar} we obtain that $%
\mathcal{A}\left( \varphi _{\ast }\right) \geq 0.$ Suppose that $\mathcal{A}%
\left( \varphi _{\ast }\right) >0.$ Then $\varphi _{\ast }\left( x,v\right)
\geq \varphi _{\ast }\left( 0,0\right) $ in a neighbourhood of the singular
point $\left( 0,0\right) .$\ Then $\varphi _{\ast }\left( 0,0\right)
<\varphi ^{\sup },$ where $\varphi ^{\sup }$ is as in Definition \ref%
{SubSet2}, because otherwise we would have $\varphi _{\ast }=\varphi ^{\sup
} $ and then $\mathcal{A}\left( \varphi _{\ast }\right) =0.$ Therefore,
there exists $\delta >0$ such that $\varphi _{\ast \ast }=\max \left\{
\varphi _{\ast },\varphi _{\ast }\left( 0,0\right) +\delta +CS\left(
x,v\right) \right\} $ with $S\left( x,v\right) $ as in Lemma \ref{super_sol}
and $C>0$ is a subsolution of (\ref{T6E6}), (\ref{T6E8}) in a neighbourhood
of the singular point satisfying $\varphi _{\ast \ast }>\varphi _{\ast },$
since $S$ is bounded by $\left( x^{\frac{2}{3}}+v^{2}\right) $ and therefore
$\mathcal{A}\left( \varphi \right) F_{\beta }$ gives a larger contribution
for $\left( x,v\right) $ close to $\left( 0,0\right) .$ However, this would
contradict the definition of $\varphi _{\ast }$ by means of (\ref{T7E2}). It
then follows from this contradiction that $\mathcal{A}\left( \varphi _{\ast
}\right) =0.$ Therefore $\varphi _{\ast }$ solves (\ref{T6E6}),\ (\ref{T6E8}%
) and the Proposition follows.
\end{proof}

\subsection{Operator $\Omega _{pt,sub}.$ Solvability of (\ref{T6E6}), (%
\ref{T6E9}).}

We consider now the problem (\ref{T6E6}), (\ref{T6E9}). Since the arguments
used are similar to the ones in the previous Sections we will just describe
in detail the points where differences arise. We prove in this Section that:

\begin{proposition}
\label{Case_m}For any $g\in C\left( X\right) $ there exists a unique $%
\varphi \in C\left( X\right) $ which solves (\ref{T6E6}), (\ref{T6E9}).
\end{proposition}

We introduce a new concept of sub and supersolutions which allows to
identify the boundary condition (\ref{T6E9}) near the singular point $\left(
x,v\right) =\left( 0,0\right) .$ The rationale behind this definition is to
obtain the inequality $\mathcal{L}\left( \varphi \right) \left( 0,0\right)
\leq \mu _{\ast }\left\vert C_{\ast }\right\vert \mathcal{A}\left( \varphi
\right) $ at the singular point. However, since we are interested in
subsolutions in $L_{b}^{\infty }\left( X\right) $ some care is needed in
order to define $\mathcal{L}\left( \varphi \right) \left( 0,0\right) .$
Notice that the inequality $\varphi -\lambda \mathcal{L}\left( \varphi
\right) \leq g,$ suggests the inequality $\varphi \left( 0,0\right) -g\left(
0,0\right) \leq \lambda \mu _{\ast }\left\vert C_{\ast }\right\vert \mathcal{%
A}\left( \varphi \right) .$

\begin{definition}
\label{SubSuperMixed}Suppose that $g\in C\left( X\right) .$ Let $\zeta
=\zeta \left( x,v\right) $ a nonnegative $C^{\infty }$ test function
supported in $0\leq x+\left\vert v\right\vert ^{3}\leq 2$, satisfying $\zeta
=1$ for $0\leq x+\left\vert v\right\vert ^{3}\leq 1$. We will say that a
function $\varphi \in L_{b}^{\infty }\left( X\right) ,$ such that the limit $%
\lim_{\left( x,v\right) \rightarrow \left( 0,0\right) }\varphi \left(
x,v\right) =\varphi \left( 0,0\right) $ exists, is a subsolution of (\ref%
{T6E6}), (\ref{T6E9}) if for all $\psi $ with $\psi \geq 0,\ \psi \left(
0,-v\right) =r^{2}\psi \left( 0,rv\right) ,\ v>0$ and having the form $\psi
=\theta \zeta G_{\gamma }+\bar{\psi}$ where $\bar{\psi}\in C^{\infty },$ $%
\theta \geq 0$\ with $G_{\gamma }$, $\gamma \in
\left\{ -\frac{2}{3},\alpha \right\} $ and with $\psi $ supported in a set
contained in the ball $\left\vert \left( x,v\right) \right\vert \leq R$ for
some $R>0$ we have:%
\begin{equation}
\int \left( \psi -\lambda \mathcal{L}^{\ast }\left( \psi \right) \right)
\varphi +\frac{\theta }{\mu _{\ast }}\left( \varphi \left( 0,0\right)
-g\left( 0,0\right) \right) \leq \int g\psi  \label{sub_solmixed}
\end{equation}%
Given $g\in C\left( X\right) ,$ we will say that $\varphi \in L_{b}^{\infty
}\left( X\right) $ is a supersolution of (\ref{T6E6}), (\ref{T6E8}) if for
all $\psi $ with the same properties as above\ we have:%
\begin{equation}
\int \left( \psi -\lambda \mathcal{L}^{\ast }\left( \psi \right) \right)
\varphi +\frac{\theta }{\mu _{\ast }}\left( \varphi \left( 0,0\right)
-g\left( 0,0\right) \right) \geq \int g\psi  \label{super_solmixed}
\end{equation}
\end{definition}

The integral on the left hand side of (\ref{sub_solmixed}), (\ref%
{super_solmixed}) is well defined in spite of the singularity of $G_{\gamma
} $ near the singular point because $\mathcal{L}^{\ast }\left( G_{\gamma
}\right) =0.$

We can then prove the following variation of Lemma \ref{TraceAphi}.

\begin{lemma}
\label{TraceAphiMixed}Let $g\in C\left( X\right) $ and $\varphi $ be a
subsolution of (\ref{T6E6}), (\ref{T6E9})\ in the sense of Definition \ref%
{SubSuperMixed}. Let:%
\begin{equation}
\Psi \left( \delta \right) =\int_{\partial \mathcal{R}_{\delta }}\left(
n_{v}G_{\alpha }\partial _{v}\varphi -n_{v}\partial _{v}G_{\alpha }\varphi
+v\varphi G_{\alpha }n_{x}\right) ds  \label{Z1E2}
\end{equation}%
which is well defined in the sense of Traces (cf. Proposition \ref{Traces}).
Moreover, the limit $\lim_{\delta \rightarrow 0^{+}}\Psi \left( \delta
\right) $ exists and we have:%
\begin{equation}
\lim_{\delta \rightarrow 0^{+}}\Psi \left( \delta \right) +\frac{1}{\lambda
\mu _{\ast }}\left( \varphi \left( 0,0\right) -g\left( 0,0\right) \right)
\leq 0\   \label{Z1E1}
\end{equation}%
Let $\mu =-\varphi +\lambda \mathcal{L}\left( \varphi \right) +g.$ Then $\mu
\geq 0$ defines a Radon measure in $X\diagdown \left\{ \left( 0,0\right)
\right\} $ satisfying:%
\begin{equation}
\int_{\mathcal{R}_{1}\diagdown \left\{ \left( 0,0\right) \right\} }G_{\alpha
}\mu dxdv<\infty   \label{Z1E3}
\end{equation}%
\bigskip
\end{lemma}

\begin{proof}
The proof is similar to the one of Lemma \ref{TraceAphi}. The only
difference is that for test functions as in Definition \ref{SubSuperMixed}
we obtain the inequality:%
\begin{equation*}
\int \left( \left( \psi -\lambda \mathcal{L}^{\ast }\left( \psi \right)
\right) \varphi -g\psi \right) +\frac{\theta }{\mu _{\ast }}\left( \varphi
\left( 0,0\right) -g\left( 0,0\right) \right) \leq 0
\end{equation*}%
instead of (\ref{M1}). Arguing then as in the proof of Lemma \ref{TraceAphi}
we obtain (\ref{Z1E1}). The rest of the argument is basically identical to
the proof of Lemma \ref{TraceAphi} with minor changes.
\end{proof}

We now remark that several of the arguments used in Section \ref{SolvAbs}
can be adapted with minor changes to the case considered in this Section.
Indeed, Lemma \ref{LimitHzero} holds if instead of considering subsolutions
of (\ref{T6E6}), (\ref{T6E8})\ in the sense of Definition \ref{SubSuper2} we
consider subsolutions of (\ref{T6E6}), (\ref{T6E9})\ in the sense of
Definition \ref{SubSuperMixed}. Moreover, the limit value $\varphi \left(
0,0\right) $ is the same value as the limit $\lim_{\left( x,v\right)
\rightarrow \left( 0,0\right) }\varphi \left( x,v\right) $ whose existence
was assumed in Definition \ref{SubSuperMixed}. Actually, the use of Lemma %
\ref{LimitHzero} could be avoided in the case of partially trapping Boundary
Conditions, due to the existence of the limit $\lim_{\left( x,v\right)
\rightarrow \left( 0,0\right) }\varphi \left( x,v\right) $ in the Definition
of subsolutions, however, it will be convenient to have also the
corresponding version of Lemma \ref{LimitHzero} in order to adapt the
argument yielding to the corresponding version of Lemma \ref{maximaRef}. The
proof of a new version of Lemma \ref{maximaRef} in which the subsolutions
considered are subsolutions of (\ref{T6E6}), (\ref{T6E9})\ in the sense of
Definition \ref{SubSuperMixed} follows with minor changes, just replacing at
several places inequalities like $\Psi _{\varphi }\left( 0\right) \leq 0$ by
$\Psi _{\varphi }\left( 0\right) +\frac{1}{\lambda \mu _{\ast }}\left(
\varphi \left( 0,0\right) -g\left( 0,0\right) \right) \leq 0$.

We can now adapt Lemma \ref{ExSubSup2} as follows:

\begin{lemma}
\label{ExSubSupMixed}For any $g\in C\left( X\right) $ there exist at least
one subsolution $\varphi ^{\mathrm{sub}}$ and one supersolution $\varphi
^{\sup }$ in the sense of Definition \ref{SubSuperMixed} such that:%
\begin{equation}
\varphi ^{\mathrm{sub}}\leq \varphi ^{\sup }  \label{ineqSubSupMix}
\end{equation}
\end{lemma}

\begin{proof}
We can just take $\varphi ^{\mathrm{sub}}=-\left\Vert g\right\Vert
_{L^{\infty }\left( X\right) },\ \varphi ^{\sup }=\left\Vert g\right\Vert
_{L^{\infty }\left( X\right) }.$ The integral terms in (\ref{sub_solmixed}),
(\ref{super_solmixed}) can the estimated as in the proof of Lemma \ref%
{ExSubSup2}. Using then that \ $\left( \varphi ^{\mathrm{sub}}-g\left(
0,0\right) \right) \leq 0$ and $\left( \varphi ^{\sup }-g\left( 0,0\right)
\right) \leq 0$ the result follows.
\end{proof}

We can then define $\mathcal{\tilde{G}}_{sub}$ as in Definition the
following subset of $L_{b}^{\infty }\left( X\right) :$

\begin{definition}
\label{SubSetMixed}Suppose that $\varphi ^{\mathrm{sub}},\ \varphi ^{\sup }$
are respectively one subsolution and one supersolution in the sense of
Definition \ref{SubSuperMixed} satisfying (\ref{ineqSubSupMix}). We then
define $\mathcal{\tilde{G}}_{sub}\subset L_{b}^{\infty }\left( X\right) $ as:%
\begin{equation}
\mathcal{\tilde{G}}_{sub}\equiv \left\{ \varphi \in L_{b}^{\infty }\left(
X\right) :%
\begin{array}{c}
\varphi \text{ sub-solution of (\ref{T6E6}),\ (\ref{T6E9})} \\
\text{ in the sense of Definition \ref{SubSuperMixed}}\ |\ \varphi ^{%
\mathrm{sub}}\leq \varphi \leq \varphi ^{\sup }%
\end{array}%
\right\} .  \label{T7E6}
\end{equation}
\end{definition}

Lemmas \ref{ContinuityWeak} and \ref{countable} hold without any changes and
the proof of Lemma \ref{max_sub} can be adapted to include this case.

\begin{lemma}
\label{max_sub_Mixed}Let $\varphi _{1},\varphi _{2},...,\varphi _{L}\in
\mathcal{\tilde{G}}_{sub}$ with $L<\infty $ and $\mathcal{\tilde{G}}_{sub}$
as in Definition \ref{SubSuperMixed}. Define:%
\begin{equation*}
\bar{\varphi}:=\max \{\varphi _{1},\varphi _{2},...,\varphi _{L}\}
\end{equation*}

Then $\bar{\varphi}\in \mathcal{\tilde{G}}_{sub}$.
\end{lemma}

\begin{proof}
It is similar to the Proof of Lemma \ref{max_sub_ref}.
\end{proof}

We remark that Proposition \ref{DirSolv} and Lemma \ref{comp_nu} are
independent of the boundary conditions imposed at the singular point $\left(
x,v\right) =\left( 0,0\right) $ and therefore can be applied to subsolutions
as in Definition \ref{SubSuperMixed}. Lemma \ref{IncSub} can be proved for
the subsolutions in the class $\mathcal{\tilde{G}}_{sub}$ given in
Definition \ref{SubSuperMixed} with minor cases with respect to the case of
trapping boundary conditions. We can now conclude the Proof of Proposition %
\ref{Case_m}.

\begin{proof}[End of the Proof of Proposition \protect\ref{Case_m}]
We define $\varphi _{\ast }$ as in (\ref{T7E2}) with $\mathcal{\tilde{G}}%
_{sub}$ as in Definition \ref{SubSetMixed}. Arguing as in the Proof of
Proposition \ref{Case_a_sub} we obtain that $\varphi _{\ast }\in \mathcal{%
\tilde{G}}_{sub}$ with $\mathcal{\tilde{G}}_{sub}$ as in Definition \ref%
{SubSetMixed}. We define $\nu =-\left( \varphi _{\ast }-\lambda \mathcal{L}%
\left( \varphi _{\ast }\right) -g\right) .$ Due to Definition \ref%
{SubSuperMixed} we have that $\nu $ is a nonnegative Radon measure. If $\nu
\left( X\diagdown \left\{ 0,0\right\} \right) >0$ we can argue as in the
Proof of Propositions \ref{Case_a_sub}, \ref{Case_r} to derive a
contradiction. Therefore $\nu \left( X\diagdown \left\{ 0,0\right\} \right)
=0$ and then, $\varphi _{\ast }$ satisfies the problem:%
\begin{equation*}
\varphi _{\ast }-\lambda \mathcal{L}\left( \varphi _{\ast }\right) -g=0\ ,\
\left( x,v\right) \in X\diagdown \left\{ 0,0\right\} \ ,\ \varphi _{\ast
}\in C\left( X\right)
\end{equation*}

We now remark that Theorem \ref{AsSingPoint} implies:%
\begin{equation*}
\varphi _{\ast }\left( x,v\right) =\varphi _{\ast }\left( 0,0\right) +%
\mathcal{A}\left( \varphi _{\ast }\right) F_{\beta }\left( x,v\right) +\psi
_{\ast }\left( x,v\right)
\end{equation*}

Using Lemma \ref{TraceAphiMixed} and Lemma \ref{ConstStar} we obtain that $%
\varphi _{\ast }\left( 0,0\right) -g\left( 0,0\right) \leq \lambda \mu
_{\ast }\left\vert C_{\ast }\right\vert \mathcal{A}\left( \varphi _{\ast
}\right) .$ Suppose that $\varphi _{\ast }\left( 0,0\right) -g\left(
0,0\right) <\lambda \mu _{\ast }\left\vert C_{\ast }\right\vert \mathcal{A}%
\left( \varphi _{\ast }\right) .$ We can obtain a larger subsolution in the
family $\mathcal{\tilde{G}}_{sub}$ as follows. If $\varphi _{\ast }\left(
0,0\right) =\varphi ^{\sup }$ we would obtain that $\varphi _{\ast \ast
}=\max \left\{ \varphi _{\ast },\varphi _{\ast }\left( 0,0\right) +\delta
+CS\left( x,v\right) +\mathcal{A}\left( \varphi _{\ast }\right) \left(
1-\delta \right) F_{\beta }\right\} $ with $S\left( x,v\right) $ as in Lemma %
\ref{super_sol} is a subsolution of (\ref{T6E6}),\ (\ref{T6E9}) in a
neighbourhood of the singular point satisfying $\varphi _{\ast \ast
}>\varphi _{\ast },$ for a constant $C$ depending only of $g$ if $\delta $
is sufficiently small, since $S$ is bounded by $\left( x^{\frac{2}{3}%
}+v^{2}\right) $ and therefore $\mathcal{A}\left( \varphi _{\ast }\right)
F_{\beta }$ gives a larger contribution for $\left( x,v\right) $ close to $%
\left( 0,0\right) .$However this contradicts the definition of $\varphi
_{\ast },$ whence $\varphi _{\ast }\left( 0,0\right) -g\left( 0,0\right)
\leq \lambda \mu _{\ast }\left\vert C_{\ast }\right\vert \mathcal{A}\left(
\varphi _{\ast }\right) .$ Then $\mathcal{L}\left( \varphi _{\ast }\right)
\left( 0,0\right) =\mu _{\ast }\left\vert C_{\ast }\right\vert \mathcal{A}%
\left( \varphi _{\ast }\right) ,$ whence $\varphi _{\ast }\in \mathcal{D}%
\left( \Omega _{pt,sub}\right) $ and the result follows.
\end{proof}

\subsection{Operator $\Omega _{\sup }.$ Solvability of (\protect\ref{T6E6}), (%
\protect\ref{T6E10}).}

In the supercritical case $r>r_{c}$ we cannot impose a boundary condition at
the singular point. However, we can use also Perron's method in this case.
We will include a reference to (\ref{T6E10}) in order to make clear that we
refer to the supercritical case. We will prove the following result:

\begin{proposition}
\label{Case_super}For any $g\in C\left( X\right) $ there exists a unique $%
\varphi \in C\left( X\right) $ which solves (\ref{T6E6}), (\ref{T6E10}).
\end{proposition}

The concept of sub and supersolutions which we will use in this case is the
following:

\begin{definition}
\label{SubSuperSuper}Suppose that $g\in C\left( X\right) .$ We will say that
a function $\varphi \in L_{b}^{\infty }\left( X\right) ,$ is a subsolution
of (\ref{T6E6}), (\ref{T6E10}) if for all $\psi \in C^{\infty }\left(
\overline{\mathcal{U}}\right) $ with $\psi \geq 0,\ \psi \left( 0,-v\right)
=r^{2}\psi \left( 0,rv\right) ,\ v>0$ and with $\psi $ supported in a set
contained in the ball $\left\vert \left( x,v\right) \right\vert \leq R$ for
some $R>0$ we have:%
\begin{equation}
\int \left( \psi -\lambda \mathcal{L}^{\ast }\left( \psi \right) \right)
\varphi \leq \int g\psi \   \label{Z1E4}
\end{equation}
Given $g\in C\left( X\right) ,$ we will say that $\varphi \in L_{b}^{\infty
}\left( X\right) $ is a supersolution of (\ref{T6E6}), (\ref{T6E10}) if for
all $\psi $ with the same properties as above\ we have:%
\begin{equation}
\int \left( \psi -\lambda \mathcal{L}^{\ast }\left( \psi \right) \right)
\varphi \geq \int g\psi  \label{Z1E5}
\end{equation}
\end{definition}

\begin{remark}
Notice that the assumption $\psi \in C^{\infty }\left( \overline{\mathcal{U}}%
\right) $ does not imply any condition in the values of the derivatives of $%
\psi $ at $\left\{ x=0\right\} ,$ except those which follow from the
boundary condition $\psi \left( 0,-v\right) =r^{2}\psi \left(
0,rv\right) ,\ v>0.$
\end{remark}

We can now construct easily bounded sub and supersolutions in the sense of
Definition \ref{SubSuperSuper}:

\begin{lemma}
\label{ExSubSupSuper}For any $g\in C\left( X\right) $ there exist at least
one subsolution $\varphi ^{\mathrm{sub}}$ and one supersolution $\varphi
^{\sup }$ in the sense of Definition \ref{SubSuperSuper} such that:%
\begin{equation}
\varphi ^{\mathrm{sub}}\leq \varphi ^{\sup }  \label{Z1E6}
\end{equation}
\end{lemma}

\begin{proof}
We can just take $\varphi ^{\mathrm{sub}}=-\left\Vert g\right\Vert
_{L^{\infty }\left( X\right) },\ \varphi ^{\sup }=\left\Vert g\right\Vert
_{L^{\infty }\left( X\right) }.$ The proof can be made then by means of a
small adaptation of the one of Lemma \ref{ExSubSup2}.
\end{proof}

We can then define $\mathcal{\tilde{G}}_{sub}$ as in Definition the
following subset of $L_{b}^{\infty }\left( X\right) :$

\begin{definition}
\label{SubSetSuper}Suppose that $\varphi ^{\mathrm{sub}},\ \varphi ^{\sup }$
are respectively one subsolution and one supersolution in the sense of
Definition \ref{SubSuperSuper} satisfying (\ref{Z1E6}). We then define $%
\mathcal{\tilde{G}}_{sub}\subset L_{b}^{\infty }\left( X\right) $ as:%
\begin{equation}
\mathcal{\tilde{G}}_{sub}\equiv \left\{ \varphi \in L_{b}^{\infty }\left(
X\right) :%
\begin{array}{c}
\varphi \text{ sub-solution of (\ref{T6E6}),\ (\ref{T6E10})} \\
\text{ in the sense of Definition \ref{SubSuperSuper}}\ |\ \varphi ^{%
\mathrm{sub}}\leq \varphi \leq \varphi ^{\sup }%
\end{array}%
\right\} .  \label{Z1E7}
\end{equation}
\end{definition}

In this case we do not need to study the subsolutions near the singular
point with the same level of detail as in the subcritical cases. However, we
will need to be able to solve the Dirichlet problem with boundary values in
the admissible boundaries for domains containing the singular point. The
following result generalizes Proposition \ref{DirSolv} to the case of
domains containing the singular point if $r>r_{c}:$

\begin{proposition}
\label{DirSolvSuper}Suppose that $r>r_{c}$ and let $\Lambda _{R}$ as in (\ref%
{LambdaSing}) for some $R>0$. Let us assume also that $g\in C\left(
\overline{\Lambda _{R}}\right) ,$ and $h\in L^{\infty }\left( \partial
_{a}\Lambda _{R}\right) $ where $\partial _{a}\Lambda _{R}$ is the
corresponding admissible boundary of $\Lambda _{R}.$ Then, there exists a
unique classical solution of the problem (\ref{solv_D1}), (\ref{solv_D2})
where the boundary condition (\ref{solv_D2}) is achieved in the sense of
trace defined in Proposition \ref{Traces}.
\end{proposition}

\begin{proof}
The existence of solutions can be obtained considering the sequence of
solutions of the problems (\ref{solv_D1}), (\ref{solv_D2}) in the domains $%
\Lambda _{R}\diagdown \mathcal{R}_{\delta }$ with $\delta >0$ small and
boundary conditions $\varphi =h~\ $on $\partial _{a}\Lambda _{R}$ and, say, $%
\varphi =0$ on $\partial _{a}\left( \Lambda _{R}\diagdown \mathcal{R}%
_{\delta }\right) \cap \partial \mathcal{R}_{\delta }.$ The solutions of
these problems, which will be denoted as $\varphi _{\delta }$ are uniformly
bounded by $\left\Vert g\right\Vert _{L^{\infty }}+\left\Vert h\right\Vert
_{L^{\infty }}$ due to the maximum principle. Moreover, the derivatives of
the functions $\varphi _{\delta }$ are bounded in $W_{loc}^{1,p}$ due to
Theorem \ref{Hypoell}.$\ $Therefore, a standard compactness argument allows
to find a subsequence $\varphi _{\delta _{n}}$ converging uniformly to a
weak solution of (\ref{solv_D1}), (\ref{solv_D2}).

In order to prove uniqueness we consider the difference $\psi $ of two
bounded solutions of (\ref{solv_D1}), (\ref{solv_D2}) in $\Lambda _{R}$
satisfies $\left( \psi -\lambda \mathcal{L}\psi \right) =0$ in $\Lambda
_{R}, $ $\psi =0$ on $\partial _{a}\Xi .$ We can construct now a positive
supersolution with the form $\varepsilon F_{\beta },$ $\varepsilon >0,$
since $\varepsilon \left( F_{\beta }-\lambda \mathcal{L}F_{\beta }\right)
=\varepsilon F_{\beta }>0.$ Since $\psi $ is bounded, and $\beta <0$ for $%
r>r_{c}$ we obtain that $\left\vert \psi \right\vert \leq \varepsilon
F_{\beta }$ if we choose $\delta $ small enough (depending on $\varepsilon $%
). We can then use a comparison argument in the domain $\Lambda
_{R}\diagdown \mathcal{R}_{\delta }$ whence $\left\vert \psi \right\vert
\leq \varepsilon F_{\beta }$ in $\Lambda _{R}$ for $\varepsilon >0$
arbitrary. Taking the limit $\varepsilon \rightarrow 0$ we obtain $\psi =0$
and the uniqueness result follows.
\end{proof}

\begin{proof}[End of the Proof of Proposition \protect\ref{Case_super}]
We define $\varphi _{\ast }$ as in (\ref{T7E2}) with $\mathcal{\tilde{G}}%
_{sub}$ as in Definition \ref{SubSetSuper}. Arguing as in the Proof of
Propositions \ref{Case_a_sub}, \ref{Case_r}, \ref{Case_m} we obtain that $%
\varphi _{\ast }\in \mathcal{\tilde{G}}_{sub}.$We define $\nu =-\left(
\varphi _{\ast }-\lambda \mathcal{L}\left( \varphi _{\ast }\right) -g\right)
.$ Due to Definition \ref{SubSetSuper} we have that $\nu $ is a nonnegative
Radon measure. If $\nu \neq 0$ we consider a domain $\Xi $ that is, either
one of the admissible domains $\Xi $ in Definition \ref{admissible} or one
of the domains $\Lambda _{R}$ containing the singular point as in the
statement of Proposition \ref{DirSolvSuper}. We can then argue as in the
Proof of Propositions \ref{Case_a_sub}, \ref{Case_r}, \ref{Case_m} to derive
a contradiction. Therefore $\nu =0$ and then, $\varphi _{\ast }\in C\left(
X\right) $ gives the desired solution of the problem (\ref{T6E6}), (\ref%
{T6E10}).
\end{proof}

\subsection{Solvability of the adjoint problems (\protect\ref{T2E8})-(\protect
\ref{T2E10}), (\protect\ref{T3E1})-(\protect\ref{T3E4}), (\protect\ref{T3E5}%
)-(\protect\ref{T3E9}), (\protect\ref{T2E3})-(\protect\ref{T2E6}).}

We can now prove Theorem \ref{AAW}.

\begin{proof}[Proof of Theorem \ref{AAW}]
Theorem \ref{AAW} is a consequence of the Hille-Yosida Theorem (cf. Theorem %
\ref{HY}), as well as Proposition \ref{PropPG}, Proposition \ref{Closed} and
Propositions \ref{Case_a_sub}, \ref{Case_r}, \ref{Case_m} and \ref%
{Case_super}.
\end{proof}


%
%
%

\section{Weak solutions for the original problem.\label{weakSolDef}}

Our next goal is to define suitable measure valued solutions of the problem
(\ref{S0E1})-(\ref{S0E3}) by means of the corresponding adjoint problems. 
To this end, we argue by duality. We will use the index $%
\sigma $ to denote each of the four cases considered in Subsection \ref%
{Operators} of Section \ref{AdjProblems}, namely, in the case of subcritical
values of $r,$ we can use trapping, nontrapping or partially trapping
boundary conditions. We will consider also the supercritical case. The
following definition will be used to define measure valued solutions of the
problem (\ref{S0E1})-(\ref{S0E3}) with all the boundary conditions
considered above.

\begin{definition}
\label{weakSol}Given $P_{0}\in \mathcal{M}_{+}\left( X\right) $ we define a
measure valued function $P_{\sigma }\in C\left( \left[ 0,\infty \right) :%
\mathcal{M}_{+}\left( X\right) \right) $ by means of:%
\begin{equation}
\int \varphi \left( dP_{\sigma }\left( \cdot ,t\right) \right) =\int \left(
S_{\sigma }\left( t\right) \varphi \right) \left( dP_{0}\right) \ \ ,\ \
t\geq 0  \label{A3}
\end{equation}%
for any $\varphi \in C\left( \left[ 0,\infty \right) :C\left( X\right)
\right) $
\end{definition}

\begin{remark}
Notice that the notation in (\ref{A3}) must be understood as follows. Let $%
\psi \left( \cdot ,t\right) =S_{\sigma }\left( t\right) \varphi .$ Then, the
right-hand side of (\ref{A3}) is equivalent to $\int \psi \left( \cdot
,t\right) dP_{0}.$ The left-hand side of (\ref{A3}) is just $\int \varphi
\left( \cdot \right) dP_{\sigma }\left( t\right) .$
\end{remark}

It is convenient to write in detail the weak formulation of (\ref{S0E1})-(%
\ref{S0E3}) satisfied by each of the measures $P_{\sigma }.$

\begin{definition}
\label{WAC}Suppose that $0<r<r_{c}$ and $P_{0}\in \mathcal{M}_{+}\left(
X\right) .$ We will say that $P\in C\left( \left[ 0,\infty \right) :\mathcal{%
M}_{+}\left( X\right) \right) $ is a weak solution of (\ref{S0E1})-(\ref%
{S0E3}) with trapping boundary conditions if for any $T\in \left[ 0,\infty
\right) $ and $\varphi \ $such that, for any time $t\in \left[ 0,T\right] ,$
$\varphi \left( \cdot ,t\right) ,\mathcal{L}\varphi \left( \cdot ,t\right)
,\varphi _{t}\left( \cdot ,t\right) \in C\left( X\right) ,\ \varphi \left(
\cdot ,t\right) $ satisfies (\ref{phi_decom})\ and such that $\mathcal{L}%
\varphi \left( 0,0,t\right) =0$ for any $t\geq 0,$ the following identity
holds:%
\begin{equation}
\int_{0}^{T}\int_{X\diagdown \left\{ \left( 0,0\right) \right\} }\left[
\varphi _{t}\left( dP\left( \cdot ,t\right) ,t\right) +\mathcal{L}\varphi
\left( dP\left( \cdot ,t\right) ,t\right) \right] +\int_{X}\varphi \left(
dP_{0}\left( \cdot \right) ,0\right) -\int_{X}\varphi \left( dP\left( \cdot
,T\right) ,T\right) =0,  \label{S7E5}
\end{equation}%
\ where $\mathcal{L}$ is as in (\ref{diffOp}).
\end{definition}

\begin{definition}
\label{WRC}Suppose that $0<r<r_{c}$ and $P_{0}\in \mathcal{M}_{+}\left(
X\right) .$ We will say that $P\in C\left( \left[ 0,\infty \right) :\mathcal{%
M}_{+}\left( X\right) \right) $ is a weak solution of (\ref{S0E1})-(\ref%
{S0E3}) with nontrapping boundary conditions if for any $T\in \left[
0,\infty \right) $ and $\varphi $ such that, for any time $t\in \left[ 0,T%
\right] ,$ $\varphi \left( \cdot ,t\right) ,\mathcal{L}\varphi \left( \cdot
,t\right) ,\varphi _{t}\left( \cdot ,t\right) \in C\left( X\right) ,\
\varphi \left( \cdot ,t\right) $ satisfies (\ref{phi_decom})\ and\ $\mathcal{%
A}\left( \varphi \right) \left( \cdot ,t\right) =0$ the identity (\ref{S7E5})
holds,\ where $\mathcal{L}$ is as in (\ref{diffOp}).
\end{definition}

\begin{definition}
\label{WMC}Suppose that $0<r<r_{c}$ and $P_{0}\in \mathcal{M}_{+}\left(
X\right) .$ We will say that $P\in C\left( \left[ 0,\infty \right) :\mathcal{%
M}_{+}\left( X\right) \right) $ is a weak solution of (\ref{S0E1})-(\ref%
{S0E3}) with partially trapping boundary conditions if for any $T\in \left[
0,\infty \right) $ and and $\varphi $ such that, for any time $t\in \left[
0,T\right] ,$ $\varphi \left( \cdot ,t\right) ,\mathcal{L}\varphi \left(
\cdot ,t\right) ,\varphi _{t}\left( \cdot ,t\right) \in C\left( X\right) ,\
\varphi \left( \cdot ,t\right) $ satisfies (\ref{phi_decom})\ and\ $\mathcal{%
L}\varphi \left( 0,0,t\right) =\mu _{\ast }\left\vert C_{\ast }\right\vert
\mathcal{A}\left( \varphi \right) ,$ the identity (\ref{S7E5}) holds, where $%
\mathcal{L}$ is as in (\ref{diffOp}).
\end{definition}

\begin{definition}
\label{WSuper}Suppose that $r>r_{c}$ and $P_{0}\in \mathcal{M}_{+}\left(
X\right) .$ We will say that $P\in C\left( \left[ 0,\infty \right) :\mathcal{%
M}_{+}\left( X\right) \right) $ is a weak solution of (\ref{S0E1})-(\ref%
{S0E3}) for supercritical boundary conditions if for any $T\in \left[
0,\infty \right) $ and $\varphi \in C_{c}^{2}\left( \left[ 0,T\right) \times
X\right) $ the identity (\ref{S7E5}) holds, where $\mathcal{L}$ is as in (%
\ref{diffOp}).
\end{definition}

Our next goal is to prove the following result:

\begin{proposition}
\label{CharMeas}Suppose that we define measure valued functions
\begin{equation*}
P_{t,sub},\ P_{nt,sub},\ P_{pt,sub},\ P_{sup}\in C\left( \left[ 0,\infty
\right) :\mathcal{M}_{+}\left( X\right) \right)
\end{equation*}
as in Definition \ref{weakSol} and initial datum $P_{0}$. Then, they are
weak solutions of (\ref{S0E1})-(\ref{S0E3})\ with trapping, nontrapping,
partially trapping and supercritical boundary conditions respectively.
Moreover, the functions $P_{t,sub},\ P_{nt,sub},\ P_{pt,sub},\ P_{sup}$ are
the unique solutions of (\ref{S0E1})-(\ref{S0E3})\ with trapping,
nontrapping, partially trapping and supercritical boundary conditions in the
sense of Definitions \ref{WAC}, \ref{WRC}, \ref{WMC} and \ref{WSuper}
respectively.
\end{proposition}

\begin{proof}
We take a smooth test function $\varphi $ depending on the variables $\left(
x,v,t\right) $. We will assume that the function $\varphi \left( \cdot
,t\right) \in \mathcal{D}\left( \Omega _{\sigma }\right) $ for any $t\in %
\left[ 0,T\right] .$ The definition of the Probability measure $P$ yields:%
\begin{equation*}
\int \varphi \left( dP_{\sigma }\left( \cdot ,t\right) ,\bar{t}\right) =\int
\left( S_{\sigma }\left( t\right) \varphi \right) \left( dP_{0},\bar{t}%
\right)
\end{equation*}%
for any fixed $\bar{t}>0,$ $t>0.$ In particular, taking for $\bar{t}$ the
values $\left( t+h\right) $ with $h$ small, we obtain:%
\begin{equation*}
\int \varphi \left( dP_{\sigma }\left( \cdot ,t\right) ,t+h\right) =\int
\left( S_{\sigma }\left( t\right) \varphi \right) \left( dP_{0},t+h\right)
\end{equation*}%
Differentiating this equation with respect to $h$ and taking $h=0,$ it then
follows that:%
\begin{equation}
\int \varphi _{t}\left( dP_{\sigma }\left( \cdot ,t\right) ,t\right) =\int
\left( S_{\sigma }\left( t\right) \varphi _{t}\right) \left( dP_{0},t\right)
\label{A1}
\end{equation}%
For a regularity of $S_{\sigma }$ and a general semigroup theory, see \cite{P}.
We now write $S_{\sigma }\left( t\right) \varphi _{t}\left( \cdot ,t\right) =%
\frac{d}{dt}\left( S_{\sigma }\left( t\right) \varphi \left( \cdot ,t\right)
\right) -\partial _{t}\left( S_{\sigma }\left( t\right) \right) \varphi
\left( \cdot ,t\right) .$ Integrating then (\ref{A1}) in the time interval $%
\left( 0,T\right) $ it then follows that:%
\begin{eqnarray*}
\int_{0}^{T}\int \varphi _{t}\left( dP_{\sigma }\left( \cdot ,t\right)
,t\right) dt &=&\int_{0}^{T}\int \left[ \frac{d}{dt}\left( S_{\sigma }\left(
t\right) \varphi \left( dP_{0},t\right) \right) -\partial _{t}\left(
S_{\sigma }\left( t\right) \right) \varphi \left( dP_{0},t\right) \right] dt
\\
&=&\int_{0}^{T}\left[ \frac{d}{dt}\left( \int \left( S_{\sigma }\left(
t\right) \varphi \left( dP_{0},t\right) \right) \right) -\int \partial
_{t}\left( S_{\sigma }\left( t\right) \right) \varphi \left( dP_{0},t\right) %
\right] dt \\
&=&\int S_{\sigma }\left( T\right) \varphi \left( dP_{0},T\right) -\int
\left( S_{\sigma }\left( 0\right) \varphi \left( dP_{0},0\right) \right)  \\
&&-\int_{0}^{T}\int \partial _{t}\left( S_{\sigma }\left( t\right) \right)
\varphi \left( dP_{0},t\right) dt \\
&=&\int \varphi \left( dP_{\sigma }\left( \cdot ,T\right) ,T\right) -\int
\varphi \left( dP_{0},0\right) -\int_{0}^{T}\int \Omega _{\sigma }S_{\sigma
}\left( t\right) \varphi \left( dP_{0},t\right) dt \\
&=&\int \varphi \left( dP_{\sigma }\left( \cdot ,T\right) ,T\right) -\int
\varphi \left( dP_{0},0\right) -\int_{0}^{T}\int \Omega _{\sigma }\varphi
\left( dP_{\sigma }\left( \cdot ,t\right) ,t\right) dt
\end{eqnarray*}%
where we have used that for any function $\psi \in \mathcal{D}\left( \Omega
_{\sigma }\right) $ we have $\partial _{t}\left( S_{\sigma }\left( t\right)
\right) \psi =\Omega _{\sigma }S_{\sigma }\left( t\right) \psi ,$ whence:%
\begin{eqnarray*}
&&\int \varphi \left( dP_{\sigma }\left( \cdot ,T\right) ,T\right) -\int
\varphi \left( dP_{0},0\right)  \\
&=&\int_{0}^{T}\int \left[ \varphi _{t}\left( dP_{\sigma }\left( \cdot
,t\right) ,t\right) +\Omega _{\sigma }\varphi \left( dP_{\sigma }\left(
\cdot ,t\right) ,t\right) \right] dt
\end{eqnarray*}

Using that the action of the operators $\Omega _{\sigma }$ in their
respective domains is given by $\mathcal{L}$ we obtain the existence
result.\

Uniqueness can be proved using the solvability of the adjoint problems
obtained in \cite{HVJ2}. Suppose that $P_{1},P_{2}$ are two solutions
of (\ref{S0E1})-(\ref{S0E3})\ with trapping boundary conditions in the sense
of Definition \ref{WAC}. Then $P=P_{1}-P_{2}$ satisfies (\ref{S7E5}) with $%
P_{0}=0$ for any function $\varphi $ with the regularity requested in the
Definition \ref{WAC}. Suppose that $P$ does not vanish identically. Then,
for some $T>0$ there exists a function $\tilde{\psi}\in C\left( X\right) $
such that $\int_{X}\tilde{\psi}\left( dP\left( \cdot ,T\right) \right) \neq
0.$ Since $\mathcal{D}\left( \Omega _{t}\right) $ is dense in $C\left(
X\right) $ there exists $\psi \in \mathcal{D}\left( \Omega _{t}\right) $
such that $\int_{X}\psi \left( dP\left( \cdot ,T\right) \right) \neq 0.$
The main result in \cite{HVJ2} implies the existence of a function $\varphi $ $%
C^{1}\left( \left[ 0,T\right] ;C\left( X\right) \right) $ satisfying $%
\varphi _{t}\left( \cdot ,t\right) +\mathcal{L}\varphi \left( \cdot
,t\right) =0$ in $X,$ for $t\in \left[ 0,T\right] $ and $\mathcal{L}\varphi
\left( 0,0,t\right) =0$ for any $t\in \left[ 0,T\right] ,$ with the initial
condition $\varphi \left( \cdot ,T\right) =\psi \left( \cdot \right) .$ This
function satisfies the regularity conditions required in Definition \ref{WAC}%
. It then follows from (\ref{S7E5}) that $\int_{X}\psi \left( dP\left( \cdot
,T\right) \right) =0,$ but this gives a contradiction, whence the uniqueness
result follows for trapping boundary conditions. The proof in the case of
non-trapping, partially trapping and supercritical boundary conditions is
similar and the details will be omitted.
\end{proof}

We can also prove the following result.

\begin{theorem}
Suppose that we define measures $P_{t,sub},\ P_{nt,sub},\ P_{pt,sub},\
P_{sup}\in C\left( \left[ 0,\infty \right) :\mathcal{M}_{+}\left( X\right)
\right) $ as in Definition \ref{weakSol} and initial datum $P_{0}$. Then,
these measures can be decomposed as:%
\begin{equation}
dP_{\sigma }\left( \cdot ,t\right) =m_{\sigma }\left( t\right) \delta
_{\left( 0,0\right) }+p_{\sigma }\left( x,v,t\right) dxdv\ \ ,\ \ \ t\geq 0\
\label{A4}
\end{equation}%
where for each $\sigma $ we have $p_{\sigma }\left( \cdot ,t\right) \in
L^{1}\left( X\right) $ and $m_{\sigma }\left( t\right) \geq 0.$ If $%
m_{\sigma }\left( 0\right) =0$ we have also $m_{nt,sub}\left( t\right)
=m_{sup}\left( t\right) =0$ for any $t\geq 0.$ Moreover, if $P_{0}$ is not
identically zero, we have also $m_{t,sub}\left( t\right) >0,\
m_{pt,sub}\left( t\right) >0$ for any $t>0.$

The functions $p_{\sigma }\left( \cdot ,t\right) $ are infinitely
differentiable for $\left( x,v\right) \neq \left( 0,0\right) $ and they
satisfy (\ref{S0E2}).
\end{theorem}

\begin{proof}
Using in (\ref{S7E5}) the test function $\varphi =1$ we obtain that, in the
case of the four considered operators:%
\begin{equation*}
\int_{X}dP_{\sigma }\left( \cdot ,t\right) =\int_{X}dP_{0}\ \ ,\ \ t\geq 0
\end{equation*}

We define $m_{\sigma }\left( t\right) =\int_{\left\{ \left( 0,0\right)
\right\} }dP_{\sigma }\left( \cdot ,t\right) .$ We define also the measures $%
dP_{\sigma }\left( \cdot ,t\right) -m_{\sigma }\left( t\right) \delta
_{\left( 0,0\right) }.$ Using Proposition \ref{CharMeas} it then follows
that this family of measures solves (\ref{S0E1}), (\ref{S0E2}) outside the
singular point $\left( x,v\right) =\left( 0,0\right) .$ Using classical
interior hypoellipticity results (\cite{HBook}, \cite{H}), we obtain the representation (\ref{A4})
where $p_{\sigma }\in C^{\infty }$ in the set $\left\{ x>0\right\} $. On the
other hand, we can prove that $p_{\sigma }\in C^{\infty }$ for $x=0,\ v<0$
using the fact that the characteristic curves move away from the domain $%
\left\{ x>0\right\} $ in that region. Actually, it is possible to argue
basically as in the proof of boundary regularity derived in \cite{HVJ}. More
precisely, introducing a cutoff function supported in the region $\left\{
v<0\right\} $ which takes the value $1$ in the region where the regularity
of $p$ can be obtained, we would obtain a new function $\tilde{p}$ solving
the same Fokker-Planck equation, but with some source terms containing the
derivative $p_{v}$ and the function $p.$ The arguments in \cite{HVJ} then
allow to prove that $p\in C^{\infty }\left\{ x\geq 0,\ v<0\right\} .$ Using
then (\ref{S0E2}) we would obtain also regularity in $\left\{ x=0,\
v>0\right\} $. Uniform regularity in compact sets of $\left\{ x=0,\
v>0\right\} $ would then follow as in \cite{HVJ}.

It then follows, since $dP_{\sigma }\left( \cdot ,t\right) $ is a Radon
measure, and therefore, outer regular that $p_{\sigma }\in L^{\infty }\left(
\left( 0,T\right) :L^{1}\left( X\right) \right) .$

In order to obtain the stated results for the masses $m_{\sigma }\left(
t\right) $ we need to study the asymptotic behaviour of some solutions of
the adjoint problems with initial data approaching the characteristic
function of the singular point. The specific form of the solutions under
consideration depends on the specific boundary conditions used.

In the cases of trapping and partially trapping boundary conditions we
construct some particular test functions with the form:%
\begin{equation*}
\varphi \left( x,v,t\right) =\Phi \left( \frac{x}{\left( \bar{t}-t\right) ^{%
\frac{3}{2}}},\frac{v}{\left( \bar{t}-t\right) ^{\frac{1}{2}}}\right) =\Phi
\left( \xi ,\eta \right)
\end{equation*}%
where $\Phi $ satisfies:%
\begin{equation*}
-\frac{3}{2}\xi \partial _{\xi }\Phi -\frac{1}{2}\eta \partial _{\eta }\Phi
\leq \partial _{\eta \eta }\Phi +\eta \partial _{\xi }\Phi
\end{equation*}%
in the sense of measures. Notice that this is equivalent to obtaining the
inequalities $\varphi _{t}\left( \cdot ,t\right) +\mathcal{L}\varphi \left(
\cdot ,t\right) \geq 0$. In the case of trapping boundary conditions we look
for functions with the form:%
\begin{equation*}
\Phi \left( \xi ,\eta \right) =a-F_{\beta }\left( \xi ,\eta \right) -W\left(
\xi ,\eta \right)
\end{equation*}%
where we choose $W$ satisfying:%
\begin{equation*}
\partial _{\eta \eta }W+\eta \partial _{\xi }W\leq C_{0}\left( \xi
\left\vert \partial _{\xi }F_{\beta }\right\vert +\left\vert \eta
\right\vert \left\vert \partial _{\eta }F_{\beta }\right\vert \right)
\end{equation*}%
and the function $W$ satisfies $\left\vert W\left( \xi ,\eta \right)
\right\vert \leq C\left( \left\vert \eta \right\vert ^{2}+\left\vert \xi
\right\vert ^{\frac{2}{3}}\right) F_{\beta }.$ The existence of the function
$W$ and the previous estimate can be proved as Lemma \ref{super_sol}.
Choosing $a>0$ sufficiently small we obtain that $\left\vert W\left( \xi
,\eta \right) \right\vert <F_{\beta }\left( \xi ,\eta \right) $ in the
neighbourhood of the singular point where $\Phi >0.$ Moreover, by
construction we obtain $\mathcal{L}\left( \Phi \right) \left( 0,0\right) =0$
due to the fact that $\mathcal{L}\left( F_{\beta }\right) =0.$ We then
define a test function $\tilde{\Phi}$ as:%
\begin{equation*}
\tilde{\Phi}=\frac{1}{a}\max \left\{ \Phi ,0\right\}
\end{equation*}

Then $\tilde{\Phi}\left( 0,0\right) =1$. Using the duality formula (\ref{A3}%
) we then obtain:%
\begin{equation}
\int_{\mathcal{R}_{\delta }}P_{\sigma }\left( dxdv,\bar{t}\right) \geq \int_{%
\mathcal{R}_{\rho }}P_{\sigma }\left( dxdv,\frac{\bar{t}}{2}\right) \ \ ,\ \
\bar{t}>0\   \label{A6}
\end{equation}%
where $\rho >0$ is small but it can be chosen independently of $\delta ,$
and $\delta >0$ can be taken arbitrarily small. The integral on the right
can be estimated uniformly from below, because $P$ is strictly positive in a
set with the form $\mathcal{R}_{\rho }\diagdown \mathcal{R}_{\frac{\rho }{2}%
}\cap \left\{ x>0\right\} $ for any fixed $\bar{t}>0$ due to the strong
maximum principle. This type of strong maximum principle arguments in
interior domains have been used also in \cite{HVJ}. Taking the limit $\delta
\rightarrow 0$ we obtain $\int_{\left\{ \left( 0,0\right) \right\}
}P_{\sigma }\left( dxdv,\bar{t}\right) >0$ for any $\bar{t}>0.$ This gives $%
m_{t,sub}\left( \bar{t}\right) >0.$

In the case of partially trapping boundary conditions we will look for test
functions satisfying $\varphi _{t}\left( \cdot ,t\right) +\mathcal{L}\varphi
\left( \cdot ,t\right) \geq 0$ in the sense of measures, constructed by
means of the auxiliary function:
\begin{equation*}
\varphi \left( x,v,t\right) =a-F_{\beta }\left( \xi ,\eta \right) -\left(
\bar{t}-t\right) ^{1-\frac{3\beta }{2}}W\left( \xi ,\eta \right)
\end{equation*}%
where $\left( \xi ,\eta \right) =\left( \frac{x}{\left( \bar{t}-t\right) ^{%
\frac{3}{2}}},\frac{v}{\left( \bar{t}-t\right) ^{\frac{1}{2}}}\right) .$
Notice that for this function we have $\mathcal{A}\left( \varphi \right)
=-\left( \bar{t}-t\right) ^{-\frac{3\beta }{2}}$ and $\mathcal{L}\left(
\varphi \right) \left( 0,0\right) =-$ $\left( \bar{t}-t\right) ^{-\frac{%
3\beta }{2}}\mathcal{L}\left( W\right) \left( 0,0\right) .$ Moreover, since $%
\beta <\frac{1}{6}$ we have $\left( 1-\frac{3\beta }{2}\right) >0.$ The
boundary condition for $\varphi $ then becomes:

\begin{equation}
\mathcal{L}\left( W\right) \left( 0,0\right) =\mu _{\ast }\left\vert C_{\ast
}\right\vert  \label{A5}
\end{equation}

Writing the equation we obtain that $\varphi $ satisfies $\varphi _{t}\left(
\cdot ,t\right) +\mathcal{L}\varphi \left( \cdot ,t\right) \geq 0$ if:\
\begin{equation*}
\mathcal{L}\left( W\right) \leq -C\left( \left\vert W\right\vert +F_{\beta
}\right)
\end{equation*}%
for a suitable constant $C>0$ and where we have to impose the boundary
condition (\ref{A5}). The resulting function $W$ is quadratic near the
singular point. Therefore, for small $\left\vert \left( \xi ,\eta \right)
\right\vert ,$ and given that $\left( 1-\frac{3\beta }{2}\right) >0$ it
follows that $F_{\beta }\left( \xi ,\eta \right) >>\left( \bar{t}-t\right)
^{1-\frac{3\beta }{2}}W\left( \xi ,\eta \right) .$ Then, if we choose $a>0$
small enough, \ and we define:%
\begin{equation*}
\tilde{\varphi}=\max \left( \varphi ,0\right)
\end{equation*}%
it follows that $\tilde{\varphi}$ is a test function, whose support expands
in a self-similar way and that it can be used as in the case of trapping
boundary conditions to prove that $m_{pt,sub}\left( t\right) >0$.

Notice that these test functions yield that $\int_{\left\{ 0\right\}
}dP_{\sigma }\left( \cdot ,\bar{t}\right) >0$ for any $\bar{t}>0,$ using (%
\ref{A6}), as well as the strong maximum principle which guarantees that $%
P\left( \cdot ,\frac{\bar{t}}{2}\right) >0$ in a neighbourhood of the origin
in the set $\left\{ x>0\right\} .$

In order to prove that $m_{nt,sub}\left( t\right) =m_{sup}\left( t\right) =0$
we need to obtain test functions satisfying $\varphi _{t}\left( \cdot
,t\right) +\mathcal{L}\varphi \left( \cdot ,t\right) \leq 0$ as well as $%
\mathcal{A}\left( \varphi \right) =0$ and taking initial values "close" to
the characteristic function of the singular point and approaching to zero
very fast.

It is natural to look for test functions which at least in some regions will
have the form:%
\begin{equation}
\varphi \left( x,v,t\right) =2\delta ^{\gamma }\left[ \left( \bar{t}%
-t\right) +\delta \right] ^{-\gamma }\Phi \left( \frac{x}{\left( \bar{t}%
-t+\delta \right) ^{\frac{3}{2}}},\frac{v}{\left( \bar{t}-t+\delta \right) ^{%
\frac{1}{2}}}\right) \   \label{A7a}
\end{equation}%
for $\delta >0$ fixed. Plugging this in the equation $\varphi _{t}+\mathcal{L%
}\left( \varphi \right) =0,$ and looking for functions satisfying $\varphi
_{t}\left( \cdot ,t\right) +\mathcal{L}\varphi \left( \cdot ,t\right) \leq 0$
we obtain the inequality:%
\begin{equation}
\gamma \Phi +\frac{3}{2}\xi \Phi _{\xi }+\frac{1}{2}\eta \Phi _{\eta }+%
\mathcal{L}\left( \Phi \right) \leq 0  \label{A7}
\end{equation}

We need to obtain a function $\Phi $ satisfying (\ref{A7}) in the sense of
measures. We consider first the case $r<r_{c}$ with non-trapping boundary
conditions. To this end we consider a function $\Phi $ smooth outside the
origin, and homogeneous function, say $\Phi _{1}$ satisfying $\frac{3}{2}\xi
\Phi _{\xi }+\frac{1}{2}\eta \Phi _{\eta }=-a\Phi $ with $a>0$ fixed,
independent of $\gamma .$ Notice that we can assume that $\Phi _{1}$ tends
to infinity as $\left( \xi ,\eta \right) $ approaches the singular point.
Then, the inequality (\ref{A7}) holds for $\left\vert \xi \right\vert
+\left\vert \eta \right\vert $ large, because the contribution of the term $%
\mathcal{L}\left( \Phi \right) $ becomes smaller than $-a\Phi $ for large $%
\left\vert \xi \right\vert +\left\vert \eta \right\vert $ and $\gamma \Phi $
gives a small contribution if $\gamma $ is small. On the other hand, in
order to obtain a bounded function for $\left\vert \xi \right\vert
+\left\vert \eta \right\vert $ of order one we construct a function $\Phi
_{2}$ which takes the value $1$ at the origin and is corrected by means of
one function $W$ satisfying $C\gamma +\mathcal{L}\left( W\right) \leq 0.$
This function can be constructed satisfying (\ref{A7}) in any bounded
region, as large as desired, if $\gamma $ is chosen small, because then the
contribution of the corrective term is small. Taking the minimum of $\Phi
_{1},\ \Phi _{2}$ we obtain the desired test function satisfying (\ref{A7})
in the sense of measures and the result follows.

In the supercritical case we can obtain a functions satisfying (\ref{A7}) as
follows. The function $F_{\beta }$ satisfies $\mathcal{L}\left( F_{\beta
}\right) =0$ and in the supercritical case $\beta <0.$ Then the contribution
of the terms $\frac{3}{2}\xi \Phi _{\xi }+\frac{1}{2}\eta \Phi _{\eta }$ is
negative close to the origin. Moreover, if we choose $\gamma $ sufficiently
small we obtain that $\gamma \Phi +\frac{3}{2}\xi \Phi _{\xi }+\frac{1}{2}%
\eta \Phi _{\eta }$ is negative. In order to obtain a bounded test function
we construct an auxiliary function, taking the value $\Phi \left( 0,0\right)
=1$ and corrected by a term $W$ satisfying $C\gamma +\mathcal{L}\left(
W\right) \leq 0$ in the usual manner. Taking then the minimum between both
functions we obtain a new bounded test function satisfying (\ref{A7}) in the
sense of measures .

Notice that in both cases (supercritical case and subcritical with
nontrapping boundary conditions), the function (\ref{A7a}) is larger than
one in a neighbourhood of the singular point for short times, but it
decreases to values of order $\delta $ in times of order, say $\sqrt{\delta }
$ if $\delta $ is small. This implies, taking the limit $\delta \rightarrow 0
$ that the mass at the singular point at any time $\bar{t}>0$ is zero, using
the duality formula $\int_{\left\{ \left( 0,0\right) \right\} }dP_{\sigma
}\left( \cdot ,\bar{t}\right) \varphi \left( \bar{t},\cdot \right) \leq
\int_{X}dP_{\sigma }\left( \cdot ,\frac{\bar{t}}{2}\right) \varphi \left(
\cdot ,\frac{\bar{t}}{2}\right) \leq \varepsilon _{0},$ where $\varepsilon
_{0}$ can be made arbitrarily small if $\delta $ is small.
\end{proof}


\noindent{\bf Acknowledgements.} 
The authors would like to thank the Hausdorff Center for Mathematical
Sciences of the University of Bonn and Pohang Mathematics Institute where
part of this work was done. HJH is partly supported by the Basic Science
Research Program (2015R1A2A2A0100251) through the National
Research Foundation of Korea (NRF). JJ is supported in part by NSF grants
DMS-1608492 and DMS-1608494. The authors acknowledge support through the CRC
1060 The mathematics of emergent effects at the University of Bonn, that is
funded through the German Science Foundation (DFG). We thank Seongwon Lee
for helping with figures in the paper.






\end{document}